\documentclass[11pt,reqno]{amsart}
\usepackage{amsfonts}
\usepackage{amsthm}
\usepackage{amsmath}
\usepackage{amssymb}
\usepackage{comment}
\usepackage{amscd}
\usepackage[noadjust]{cite}
\usepackage[ngerman,english]{babel}
\usepackage[mathscr]{eucal}
\usepackage{indentfirst}
\usepackage{graphicx}
\usepackage{graphics}
\usepackage{pict2e}
\usepackage{epic}
\numberwithin{equation}{subsection}
\usepackage{epstopdf} 
\usepackage{bbm}
\usepackage{xcolor}

\usepackage{lmodern}

\let\OLDthebibliography\thebibliography
\renewcommand\thebibliography[1]{
	\OLDthebibliography{#1}
	\setlength{\parskip}{0pt}
	\setlength{\itemsep}{8pt}
}

\usepackage{setspace}

\usepackage[hidelinks, hypertexnames = false]{hyperref}

\usepackage[most]{tcolorbox}

\tcbset{colback=white!10!white, colframe=black!50!black, 
	highlight math style= {enhanced, 
		colframe=black,colback=black!10!white,boxsep=0pt}
}

\usepackage{tikz}

\calclayout

\numberwithin{equation}{section}

\theoremstyle{plain}
\newtheorem{Th}{Theorem}[section]
\newtheorem{Lemma}[Th]{Lemma}
\newtheorem{Cor}[Th]{Corollary}
\newtheorem{Prop}[Th]{Proposition}
\DeclareMathOperator{\R}{\mathbb{R}}
\DeclareMathOperator{\Z}{\mathbb{Z}}
\DeclareMathOperator{\C}{\mathbb{C}}
\DeclareMathOperator{\N}{\mathbb{N}}

\DeclareMathOperator{\Ran}{\text{Ran}}

\newcommand{\norm}[2]{\left\lVert #1 \right\rVert_{#2}}

\newcommand{\f}[2]{\frac{#1}{#2}}

\newcommand{\leqnomode}{\tagsleft@true\let\veqno\@@leqno}
\newcommand{\reqnomode}{\tagsleft@false\let\veqno\@@eqno}

\theoremstyle{plain}
\newtheorem{Def}[Th]{Definition}
\newtheorem{As}[Th]{Assumption}

\newtheorem{Corollary}[Th]{Corollary}
\newtheorem{Rem}[Th]{Remark}
\newtheorem{?}[Th]{Problem}

\newcommand{\nocontentsline}[3]{}
\newcommand{\tocless}[2]{\bgroup\let\addcontentsline=\nocontentsline#1{#2}\egroup}

\newenvironment{psmallmatrix}
{\left(\begin{smallmatrix}}
	{\end{smallmatrix}\right)}

\newenvironment{mmatrix}
{\left(~\begin{matrix}}
	{\end{matrix}~\right)}

\newcommand{\smallbinom}[2]
{\left(\begin{smallmatrix}  #1\\[1pt] #2 \end{smallmatrix}\right)}

\renewcommand{\rho}{R}

\begin{document}


	\title[Blow up dynamics for energy critical NLS]{Blow up dynamics for the $3$D energy-critical\\ Nonlinear Schr\"odinger equation 
	}
	\author{Tobias Schmid}

	\address{EPFL SB MATH PDE,
		Bâtiment MA,
		Station 8,
		CH-1015 Lausanne}
	\email{tobias.schmid@epfl.ch}	

	\begin{abstract} 
		We construct a two-parameter continuum of type II  blow up solutions for the 
		energy-critical focusing NLS in dimension $ d = 3$. The solutions collapse to a single  energy bubble in finite time, precisely they have the form
		\[ u(t,x) = e^{i \alpha(t)}\lambda(t)^{\f{1}{2}}W(\lambda(t) x) + \eta(t, x ),~ ~ t \in[0, T),~ x \in \R^3,\]
		where $ W( x)  = \big( 1 + \f{|x|^2}{3}\big)^{-\f{1}{2}}$ is the ground state solution, $\lambda(t) = (T -t)^{- \f12 - \nu} $ for suitable $ \nu > 0 $, ~$  \alpha(t) = \alpha_0 \log(T - t)$ and $ T= T(\nu, \alpha_0) > 0 $.  Further  $ \|\eta(t) - \eta_T\|_{\dot{H}^1 \cap \dot{H}^2} = o(1)$
		as $ t \to T^-$ for some $ \eta_T \in \dot{H}^{1} \cap \dot{H}^2$.
	\end{abstract}

	\maketitle
	\tableofcontents
    
\section{Introduction \& Notation}
In this article we construct (finite-time) blow up solutions for the energy critical focusing nonlinear Schr\"odinger equation  
\begin{align} \tag{NLS} \label{NLSequation}
	\begin{cases}
		&i\partial_tu = - \Delta u - |u|^{\f{4}{d-2}}u~~~\text{in} ~~  (0,T) \times \R^{d},\\[2pt] \nonumber
		&u(0) = u_0 \in \dot{H}^1(\R^d)~~~\text{on} ~~ \R^{d}.
	\end{cases}
\end{align}
in  $ d = 3$ dimensions.\\[2pt]
\emph{Local wellposedness}. The Cauchy problem for the equation in \eqref{NLSequation} is locally wellposed  in $\dot{H}^1(\R^d)$ due to Cazenave-Weissler (see \cite{Cazenave-Wei}), i.e. for $ u_0 \in \dot{H}^1$ and $ t_0 \in \R$, there exists an  interval $ I \subset \R$,~ $ t_0 \in I^{\circ}$ and a weak solution $u \in C(I, \dot{H}^1(\R^d))$ of \eqref{NLSequation}
with (1) ~ $u(t_0) = u_0$ and (2)~uniqueness in a closed subspace $ X \subset C(I, \dot{H}^1(\R^d))$.  The local wellposedness theory for the Cauchy problem is also well presented in \cite{KM} and \cite{Killip-Visan}. %
The equation in \eqref{NLSequation} in invariant under the scaling
$$ u(t,x) \mapsto u_{\lambda}(t,x)  = \lambda^{\f{d-2}{2}}u(\lambda^2 t, \lambda x), $$
in the sense that if $u$ solves \eqref{NLSequation}, then $ u_{\lambda }$ also solves \eqref{NLSequation} with initial data $(u_0)_{\lambda}$.
The following  \emph{energy} functional  $E(u)$ is conserved in time along solutions $u$ 
\begin{align}
	E(u(t)) &= \int [\f{1}{2}|\nabla u(t)|^2 -  (\f12 - \f1d)| u(t)|^{\f{2d}{d-2}}]~dx = E(u_0).
\end{align}
Further with $u_0 \in L^2$ the mass $ \| u\|_{L^2}^2$ of $u$ is an additional conservation law and with the power law in \eqref{NLSequation}, we obtain $ E(u_{\lambda}(t)) = E(u(\lambda^2 t)) $. Hence we refer to \eqref{NLSequation} as \emph{energy-critical}.
A 2-parameter family of ground state solutions for the energy $E$ is given by
\begin{align*}
	&e^{i \alpha}W_{\lambda}(x) = \pm e^{i \alpha}\lambda^{\f{d-2}{2}} W(\lambda x), ~ \lambda >0,~\alpha \in \R,
\end{align*}
where the 
profile 
\begin{align}
	W(x) = \bigg( 1 + \f{|x|^2}{ d(d-2)}\bigg)^{\f{2-d}{2}},~ x \in \R^d, 
\end{align}  
is the (up to symmetry) unique positive solution of $ \Delta W + W ^{\f{d+2}{d-2}} = 0$. It has been found by Aubin \cite{Aubin} and Talenti \cite{Talenti} in the context of extremizer for the Sobolev inequality.\\[4pt]
For small initial data $ u_0 \in \dot{H}^1(\R^d)$ the solution of \eqref{NLSequation} is known to scatter globally in time as $ t \to \pm \infty$. For large $\dot{H}^1(\R^d)$ data, however, blow up 
and energy concentration may occur. In particular, the ground states $e^{i \alpha}W_{\lambda}$ pose obstructions to global scattering for the Cauchy problem. 
In the present article, we will construct solutions of \eqref{NLSequation} in dimension $d =3$ of the form
\begin{align}\label{form}
	u(t,r) = e^{i \alpha(t)}W_{\lambda(t)}(x) + \eta(t, x), ~~~(t,x) \in  [0,T) \times \R^3,
\end{align}
where $\lambda(t) = (T -t)^{-\f12 - \nu},~ \alpha(t) = \alpha_0 \log(T-t)$ and $\eta(t, \cdot) \to\eta_0$  in $\dot{H}^{2}\cap \dot{H}^1 $, i.e. the local solution \emph{collapses} along $ (\alpha(t),~\lambda(t)) $ to the rescaled bulk term $W$ as $ t \to T$ . The method we use is based on an approximation scheme for  \eqref{form} solving \eqref{NLSequation} up to a fast decaying error. Such a scheme was developed in \cite{Perelman}, \cite{OP} for Schr\"odinger systems and  originated  in the seminal  work of  Krieger-Schlag-Tataru in \cite{KST1}, \cite{KST-slow}, \cite{KST2-YM} for corotational critical wave maps, the energy-critical nonlinear wave equation and the corotational Yang-Mills equation. See also \cite{K-S-full} and \cite{Don-Kr} for further work related to this approximation scheme. 
The main result of this article is stated as follows.
\begin{Th} \label{main} Let $ \alpha_0 \in \R$,~$ d = 3$ and $~ \nu >  1 $.  Then there exists $M = M(\alpha_0, \nu) > 0 $ small such that the following holds. Let $ \delta \in (0,M)$. Then for some small $T_0 >0 $ and  any $ T \in (0, T_0)$, there is a solution $ u(t) = u_{T, \alpha_0, \nu}(t)$  of \eqref{NLSequation} with $ u \in C^0([0, T), \dot{H}^1(\R^3) \cap \dot{H}^{2}(\R^3))$ of the form
	\[  u(t,x) = e^{i \alpha(t)}W_{\lambda(t)}(x) + \eta(t, x), ~~~(t,x) \in [0,T) \times \R^3,  \]
	where $ \lambda(t) = (T-t)^{- \f12- \nu},~\alpha(t) = \alpha_0 \log(T-t)$ and $ \eta \in C^0([0,T), \dot{H}^1 \cap \dot{H}^{2})$.  Further there holds
	\begin{align}
		&\sup_{t \in [0,  T)}\norm{\eta(t)}{\dot{H}^1 \cap \dot{H}^2} \leq \delta,
	\end{align}
	and there exists $ \eta_T  \in  H^{1 + \nu-}(\R^3)$ with $ \lim_{t \to T^-}\|\eta(t) - \eta_T \|_{\dot{H}^1 \cap \dot{H}^2} = 0$. 
\end{Th}
\begin{Rem} (i) The solution $u(t)$ in fact satisfies $u \in C^0([0, T), \dot{H}^1(\R^3) \cap \dot{H}^{s}(\R^3))$ for all $ 1 < s < 1 + \nu$. We can show the statements in Theorem \ref{main} are true if we replace $ \dot{H}^2$ by $ \dot{H}^s$.  Further,  we give the prove of Theorem  \ref{main} in case $ \nu >1 $ is irrational. The rational case requires a slight modification, see the remark below Proposition \ref{main-prop-approx}.\\[3pt]
	(ii)~Clearly, the solutions  $u(t)$   in Theorem \ref{main} satisfy
	\begin{align}
		& \sup_{t \in [0, T)}\| u(t) \|_{\dot{H}^1(\R^3)}< \infty,~~\|u \|_{L^{10}([0,T) \times \R^3)} = \infty,
	\end{align}
	and $ T = T_+(u(0)) < \infty$ is the maximal forward existence time. \\[3pt]
	(iii)~We may verify the concentration inequalities
	\begin{align}\label{con-ult1}
		&\lim_{t \to T^-}\int_{|x| \leq R} |\nabla u(t)|^2 ~dx \geq \int |\nabla W|^2~dx + \int_{|x| \leq R} |\nabla \eta_T|^2 ~dx,~~ R > 0\\ \label{con-ult2}
		&\limsup_{ t \to T^-}\int_{|x| \geq  R}  \big[|\nabla u(t)|^2 + |u(t)|^{6}\big]~dx  \leq  C_0~T,~~~~R > 0.
	\end{align}
	In fact, concentration of the $ \dot{H}^1$ norm is obtained in the sense
	\begin{align}\label{con-ult4}
		\lim_{t \to T^-}\| \nabla u(t)\|_{L^2(|x|\leq R)} = \| \nabla W\|_{L^2} + \|\nabla \eta_T\|_{L^2(|x|\leq R)},~ 0 < R \leq \infty.
	\end{align}
\end{Rem}
~~\\
The solutions $u(t)$ in Theorem \ref{main} can be chosen with $ E(u)$ arbitrarily close to $E(W)$. In fact we have the following Corollary implied by Theorem \ref{main}.
\begin{Corollary} \label{main-cor} Let $\alpha_0 \in \R $,~$ \nu > 1 $. Then there exists $\tilde{\delta} = \tilde{\delta}_{\alpha_0, \nu}  > 0$  such that the following holds.  For all $ \delta \in (0, \tilde{\delta}) $ there exists $ T = T_{\delta} > 0$ and a type II solution $ u(t)$ of \eqref{NLSequation} on $[0,T)$ as in Theorem \ref{main} with $T = T_+(u(0)) < \infty$  satisfying 
	\[
	|E(u) - E(W)| < \delta,
	\]
	and for some $ \tilde{C} = \tilde{C}_{\alpha_0, \nu} > 0$ there holds
	\begin{align}\label{con-ult3}
		& 	\lim_{t \to T^-} \bigg(\int_{|x|\leq R} |\nabla u(t)|^2~dx -  \int |\nabla W|^2~dx \bigg) \in [0,\tilde{C}\cdot \delta),~~R > 0\\
		&\limsup_{ t \to T^-}\int_{|x| \geq  R}  \big[|\nabla u(t)|^2 + |u(t)|^{6}\big]~dx  \leq \tilde{C} \cdot \delta,~~~ R > 0.
	\end{align}
\end{Corollary}
~~\\
Some remarks are in order.
\begin{Rem}~~(i)~~The result of Theorem \ref{main}  was suggested by Ortoleva-Perelman in \cite[Remark 1.5]{OP}.  The approach for the proof is that of \cite{OP} and further following G.Perelman's work on the Schr\"odinger maps flow \cite{Perelman}.\\[5pt]
	(ii) The inequalities   \eqref{con-ult1}, \eqref{con-ult2} and \eqref{con-ult4} can be interpreted  as concentration results. In fact  Kenig-Merle \cite[Corollary 5.18]{KM} gave a description of the $\dot{H}^1$ concentration for type II finite time blow up, i.e. type II solutions $u(t)$ of \eqref{NLSequation} with maximal existence time $ T_+ < \infty$ must satisfy
	\begin{align}\label{con-in1}
		&\inf_{R > 0} \liminf_{ t \to T^-} \int_{|x| \leq R} |\nabla u(t)|^2~dx \geq \frac{2}{d} \int | \nabla W|^2 ~dx,\\[2pt] \label{con-in2}
		&\inf_{R > 0} \limsup_{ t \to T^-} \int_{|x| \leq R} |\nabla u(t)|^2~dx \geq \int | \nabla W|^2 ~dx.
	\end{align}
	Theorem \ref{main} and Corollary \ref{main-cor} confirms 
	\cite[Remark 5.19]{KM} 
	and thus  the occurrence of the $ \dot{H}^1$ concentration \eqref{con-in1} and \eqref{con-in2} in \cite[Corollary 5.18]{KM} (here through \eqref{con-ult4}).\\[5pt]
	(iii) In Theorem \ref{main} and Corollary \ref{main-cor}, if $ |\alpha_0| >0$ or $ \nu > 1 $ increase, then $ M, ~\tilde{\delta} > 0 $ decrease (and likewise also $ T_{\delta}>0$ and $ T_0 > 0$ decrease in this case). The parameter $ \delta> 0$ 
	controls  $\eta(t)$, where the blow up concentrates the $\dot{H}^1$ norm at the origin (in the sense of \eqref{con-ult4})  through `ground state renormalization' in finite time $  t = T$. In particular $ u_T = \lim_{t \to T^-}u(t)$ can not exist strongly in $\dot{H}^1(\R^n)$ since this would contradict \eqref{con-ult4}.
\end{Rem}
~~\\
\subsection{Brief overview of the Cauchy problem}. We now give an overview with results on the global v.s. local wellposedness theory of the energy-critical focusing nonlinear Schr\"odinger equation
\[
i\partial_tu = - \Delta u - |u|^{\f{4}{d-2}}u,~ x \in \R^d
\]
with data in $\dot{H}^1(\R^d)$. The description may be incomplete due to the vast literature on the NLS. We first explain the following notation.\\[2pt]
\emph{Notation}.~ From the $\dot{H}^1(\R^d)$ local wellposedness theory above (cf. \cite{Cazenave-Wei}, \cite{Cazenav}), we denote by $ T_+ = T_+(u_0) $ the maximal forward existence time  and likewise by $ T_- = T_-(u_0)$  the maximal backward existence with initial data $  u_0 \in \dot{H}^1(\R^d)$. If $ u_0 \in \dot{H}^1\cap \dot{H}^s$, then also $ u \in C((-T_-, T_+), \dot{H}^1\cap \dot{H}^s ) $ and in case $T_{\pm} < \infty$, then $u(t) $ leaves any compact subset of $ \dot{H}^1(\R^d)$ as $ t \to T_{\pm}$. Further on any  interval $ I \subset \R$ of existence, we write 
$$\| u\|_{S(I)} = \| u\|_{L^{\frac{2(d+2)}{d-2}}(I \times \R^d)}$$
for the  Strichartz norm. Finiteness of the latter is known to give a continuation criterion in $ C(I, \dot{H}^1(\R^d))$ and in particular $T_+ = \infty$ or $T_- =  \infty$ implies \emph{global scattering}
\[
\lim_{t \to \pm \infty}\| u(t) - e^{i t \Delta} u_0^{\pm} \|_{\dot{H}^1(\R^d)} = 0
\]
in the respective time direction, see \cite{KM}. We call a solution $u$ of  \eqref{NLSequation} a \emph{type II} solution, if 
\[ 
\sup_{t \in [ 0, T_+)} \| u(t)\|_{\dot{H}^1(\R^d)} < \infty,
\]
and in particular a \emph{type II blow up} if $ T_+ < \infty$, ie. $ \| u\|_{S([0, T_+)} = \infty$. A \emph{type I} blow up is such that $ T_+ < \infty$ and the $ \dot{H}^1(\R^d)$ blows up as $ t \to T_+^-$. 
\\[2pt]
\emph{Scattering v.s. blow up}. For the energy critical defocusing  (i.e. $+|u|^{\f{4}{d-2}}u$) 
nonlinear Schr\"odinger equation, $ \dot{H}^1(\R^d)$ solutions are expected to exist global in time and scatter at infinity. Such results were obtained (with radial data) in the work of Bourgain \cite{Bour} ($d = 3$) , Tao \cite{Tao} ($d \geq 5$)  and also in the work of Ryckman-Visan \cite{Ryckman}, Visan \cite{Visan} and  \cite{Colliand-I} (where radial symmetry was removed in the three latter references).\\[3pt]
As mentioned above, global scattering as $ t \to \pm \infty$ for the focusing \eqref{NLSequation} with small initial data $ u_0 \in \dot{H}^1(\R^d)$ has been proved by Cazenave-Weissler in \cite{Cazenave-Wei}. For large data $ u_0 \in \dot{H}^1(\R^d)$, the solution does not scatter due to the existence of the ground state solutions
\begin{align}	
	&u(t,x) = e^{i \theta}W_{\lambda}(x),~~~W(x) = \big( 1 + \f{|x|^2}{ d(d-2)}\big)^{\f{2-d}{2}},~ x \in \R^d,
\end{align} 
which is relevant in the energy-trapping phenomena described below.
Further, finite time blow up in $\dot{H}^1(\R^d)$ is provided by the \emph{virial identities}

\begin{align}
	&\frac{d}{dt}\int |x|^2 |u(t,x)|^2~dx  = 4 \text{Im} \int x \cdot \nabla u(t,x) \bar{u}(t,x)~dx\\[2pt]
	&\frac{d^2}{dt^2}\int |x|^2 |u(t,x)|^2~dx = 8 \bigg(\int |\nabla u(t,x)|^2~dx - \int | u(t,x)|^{\frac{2d}{d-2}}~dx\bigg)
\end{align}
~~\\
in the  class $\Sigma = \{  u \in \dot{H}^1~|~ |x| u(x) \in L^2\} $ of finite variance data. In particular, the solution of \eqref{NLSequation} with $ u_0 \in \Sigma$ has finite time blow up, see  \cite{Cazenav},  if the convexity assumption $ E(u_0) < 0 $ holds.\\[2pt]
In the radial case, the situation of blow up v.s. scattering is well described below the ground state energy $ E(u_0) < E(W)$ by the celebrated  dichotomy of Kenig-Merle \cite{KM} for $ 3 \leq d \leq 5$ (see  Keraani \cite{Keerani}, Killip-Visan \cite{Killip-Visan} in $ d \geq 6$ for extensions). For  data $ u_0  \in \dot{H}^1_{rad}(\R^d)$ and subthreshold energy $E(u_0) < E(W)$, there holds
\begin{itemize}
	\setlength\itemsep{3pt}
	\item[(a)] The unique solution of \eqref{NLSequation} exists globally, $T_{\pm} = \infty$, and scatters to zero in both directions if $ \norm{\nabla u_0}{L^2}^2 < \norm{\nabla W}{L^2}^2$.
	\item[(b)] The unique solution of \eqref{NLSequation} blows up in finite time in both directions, $T_{\pm} < \infty$, if $ u_0 \in H^1(R^d)$ and $ \norm{\nabla u_0}{L^2}^2 > \norm{\nabla W}{L^2}^2$.
\end{itemize}
~~\\
The dynamics exactly at $E(u_0) = E(W)$ was considered for radial data by Duyckaerts and Merle for $ 3 \leq d \leq 5 $ in \cite{DM}. Besides  $ u(t) = W$, they add the possibility of $ u(t) = W_{\pm}(t)$ to the  dichotomy  $\pm  \|\nabla u_0 \|_{L^2 }  > \pm \| \nabla W\|_{L^2}$, where $ W_{\pm}$ are unique global solutions with $ W_{\pm }(t)  \to W $ as $ t \to \infty$ in $\dot{H}^1(\R^d)$.
This result  has been extended to  $ d \geq 6$ by Li-Zhang \cite{Li-Zhang-thre} with radial data $ u_0 \in \dot{H}^1(\R^d)$ and an approach to remove the radiality assumption in $ d \geq 5$ has been presented by Su-Zhao in \cite{Su-thre}.\\[5pt]
On the above ground state dynamics $ E(u_0) > E(W)$, we first refer to two Corollaries  of Kenig-Merle \cite[Corollary 3.15, Corollary 5.18]{KM}.  Let $u(t)$ be a radial type II solution of \eqref{NLSequation} with initial data $ u_0 \in \dot{H}^1(\R^d)$ (no further energy restriction), i.e. such that 
\begin{align}\label{ineq-KM}
	\sup_{t \in (T_-, T_+)}\| u(t)\|_{\dot{H}^1}^2 \leq  C_0,
\end{align}
for some constant $ C_0 > 0$. (i) If the inequality in  \eqref{ineq-KM} holds with $ C_0 < \| W\|_{\dot{H}^1}^2  $, then $(-T_-, T_+) = (-\infty, \infty)$ and the global Strichartz norm $ \| u \|_{S(- \infty, \infty)} < \infty$, i.e. we obtain scattering in both time directions. (ii) If $T_+ < \infty$, then the $\dot{H}^1(\R^d)$ concentration inequalities \eqref{con-in1} and \eqref{con-in2} hold. We secondly refer to the work of Duyckaerts-Roudenko \cite{DR} involving finite variance
$$
V(t) = \int |x|^2 |u(t,x)|^2~dx. 
$$
In particular,  let $u(t)$ solve \eqref{NLSequation} with $ u_0 \in H^1(\R^d)$ and $V(0) < \infty$. Then, if there holds
\[
E(u)\big(1 + \frac{V_t(0)^2}{32 E(u) V(0)}\big) < E(W),~~~\pm \| u_0\|_{L^{\frac{2d}{d-2}}}  > \pm \| W\|_{L^{\frac{2d}{d-2}}},~~ \pm V_t(0) \leq 0,  
\]
there is finite-time blow up ($(+)$-case) or global scattering ($(-)$-case) in $\dot{H}^1(\R^d)$ in forward time, where for the latter $ d \geq 5$ or  radiality in $ d = 3,4$ is required. In fact the scattering result follows from a stronger statement. Let $u(t) $ be a solution of \eqref{NLSequation} ( in $ d \geq 5$ and radial in $ d = 3,4$) with
\[
\limsup_{t \to T_+^-} \| u(t)\|_{L^\frac{2d}{d-2}} < \| W\|_{L^\frac{2d}{d-2}},
\]
then $T_+ = \infty$ and  $u(t)$ scatters in $\dot{H}^1(\R^d)$ in forward time. This result is  complementary to the Kenig-Merle Corollary above and is obtained by a rather quantitative control of the Strichartz norm $\| u\|_{S(I)}$, see \cite[Theorem 3.2]{DR}. Finally, by a result of Nakanishi-Roy \cite{Nak-ro} on \eqref{NLSequation} in $ d =3$ with radial $u_0 \in \dot{H}^1(\R^d)$, it is known that energy slightly above the ground state 
\[
E(u) < E(W) + \varepsilon,
\]
leads the solutions $u(t)$ to stay in a neighborhood of the ground state manifold, for an interval  of existence times. If $u(t)$ leaves the neighborhood, then either (i) there is global scattering as $ t \to \pm \infty$, or (ii) there is finite time blow up. The proof is established by a one-pass Lemma, which  provides the above classification near  the ground state.\\[10pt]
\emph{On blow up solutions}. To the authors knowledge, apart from the above description, the only blow up solutions for \eqref{NLSequation} in the literature are as follows. Rafa\"el-Szeftel  \cite{Raph-Szeft} derived a class of radial blow up solutions of \eqref{NLSequation} in dimension $ d = 3 $, emerging from initial data in an open subset of $ H^3(\R^3)$ and contracting on a sphere of fixed positive radius (\emph{standing ring blow up}). Here the blow up takes places at a log-log scale. 
Subsequently, an open set of radial data in $ H^1(\R^3)$ leading to such blow up solutions for \eqref{NLSequation}  was constructed by Holmer-Roudenko \cite{Hol-Rud}. For related equations, we refer to  the construction of  stable smooth blow up for the critical Schr\"odinger maps flow by Merle-Rafa\"el-Rodnianski  \cite{Rod-Raph} and to  blow up  for $L^2$ critical Schr\"odinger equations by Merle-Rafa\"el \cite{Merle-Rapha1}, \cite{Merle-Rapha2} and Merle \cite{Merle-k-bl}.\\[10pt]
\emph{Type II solutions}. The  contribution of energy bubbles to the asymptotics of \eqref{NLSequation} is conjectured to be of the following schematic form.
\begin{align}\label{global-bounded}
	&  u(t) = \sum_{j}   e^{i \alpha_j(t)} \iota_j W_{\lambda_j(t)} + v(t) + o_{\dot{H}^1}(1),~~\iota_j \in\{ \pm 1 \},
\end{align}
where 
$ v(t,x) = e^{i t \Delta }v_0(x)$ if $T_+ = \infty$ and $ v(t,x) = v_0(x)$ if $T_+ < \infty$.
To the authors knowledge, apart form $ e^{i \theta}W_{\lambda}$ and scattering solutions, the only examples known to be of the form \eqref{global-bounded}, are as follows. There are pure ($v_0 = 0$) global one-bubble solutions $W_{\pm}(t) $ given by Duyckaerts-Merle in  \cite{DM} for $ 5 \geq d \geq 3$ at ground-state energy $E(W_{\pm}) = E(W)$. Ortoleva-Perelman constructed  vanishing non-dispersive solutions in \cite{OP} for $ d = 3$. Further,  global two-bubble solutions have been constructed by 
J. Jendrej in \cite{Jendrej} for $ d \geq 7$.\\[10pt]
\emph{Result in this article}. We obtain a 2-parameter family $\{ u_{\nu, \alpha_0}\}$ of  single-bubble finite-time blow up solutions of \eqref{NLSequation}  
in dimension $ d =3 $. The energy $E(u)$ can be located arbitrarily close to the ground state energy $ E(W)$. The approach we take on is motivated by Perelman's work \cite{Perelman} 
and Ortoleva-Perelman's work \cite{OP} on 
critical Schr\"odinger equations. Further, this approach for the parametrix is schematically related to a forthcoming result \cite{KSch} on the (critical) Zakharov system in $ d =4$  obtained jointly with J. Krieger.

\subsection{Notation and Outline}~~\\

In Section \ref{sec:approx} we prove the existence of approximate solutions of \eqref{NLSequation} with a small error that decays sufficiently. In order to find exact solutions in Section \ref{sec:linearized},  it is necessary to provide suitable estimates for the perturbed Schr\"odinger flow, which will be done as well in Section \ref{sec:linearized}.  In Section \ref{sec:spec} and Section \ref{sec:Scat}, we  discuss spectral properties of the linearized operator and the distorted Fourier transform used in Section \ref{sec:linearized}. The construction of Jost solutions for $1D$ Schr\"odinger systems will be included in Section \ref{sec:Scat}. We briefly give some preliminary details on how the approximation will be obtained.\\[5pt] \textbf{The approximate solution.}
Let  $ \nu > 1 $,~ $ 0 < \epsilon_1 \ll 1 $  sufficiently small (to be chosen below) and  $ C > 1 $ a fixed constant.
We subdivide $ \R^d \times (0, T)$  into three distinct regions. We start with  the \emph{Interior region} in Section \ref{subsec:inner}, 
\begin{align}
	I &:= \{ (x,t)~|~ 0 \leq |x| \leq C(T-t)^{\epsilon_1 + \f{1}{2}} \},
\end{align}
in which we solve \eqref{NLSequation} for $ \eta= \eta(t,\rho)$ with $ \rho = \lambda(t) r$,~ $ r = |x|$, ~ $ \lambda(t) = (T-t)^{- \f12 - \nu} $ and $ \alpha(t) =  \alpha_0 \log(T - t)$ of the form
\begin{align*}
	u(\rho,t) = \lambda(t)^{\f{d-2}{2}} e^{i \alpha(t)}(W(\rho ) + \eta(\rho,t)),~~ t \in (0, T).
\end{align*}
Here $ \eta(t,R)$ \emph{corrects} the large error $ \partial_t(\lambda(t)^{\f{1}{2}} e^{i \alpha(t)}W(\lambda(t) r ))$ as $ t \to T^-$ 
and $T> 0 $ is arbitrary but will be fixed later. 
Note that we have $ 0 \leq \rho \leq C (T-t)^{\epsilon_1  -  \nu}$ and $\eta(t,R)$ will be chosen  of the form
\[
\eta(t, \rho)  = (T-t)^{2 \nu}\tilde{\eta}_1(R) + (T-t)^{4 \nu}\tilde{\eta}_2(R) + (T-t)^{6 \nu}\tilde{\eta}_3(R)  + \dots,
\]
for which \eqref{NLSequation} reduces to an iterative system of ODE equations for 
$ (\tilde{\eta}_k)_k$.\\[5pt]
In the \emph{Self-similar region} of Section \ref{subsec:SS}
\begin{align}
	S &:= \{ (x,t)~|~ \f{1}{C} (T-t)^{\epsilon_1+ \f{1}{2}} \leq |x| \leq C (T-t)^{-\epsilon_2 + \f{1}{2}} \},
\end{align}
we solve \eqref{NLSequation} in the self-similar variable $ y = r (T-t)^{- \f{1}{2}}$ for $  u(y,t) = (T-t)^{- \f{1}{4}} w( t,y), $ where  $ \f{1}{C} (T-t)^{\epsilon_1} \leq y \leq C (T-t)^{- \epsilon_2}$. Hence in order to solve the resulting self-similar system for $ w = w(t,y) $, we note that for $ 0 < t < T $ and $ (x,t) \in I \cap S $ there holds $| x| \sim_C (T-t)^{\epsilon_1 + \f{1}{2}}$ and thus 
$$  \rho \sim_C  (T-t)^{\epsilon_1 -  \nu}~~\text{and}~~ y \sim_C  (T-t)^{\epsilon_1},~~\text{where}~~\varepsilon_1 <  \nu.$$
Then, as $ t \to T^-,~ y \to 0$ we obtain a series expression of $w(y,t) $ by a matched asymptotics ansatz from  $ \eta(\rho,t) $ as $ \rho \to \infty$.\\[5pt]
Finally, in the \emph{Remote region} in Section \ref{subsec:RR}
\begin{align}
	R &:=  \{ (x,t)~|~  (T-t)^{- \epsilon_2 + \f{1}{2}} \leq |x|  \},
\end{align}
we  change the coordinates back  to $(t,r)$ and  control the remaining radiation part perturbatively, i.e. to leading order we consider the non-vanishing part in the $  y \to \infty $ asymptotics as $ t \to T^{-}$.\\[2pt]
\textbf{Notation.} We write $ f \lesssim g $ and $ g \gtrsim f $ for $ f \leq C g$ and $ f \geq C g $ respectively, where $C > 0$ is a constant. We use the typical $ O$-notation, i.e. $ f(x) = O(g(x)) $ as $ |x| \to \infty$ if $ |f(x)| \lesssim |g(x)| $ for $ |x|  \gg1$. Let $ \ell \in \Z_+$ and $ g(x)$ be $\ell$-times differentiable if $ |x| \gg1 $.  We then define $ O_{\ell}(g(x)) $ to be the class of $ f \in C^{\ell}(\{ |x| \gg1 \} ) $ with $ \partial_k f(x)  = O(\partial_k g(x)) $ as$ |x| \to \infty $ and all $ 0 \leq k \leq \ell$. Similar notation is used as $ |x| \to 0$.\\[2pt]
We may also refer to the asymptotics $ f(x) = O(g(x)) $ as smooth, or just differentiable,  if $ f(x) = O_{\ell}(g(x))$ for all $ \ell \in \Z_+$. In the  first subsection \ref{subsec:inner} of the following Section \ref{sec:approx}, we also use $f(x) = O(g(x))$ in a slightly stronger sense, requiring
\begin{align}
	&f(x) \slash g(x) = a_0 + xa_1 + x^2 a_2 + \dots, ~~|x| \ll 1, ~~\text{or}\\[2pt]
	&f(x) \slash g(x) = b_0 + x^{-1}b_1 + x^{-2}b_2 + \dots,~~ |x| \gg1,
\end{align}  to converge absolutely. This  is clarified below.
\section{Construction of approximate solutions near the ground state } \label{sec:approx}
In this Section, we first present the iterative construction of a parametrix $ u_{\text{app}}(t,x)$ for   \eqref{NLSequation} in dimensions $ d = 3$  suggested by Perelman \cite{Perelman} for the Schr\"odinger maps flow in $d =2$ and Ortoleva-Perelman  \cite{OP} for global solutions of the quintic NLS in $ d =3$. This parametrix, constructed iteratively, is essentially inspired by the seminal work of Krieger-Schlag-Tataru \cite{KST1}, \cite{KST-slow} that is, for any $ N \in \Z_+$ we find $ \tilde{N} = \tilde{N}(N) \in \Z_+$ such that $ u_{\text{app}}(t,x) = e^{i \alpha(t)} W_{\lambda(t)}(x) + \zeta(t,x)$ satisfies \eqref{NLSequation} \emph{approximately}, i.e.
\[
i \partial_t u_{\text{app}}(t,x) + \Delta u_{\text{app}}(t,x) + |u_{\text{app}}(t,x)|^4 u_{\text{app}}(t,x) = O((T-t)^N),
\]
~\\
on some interval $[0,T)$ with small $ 0 < T\ll1$,  after carrying out $\tilde{N}$-steps in the iteration procedure. More precisely, we have the following Proposition.
\begin{Prop}[Approximation] \label{main-prop-approx}Let $ d  = 3$. For $ \nu>1 $ irrational and $\alpha_0 \in \R$, there exists  a positive number $M = M_0(\alpha_0, \nu) > 0$ such  such that the following is true. For  any $ \delta \in (0, M_0) $ and  $ N \in \Z_+$ large enough, there exists $ 0 < T_0(N, \delta, \nu,\alpha_0) \leq 1$ and for any $ 0 < T \leq T_0$  a radial function $ u_{\text{app}}^N \in C^{\infty}(\R^3 \times [0, T))$ of the form
	\begin{align}
		&u_{\text{app}}^N(x,t) = \lambda(t)^{\f12} e^{i \alpha(t)} \big( W(\lambda(t) x) + \zeta^N(\lambda(t) x, t)\big),\\ \nonumber
		& x \in \R^3,~ t \in (0, T),~ \lambda(t) = (T-t)^{- \f12 - \nu},~\alpha(t) = \alpha_0 \log(T-t),
	\end{align}
	such that\\[2pt]
	(a)~There holds $ \zeta^N(t) \in \dot{H}^{1} \cap \dot{H}^{2}$ and  for $ \zeta^N(\lambda(t) x,t) = z^N(R,t),~R = \lambda(t) |x|$ we have 
	\begin{align}
		&\| z^N \|_{L^{\infty}_R} +  	\|  R \partial_R z^N(t)  \|_{L^{\infty}_R} \leq C \delta^{m(\nu) }(T-t)^{\frac{\nu}{2} + \f14 },~~\\
		&  \|  R^{-l}\partial^k_R z^N(t)  \|_{L^{2}(R^2 dR)} \leq C \delta^{m(\nu) } (T-t)^{\frac{\nu}{2}+ \f14 },~~ 0 \leq k +l \leq 1,\\
		&\|  R^{-l} \partial_R^k z^N(t)  \|_{L^{2}(R^2 dR)} \leq C \delta^{m(\nu)} (T-t)^{\nu },~~ k + l = 2,\\
		&\|  \partial_R z^N(t)  \|_{L^{\infty}_R} + \|  R^{-1} z^N(t)  \|_{L^{\infty}_R}  \leq C \delta^{m(\nu)} (T-t)^{\nu},\\
		&\|  R^{-3} z^N(t) \|_{L^{\infty}_R} + \|  R^{-l} \partial_R^k z^N(t)  \|_{L^{\infty}_R} \leq C (T-t)^{2\nu},~~k + l = 2,
	\end{align}
	where $ m(\nu) > 0$ (increasing in $\nu$) and $C>0$ depends only on $ \alpha_0, \nu$ (not on $N$ or $\delta$).
	Also there exists $ \zeta_*^N \in H^{1 +\nu -}(\R^3)$ such that 
	\begin{align}
		e^{i \alpha(t)} \lambda^{\frac{1}{2}}(t) \zeta^N(\lambda(t)\cdot , t) = \zeta_*^N + o(1),~~ t \to T^{-}
	\end{align}
	in $\dot{H}^1(\R^3)\cap \dot{H}^2(\R^3)$.\\[2pt]
	(b)~ The error function defined by 
	\[
	e^N(\cdot,t) := i \partial_t u_{\text{app}}^N( \cdot,t) +\Delta u_{\text{app}}^N(\cdot, t) + | u_{\text{app}}^N( \cdot, t)|^{4} u_{\text{app}}^N( \cdot, t),
	\]
	satisfies for $ 0 \leq t < T$
	\begin{align}
		&\|   e^N(t)    \|_{\dot{H}^2}  
		+	\|  \langle x \rangle  e^N(t)    \|_{L^2}  \leq C (T -t)^{N}.
	\end{align}

\end{Prop}
\begin{Rem} As in \cite{Perelman}, we observe the following conclusion on Proposition \ref{main-prop-approx}.
	(i)~ We may improve  $ \zeta^N \in \dot{H}^s(\R^3)$ for all $ 1 < s < 1+ \nu$ and with corresponding estimates as in the above Proposition.
	(ii)~ Further, the error estimate may be improved as follows. 
	\begin{align}
		&\forall~~N \geq c:~~ \|  \langle x \rangle^{\beta}  e^N(t)     \|_{\dot{H}^s_x} \leq c (T-t)^{N - c},
	\end{align}
	where $ s, \beta >0$ and $c = c(s, \beta) > 0$ is some constant. (iii) The restriction to irrational $ \nu >1$ is only for technical reasons in the proof of Proposition \ref{main-prop-approx} and can be removed by slight modification.
\end{Rem}
\begin{Rem} The approximation constructions in this Section also work  for the cubic focusing NLS in dimension $d=4$. Some parts, especially in the beginning, are therefore carried out with $ d \in \{3,4\}$, where differences and similarities of the approach  in dimensions $ d =3$ and $d=4$ are visible. However, the main Propositions are only provided for the quintic case in $ d =3$ and the cubic case is provided elsewhere.
\end{Rem}
\subsection{Interior region $ r \lesssim (T -t)^{\f12 + \epsilon_1 } $}\label{subsec:inner}~Let $ C > 0 ,~ \epsilon_1 > 0$ and for the rescaled variable $ R = r \lambda(t) $ with $ t \in (0, T)$ we require
\begin{align}\label{condition-inner}
	0 \leq R \leq C (T-t)^{\epsilon_1 - \nu}.
\end{align}
Then, for a solution of \eqref{NLSequation}, we make the following ansatz
\[
u(R, t) = \lambda(t)^{\f{d-2}{2}} e^{i \alpha(t)} \tilde{u}(R, t). 
\] 
We first calculate
\begin{align*}
	\begin{alignedat}{2}
		\partial_t u(R, t) =&~~ e^{i \alpha(t)} \lambda(t)^{\f{d-2}{2}}  \bigg[i \dot{\alpha}(t) \tilde{u}(R, t) + \frac{\dot{\lambda}(t) }{\lambda(t)}\bigg(\f{d-2}{2}  \tilde{u}(R, t) + R  \partial_{R}\tilde{u}(R, t)\bigg) +  \partial_t\tilde{u}(R, t)\bigg],\\[4pt]
		\Delta u(R, t)  =&~~ e^{i \alpha(t)} \lambda(t)^{\f{d+2}{2}} (\partial_{R}^2 + \f{\partial_{R}}{\rho}) \tilde{u}(R, t).
	\end{alignedat}
\end{align*}
Hence, substituting $ \lambda(t) = (T - t)^{- \f12 - \nu},~ \alpha(t) = \alpha_0\log(T-t)$, we infer from \eqref{NLSequation} 
~\\
\begin{align}
	- \Delta \tilde{u}(R, t) -& |\tilde{u}(R, t)|^{\f{4}{d-2}}\tilde{u}(R, t)\\ \nonumber
	=&~ \alpha_0 (T-t)^{2\nu} \tilde{u}(R, t) + i (T-t)^{2\nu}  ( \f12 + \nu)(\f{d-2}{2}  + \rho\partial_{\rho})\tilde{u}(R, t)\\  \nonumber
	&~+ i (T-t)^{1 + 2\nu} \partial_t \tilde{u}(R, t) .
\end{align}
Further substituting  $ \tilde{u}(R, t) = W(\rho) + \eta(\rho, t) $ and using $ \Delta W+ W^{\f{d+2}{d-2}} = 0 $, we have
\begin{align}\label{main-eqn-interior-vor}
	i (T-t)^{1 + 2\nu} \partial_t \eta(\rho, t) = H \eta(\rho, t)  +  \mathcal{N}(\eta)(\rho, t),
\end{align}
where the operator $ H $ is of the form
\begin{align}\nonumber
	& H \eta = - \Delta \eta - \f{d}{d-2} W^{\f{4}{d-2}}\eta - \f{2}{d-2}W^{\f{4}{d-2}}\bar{\eta}.
\end{align}
The expression $ \mathcal{N} = \sum_{J=0}^2 N_{J},$ is subdivided into the following contributions. First we have the initial error contribution to $ (\eta, \bar{\eta})$
\begin{align}
	& N_0(t, \rho) = - \alpha_0 (T-t)^{2 \nu}W(\rho) - i (T-t)^{2\nu}  ( \f12 + \nu)(\f{d-2}{2}  + \rho\partial_{\rho})W(\rho).
\end{align}
The linear term $\mathcal{N}_1$ then contributes 
\begin{align}
	& N_1(\eta) = - \alpha_0 (T-t)^{2 \nu}\eta - i (T-t)^{2\nu}  ( \f12 + \nu)(\f{d-2}{2}  + \rho\partial_{\rho})]\eta.
\end{align}
Finally we can write the nonlinear expression in $ \eta, \bar{\eta}$ as
\begin{align}
	&N_2(\eta) = - |\tilde{u}(\rho, t)|^{\f{4}{d-2}}\tilde{u}(\rho, t) + W^{\f{d+2}{d-2}} + \f{d}{d-2} W^{\f{4}{d-2}}\eta + \f{2}{d-2}W^{\f{4}{d-2}}\bar{\eta}.
\end{align}
The perturbation
$$\eta(R,t) = \eta^{(1)}(R,t) + i \eta^{(2)}(R,t) $$
is indifferently identified with $  \eta= (\eta^{(1)},  \eta^{(2)} )^t$ where $ \eta^{(1)} = Re(\eta),~ \eta^{(2)} = Im(\eta)$. Then  \eqref{main-eqn-interior-vor} is rewritten into a system
\begin{align}\label{main-eqn-interior-pre}
	& (T-t)^{1 + 2\nu} \partial_t \eta  + \mathcal{L}_W \eta   =  \tilde{\mathcal{N}}(\eta),~~~~\mathcal{L}_W = \begin{pmatrix}
		0 &  \Delta  + W^{p_c -1}\\
		-\Delta  - p_cW^{p_c -1} & 0
	\end{pmatrix},\\ \nonumber
	&\tilde{\mathcal{N}}(\eta) = \begin{pmatrix}
		Im(\mathcal{N}(\eta))\\
		- Re(\mathcal{N}(\eta))
	\end{pmatrix},~~p_c = \frac{d +2}{d-2}.
\end{align}
~~\\
We consider \eqref{main-eqn-interior-pre} in this section, however it will be convenient below to switch equivalently between  \eqref{main-eqn-interior-pre} and the following representation

\begin{align}\label{main-eqn-interior}
	& i (T-t)^{1 + 2\nu} \partial_t z = H z   +  \tilde{\mathcal{N}}(z),~~~~\tilde{\mathcal{N}}(z) = \begin{pmatrix}
		\mathcal{N}(\eta)\\
		- \overline{\mathcal{N}(\eta)}
	\end{pmatrix},~ z(t,r) = \begin{pmatrix}
		\eta(t,r)\\
		\overline{\eta(t,r)}
	\end{pmatrix},\\[2pt] \nonumber
	&H = H_0 + V(x) =  \begin{pmatrix}
		- \Delta  & 0 \\
		0 & \Delta  
	\end{pmatrix} + \begin{pmatrix}
		V_1(r) & V_2(r)\\
		- V_2(r) & -V_1(r)
	\end{pmatrix},\\[4pt] \nonumber 
	&D(H) = H^2(\R^d, \C^2) \subset L^2(\R^d, \C^2),
\end{align}

where 
\begin{align*}
	&V_1(x) = -  \f12(p_c + 1) W^{p_c -1}(r) ,~ ~V_2(r) = -  \f12(p_c -1) W^{p_c -1}(r).
\end{align*}
With the Pauli matrices 
\[ \sigma_1 =  \begin{pmatrix}
	0 & 1\\
	1 & 0
\end{pmatrix},~~\sigma_2 = \begin{pmatrix}
	0 & -i\\
	i & 0
\end{pmatrix},~~
\sigma_3 =  \begin{pmatrix}
	1 & 0\\
	0 & -1
\end{pmatrix},
\]
we may write
$$ H = - \Delta \sigma_3 + V(r) $$
and obtain $\sigma_1 H \sigma_1 = - H,~ \sigma_3 H \sigma_3 = H^*$.
\begin{Rem}

	A similar notation as above is adapted for $\tilde{N}_J$, i.e. such that  $\tilde{\mathcal{N}} = \sum_{J =0 }^2 \tilde{N}_J$.
	Note that $ \tilde{N}_0 $ and $ \tilde{N}_1 $ depend linearly on $ \eta$, whereas $ \eta$ appears at least quadratically in $ \tilde{N}_2$.
\end{Rem}
~~\\
\emph{Elliptic modifier in $I$}. As indicated above, we will subsequently add corrections to the bulk term
\[
\tilde{u}(t,R) = W(R) + \eta(R,t),~~~~\eta(R,t) = \eta_1(t, R) + \eta_2(t,R) + \dots 
\]
and set 
\[
\eta_K^*(t,R) =  \eta_1(t, R) + \eta_2(t,R) + \dots + \eta_K(t,R).
\]
In the inner region, a similar heuristic as in \cite{KST1}  applies and leads us to restrict to an elliptic construction of $ \eta_k $ (see also \cite{OP}). For the iterations we choose to use notation similar to \cite{KST1}. In the latter work, the typical asymptotics (for $R \gg1$) of a correction $f(R)$ with leading term $ R^k \log^{\ell}(R)$ reads
\begin{align} \label{previous-exp}
	f(R) = \sum_{j = 0}^{\ell} \sum_{r \geq 0}  c_{jr} R^{k -r} \log^j(R) + L.O.T,~~~ R \gg 1,
\end{align}
in an absolute sense where $c_{\ell r} \neq 0$ for at least one $r \geq 0$. The `L.O.T ' terms further  have the form
\begin{align}
	\sum_{j \geq 1} \sum_{r \geq j } c_{r, j + \ell}R^{k -r} \log^{j + \ell}(R),~~ R \gg1,
\end{align}
where $ c_{r j} = 0$ except for finitely many $j > \ell$. We now introduce a refined version of this asymptotics. For two positive integer $l = (l_1, l_2)$ we let  $r_l := \lfloor \tfrac{l_1}{l_2}r \rfloor $.
\begin{Def} \label{deff} For $ m, l_1, l_2 \in\Z_+,~ l = (l_1, l_2),~R_0 \in [1, \infty)$, we denote by $ S^m_l(R_0; R^k \log^{\ell}(R))$ the space of functions $f (R)$ analytic on $(0, \infty)$ and such that
	\begin{itemize}
		\item[(i)] $R^{-m}f(R)$ has an even Taylor expansion at $ R = 0$, converging absolutely for $ R^{-1} > R_0$.
		\item[(ii)]  The function $f(R)$ has the expansion
		\begin{align}\label{ex}
			f(R) = \sum_{r = 0}^{\infty} \sum_{j = 0}^{r_l+ \ell} c_{j r} R^{k - l_1 \cdot r} \log^{j}(R),~~~ c_{j r } \in \C,
		\end{align}
		in an absolute sense for $ R > R_0 $ and such that $ c_{j r} = 0$ except for finitely many $ j \in \N_0$.
	\end{itemize}
\end{Def}
The following useful Lemma is directly implied by  Definition \ref{deff} (see Remark \ref{remark})
\begin{Lemma} Let  $ l_1,l_2 \in \Z_+$ and $ k,m \in \Z$. For $h, \tilde{h} \in \Z$ 
	there holds
	\[
	S^m_{l_1,l_2}(R^k \log^{\ell}(R)) \subset S^m_{l_1,l_2}(R^{k + \tilde{h} \cdot l_1}\log^{\ell - h \cdot l_2}(R)),
	\]
	if we have $  \ell \geq  h \cdot l_1,~ \tilde{h} \geq \max\{h, 0\} $.
\end{Lemma}
\begin{Rem}\label{remark}
	(a)~The sum in Definition \ref{deff} can be written as 
	\[
	\sum_{j = 0}^{ \ell} 	\sum_{r = 0}^{\infty}  c_{j r} R^{k -l_1 \cdot r} \log^{j}(R) + \sum_{j \geq 1} \sum_{\substack{ l_1 r \geq l_2 j \\  r \geq 0}}c_{ j +\ell, r}R^{k - l_1 \cdot r} \log^{ \ell + j}(R)
	\]
	motivating the leading term $ R^k (\log(R))^{\ell}$. We expect Definition \ref{deff} to be fairly  useful  in describing elliptic approximation schemes similar to the one described below.\\[2pt]
	(b) We suppress the parameter $ R_0> 0$ (when it is unambiguous) and in the following, only the spaces 
	\[
	S^m_{2,1}( R^k \log^{\ell}(R)),~  S^m_{1,2}( R^k \log^{\ell}(R))
	\]
	are relevant in dimension $ d =4$ and $ d =3$ respectively. The former is representing expansions of `even order', up to the leading factor $R^k$, throughout the iteration, i.e.  
	\[
	S^m_{2,1}( R^k \log^{\ell}(R)) \subset S^m_{1,1}( R^k \log^{\ell}(R)),
	\]
	in the sense of \eqref{previous-exp} with $ c_{j, 2r +1} = 0$ if  $ r\in \Z_{\geq 0}$.
\end{Rem}
We set $ \eta_0 = 0$. Considering the system \eqref{main-eqn-interior-pre} for real and imaginary parts, we note that for 
\begin{align}
	&\mathcal{L}_+ = -  \partial_R^2  - \frac{d-1}{R}\partial_R - p_c W^{p_c-1}(R),\\
	& \mathcal{L}_- = -  \partial_R^2  - \frac{d-1}{R}\partial_R - W^{p_c-1}(R),
\end{align}
there holds the following.
\begin{Lemma}\label{FS-inner-re} The operator $ \mathcal{L}_{\pm}$ have fundamental solutions $ \Theta^{(d)}_{\pm}(R), \Phi^{(d)}_{\pm}(R) $, where 
	\begin{align}\label{FS}
		&\Phi_{-}^{(d)}(\rho) = W(\rho),~~~~\Phi_{+}^{(d)}(\rho) = ( \frac{d-2}{2} + \rho \partial_{\rho})W(\rho) = W_1(R),\\[2pt]
		& \Theta_-^{(3)}(\rho) = W(\rho)Q_-^{(3)}(R), ~~~~\Theta_-^{(4)}(\rho) = W(\rho)( Q_-^{(4)}(R) + \log(R)), \\[2pt]
		& \Theta_+^{(3)}(\rho) = W^3(R)Q_+^{(3)}(R),~~~~\Theta_+^{(4)}(\rho) = \gamma_+ W^{4}(R)Q_+^{(4)}(R)+ W_1(R) \log(\rho)
	\end{align}
	such that $ P_{\pm}^{(d)}(R) =  R^{d-2} \cdot~ Q_{\pm}^{(d)}(R) $ are even,  real polynomials with $ \deg(P^{(3)}_+ ) = 2,~ \deg(P^{(3)}_+) = \deg(P^{(4)}_-) = 4,~ \deg(P^{(4)}_+) = 6$ and $ P^{(d)}_{\pm}(0) \neq 0$.
\end{Lemma}
\begin{proof}[Proof of Lemma \ref{FS}]  The solutions $ \mathcal{L}_-(W)(R) =  \mathcal{L}_+( W_1)(R) = 0$ are easily verified. Then in $d =3$ dimensions the Ansatz $ \mathcal{L}_-(W \cdot f)(R) = \mathcal{L}_+(W_1 \cdot g)(R) = 0$ directly shows 
	\[
	f''(R) = - \frac{2 W'(R) + \frac{2}{R}W(R)}{W(R)}f'(R),~~g''(R) = - \frac{2 W_1'(R) + \frac{2}{R}W_1(R)}{W_1(R)}g'(R),
	\]
	where we need to exclude $ W_1(R_*) = 0$ for the latter. In both cases we integrate $ f'(R) \sim R^{-2}W^{-2}(R)$ and $g'(R) \sim R^{-2} W_1^{-2}(R)$. In $d=4$ dimensions we set $ \mathcal{L}_-(W \cdot (f + \beta \log))(R) = \mathcal{L}_+(W_1 \cdot (g + \beta \log))(R) = 0$, i.e. 
	\[
	f'(R) \sim R^{-3}W^{-2}(R) - \beta R^{-1},~~g'(R) \sim R^{-3}W_1^{-2}(R) - \tilde{\beta} R^{-1},
	\]
	which have rational primitives by choice of $\beta, \tilde{\beta}$.
\end{proof}
\begin{Rem}\label{anfangsre}
	It is practical to  introduce the notation $ f(R) = O(R^m)~,m \in \Z $ if 
	\begin{align*}
		&R^{-m}f(R) = c_0 + c_1R^{-1} + c_3 R^{-3} +\dots,~~ R > 1,~~\text{or}~~\\
		& R^{-m}f(R) = d_0 + d_1R + d_2 R^{2} +\dots,~~ R < 1, 
	\end{align*}
	respectively. Then we observe the following asymptotics in dimension $ d = 3$
	\begin{align*}
		&\Phi_{\pm}^{(3)}(R) = O(R^{-1}),~ \Theta^{(3)}_{\pm}(\rho)=  O(1),~\text{if}~ R > 1,~ \\
		&\Phi_{\pm}^{(3)}(R)  = O(1),~~ \Theta_{\pm}^{(3)}(\rho)=  O(R^{-1}),~\text{if}~ R < 1,
	\end{align*}
	and in dimension $ d =4 $
	\begin{align*}
		&\Phi_{\pm}^{(4)}(R) = O(R^{-2}),~~ \Theta_{\pm}^{(4)}(\rho) = O(1) + \log(R) O(R^{-2}),~\text{if}~ R > 1,\\
		&\Phi_{\pm}^{(4)}(R)  = O(1),~~ \Theta_{\pm}^{(4)}(\rho) = O(R^{-2}) + \log(R)) O(1),~\text{if}~ R < 1.
	\end{align*}
	Further in the expansions $c_{2m+1} =  d_{2m+1} = 0,~ m \in \N_0$ for all of the above cases.
\end{Rem}

By the variation of constants formula we solve
\[
\mathcal{L}_{\pm}v^{\pm} (R) = f(R),~~ R \in (0, \infty),~~ v^{\pm}(0) = \partial_Rv^{\pm}(0) = 0,
\]
with the operator 
\begin{align}\label{formula1}
	&v_{\pm}(\rho) = C_{\pm} \int_0^{\rho} s^{d-1}G_{\pm}(R,s) f(s)~ds
\end{align}
where $ G_{\pm}(R,s)$ denotes the Greens function
\[
G_{\pm}(R,s) = \Theta_{\pm}^{(d)}(\rho)\Phi_{\pm}^{(d)}(s) - \Theta_{\pm}^{(d)}(s)\Phi_{\pm}^{(d)}(\rho),
\]
and the Wronskian satisfies $w(s) \sim s^{-(d-1)}$. The integral \eqref{formula1} is well defined in $ d = 3,4 $ if for instance $f(s) = O(s^{-1})$ as $ s \ll 1  $ and in the following we set $ C_{\pm} = \f12$.
\subsubsection{The first two corrections}    
We split into
\begin{align}
	\big(-  \partial_R^2  - \frac{d-1}{R}\partial_R\big) \eta_1^{(1)} - p_c W^{p_c-1}(R) \eta_1^{(1)} = \tilde{N}_0^{(1)}(t,R) = - \omega_0 (T - t)^{2 \nu} W(R)\\
	\big(-  \partial_R^2  - \frac{d-1}{R}\partial_R\big) \eta_1^{(2)} - W^{p_c-1}(R) \eta_1^{(2)} = \tilde{N}_0^{(2)}(t,R) = ( \f12 + \nu) (T - t )^{2\nu} W_1(R)
\end{align}
and hence 
\begin{align*}
	&\eta_1^{(1)}(t,r) = - \f12 \alpha_0 (T - t)^{2\nu} \int_0^R s^{d-2}G_+(R,s) W(s)~ds,\\
	&\eta_1^{(2)}(t,r) = \f12 ( \f12 + \nu) (T - t )^{2\nu} \int_0^R s^{d-2}G_-(R,s) W_1(s)~ds.
\end{align*}
Since both $ W_1(R) = O(R^{2-d}),~ W(R) = O(R^{2-d}),~~ W_1(R) = O(1),~W(R) = O(1)$ as $ R > 1,~ R < 1$ respectively we write for $ R \geq R_* > 1$ fix
\begin{align*}
	\int_0^R s^{d-1}G_+(R,s) W(s)~ds =& \int_0^{R_*} s^{d-1}G_+(R,s) W(s)~ds  + \int_{R_*}^R s^{d-1}G_+(R,s)  W(s)~ds
\end{align*}
with 
\begin{align}
	\int_0^{R_*} s^{d-1}G_+(R,s)  W(s)~ds = \begin{cases}
		O(1), & d = 3\\
		O(1) + \log(R) O(R^{-2}), & d =4.
	\end{cases}
\end{align}
We further obtain up to leading order in $ d = 3$ for the second integral
\begin{align*}
	\int_{R_*}^R s^2 G_+(R,s)  W(s)~ds =& ~O(1) \cdot \int_{R_*}^R s^2 \cdot \Phi_{+}^{(3)}(s)W(s)~ds - O(R^{-1}) \cdot \int_{R_*}^R s^2 \cdot  \Theta_{+}^{(3)}(s) W(s)~ds\\
	= & ~O(1) \cdot \int_{R_*}^R~ O(1)~ds - O(R^{-1}) \cdot \int_{R_*}^R~O( s)~ds\\
	=&~ O(1) \cdot ( O(R)) + O(R^{-1}) \cdot ( \log(R) + O(R^2) ).
\end{align*}
Further in $ d = 4$ we have similarly
\begin{align*}
	\int_{R_*}^R s^3 G_+(R,s)  W(s)~ds = & ~(O(1) + \log(R) O(R^{-2})) \cdot \int_{R_*}^R s^3 \cdot s^{-2 -2r}W(s)~ds\\
	&~~ - O(R^{-2}) \cdot \int_{R_*}^R s^{3 }(s^{ - 2r} + \log(s) s^{-2 - 2r }) \cdot W(s)~ds\\[2pt]
	= & ~(O(1) + \log(R) O(R^{-2}))  \cdot ( O(1) + \log(R) )\\
	&~~~~~- O(R^{-2}) \cdot \big[O(R^2) + \log(R) O(1) + \log^2(R)\big].
\end{align*}
The same argument applies to the integral for $ \eta^{(2)}(t,R)$. We also conclude straight forward with the above asymptotics for $G_{\pm}(R,s)$ (in $ s,R$) and $W(R), W_1(R)$ at $R = 0$, that $ \eta(t,R)$ has an even Taylor expansion decaying of order two at $ R = 0$.\\

Hence we obtain
\begin{align*}
	\eta_1 \in (T - t)^{2 \nu} S^2(1; R),~~d = 3,~~\eta_1 \in (T - t)^{2 \nu}  S^2(1; \log(R)).~~~d = 4.
\end{align*} 
In the following we will write $ S^m(R^k \log^{\ell}(R))$ instead of $ S^m(1; R^k \log^{\ell}(R))$. The radius of convergence is inductively observed to be $ R_0 = 1$ in the subsequent arguments.\\
For the next iterate $ \eta_2$, we define the error $e_1(t,R) $ from the previous step 
\begin{align*}
	e_1(t,R) =& \big[(T - t)^{1 + 2\nu} \partial_t + \mathcal{L}_W\big] \eta_1 - \tilde{\mathcal{N}}(\eta_1)\\
	=& (T - t)^{1 + 2\nu} \partial_t \eta_1 + \tilde{N}_0 - \tilde{\mathcal{N}}(\eta_1).
\end{align*}
We then identify a `lowest order' term such that
\[
e_1(t,r) - e_1^0(t,r) = (\sqrt{T -t}\lambda)^{-3}f_1(\alpha_0, \nu, R) + (\sqrt{T - t}\lambda)^{-4} f_2(\alpha_0, \nu, R)  \dots,
\]
and solve
$$ \mathcal{L}_W \eta_2(t,R) = - e_1^0(t,R) $$
with $ \eta_2(t,0) = \partial_R \eta_2 (t,0) = 0$. Hence we need to consider two source terms and for the first 
\begin{align}
	(T - t)^{1 + 2\nu} \partial_t \eta_1 \in \begin{cases}
		2 \nu (T - t)^{4 \nu} \cdot S^2(R) & d = 3\\
		2 \nu (T - t)^{4 \nu} \cdot S^2( \log(R)) & d = 4,
	\end{cases}
\end{align}
where we note the dependence $R = R(t)$ is discarded, since it has already been accounted for in the derivation of \eqref{main-eqn-interior-pre}. Further (slightly abusing the notation) for the second source term we write
\begin{align}\label{source2}
	\tilde{\mathcal{N}}(\eta_1) - \tilde{N}_0 =~& N_1(\eta_1) + N_2(\eta_1) \\ \nonumber
	=~& - \alpha_0 (T-t)^{2 \nu}\eta_1 - i (T-t)^{2\nu}  ( \f12 + \nu)(\f{d-2}{2}  + \rho\partial_{\rho})\eta_1\\ \nonumber
	~~&- |W + \eta_1|^{\f{4}{d-2}} (W + \eta_1) + W^{\f{d+2}{d-2}} + \f{d}{d-2} W^{\f{4}{d-2}}\eta_1 + \f{2}{d-2}W^{\f{4}{d-2}}\bar{\eta_1}.
\end{align}
The third line consists of terms of the form 
\begin{align} 
	\begin{cases}
		(T - t)^{2k\nu } \chi^{(k)} \cdot W^{5 -k },   & 2 \leq k \leq 5,~ ~d = 3,\\[3pt]
		(T - t)^{2j\nu } \chi^{(j)}\cdot  W^{3 -j }, & 2 \leq j \leq 3,~~ d = 4,
	\end{cases}
\end{align}
where we write
$
\chi^{(k)} \cdot W^{5 -k } = \Pi_{i - 1}^k \chi_i(R) \cdot W^{5 -k}(R),~~ \chi_i(R) \in \{  \eta_1, \bar{\eta}_1\}.
$
and similar in case $ d = 4$.
\begin{Rem} From the second and third line in \eqref{source2}, as well as the asymptotics  for $ \Theta_{\pm}, \Phi_{\pm} $ in remark \ref{anfangsre}, the asymptotics for the real and imaginary parts $ \eta_2 = \eta_2^{(1)} + i \eta_2^{(2)}$, are obtained along the same line of calculations.
\end{Rem}
We first consider the case $ d = 3 $ and schematically write  $ \eta_1(t,R) $ instead of $\eta_1^{(1)}, \eta_1^{(2)}$. By the above calculation, we have 
\[
(T - t)^{-2 \nu}\eta_1(t,R) =  O(R) + \log(R) \cdot O(1),~~~ R > 1,
\]
thus 
\[
\chi^{(2)} \cdot W^{3} (R) = O(R^{-1 }) +  \log(R) \cdot O(R^{-1}),~~ R > 1,
\]
and $ \chi^{(2)} \cdot W^{3}  \in S^4(R^{-1})$. For $f \in S^m(R^k \log^{\ell}(R))$ we clearly have 
\[
(\f{d-2}{2}  + \rho\partial_{\rho})f(R) \in  S^m(R^k \log^{\ell}(R)).
\]
Hence, we obtain $ e_1^0 \in (T-t)^{4 \nu} S^2(R)$. Similar in $ d = 4$ we observe
\[
\chi^{(2)} \cdot W (R) = O( R^{-2 }) +  O(R^{-2 }) \cdot \log(R) +  \log^{2}(R) \cdot O(1),~~ R > 1,
\]
and thus $ \chi^{(2)} \cdot W  \in S^4 ( R^{-2} \log(R))$ , which implies
\[
e_1^0 \in (T-t)^{4 \nu} S^2(\log(R)).
\]
Calculating 
\[
v_{\pm} (R) = - \int_{0}^R s^{d -1} G_{\pm}(R,s) e_1^0(t,s) ~ds,
\] 
at $ R =  \infty $, then gives the asymptotics of $ \eta_2$.  For instance, we find 
at $R \gg1 $   ($d = 3$) 
\begin{align*}
	\int_{R_*}^R s^2 G_{\pm}(R,s)  s^{1 -r}~ds =& ~O(1) \cdot \int_{R_*}^R s^2 \cdot s^{-r}~ds - O(R^{-1}) \cdot \int_{R_*}^R s^{2 - 2r} \cdot s^{1 -r}~ds\\
	= & ~O(1) \cdot ( O(R^3) + \log(R)) + O(R^{-1})(O(R^4) + \log(R)),
\end{align*}
and similar in  $ d = 4$ 
\begin{align*}
	\int_{R_*}^R s^3 G_{\pm}(R,s)  s^{-r}\log(s)~ds = & ~(O(1) + \log(R) O(R^{-2})) \cdot \int_{R_*}^R s^{1 -r}\log(s) ~ds\\
	&~~ - O(R^{-2}) \cdot \int_{R_*}^R (s^{ 3 - r} + \log(s) s^{1 - r }) \cdot \log(s)~ds\\[2pt]
	= & ~O(1) \log(R) +  O(R^{-2})\log^2(R) + O(R^{-2})\\
	&~- O(R^{-2}) \cdot \big[O(R^4) \log(R) + O(1) + O(R^2)\log^2(R)  + \log^3(R)\big].
\end{align*}
Integrating the remaining terms in the expansion of $ e_1^0 $ in the  $S$-space implies
\begin{align}
	\eta_2(t,R) \in 
	\begin{cases}
		(T - t)^{4\nu} S^4( R^3) & d = 3,\\[3pt]
		(T - t)^{4 \nu} S^4(R^2 \log(R)) & d = 4.
	\end{cases}
\end{align}
Now the subsequent error $ e_2(t,R)$ has the form
\begin{align*}
	e_2(t,R) =& \big[(T - t)^{1 + 2\nu} \partial_t + \mathcal{L}_W\big] \eta_2^*- \tilde{\mathcal{N}}(\eta_2^*)\\
	=& (T - t)^{1 + 2\nu} \partial_t \eta_2 + e_2(t,R) - e_1^0(t,R) +  \tilde{N}(\eta_1) - \tilde{\mathcal{N}}(\eta_2^*),
\end{align*}
where $ \eta_2^* = \eta_2 + \eta_1 $. By construction of $ \eta_2, e_1^0$ we conclude
\begin{align*}
	e_2(t,R) \in 
	\begin{cases}
		\sum_{(*)}(T - t)^{2 \nu l } S^{2l -2}(R^{2l - 3})& d =3\\[3pt]
		\sum_{(*)}(T - t)^{2 \nu l } S^{2l -2}(R^{2l -4}\log^3(R)) & d =4,
	\end{cases}
\end{align*}
where  $(*)$ sums over $ 3 \leq l \leq \frac{2d+4}{d-2}$.
\subsubsection{The general induction step} Now we describe the induction for the approximating sequence
\[
\eta_k^*(t,R) = \eta_1(t,R) + \eta_2(t,R) + \dots + \eta_k(t,R).
\]
First  we define the expression 
$$ \triangle_kF(\eta) = F(\eta_{k}^*) - F(\eta_{k-1}^*) $$
and the  $k^{\text{th}}$  correction  $ \eta_k = \eta_k^* - \eta_{k-1}^* $ (with $ \eta_0 : = 0$ )  such that

\begin{align}\label{defin-cor}
	\begin{cases}
		\mathcal{L}_W \eta_k(t,R) =~- e^0_{k-1}(t,R)~~& k \geq 2\\[4pt] 
		\eta_k(t, 0) =~ \partial_R \eta_k(t,0) = 0.&~~
	\end{cases}
\end{align}
~~\\
The $k^{\text{th}}$ error function $e_k$ is given by
\begin{align}\label{defin-error}
	e_k(t,R)  = &~ \big[(T - t)^{1 + 2\nu} \partial_t + \mathcal{L}_W\big] \eta_k^*- \tilde{\mathcal{N}}(\eta_k^*)\\[4pt] \nonumber
	= &~ (T - t)^{1 + 2\nu} \partial_t \eta_k + e_{k -1}(t,R) - e_{k-1}^0(t,R) -  \triangle_k\tilde{N}(\eta),~~k \geq 2.
\end{align}
The lower order version $e_k^0$ of $ e_k$ is defined such that $e_k - e_k^0$ has the expansion

\begin{align}\label{def-red-error}
	(e_k - e_k^0)(t,R) = (\sqrt{T -t}\lambda)^{-2(k + 2)}f_{1,k}(R) + (\sqrt{T - t}\lambda)^{-2 (k + 3)} f_{2,k}(R) +  \dots,
\end{align}
~~\\
where $ \sqrt{T -t}\lambda(t) =  (T - t)^{-  \nu}$. The error $e_k^0(t,R) $ is required to consist of the \emph{minimal} set of terms such that \eqref{def-red-error} holds.
We now  state the induction.
\begin{Lemma} \label{main-inner-Lemma}Let  $\eta_0 = 0$ and $ \eta_1(t,R)$ be as in the previous subsection. Then there exists unique solutions $ \eta_k(t,R) $ of  \eqref{defin-cor} -  \eqref{def-red-error} with $ k \geq 1$,  such that
	
	\begin{align}\label{corrections-rep}
		\eta_k(t,R) \in
		\begin{cases}
			(T-t)^{2 k\nu} S^{2k}_{1,2}(R^{2k -1}), & d =3\\[6pt]
			(T-t)^{2 k\nu} S^{2k}_{2,1}(R^{2k -2}\log(R)) & d = 4,
		\end{cases}
	\end{align}
	Further $ e_k(t,R) $  satisfies 
	$
	e_k(t,R)  = \sum_{l}~ (T -t)^{2  \nu l} \chi_l(R)
	$
	with   
	
	\begin{align}\label{error-rep}
		\chi_l(R) \in 
		\begin{cases}
			S^{2l -2}_{1,2}(R^{2l - 3})& d =3\\[6pt]
			S^{2l -2}_{2,1}(R^{2l -4} \log(R)) & d =4.
		\end{cases}
	\end{align}
	and $ l = k+1 , \dots , k_d$ where $  k_d = \frac{d + 2}{d -2} k $.
\end{Lemma}
\begin{Rem}
	Since we are concerned with leading order asymptotics (as stated in Definition \eqref{deff}), we interchange the `even' S-space in dimension $ d = 4$  with
	\begin{align}
		&\eta_k(t,R) \in (T-t)^{2 k\nu} S^{2k}(R^{2k -2}  \log(R)),\\
		&\chi_l(R) \in  S^{2l -2}(R^{2l -4}\log(R)).
	\end{align}
	The statement in Lemma \ref{main-inner-Lemma} follows  however directly from the `even' order expansion of the Greens function in Lemma \ref{FS} . 
\end{Rem}
\begin{proof}[Proof of Lemma \ref{main-inner-Lemma}]
	Assume \eqref{corrections-rep}, \eqref{error-rep} and the formula for the error hold true for any $ 1 \leq \tilde{k} \leq k$. For simplicity we write
	\[
	\eta_{ k +1}(t,R) = - \f12 \int_0^R s^{d-1} G_{\pm}(s,R) e_k^0(t,s)~ds.
	\]
	Now by definition of $e_k^{0}(t,R)$ and \eqref{error-rep}, we have  $e_k^{0}(t,R) =  (T-t)^{2\nu (k+1)} \chi_{k+1}(R)$, where $ \chi_{k+1} \in S^{2}(R^{2k - 1})$  in $ d = 3$ and $ \chi_{k+1} \in S^{2}(R^{2k -2 } \log(R))$  in $ d = 4$. Then we split for $ R > R_* > 1 $ as before
	\[
	\int_0^R s^{d-1} G_{\pm}(s,R) e_k^0(t,s)~ds =  \int_0^{R_*} s^{d-1} G_{\pm}(s,R) e_k^0(t,s)~ds + \int_{R_*}^R s^{d-1} G_{\pm}(s,R) e_k^0(t,s)~ds.
	\]
	The first integral corrects initial values at $ R = R_*$ and has the asymptotics (as $ R \gg 1$)
	\begin{align*}
		&\int_0^{R_*} s^{d-1} G_{\pm}(s,R) e_k^0(t,s)~ds =	(T-t)^{2\nu (k+1)} O(1),~~~ d= 3,\\
		& \int_0^{R_*} s^{d-1} G_{\pm}(s,R) e_k^0(t,s)~ds  = (T-t)^{2\nu (k+1)} \big(O(1) + \log(R) \cdot O(R^{-2})\big),~~~d=4.
	\end{align*}
	For the second integral, we expand 
	\begin{align}\label{expanding}
		\chi_{k+1}(R) = O( R^{2k - 1 }) + \sum_{j \geq 1}  O( R^{2 k -1 - j}) \log^j(R),
	\end{align}
	where the latter sum is finite. Simplifying the notation as above for $ \eta_1, \eta_2$, we calculate in $ d = 3$ for the leading order term
	\begin{align*}
		\int_{R_*}^R s^2& G_{\pm}(R,s)  s^{2k - 1 -r}~ds =  ~O(1) \cdot \int_{R_*}^R s^2 \cdot s^{2k -2 -r}~ds - O(R^{-1}) \cdot \int_{R_*}^R s^{2} \cdot s^{2k - 1 -r}~ds,
	\end{align*}
	and hence
	\begin{align*}
		\int_{R_*}^R s^2 G_{\pm}(R,s)  s^{2k - 1 -r}~ds =& ~O(1) \cdot \int_{R_*}^R s^2 \cdot s^{2k -2 -r}~ds - O(R^{-1}) \cdot \int_{R_*}^R s^{2 - 2r} \cdot s^{2k - 1 -r}~ds\\
		= & ~O(1) \cdot ( O(R^{2k +1 }) + \log(R)) + O(R^{-1})\cdot (O(R^{2k + 2}) + \log(R)),
	\end{align*}
	and similar in  $ d = 4$ 
	\begin{align*}
		\int_{R_*}^R s^3 G_{\pm}(R,s)  s^{2k -2-r}\log(s)~ds = & ~(O(1) + \log(R) O(R^{-2})) \cdot \int_{R_*}^R s^{2k -1 -r}\log(s) ~ds\\
		&~~ - O(R^{-2}) \cdot \int_{R_*}^R (s^{ 2k +1 - r} + \log(s) s^{2k -1 - r }) \cdot \log(s)~ds\\[2pt]
		= &  O(R^{2k}) \log(R) + O(R^{2k-2})\log^2(R) + O(R^{-2})\log^3(R) + O(1).
	\end{align*}
	For the second sum in the expansion \eqref{expanding}, we need to integrate  $ R^{2k -1 - j -r} \log^j(R),~~ r \geq 0$, for finitely many $ j \geq 1$, i.e. it follows similarly
	\begin{align*}
		\int_{R_*}^R &s^2 G_{\pm}(R,s)  s^{2k -1 - j -r} \log^j(R)~ds\\
		& = ~O(1) \cdot \int_{R_*}^R s^{2k -j -r} \log^j(s)~ds - O(R^{-1}) \cdot \int_{R_*}^R s^{2k + 1 -j  -r} \log^j(s)~ds\\
		= & ~\log^j(R) \cdot ( O(R^{2k +1 -j }) + \log(R)) +O(1)\\
		&~+ O(R^{-1}) \log^j(R) \cdot (O(R^{2k + 2 -j }) + \log(R)),
	\end{align*}	 
	~~\\
	in the dimension $ d=3$ and where we assumed $ j \leq 2k$. If $ j > 2k $ the latter  two lines simply read
	$
	\log^j(R) O(R^{2k +1 -j })  + O(1). 
	$
	Hence  $ \eta_{k+1} \in S^{2k +2 }(R^{2k +1})$ in $ d =3$ as required in \eqref{corrections-rep}. The case of $ d =4$ is very similar and  $ \eta_{k+1} \in S^{2k +2 }(R^{2k } \log(R))$,  which again  verifies \eqref{corrections-rep}. We omit the details here and turn to the error function \eqref{defin-error}, i.e.
	~\\
	\begin{align}\label{er}
		e_{k +1 }(t,R) = &~ \big[(T - t)^{1 + 2\nu} \partial_t + \mathcal{L}_W\big] \eta_{k +1}^*- \tilde{\mathcal{N}}(\eta_{k +1}^*)\\[4pt] \nonumber
		= &~ (T - t)^{1 + 2\nu} \partial_t \eta_{k+1} + e_{k }(t,R) - e_{k}^0(t,R) -  \triangle_{k +1}\tilde{N}(\eta).
	\end{align}
	By assumption and the calculation for $ \eta_{k+1}$, we have in $ d =3 $
	\begin{align}\label{time-update}
		&(T - t)^{1 + 2\nu} \partial_t \eta_{k+1} \in (T -t)^{2\nu(k +2)}S^{2k + 2}(R^{2k + 1})\\ 
		&e_{k }(t,R) - e_{k}^0(t,R) \in \sum_{l} (T -t)^{2\nu l}S^{2l -2} ( R^{2l -3}),~~l = k+2, \dots 5k,
	\end{align}
	and for the case  of dimension $ d = 4$ we similarly have 
	\begin{align}\label{time-update2}
		&(T - t)^{1 + 2\nu} \partial_t \eta_{k+1} \in (T -t)^{2\nu(k +2)}S^{2k + 2}(R^{2k } \log(R))\\ 
		&e_{k }(t,R) - e_{k}^0(t,R) \in \sum_{l} (T -t)^{2\nu l}S^{2l -2} ( R^{2l -4}\log^3(R)),~~l = k+2, \dots 3k,
	\end{align}
	
	Further for the last term  in  \eqref{er}, we apply the following simple Lemma
	\begin{Lemma}[power law]\label{power-law} Let $  j_1, j_2, \dots j_m \in \N,~~m =  1, 2, \dots, \frac{d + 2}{d -2}$ and\\
		$ j = j_1 + j_2 + \dots + j_m$. If we choose functions
		\begin{align*}
			&f_i(R) \in S^{2i}(R^{2i -1}) ,~~ ~ i = j_1, j_2, \dots, j_m,~~~d =3,\\
			& f_i(R) \in S^{2i}(R^{2i -2} \log(R)) ,~~~ i = j_1, j_2, \dots, j_m,~~~ d =4,
		\end{align*}
		then with $ W(R) = ( 1 + \frac{R^2}{d (d -2)})^{\frac{2-d}{d}}$ there holds
		\begin{align*}
			f_{j_1}(R ) \cdot f_{j_2}(R ) \cdots f_{j_m}(R) \cdot W^{\frac{d + 2}{d-2} -m}(R) \in 
			\begin{cases}
				S^{2j} ( R^{2j -5} ) & d = 3\\
				S^{2j}( R^{2j- 6}\log^m(R) ) & d =4.
			\end{cases}
		\end{align*}
	\end{Lemma}
	\begin{proof} The statement follows immediately from Definition \ref{deff} of the $S$-space and the Cauchy products for the expansions of $ f_{j_1}(R) , \dots f_{j_m}(R) $ were $ R \ll1 $ and $ R\gg 1$ respectively. Then in $ d = 3$, considering the leading order, we conclude the product is an element of 
		\begin{alignat*}{1}
			&S^{2j_1}(R^{2j_1 -1})\cdot  S^{2j_2}(R^{2j_2 -1})\\
			&  \hspace{2cm} ~ \times~S^{2j_3}(R^{2j_3 -1})  \cdots S^{2j_m}(R^{2j_m -1}) \cdot W^{5 -m }(R)\\[2pt]
			&\subseteq S^{2j }( R^{2j -m}) \cdot W^{5 -m}(R) \subseteq S^{2j }( R^{2j -5 }).
		\end{alignat*}
		Likewise in $ d  = 4$ we have an element of 
		\begin{alignat*}{1}
			&S^{2j_1}(R^{2j_1 -2} \log(R))\cdot S^{2j_2}(R^{2j_2 -2} \log(R))\\
			&  \hspace{3cm} ~\times~ \cdots S^{2j_m}(R^{2j_m -2}\log(R)) \cdot W^{3 -m }(R)\\[2pt]
			&\subseteq  S^{2j }( R^{2j -2 m }\log^{m}(R)) \cdot W^{3-m}\subseteq  S^{2j }( R^{2j -6 }\log^{m}(R)).
		\end{alignat*}
	\end{proof}
	Now  the interaction term splits into  
	$
	\triangle_{k +1}\tilde{N}(\eta) = \tilde{N}_1(\eta_{ k+1}) +  \Delta_{k+1}\tilde{N}_2(\eta),
	$
	where 
	\begin{align*}
		\tilde{N}_1(\eta_{ k+1})  =& - \omega_0 (T-t)^{2 \nu}\eta_{k +1} - i (T-t)^{2\nu}  ( \f12 + \nu)(\f{d-2}{2}  + \rho\partial_{\rho})\eta_{k +1},
	\end{align*}
	and thus 
	\begin{align*}  \tilde{N}_1(\eta_{ k+1})  \in 
		\begin{cases}
			(T - t)^{2\nu (k +2)} S^{2k + 2}(R^{2k + 1}) & d = 3\\[2pt]
			(T - t)^{2\nu (k +2)} S^{2k +2}(R^{2k }\log(R)) &  d = 4.
		\end{cases}
	\end{align*}
	The term $ \Delta_{k+1}\tilde{N}_2(\eta) $ consists of expressions of the form
	\begin{align*}
		&  \tilde{\eta}_{k +1} \cdot \tilde{\eta}_{i_1} \cdots  \tilde{\eta}_{i_m}\cdot W^{\f{4}{d -2} -m},~~~\tilde{\eta} \in \{\eta, \bar{\eta} \}
	\end{align*}
	where $ 1 \leq  i_1, i_2, \dots, i_m \leq k +1,$ and $1 \leq m \leq \frac{4}{d-2}$.
	By assumption Lemma \ref{power-law} implies
	\begin{align*}
		& \tilde{\eta}_{k +1} \cdot \tilde{\eta}_{i_1} \cdots  \tilde{\eta}_{i_m}\cdot W^{4 -m} \in (T - t)^{2\nu j} S^{2j }( R^{2j-5 }),~~ d =3,\\
		& \tilde{\eta}_{k +1} \cdot \tilde{\eta}_{i_1} \cdots  \tilde{\eta}_{i_m}\cdot W^{2 -m} \in (T - t)^{2\nu j} S^{2j }( R^{2j-6 }\log^{m + 1}(R)),~~ d = 4,
	\end{align*}
	where $j : = i_1 + i_2 \dots + i_m + k +1$.
	The interaction terms grow slower  to leading order in $R$ than represented in \eqref{error-rep}. However, not so for the `lowest ' order term attained in \eqref{time-update} and the linear contribution $ \tilde{N}_1 $. For $ d = 4 $ especially, we note the embeddings
	\begin{align*}
		&S^{2j }( R^{2j-6 }\log^{2}(R)) \subseteq  S^{2j }( R^{2j-5 }\log(R))\\
		&S^{2j }( R^{2j-6 }\log^{3}(R))  \subseteq   S^{2j }( R^{2j-4 }\log(R)).
	\end{align*}
	which are straight forward from the definition \eqref{deff}.
	The `highest' order term is given by the case of $ i _{\ell} = k+1,~ \ell = 1, \dots, \f{4}{d-2}$, which attains the upper  sum index in  \eqref{error-rep}.

\end{proof}
The following estimates are an immediate consequence of Lemma \ref{main-inner-Lemma}. First we note in the inner region 
\[
(\star)~~ 0 \leq R \leq C \lambda(t) (T-t)^{\f12 + \epsilon } = C (T -t)^{\epsilon - \nu},
\]
and  observe
\begin{Lemma}\label{simple-Lemma} Let  $ T-t \in [0,1) $ and $ R \in [0, C (T -t)^{\epsilon - \nu})$, then
	\begin{align}
		& | R^{- j}\partial_R^i (e_k(t,R))| \leq C_{k, j, i} (T-t)^{2 \nu (k+1)} \langle R \rangle^{2k -1 - i - j},~~~ d =3,\\
		& | R^{- j}\partial_R^i (e_k(t,R))| \leq C_{k, j, i} (T-t)^{2 \nu (k+1)} \langle R \rangle^{2k -2 - i - j}( 1+  \log(1 + R)),~~~ d =4,
	\end{align}
	where  $ \langle R\rangle = (1 +R^2)^{\f12}$,~ $ j + i \leq 2k$.
\end{Lemma}
\begin{Rem}
	The constants $ C_{k,j,i}= C(k,j,i, |\alpha_0|,\nu)$ where the dependence on $ |alpha_0|, \nu $ is analytic. 
\end{Rem}
\begin{proof}
	By Lemma \ref{main-inner-Lemma} for $ d = 3$ the error $ e_k(r,R)$ consists of  $ (T-t)^{2\nu l} \chi_l(R)$ such that
	\[
	\partial_R^m\chi_l(R) = R^{-m -3} \cdot \big(O(R^{2 l}) +  \sum_{j =1}^{k_l}O(R^{2l -j}) \log^j(R) \big),~~ m \geq 0,~R \gg 1.
	\]
	For the leading order $O(R^{2l})$ where $ l = k+1, \dots , k_d$ we then (by $(\star)$) use
	\begin{align*}
		(T-t)^{2\nu l} R^{2l -m -3} \leq&\;\tilde{C}_k (T-t)^{2\nu(k+1)}R^{2k -m -1}  (T-t)^{2 \epsilon( l - k -1)}\\
		\leq& \;C_k (T-t)^{2 \nu(k+1)}R^{2k -m -1} 
	\end{align*}
	combined with the absolute convergence  of $\chi_l(R)$ where $ R > 1$.
	In the case of the sum on the right $R^{-m -3}\cdot O(R^{2l -j}) \log^j(R)$ for $j \geq 1$, we observe the same upper bound  via estimating
	\begin{align*}
		(T-t)^{2\nu l} R^{2l -m - j -3} \log^j(R) \leq& \;\tilde{C}_k (T-t)^{2 \nu(k+1)}R^{2k -m -1}  (R^{-j} \log^j(R))\\
		\leq &\;\tilde{C}_{k j} (T-t)^{2 \nu(k+1)}R^{2k -m -1}. 
	\end{align*}
	Likewise in $ d =4$ we have for $m \geq 0,~R > 1$
	\[
	\partial_R^m\chi_l(R) = R^{-m -4} \cdot \big(O(R^{2 l})( 1 + O(\log(R))) +  \sum_{j =1}^{\tilde{k}_l}O(R^{2l -j}) \log^{j+1}(R) \big),
	\]
	from which we infer the desired bound as above in the region $ R \gg1 $. Similarly we obtain in dimension $ d = 3,4$
	\begin{align*}
		&\partial_R^m\chi_l(R) = R^{-m} \cdot O(R^{2l -2}),~~ m \geq 0,~R \ll 1,\\[3pt]
		& |(T-t)^{2 \nu l} \partial_R^m\chi_l(R)| \leq C_K (T -t)^{2 \nu (k+1)} R^{2k-m}.
	\end{align*}
	Hence taking $ m \leq 2k$,~ $ i \leq 2k -m$, we infer the bound $| R^{-i}\partial_R^m \chi_l(R) | \lesssim_{k, m, i} 1 $ in the region $ R \ll1 $. The estimates are then put together noting $ T-t \in [0, 1)$ and the  $\chi_l(R)'s$ are analytic functions on $[0, \infty)$.
\end{proof}
\begin{Def}
	We set for the interior approximation
	\begin{align}
		&u^{N_1}_{In}(t,R) := W(R) + \eta_{N_1}^*(t,R),~ N_1 \gg 1,\\[4pt] \label{complete-error-inner}
		& e^{N_1}_{In}(t,R) :=  i (T-t)^{1 + 2\nu} \partial_t u^{N_1}_{In} + \Delta u^{N_1}_{In}+ |u^{N_1}_{In}|^{\f{4}{d-2}}u^{N_1}_{In} - \alpha_0 (T-t)^{2\nu} u^{N_1}_{In}\\[2pt] \nonumber
		&~~~~~~~~~~~~~~~~~~~~ + i (T-t)^{2\nu}  ( \f12 + \nu)(\f{d-2}{2}  + \rho\partial_{\rho})u^{N_1}_{In} = e_{N_1}(t,R),
	\end{align} 
	where $ N_1\in \Z_+$ will be fixed below  and we note the identity \eqref{complete-error-inner} holds again up to identifying real and imaginary parts $\eta = \eta^{(1)} + i \eta^{(2)}$. 
\end{Def}
For a fixed $ T > 0$ and $ t \in [0, T)$, we set the indicator function
\[
\chi_{In}(t,R) = \begin{cases}
	1 & 0 \leq R \leq C (T -t)^{\epsilon - \nu} \\
	0 & \text{otherwise}.
\end{cases}
\]
\begin{Corollary}\label{estimates} For $ \alpha_0 \in \R,~ \nu > \epsilon_1 > 0$  there exists $  0 < T_0(\alpha_0, \nu,k) \leq 1 $ such that for all $ 0 < T \leq T_0$ the functions
	$ u^{(k)}_{In}, e^{(k)}_{In}$ with $ k  \gg1 $ satisfy for all $ t \in [0,T)$
	\begin{align}\label{inner-est1}
		&\|  \chi_{In}(t,R) \cdot (W -u^{(k)}_{In})  \|_{L^{\infty}_R} \leq C_{\nu, |\alpha_0|} (T - t )^{\nu}\\[2pt]\label{inner-est2}
		&\|  \chi_{In}(t,R) \cdot R^{-j}\partial_R^{i}(W -u^{(k)}_{In})  \|_{L^{\infty}_R} \leq C_{\nu, |\alpha_0|} (T - t )^{2\nu},~~~1 \leq j + i \leq 2\\[2pt]\label{inner-est3}
		&\|  \chi_{In}(t,R) \cdot R^{-j}\partial_R^{i}(W -u^{(k)}_{In}) \|_{L^{2}(R^{d-1}dR)} \leq C_{\nu, |\alpha_0|} (T - t )^{\nu(j + i - \f{d-2}{2}) },~~~0 \leq j + i \leq 2\\[2pt]\label{inner-est4}
		&\|  \chi_{In}(t,R) \cdot R^{-j}\partial_R^{i}e^{(k)}_{In}\|_{L^{2}(R^{d-1}dR)} \leq C_{\nu, |\alpha_0|} (T - t )^{\f{6-d}{2}\nu + 2  \epsilon_1 k  },~~~0 \leq j + i \leq 2k
	\end{align}
\end{Corollary}
\begin{Rem}\label{rem} For the bounds in the case $ d =4 $, we use the embedding $ S^{2l}(R^{2l-2}\log(R)) \subset S^{2l}(R^{2l-1}) $. The  upper bounds in $  \eqref{inner-est1}, \eqref{inner-est3}, \eqref{inner-est4}  $ may be improved by the factor
	\[
	(T - t )^{\nu}( 1+ |\log(T-t)|).
	\]
\end{Rem}
\begin{proof}[Proof of Corollary \ref{estimates}]  The $L^{\infty}_R $ bounds are obtained as above in the proof of Lemma \ref{simple-Lemma}. By construction $ W - u_{In}^{(k)}$ has the form $ (T - t)^{2\nu l} \chi_l(R),~ l = 1, \dots, k$, where $ \chi_l \in S^{2l}(R^{2l -1})$ or $ \chi_l \in S^{2l}(R^{2l -2}\log(R))$  in dimension $d = 4$.
	Especially, we observe for the leading order in the  $ R \gg 1 $ regime 
	\begin{align*}
		&(T -t)^{2 \nu l} R^{2l -1} \lesssim (T -t)^{(2l -1) \epsilon + \nu}  \lesssim (T-t)^{\nu},\\
		&(T -t)^{2 \nu l} R^{2l -1 - j -i} \lesssim (T -t)^{( \nu - \epsilon) ( j + i -1)} \cdot (T -t)^{2 \nu} \lesssim (T-t)^{2\nu},~ 1 \leq i+j.
	\end{align*}
	In order to prove the estimate where $ R \ll 1$, we restrict to $1 \leq i+j \leq 2 l $, hence since $ l \geq 1$ we require $1 \leq i+j \leq 2  $. Likewise, integrating the leading order term where $ R \gg1 $, we infer for a large $ R_0 >1 $ 
	\begin{align}\label{intt}
		\bigg( \int_{R_0}^R O( s^{4l - 2-2j - 2i}) s^{d-1}~ds \bigg)^{\f12} &= O( R^{2l - j -i + \f{d-2}{2}}).
	\end{align}
	Hence for $ l = 1, \dots , k$
	\begin{align*}
		(T -t)^{2 \nu l } \| \chi\{  R \gg 1 \} \cdot O( R^{2l - 2-2j - 2i}) \|_{L^2(R^{d-1}dR)} \lesssim_l  (T -t)^{2\epsilon l} (T -t)^{( \epsilon - \nu)(  \f{d-2}{2} -j -i)},
	\end{align*}
	from which we infer \eqref{inner-est3}. Now in dimension $ d = 4$ we estimate $ R^{-1}\log(R) \lesssim1 $, see the remark \ref{rem} above. For the latter estimate \eqref{inner-est4}, we integrate the upper bound in Lemma \ref{simple-Lemma} and infer from taking $l = k$ in \eqref{intt}
	\begin{align*}
		(T -t)^{2 \nu (k +1)} \| \chi\{  R \gg 1 \} \cdot O( R^{2k - 2-2j - 2i}) \|_{L^2(R^{d-1}dR)} \lesssim (T -t)^{2 \nu + 2\epsilon k}  (T -t)^{( \epsilon - \nu)( \f{d-2}{2} - j -i)},
	\end{align*}
	which in turn implies \eqref{inner-est4}.
\end{proof}
\subsection{Self-similar region $ (T -t)^{\f12 + \epsilon_1  }  \lesssim r \lesssim (T -t)^{\f12 - \epsilon_2  }  $}\label{subsec:SS}~For $ \tilde{C} > C^{-1} > 0$ large and $ 0 < \epsilon_2 < 1 $  we restrict to the region
\begin{align}\label{condition}
	S = \big\{ (t,r) ~|~	\frac{1}{\tilde{C}} (T -t)^{ \epsilon_1} \leq r (T-t)^{- \f12} \leq \tilde{C} (T -t)^{- \epsilon_2 } \big\}
\end{align}
and then consider \eqref{NLSequation} with 
\begin{align}\label{ansatzSS}
	u(t, y) = e^{i \alpha(t)} (T -t)^{-\f{d-2}{4}} w(t, y),~~  y = r(T-t)^{- \f12}.
\end{align}
The functions $w(t,y)$  solves
\begin{align}\label{SelfS}
	& i (T -t) \partial_t w = - (\mathcal{L}_S + \alpha_0)w - |w|^{\frac{4}{d-2}}w,\\[3pt] \nonumber 
	& \mathcal{L}_S = \partial_{yy} + (d-1)\tfrac{1}{y} \partial_y + i\big(\tfrac{d-2}{4} + \tfrac12 y \partial_y\big).
\end{align}
In order to approximate solutions of \eqref{SelfS} to high order, we intend to inductively add corrections to a solution of
\[
i (T -t) \partial_tA_0  = - (\mathcal{L}_S + \alpha_0) A_0
\]
in the variable $(t,y)$, i.e. we calculate
\begin{align*}
	&A_0(t,y) + A_1(t,y) + A_2(t,y) + \dots + A_N(t,y),
\end{align*}
where $ N \gg1 $  is potentially large. We set the corrections to be of the form
\[
A_{n}(t,y) = (T - t)^{\nu( n + \f{4-d}{2})} \sum_{j \leq n} \big(\log(y) - \nu \log(T -t)\big)^{j} A_{n,j}(y)
\]
which will be an effective choice of the $(t,y)$ dependence of $ A_k(t,y) $ based on a \emph{formal} asymptotic matching argument. \\[5pt]
\emph{Asymptotic matching heuristics}.~The  intersection $ I \cap S$ of the inner region \eqref{condition-inner} and \eqref{condition}  implies 
$$ R \sim (T-t)^{\epsilon_1  - \nu},~ y \sim (T-t)^{ \epsilon_1},~~ (t,r) \in I \cap S.$$
We thus derive an asymptotic expansion   for 
$
u_{In}^k(t,R) = u_{In}^k(t, (T-t)^{- \nu} y)
$
in the variable $y  = R \cdot (T-t)^{\nu}$ where $ 0 < y \ll1,~T - t \ll1$ from the  previous $R \gg1 $ asymptotics .\\[3pt]
First, we proceed by expanding $\lambda(t)^{\f{d-2}{2}} ( W(R) + \eta_N^*(t,R))$ with $ \eta_k \in (T-t)^{2\nu k} S^{2k}(R^{2k-1})$ in order to demonstrate the general structure of the expansion and the corrections in this Section. Then we will use the more accurate information provided by Lemma \ref{main-inner-Lemma}. Hence 
\begin{align}\label{transi}
	&\lambda(t)^{\f{d-2}{2}} ( W(R) + \eta_N^*(t,R))\\ \nonumber
	=&~ (T -t)^{-\f{d-2}{2}( \f12 + \nu) }\bigg( \sum_{l = 0}^N (T -t)^{2 \nu l }  \sum_{j = 0}^{J_l} \log^{j}(R)O(R^{2l -1 -j})\bigg)\\ \nonumber
	=&~ (T -t)^{-\f{d-2}{2}( \f12 + \nu) }\bigg( \sum_{l = 0}^N (T -t)^{2 \nu l } \sum_{j = 0}^{J_l}\sum_{r \geq 0}  c_{r,j}^{(l)} \big(\log(y) - \nu \log(T -t)\big)^{j} R^{2l -1 -j -r}\bigg)
\end{align}
We assume $ J_l$ is non-decreasing in $ l \in \Z_+$ and for $ \tilde{N} >  J_{N}$ we write
\begin{align}
	\sum_{r = 0}^{\infty} c_{r,j}^{(l)}R^{2l -1 -j -r} =&~  \sum_{ r = 0}^{\tilde{N} - j} c_{r,j}^{(l)}R^{2l -1 -j -r} +  O (R^{2 l -2 - \tilde{N} })\\ \nonumber
	= &~\sum_{ n = j}^{\tilde{N} } c_{n -j,j}^{(l)}y^{2l -1 -n} (T-t)^{- \nu(2l -1 -n)}  +  O (R^{2 l -2 - \tilde{N} }).
\end{align} 
To be precise, we infer for $ \tilde{N} \gg1$ large enough
\begin{align}\label{transition}
	&\lambda(t)^{\f{d-2}{2}} ( W(R) + \eta_N^*(t,R))\\ \nonumber
	=&~ (T -t)^{-\f{d-2}{4}}\bigg( \sum_{n = 0}^{\tilde{N}}  (T -t)^{\nu( n  - \f{d-4}{2})} \sum_{j \leq n} \sum_{l = 0}^N c_{n-j, j}^{(l)}\big(\log(y) - \nu \log(T -t)\big)^{j}  y^{2l -1-n}\bigg)\\ \nonumber
	&+~ (T-t)^{- \frac{d-2}{2}( \f12 + \nu)} R^{-2 - \tilde{N}} \sum_{l = 0}^{N} y^{ 2l} \sum_{j = 0}^{J_l} \log^j(R) f_{j, l , \tilde{N}} (R),
\end{align}
where $ f_{j, l \tilde{N}} \in C^{\infty}$ are bounded (with bounded derivatives). Note in particular, since the third line is small as $0 < T-t \ll1$,  we derive an Ansatz from the second line.
In the second line we set $ c_{n -j, j }^{(l)} = 0$ if $ J_l < j$, i.e. especially $  c_{n -j, j}^{(0)} = 0 $ if $ j > 0$.  The case $j = l =0$ corresponds to the (even order) expansion of $R^{d-2}\cdot W(R)$, i.e. 
\begin{align*}
	R\cdot W(R) = c_{0,0}^{(0)}+ R^{-2}c_{2,0}^{(0)} + R^{-4}c_{4,0}^{(0)} + \dots,~~~ d = 3,\\
	R^2 \cdot W(R) = c_{1,0}^{(0)}+ R^{-2}c_{3,0}^{(0)} + R^{-4}c_{5,0}^{(0)} + \dots,~~~ d = 4.
\end{align*}
~~\\[2pt]
Now Lemma \ref{main-inner-Lemma} implies that the coefficients $ c_{r, j }^{(l)}  = 0 $ if $r < j$, i.e. in dimension $ d =3$, we restrict to $ j < \tfrac{n}{2}$ for the expansion of \eqref{transition}. Further in dimension $ d =4 $ there holds $ c_{r, j }^{(l)}  = 0 $ if $ n = r + j $ is an even number.
The suggested   ansatz for the solution $w(t,y)$ under \eqref{condition} hence consists of
\begin{align}\label{AnsatzSS}
	&\sum_{n = 0}^{\tilde{N}}  (T -t)^{\nu( n +1 - \f{d-2}{2})} \sum_{j \leq n} \big(\log(y) - \nu \log(T -t)\big)^{j} A_{n,j}(y),\\[5pt]
	\label{second-SS-cond}
	&A_{n,j}(y) = \sum_{l\geq 0} d_{n,j, l}\;y^{2l -1 -n},~~ 0 < y \ll1.
\end{align}
and such that
\begin{align*}
	&A_{n,j}(y) = 0,~~\text{if}~~j > \tfrac{n}{2},~~ d = 3,\\
	& A_{n,j}(y) = 0,~~\text{if}~~n \equiv 0~\text{mod}~2,~~ d = 4.
\end{align*}
The calculation  \eqref{transition} now gives us a clear idea for defining the $A_n(t,y) $ corrections with a sufficiently `good' error in the following iteration scheme. This is clarified below in Lemma \ref{consist} and through Lemma \ref{Lemma-solut-inner1}. Essentially, the asymptotic matching analysis suggests an approximation of \eqref{SelfS} (to arbitrary order) which, for any number of corrections, coincides with the interior approximation up to a well controlled decaying error in the intersection $I \cap S$ domain.\\
~~\\
\emph{Setting up the iteration}.  For $ n \in \N_0$ let 
\begin{align}\label{ansatz-for-itSS}
	A_n(t, y) = (T -t)^{\nu( n +1 - \f{d-2}{2})} \sum_{j \leq n} \big(\log(y) - \nu \log(T -t)\big)^{j} A_{n,j}(y).
\end{align}
Let us first restrict to dimension $ d = 3$, then initially we have $ A_0(t, y) = (T -t)^{\f{\nu}{2}}  A_{0,0}(y)$, where we solve
\begin{align}\label{start1}
	(\mathcal{L}_S + \mu_0)A_{0,0}(y) = 0,\;\; \mu_0 = \alpha_0  - i \nu \tfrac{1}{2},
\end{align}
with \eqref{second-SS-cond}. Similarly as in the previous Section \ref{subsec:inner}, we construct corrections for $A_0$, i.e.
\[
A_N^*(t,y) = (T -t)^{\f{\nu}{2}}  A_{0,0}(y) + A_1(t,y) + A_2(t,y) +\dots + A_N(t,y).
\]
The error satisfies
\begin{align}{\tiny }
	e_0(t,y) =& \big(i (T -t) \partial_t  + (\mathcal{L}_S + \alpha_0)\big)A_0(t,y) + |A_0(t,y)|^{4} A_0(t,y)\\ \nonumber
	=& (T - t)^{\f52 \nu} |A_{0,0}(y)|^{4} A_{0,0}(y).
\end{align}
Following the ansatz \eqref{ansatz-for-itSS} and neglecting the logarithmic factors, we subsequently only remove the $ (n+1)$st order terms of the form $ (T -t)^{\nu( n + \f32)} f(y)$ from the error $e_n(t,y)$ in the $n$-th iteration. Precisely, up to logarithmic factors,  we set $e_n^0(t,y) $ as the minimal set of terms, such that
\[
e_n(t,y) = e_n^0(t,y) + o((T -t)^{\nu(n + \f32)})f(y),
\]
for a function $ f(y)$ analytic on $(0,\infty)$ up to a pole singularity at $ y = 0$. Especially, for $ n = 0$, we choose $ e_0^0(t,y) = 0$ and thus solve
\begin{align}\label{start2}
	(\mathcal{L}_S + \mu_1)A_{1,0}(y) = 0,\;\; \mu_1 = \alpha_0  - i \nu\tfrac{3}{2}.
\end{align}
The  solutions of \eqref{start1} and \eqref{start2} are given below in Lemma \ref{FS-Inner}. Then
\begin{align}{\tiny }
	e_1(t,y) =& \big(i (T -t) \partial_t  + (\mathcal{L}_S + \alpha_0)\big)A_1^*(t,y) + |A_1^*(t,y)|^{4} A_1^*(t,y)\\ \nonumber
	=& (T - t)^{\f52 \nu} |A_{0,0}(y)|^{4} A_{0,0}(y) + \triangle_1|A(t,y)|^{4} A(t,y),
\end{align}
where
\[
\triangle_N\big(|A|^{4} A\big) = |A_N^*|^{4} A_N^*  - |A_{N-1}^*|^{4} A_{N-1}^*.
\]
We hence proceed with
\begin{align}\label{start31}
	&(\mathcal{L}_S + \mu_2)A_{2,1}(y) = 0,~~ \mu_2 = \alpha_0  - i \nu\tfrac{5}{2},\\\label{start3b}
	& (\mathcal{L}_S + \mu_2)A_{2,0}(y) =   - \big( i(\tfrac{1}{2} + \nu) + \tfrac{1}{y^2}\big) A_{2,1}(y) -\tfrac{2}{y} \partial_y A_{2,1}(y) - |A_{0,0}(y)|^{4} A_{0,0}(y),
\end{align}
where the second line includes the error from the logarithmic factor omitted  in the first line.\\[4pt]
Let us turn to the case of dimension $ d = 4$. We then have initially 
\[
A_1(t,y) = (T -t)^{\nu} \big(A_{1,0}(y) + (\log(y)  - \nu \log(T-t)) A_{1,1}(y)\big) =: (T -t)^{\nu}\tilde{A}_1(t,y).
\]
Thus we solve 
\begin{align}
	&(\mathcal{L}_S + \tilde{\mu}_1)A_{1,1}(y) = 0,~~ \tilde{\mu}_1 = \alpha_0  - i \nu,\\
	& (\mathcal{L}_S + \tilde{\mu}_1)A_{1,0}(y) =    - \big( i(\tfrac{1}{2} + \nu) + \tfrac{2}{y^2}\big) A_{1,1}(y) -\tfrac{2}{y} \partial_y A_{1,1}(y).
\end{align}
The error now reads
\begin{align}{\tiny }
	e_1(t,y) =& \big(i (T -t) \partial_t  + (\mathcal{L}_S + \alpha_0)\big)A_1^*(t,y) + |A_1^*(t,y)|^{2} A_1^*(t,y)\\ \nonumber
	=& (T - t)^{3 \nu} |\tilde{A}_1(t,y)|^{2} \tilde{A}_1(t, y).
\end{align}
Reasoning similarly as above in case $ d = 3$, we hence first solve
\begin{align}
	&(\mathcal{L}_S + \tilde{\mu}_3)A_{3,3}(y) =  - A_{1,1}^2 \overline{A}_{1,1} ,~~ \tilde{\mu}_3 = \alpha_0  - 3 i \nu.
\end{align}
Now the subsequent correction of the third step is 
\begin{align}
	(\mathcal{L}_S + \tilde{\mu}_3)A_{3,2}(y) &=  - 3\big( i(\tfrac{1}{2} + \nu) + \tfrac{2}{y^2}\big) A_{3,3}(y) - \tfrac{6}{y} \partial_y A_{3,3}(y) - A_{1,1}^2\cdot  A_{1,0},
\end{align}
where we use a simplification for the interaction term $ A_{1,1}^2\cdot  A_{1,0} $. This term should really be of the form
\[
A_{1,1}^2\cdot  \overline{A_{1,0}},~~\text{and}~~ \overline{A_{1,1}} \cdot A_{1,1} \cdot  A_{1,0}.\\[5pt]
\]
Lastly, two more corrections are calculated 
\begin{align}
	(\mathcal{L}_S + \tilde{\mu}_3)A_{3,1}(y) &=  - 2\big( i(\tfrac{1}{2} + \nu) + \tfrac{2}{y^2}\big) A_{3,2}(y) - \tfrac{4}{y} \partial_y A_{3,2}(y)  + \tfrac{3}{y^2}A_{3,3}(y) - A_{1,0}^2 \cdot A_{1,1}\\[4pt]
	(\mathcal{L}_S + \tilde{\mu}_3)A_{3,0}(y) &=  - \big( i(\tfrac{1}{2} + \nu) + \tfrac{2}{y^2}\big) A_{3,1}(y) - \tfrac{2}{y} \partial_y A_{3,1}(y)  + \tfrac{2}{y^2}A_{3,2}(y) - A_{1,0}^3 .
\end{align}
Again the `linear' error terms on the right appear from omitting the logarithmic factor, which will be described in more detail below.
\begin{Rem} In  \cite{OP}, not many iteration steps were necessary in order to obtain a decent approximation. However, we consider the general structure from \cite{Perelman} and present the induction step in the following.
\end{Rem}
~~\\
We now outline the step for higher order corrections in  the iteration.\\[4pt]
\emph{The general induction}. We restrict to dimension $ d = 3$ first. By direct inspection of \eqref{ansatz-for-itSS} in $ d = 3$, the error in the  $(n-1)$st step for $ n \geq 2$ will be rearranged into 
\begin{align}
	e_{n-1}(t,y) = \sum_{k \geq n}(T-t)^{\nu(k + \f12)} \sum_{j \leq \tfrac{k}{2} -1}  ( \log(y) - \nu \log(T-t))^j e_{n-1,j}^{k-n}(y),
\end{align}
where $ e_{n-1,j}^{k-n}(y)$ are interaction terms from the quintic  power law. The outer sum $ \sum_{k \geq n }$ ranges up to $ k = 5n-3$, however the upper bound is not relevant. In order to define the source terms, the `lowest order' error will be set as follows.
\begin{align}
	e_{n-1}^0(t,y) = (T-t)^{\nu(n + \f12)}\sum_{j \leq \tfrac{n}{2} -1} ( \log(y) - \nu \log(T-t))^j e^0_{n-1,j}(y),
\end{align}  
where  $ e^0_{n-1,j} $ is a sum of interaction terms of the form
\[
\tilde{A}_{n_1, k_1} \cdot \tilde{A}_{n_2, k_2}  \cdots \tilde{A}_{n_5, k_5},~~ \tilde{A} \in \{A, \bar{A}\},
\]
such that $ n_1 + \dots +  n_5 = n-2$ and $ k_1 + k_2 + \dots + k_5 = j$. 
In the following let $n \in \N$ and $ n \geq 2$.  We choose $ m_n = \tfrac{n}{2}$ , if $ n$ is even and $m_n = \tfrac{n-1}{2}$ if $n$ is odd.
We further set $ \mu_n = \alpha_0  - i \nu(n + \tfrac{1}{2}) $ and calculate 
\begin{align*}
	&(i (T -t) \partial_t + \mathcal{L}_S + \mu_n) \big[(\log(y) - \nu \log(T-t))^k A_{n,k}(y)\big]\\
	&\hspace{4cm} - (\log(y) - \nu \log(T-t))^k (\mathcal{L}_S + \mu_n) A_{n,k}(y)\\[3pt]
	&= (\log(y) - \nu \log(T-t))^{k-1}k D_y A_{n,k}(y)\\[3pt]
	& \hspace{.5cm} + (\log(y) - \nu \log(T-t))^{k-2} k(k-1) \tfrac{1}{y^2}A_{n,k}(y),
\end{align*}
where 
\[
D_y =   \big( i(\tfrac{1}{2} + \nu) + \tfrac{d-2}{y^2}\big) +  \tfrac{2}{y} \partial_y.
\]
~~\\
Then we set for the $n^{\text{th}}$ correction $ A_n(t,y)$
\begin{align}\label{three-first}
	(\mathcal{L}_S + \mu_n)A_{n,m_n}(y) &=\; 0,\\[4pt] \label{three-second}
	(\mathcal{L}_S +\mu_n)A_{n, m_n -1}(y) &= \; - m_n D_y A_{n,m_n}(y) - e_{n-1, m_n-1}^0(y),\\[4pt]
	(\mathcal{L}_S + \mu_n)A_{n, m_n -2}(y) &=\;  - (m_n -1) D_y A_{n,m_n -1}(y) -  e_{n-1, m_n -2}^0(y),\\[3pt] \nonumber
	&~~~~ - m_n(m_n-1) \tfrac{1}{y^2} A_{n, m_n}(y),\\[4pt] \nonumber
	\hspace{1cm}\vdots &\hspace{2cm}\vdots\\[4pt] \label{three-last}
	(\mathcal{L}_S + \mu_n)A_{n,0}(y) &=\; -  D_y A_{n,1}(y) -  e_{n-1, 0}^0(y) -  \tfrac{2}{y^2} A_{n, 2}(y) 
\end{align} 
The error in the $n^{\text{th}}$ step then satisfies 
\begin{align}
	e_n(t,y) = & \big(i (T -t) \partial_t  + (\mathcal{L}_S + \alpha_0)\big)A_n^*(t,y) + |A_n^*(t,y)|^{4} A_n^*(t,y)\\ \nonumber
	=& e_{n-1}(t,y) - e_{n-1}^0(t,y) + \triangle_n ( |A|^4 A).
\end{align}
~~\\
Now in dimension $ d = 4$ as before, we rearrange for 
\begin{align}
	e_{n-2}(t,y) = \sum_{k \geq n}(T-t)^{\nu k} \sum_{j \leq k}  ( \log(y) - \nu \log(T-t))^j e_{n-2,j}^{k-n}(y),
\end{align}
where the sum ranges up to $ k = 3n-6 $ and $ k \geq n \geq 3 $ are  odd integer. Then we set the `lowest' order error  as follows
\begin{align}
	e_{n-2}^0(t,y) = (T-t)^{n\nu}\sum_{j \leq n} ( \log(y) - \nu \log(T-t))^j e_{n-2,j}^0(y),
\end{align}
where $ e_{n-2,j}^0(y)$ are interaction terms from the cubic  power law.  They have the form
\[
\tilde{A}_{n_1, k_1} \cdot \tilde{A}_{n_2, k_2}  \cdot\tilde{A}_{n_2, k_3},~~ \tilde{A} \in \{A, \bar{A}\},
\]
such that $ n_1 +  n_2 +  n_3  = n$ and $ k_1 + k_2  + k_3 = j$. Then we set $ \tilde{\mu}_n = \alpha_0  - i  \nu n  $ and calculate for the $n$th correction $ A_n(t,y)$  

\begin{align}\label{four-first}
	(\mathcal{L}_S + \tilde{\mu}_n)A_{n,n}(y) &=\; - e_{n-2, n}^0(y),\\[4pt]
	(\mathcal{L}_S + \tilde{\mu}_n)A_{n, n -1}(y) &=\;  - n D_y A_{n,n}(y) - e_{n-2, n-1}^0(y)\\[4pt]
	(\mathcal{L}_S + \tilde{\mu}_n)A_{n, n -2}(y) &= \; - (n -1) D_y A_{n,n -1}(y) -  e_{n-2, n -2}^0(y)\\[3pt] \nonumber
	&~~~~ - n(n-1) \tfrac{1}{y^2} A_{n, n}(y)\\[4pt] \nonumber
	\hspace{2cm}\vdots &\hspace{2cm}\vdots\\[4pt]
	(\mathcal{L}_S + \tilde{\mu}_n)A_{n,0}(y) &= \;-  D_y A_{n,1}(y) -  e_{n-2, 0}^0(y) -  \tfrac{2}{y^2} A_{n, 2}(y).\label{four-last}
\end{align}
The error in the $n^{\text{th}}$ step then satisfies 
\begin{align}
	e_n(t,y) = & \big(i (T -t) \partial_t  + (\mathcal{L}_S + \alpha_0)\big)A_n^*(t,y) + |A_n^*(t,y)|^{2} A_n^*(t,y)\\ \nonumber
	=& e_{n-2}(t,y) - e_{n-2}^0(t,y) + \triangle_n ( |A|^2 A).
\end{align}
Note since $A_{n,j}(y) = 0$ for even integer $ n \geq 0$, we have $ A_n^* = A_{n+1}^*$ and  $ e_n(t,y) = e_{n+1}(t,y)$.\\[2pt]
We can write both systems via
\begin{align}\label{Matrix-Self-Sim}
	\mathit{L}_S \cdot \mathbf{A}(y) = -\mathbf{e}^0(y),
\end{align}
where $\mathit{L}_S$ has the form
\begin{align*}
	\mathit{L}_S = \begin{pmatrix}
		\mathcal{L}_S + \mu & 0 & 0 & 0& \dots &0\\
		n D_y & \mathcal{L}_S + \mu  & 0&0& \dots &0\\
		n (n-1) y^{-2}& (n-1) D_y &\mathcal{L}_S + \mu & 0 &\dots&0\\
		\vdots & \ddots\\
		0 & 0& \cdots & 2 y^{-2} & D_y & \mathcal{L}_S + \mu~
	\end{pmatrix}
\end{align*}
and 
\begin{align*}
	\mathbf{A} = (A_{n,n},A_{n,n-1}, \dots, A_{n,0})^T,~~~~~\mathbf{e }^0= (e^0_{n-1,n}, e^0_{n-1,n-1}, \dots, e^0_{n-1,0})^T.
\end{align*}
We recall in $ d = 4 $ we set $ A_{2k} = 0 $  and thus $e^0_{2k+1}(t,y) = e^0_{2k}(t,y)$ for all $ k \in \Z_{\geq0}$.
\begin{Lemma}\label{FS-Inner} Let  $ \mu \in \C $. Then $ \mathcal{L}_S + \mu$ has fundamental solutions  $ \phi_0^{(d)}(\mu,y),~ \psi_0^{(d)}(\mu,y)$ for $ y \in (0, \infty)$ such that
	\begin{itemize}
		\item[(i)] The function $ \phi_0^{(d)}(\mu,y)$ is analytic in $ (y, \mu)$ with $ \phi_0^{(d)}(\mu,y) = 1 + O(y^2)$ as $ y \to 0$ and has a smooth even extension to $\R$.
		\item[(ii)] The function $ \psi_0^{(3)}(\mu,y) = y^{-1} +  \tilde{\psi}(\mu,y)$, where $  \tilde{\psi}(y,\mu)$  is analytic in $ (y,\mu) $ and has a smooth odd extension to $\R$. Further
		\[
		\psi_0^{(4)}(\mu,y) = c_1 y^{-2} + c_2\tilde{\psi}_0(y, \mu) +   \log(y) \phi_0^{(4)}(\mu,y), 
		\] 
		where $ \tilde{\psi}_0 $ is as in (i).
	\end{itemize}
\end{Lemma}
\begin{proof} In the case of $d = 3$, it is quickly observed that for the recursive formula
	\begin{align*}
		&c_{k+1} (2k+3)(2k +2) y^{2k +1} + c_{k +1} (d-1) (2k+3) y^{2k +1}\\
		&\hspace{3cm}= - c_k  \tfrac{i}{2}(2k+1) y^{2k+1} - c_k\big( i \tfrac{d-2}{4} + \mu\big) y^{2k +1},
	\end{align*}
	the first line degenerates for $ k = -2$. Likewise we argue for the part $(i)$ with
	\begin{align}\label{even}
		&c_{k+1} (2k+2)(2k +1) y^{2k} + c_{k +1} (d-1) (2k+2) y^{2k }\\\nonumber
		&\hspace{3cm}= - c_k  \tfrac{i}{2}(2k) y^{2k} - c_k\big( i \tfrac{d-2}{4} + \mu\big) y^{2k },
	\end{align}
	and $ k = -1$. Now for (ii) in dimension $ d =4$, we use the latter recursion `corrected' by the terms  having at least one derivative  of $ \mathcal{L}_S$ on the logarithmic factor, i.e. we add
	\begin{align*}
		\tilde{c}_{k+1} y^{2k}2 (2k+2) - \tilde{c}_{k+1 } y^{2k} + (d-1) \tilde{c}_{k+1}y^{2k} + \tfrac{i}{2}\tilde{c}_k y^{2k}.
	\end{align*}
	to \eqref{even}, where the coefficients $ (\tilde{c}_k)_k$ are a solution to \eqref{even}.
\end{proof}
\begin{Lemma}\label{Lemma-solut-inner1}	The system \eqref{three-first} - \eqref{three-last}  in dimension $ d = 3$  (including the initial steps for $A_0, A_1, A_2$ ) has a unique solution $A_{n, j}(y)$  with $ 0 \leq j \leq n$  such that
	\begin{align}\label{condition-main1}
		A_{n,j}(y) = \sum_{ l \geq 0 } d_{n, j, l} y^{2l -1 -n},~~ 0 < y \ll 1
	\end{align}
	in an absolute sense given $ d_{n, 0, \f{n}{2}},~d_{n, 0, \f{n+1}{2}}  \in \mathbf{C}$ are fixed for even and odd $ n$ respectively.\\[2pt]
	The error $ e_N(t,y)$ can be written of the form\\
	\begin{align}\label{error-exp1}
		e_N(t,y) &= \sum_{k \geq N+1} (T-t)^{\nu(k + \f12)}\sum_{j \leq \f{k}{2}} (\log(y) - \nu \log(T-t))^j \chi^{(3)}_{j, k}(y),
	\end{align}
	where $ \chi_{k, j }^{(3)}(y) $ are analytic functions in $ y \in (0, \infty)$ with 
	\[
	\chi_{k, j }^{(d)}(y) = \sum_{l \geq 0} y^{2l -k -3} \tilde{d}_{k , j, l},~~ 0 <  y \ll1.~~ 
	\]
	Further $\chi^{(3)}_{k,j}(y)$ is non-vanishing for only finitely many $k \in \N$ and we have $ \tilde{d}_{n,j,l } = d_{n, j, l} = 0$ if $ l < j $ or $ l < j+1$ with even or odd $ n $, respectively. 
\end{Lemma}

Before we prove the Lemma, we note the following simple observation, similarly employed in the work \cite[Lemma 2.5]{Perelman}  and which shows how we construct the iterates $A_{n,j}$ in the induction step.
\begin{Lemma} Let $ \mu= \mu(d) \neq i \frac{d-2}{4}$ and  $ F^{(d)} \in C^{\infty}(\R_{>0})$ such that 
	\[
	F^{(d)}(y) = \sum_{l = 0}^nF^{(d)}_l y^{-d -2l} + \tilde{F}_0^{(d)} y^{2-d} + \tilde{F}^{(d)}(y), ~~0 < y \ll1, 
	\]
	where $\tilde{F}^{(d)} \in C^{\infty}(\R_{>0})$ has a smooth odd extension to $\R$ in $d=3$ and a smooth even extension to $\R$ in $d =4$. Then there exists  a linear combination $ A^{(d)}$ of $ (F_l^{(d)})_{l =0}^n $ such that
	\[
	(\mathcal{L}_S + \mu)w(y) = F^{(d)}(y) + A^{(d)} y^{-d}
	\]
	has a solution $ w \in C^{\infty}(\R_{>0})$  of the form
	\[
	w(y) = \sum_{l = 0}^{n+1}w^{(d)}_l y^{2-d + 2(l-n)} + \tilde{w}^{(d)}(y) ,~~0 < y \ll1, 
	\]
	where $ \tilde{w}^{(d)} \in C^{\infty}(\R_{>0})$  has a smooth odd extension $\tilde{w}^{(3)}(y) = O(y^3)$ to $\R$ in $d=3$ or an even extension $ \tilde{w}^{(4)}(y) = O(y^2)$ in $ d=4$. Further $ w_{n}^{(4)} \cdot (\mu - i \frac{1}{2})  = \tilde{F}_0^{(4)}$ and $ w_{n+1}^{(4)},~w^{(3)}_n$ are free to choose.
\end{Lemma}
The proof of this statement is implicitly given in the subsequent proofs of Lemma \ref{Lemma-solut-inner1}.
\begin{proof}[Proof of Lemma \ref{Lemma-solut-inner1}]
	We first start with dimension $ d =3$. Here we solve  \eqref{start1} requiring $ A_{0,0}(y) = O(y^{-1} ) $ as $ y \to 0^+$ by $ A_{0,0}(y) =  c_{0,0,0}\cdot \psi_0^{(3)}(\mu_0, y)$. Likewise, considering Lemma \ref{FS-Inner}, we solve \eqref{start2} with \eqref{condition-main1} by  $ A_{1,0}(y) = c_{1,0,1}\cdot  \phi_0^{(3)}(\mu_1, y)$. Let us consider $A_2(t,y)$.\\
	
	For $ A_{2,1}(y)$, we obtain from \eqref{start31} and \eqref{condition-main1} that $A_{2,1}(y) = \beta_{0} \cdot \psi_0^{(3)}(\mu_2, y)$ with $ \beta_0 \in \C$ to be specified. For $A_{2,0}$ we set 
	\[
	F_{2,0}(y) :=  -\big( i(\tfrac{1}{2} + \nu) + \tfrac{1}{y^2}\big) A_{2,1}(y) -\tfrac{2}{y} \partial_y A_{2,1}(y) - |A_{0,0}(y)|^{4} A_{0,0}(y),
	\]
	which satisfies 
	\[
	F_{2,0}(y) = \sum_{l \geq -2} \tilde{c}_{2,0,l} y^{2l -1},~~ 0 < y < y_0
	\]
	for some  $0 < y_0 \ll 1 $ and  both $ \tilde{c}_{2,0,-2}$ as well as $ \tilde{c}_{2,0,-1} -  \beta_0$ do not depend on $ \beta_0$. The Wronskian $w$ of $\mathcal{L}_S$ satisfies
	\[
	w^{-1}(s)  = \sum_{l \geq 0} w_l s^{2l +d-1},~~ 0 < s< y_0 .
	\]
	As in \cite{OP}, we reduce $ ( \mathcal{L}_s + \mu_2)A_{2,0} = F_{2,0}(y)$ to 
	\begin{align}\label{red}
		( \mathcal{L}_s + \mu_2)\tilde{A}_{2,0} = F_{2,0}(y) -  \tfrac{1}{6}\tilde{c}_{2,0,-2} ( \mathcal{L}_s + \mu_2)\frac{1}{y^3} =: \tilde{F}_{2,0}(y),
	\end{align}
	where $ A_{2,0}(y) =  \tilde{c}_{2,0,-2} \frac{1}{6y^3} + \tilde{A}_{2,0}(y)$. We then choose $\beta_0$ such that the right side of \eqref{red} satisfies
	\[
	\tilde{F}_{2,0}(y) = \sum_{l \geq 0} \tilde{d}_{2,0,l} y^{2l -1},~~ 0 < y < y_0.
	\]
	Thus for the solution of \eqref{red} under \eqref{condition-main1} we observe
	\begin{align*}
		\tilde{A}_{2,0}(y) &= c_{2,0,1} \cdot  \psi_0^{(3)}(\mu_2, y) + \int_0^{y} w^{-1}(s) G(s, y, \mu_2) \tilde{F}_{2,0}(s)~ds,
	\end{align*}
	with
	\[
	G(s, y, \mu_2) = \psi_0^{(3)}(\mu_2, y) \phi_0^{(3)}(\mu_2,s) - \psi_0^{(3)}(\mu_2, s) \phi_0^{(3)}(\mu_2,y),
	\]
	and where the particular solution has a smooth odd extension to $\R$. 
	Now we assume the statement in Lemma \ref{Lemma-solut-inner1} is true with $ A_{k,j} = 0$ for $ j > \f{k}{2}$ and $ k \leq n-1$. Let us first consider the case where $ 2 \tilde{n} = n $ is even. We hence solve
	\[
	(\mathcal{L}_S + \mu_n) A_{n, \f{n}{2}}(y) = 0,
	\]
	such that \eqref{condition-main1} holds via $ A_{n, \f{n}{2}}(y) = \beta_{n, \f{n}{2}} \cdot \psi^{(3)}(\mu_n, y)$ for $ \beta_{n, \f{n}{2}} \in \C$ to be specified. Then, let 
	\[
	F_{n, \f{n}{2}-2}(y) := - \tfrac{n}{2} D_y A_{n, \f{n}{2}}(y) - e^0_{n-1, \f{n}{2}-1}(y) = \sum_{l \geq 0} \tilde{c}_{n, \f{n}{2}-1, l} y^{2l -3 -n},
	\]
	be the right side of \eqref{three-second}. We note that $\tilde{c}_{n, \f{n}{2}-1, l}$ for $ l = 0, \dots, \f{n}{2}-1$ as well as $ \tilde{c}_{n, \f{n}{2}-1, \f{n}{2}} - \f{n}{2} \beta_{n, \f{n}{2}}$ are not dependent on $ \beta_{n, \f{n}{2}}$. Especially, we make the ansatz
	\begin{align}\label{Reduction-here}
		& (\mathcal{L}_S + \mu_n) \tilde{A}_{n, \f{n}{2}-1}(y) = F_{n, \f{n}{2}-1}(y) - (\mathcal{L}_S + \mu_n) \big( \sum_{l =0}^{\f{n}{2}-1} \tilde{d}_{n,l}y^{2l -1 -n}\big) =: \tilde{F}_{n, \f{n}{2}-1}(y),\\ \label{Reduction-here2}
		& \tilde{A}_{n, \f{n}{2}-1}(y) = A_{n, \f{n}{2}-1}(y) - \sum_{l =0}^{\f{n}{2}-1} \tilde{d}_{n,l} y^{2l -1 -n}.
	\end{align}	
	The coefficients $ \tilde{d}_{n,l} $ will be such that $ \tilde{F}_{n, \f{n}{2}-1}(y) = O(y^{-3})$, i.e. by the definition of $\mathcal{L}_S$ we observe the asymptotics with the choice 
	\begin{align}
		&\tilde{d}_{n,0} = \frac{\tilde{c}_{n, \f{n}{2}-1, 0}}{(n + 1 )n},\;\;\;\tilde{d}_{n,l} = \frac{\tilde{c}_{n, \f{n}{2}-1, l} - (\f{i}{2}(2(l-1) - \f12 -n)  + \mu_n) \tilde{d}_{n, l-1}}{(n + 1 -2l)(n - 2l)},
	\end{align}
	for  $ l = 1, \dots, \f{n}{2}-1$. We proceed to fix $\beta_{n, \f{n}{2}}$ such that 
	\[
	\tilde{F}_{n, \f{n}{2}-1}(y) = \sum_{l \geq 0} \tilde{d}_l y^{2l - 1}
	\]
	and thus solve 
	\begin{align}\label{sol-darst}
		\tilde{A}_{n, \f{n}{2}}(y) = \beta_{n, \f{n}{2}-1} \cdot \psi^{(3)}(\mu_n, y) +  \int_0^y w^{-1}(s) G(\mu_n, y,s) \tilde{F}_{n, \f{n}{2}-1}(s)~ds,
	\end{align}
	where the particular part is an odd  function in $C^{\infty}(\R)$ and $\beta_{n, \f{n}{2}-1}  \in \C$ is again to be fixed. The next equation in \eqref{three-first} - \eqref{three-last} follows the same ansatz, where the additional term 
	\[
	\tfrac{1}{y^2} A_{n, \f{n}{2}}(y) = \beta_{n, \f{n}{2}} \big[y^{-3} + O(y^{-1})\big]
	\]
	appears on the right. The solution is then again reduced to the `singular parts' exactly as in \eqref{Reduction-here2}, \eqref{sol-darst}, which requires us to fix $\beta_{n, \f{n}{2}-1}$ and introduces a  free parameter $\beta_{n, \f{n}{2}-2} \in \C$ . This procedure is now iterated (remaining $ \f{n}{2}-2$ times) where finally $ A_{n, 0}$ has the form
	\begin{align}
		& A_{n, 0}(y) = \tilde{A}_{n,0}(y) + \sum_{l =0}^{\f{n}{2}-1} \tilde{d}_{n,0,l} y^{2l -1 -n},\\
		&\tilde{A}_{n, 0}(y) = \beta_{n, 0} \cdot \psi^{(3)}(\mu_n, y) +  \int_0^y w^{-1}(s) G(\mu_n, y,s) \tilde{F}_{n,0}(s)~ds.
	\end{align}
	In particular, we then set $ \beta_{n, 0} = c_{n,0, \f{n}{2}}$. Let us turn to the odd case $ 2\tilde{n}-1 = n $, where  in principal we follow the same steps. The difference is that we replace $\f{n}{2}$ by $ \frac{n-1}{2}$ and the required expansions \eqref{condition-main1} are even. Thus we solve 
	\[
	(\mathcal{L}_S + \mu_n) A_{n, \f{n-1}{2}}(y) = 0,
	\]
	such that \eqref{condition-main1} holds via $ A_{n, \f{n-1}{2}}(y) = \beta_{n, \f{n-1}{2}} \cdot \phi^{(3)}(\mu_n, y) = O(1)$  as $ y \to 0^+$. Then, the subsequent solution is determined via
	\begin{align}
		& (\mathcal{L}_S + \mu_n) \tilde{A}_{n, \f{n-1}{2}-1}(y) = F_{n, \f{n-1}{2}-1}(y) - (\mathcal{L}_S + \mu_n) \big( \sum_{l =0}^{\f{n-1}{2}} \tilde{d}_{n,l}y^{2l -1 -n}\big) =: \tilde{F}_{n, \f{n-1}{2}-1}(y),\\ 
		& \tilde{A}_{n, \f{n-1}{2}-1}(y) = A_{n, \f{n-1}{2}-1}(y) - \sum_{l =0}^{\f{n-1}{2}} \tilde{d}_{n,l} y^{2l -1 -n}.
	\end{align}	
	The coefficient $ \beta_{n, \f{n-1}{2}}$ has to be fixed such that  $ \tilde{F}_{n, \f{n-1}{2}-1}(y) = O(1) $ for $ 0< y \ll1$. In particular, iterating this procedure, we find the last term $A_{n, 0}$ to be of the form
	\begin{align}
		& A_{n, 0}(y) = \tilde{A}_{n,0}(y) + \sum_{l =0}^{\f{n-1}{2}} \tilde{d}_{n,0,l} y^{2l -1 -n},\\
		&\tilde{A}_{n, 0}(y) = \beta_{n, 0} \cdot \phi^{(3)}(\mu_n, y) +  \int_0^y w^{-1}(s) G(\mu_n, y,s) \tilde{F}_{n,0}(s)~ds,
	\end{align}
	where $ \beta_{n, 0} = c_{n, 0, \f{n+1}{2}}$ and the particular solution $ \tilde{A}_{n, 0} - \beta_{n, 0} \cdot \phi^{(3)}  = O(y^2)$ is an even function in $C^{\infty}(\R)$. Further, we have by assumption $ e^0_{n-1,j}(y) = O(y^{2j - 3-n})$ for $ j = 0, \dots, \f{n}{2}-1$ if $n $ is even or $ e^0_{n-1,j}(y) = O(y^{2j - 1-n})$ for $ j = 0, \dots, \f{n-3}{2}$ if $n $ is an odd number. In particular, we then have 
	\begin{align}\label{no1-d3}
		&A_{n,j}(y) = O(y^{2j - 1-n}),\;\;\; j = 0, \dots, \f{n}{2},\;\; n ~\text{even},\\ \label{no2-d3}
		&A_{n,j}(y) = O(y^{2j +1-n})`\;\;\ j = 0, \dots, \f{n-1}{2},\;\; n~\text{odd}.
	\end{align}
	We hence directly verify the expression \eqref{error-exp1} for $ e_{n}(t,y)$ and the claimed asymptotics by the induction assumption, \eqref{no1-d3} and \eqref{no2-d3}.
\end{proof}

We now let  the corrections in dimension $ d =3$
$$ A_1(t,y),~ A_2(t,y), ~\dots,~ A_{\tilde{N}}(t,y),~ y \in (0, \infty),~ t \in [0, T),~ \tilde{N} \in \Z_{+}$$
be defined as in Lemma \ref{Lemma-solut-inner1} such that  in \eqref{condition-main1}
\[
d_{n, 0, \f{n}{2}}  = c_{n,0}^{(\f{n}{2})},~~d_{n, 0, \f{n+1}{2}}  = c_{n,0}^{(\f{n+1}{2})}.
\]
We then set the self-similar approximation, for which we may later change into $(t,R)$-coordinates as a reference.
\begin{Def}[Self-Similar approximation]
	\begin{align}
		&\tilde{u}^{N_2}_{S}(t,y) := A_{N_2}^*(t,y) = A_1(t,y) +  A_2(t,y) + \dots +  A_{N_2}(t,y) ,~ N_2 \gg 1,\\[3pt] \label{complete-error-self-similar}
		& e^{N_2}_{S}(t,y) :=  i (T-t) \partial_t \tilde{u}^{N_2}_{S} + (\partial_{yy} + (d-1) \tfrac{1}{y}\partial_y)\tilde{u}^{N_2}_{S}\\[1pt] \nonumber
		&~~~~~~~~~~~~~~~~~~~~~ + [i (\tfrac{1}{4} + \tfrac12 y \partial_y) + \alpha_0]\tilde{u}^{N_2}_{S}  + | \tilde{u}^{N_2}_{S} |^{4}\tilde{u}^{N_2}_{S} = e_{N_2}(t,y),
	\end{align} 
	where $ N_2 \in \Z_+$ is to be fixed later on. 
\end{Def}
We now show that the functions
$$ (T-t)^{- \f14} \tilde{u}^{N_2}_S(t, (T-t)^{\nu}R),~~ \lambda(t)^{\frac{1}{2}}u^{N}_{In}(t, R),~~ R = (T-t)^{- \f12 - \nu} r,$$
coincide up to a fast decaying error in the region $(t,r) \in I \cap S$.
\begin{Lemma}[Consistency in $I \cap S$]  \label{consist} Let $N \in \Z_+$, then there exists $N_1, N_2 \in \Z_+$ large, such that with some $C > 0$ there holds
	\begin{align}
		\forall m \in \Z_+:~~	\big | \partial_R^m& \big[ (T-t)^{- \f14} \tilde{u}^{N_2}_S(t, (T-t)^{\nu}R) - \lambda(t)^{\frac{1}{2}}u^{N_1}_{In}(t, R)\big]\big | \\[4pt] \nonumber
		&\leq C_{m , N_1, N_2} R^{-m} (T-t)^{N \nu}.
	\end{align}
	for all $(t,R) $ with $ (t,r) \in I \cap S$ and $ 0 < T-t \ll1 $.
\end{Lemma} 
\begin{proof} By the calculation \eqref{transition} in the beginning of Subsection \ref{subsec:SS}, we write (for $N_2 \gg1$ large)
	\begin{align}\label{first-lll}
		\lambda(t)^{\frac{d-2}{2}}u^{N}_{In}(t, R) &= (T -t)^{-\f{d-2}{4}}\big( \sum_{n = 0}^{N_2}  (T -t)^{\nu( n  - \f{d-4}{2})} \sum_{j \leq n}  \log^j(R) \tilde{A}_{n,j}(y)\big) + E_1,\\ \nonumber
		E_1 &= (T-t)^{- \frac{d-2}{2}( \f12 + \nu)} R^{-2 - N_2} \sum_{l = 0}^{N_1} y^{ 2l} \sum_{j = 0}^{J_l} \log^j(R) f_{j, l , N_2} (R),
	\end{align}
	where for $ y = r (T-t)^{-\f12}$ and $(t,r) \in I \cap S$ we set
	$
	\tilde{A}_{n,j}(y) : = \sum_{l = 0}^{N_1} c_{n-j, j}^{(l)}   y^{2l -1-n}
	$.
	We infer from \eqref{first-lll}, that $\{\tilde{A}_{n,j}\}_{n \leq N_2}$ satisfy \eqref{three-first} - \eqref{three-last} or \eqref{four-first} - \eqref{four-last} according to Lemma \ref{Lemma-solut-inner1}  up to an additional error. The error in each step is at least of order
	$$\tilde{e}(t,r) = O((T-t)^{\nu N_1 + \frac{1}{2}}),$$
	hence irrelevant in the first $\sim N_1$ steps of the iteration described in the above Lemma.  Thus, subtracting $ (T-t)^{- \f14} u^{N_2}_S(t, (T-t)^{\nu}R)$ we are left with $\sum_{j =1}^4 E_j$, where 
	\[
	(T -t)^{\f{d-2}{4}}E_4 = \sum_{n = 0}^{N_2} (T-t)^{\nu(n + \frac{1}{2})} \sum_{j \leq n} \sum_{l \geq N_1 +1} d_{n, j,l}  \log^j(R)y^{2l -1-n}.
	\]
	Likewise if $ N_2 \gg N_1$ we observe
	\begin{align*}
		& (T -t)^{\f{1}{4}}E_2 =  \sum_{n = N_1}^{N_2}  (T -t)^{\nu( n  + \f{1}{2})} \sum_{j \leq n}  \log^j(R) \tilde{A}_{n,j}(y)\\
		&(T -t)^{\f{1}{4}}E_3 = \sum_{n = N_1}^{N_2} (T-t)^{\nu(n + \frac{1}{2})} \sum_{j \leq n} \sum_{l= 0}^{\infty} d_{n, j,l}  \log^j(R)y^{2l -1-n}.
	\end{align*}
	We exchange $ y^{-1-n} = R^{-1-n} (T-t)^{-\nu n -n}$ and use $ R^{-1} \sim (T-t)^{\nu -\epsilon_1}$. For $E_4 $ we additionally  note $y^{2N_1 +2} = R^{2N_1 +2} (T-t)^{\nu(2N_1 +2)} \sim (T-t)^{\epsilon_1(2N_1 +2)}  $.
\end{proof}
\subsubsection{The asymptotic behavior at $ y \gg1 $} We now determine 
a suitable expression for $ A_n(t, y) $ in the region $ 1 \ll y < \infty $ and describe the asymptotics at  $ y \to \infty$ of 

\begin{align}\label{near-zero}
	A_N^*(t,y) 
	= \sum_{n= 0}^N (T -t)^{\nu(n + \f{4-d}{2})} \sum_{j \leq m} (\log(y) - \nu \log(T-t))^j A_{n,j}(y),
\end{align}	
~~\\
where we let $ m = m_n $ in dimension   $ d=3$. 
We hence devise a new Ansatz (compare also \cite{Perelman}), which is more convenient transforming to $(t,r)$ variables in the final region Section \ref{subsec:RR} where $ r \gtrsim  (T-t)^{ \f12 - \epsilon_2}$.  We demonstrate this Ansatz for both $d=3,4$ dimensions, where in $d=4$ we set $m =n$. The main Lemma is only provided in dimension $d=3$.\\[4pt]
We will write the approximate solution \eqref{near-zero}  into  the following form

\begin{align}\label{RO1}
	&\sum_{n = 0}^{N_2} (T-t)^{\nu(n +  \frac{d-4}{2})} \big(W_{n}^{(+)}(t,y) + W_{n}^{(-)}(t,y)\big),~~ N_2 \in \Z_+,\\
	\label{RO2}
	&W_{n}^{(\pm)}(t,y) = \sum_{\ell = 0}^{n}  (\log(y) \pm \f12 \log(T-t))^{\ell} W^{(\pm)}_{n,  \ell}(y).
\end{align}

Let us first derive a triangular system as in the previous section for the evolution of the homogeneous linear flow in  \eqref{SelfS} under the new representation \eqref{RO1} - \eqref{RO2}. Assume we have 
\begin{align}\label{RO-lin}
	i (T -t) \partial_t w + (\mathcal{L}_S + \alpha_0)w  = F(t,y)
\end{align}
where $ w(t,y), F(t,y) $ are of the above form, i.e.  
\begin{align}\label{RO-f}
	&F(t,y) = \sum_{n \geq 0} (T-t)^{\nu(n +  \frac{4 -d}{2})} \big(F_{n}^{(+)}(t,y) + F_{n}^{(-)}(t,y)\big),\\
	\label{ROg}
	& F_{n}^{(\pm)}(t,y) = \sum_{\ell = 0}^n (\log(y) \pm \f12 \log(T-t))^{\ell}  F^{(\pm)}_{n,  \ell}(y),
\end{align}
with a finite sum in \eqref{RO-f}, i.e. $F^{(\pm)}_{n, \ell} = 0$ except for finitely many $ n \geq 0$. 
We write \eqref{RO-lin} into
\begin{align} \label{sys-gen-large-y-lim}
	&(i (T -t) \partial_t + \mathcal{L}_S + \alpha_0)  \big(\sum_{\pm}\sum_{n = 0}^{N_2} (T-t)^{\nu(n +  \frac{4 -d}{2})} \sum_{\ell = 0}^{n}  (\log(y) \pm \f12 \log(T-t))^{\ell} W^{(\pm)}_{n,  \ell}(y)\big)\\ \nonumber
	& = \sum_{\pm} \sum_{n \geq 0} (T-t)^{\nu(n +  \frac{4 -d}{2})} \sum_{\ell = 0}^n (\log(y) \pm \f12 \log(T-t))^{\ell}  F^{(\pm)}_{n,  \ell}(y).
\end{align}

Hence we consider  a system for $\{ W^{(\pm)}_{n, \ell}\}_{n, 0 \leq \ell \leq n}$  of the following form

\begin{align}
	\label{Large-y-system}
	(\mathcal{L}_S + \mu)W^{(\pm)}_{n, n}(y) &=\; F^{(\pm)}_{n, n}(y) ,\\[6pt]
	(\mathcal{L}_S + \mu)W^{(\pm)}_{n, n -1}(y) &=\;  - n D^{\pm}_y W^{(\pm)}_{n, n}(y)  + F^{(\pm)}_{n, n -1}(y) \\[6pt] 
	(\mathcal{L}_S + \mu)W^{(\pm)}_{n,j}(y) &= \;-  (j+1)D^{\pm}_y W^{(\pm)}_{n,j+1}(y)  +  F^{(\pm)}_{n, j}(y) -  \tfrac{(j+2)(j+1)}{y^2} W^{(\pm)}_{n, j+2}(y) \label{last-y-large},\\[5pt] \nonumber
	0 \leq j \leq n-2,~~~~~&
\end{align}
~~\\
where  $ ( \mu,d) \in \{  ( \mu_n, 3), (\tilde{\mu}_n ,4)  \}$ and we set
\begin{align}
	&D^{\pm}_y =  i(\frac{1}{2} \mp \f12) + \frac{d-2}{y^2} +  \frac{2}{y} \partial_y.
\end{align}
For our purposes, we can thus set $ W^{(\pm)}_{n, j} = 0$ if $ j > m_n$ in $ d =3 $  or $ n \in \Z_+$ is even in $ d = 4$. 
Also, we like to note that  \eqref{Large-y-system} - \eqref{last-y-large} is  rewritten into
\begin{align}\label{Matrix-Self-Sim-Large}
	\mathbf{L}_S^{\gg1}(y)  \cdot \mathbf{W}^{\pm}(y) = \mathbf{F}^{\pm}(y)
\end{align}
where again $\mathbf{L}_S^{\gg1}(y)$ has triangular form
\begin{align*}
	\mathbf{L}_S^{\gg1}(y)  = \begin{pmatrix}
		\mathcal{L}_S + \mu & 0 & 0 & 0& \dots &0\\
		n D^{\pm}_y & \mathcal{L}_S + \mu  & 0&0& \dots &0\\
		n (n-1) y^{-2}& (n-1) D^{\pm}_y &\mathcal{L}_S + \mu & 0 &\dots&0\\
		\vdots & \ddots\\
		0 & 0& \cdots & 2 y^{-2} & D^{\pm}_y & \mathcal{L}_S + \mu~
	\end{pmatrix},
\end{align*}
which allows us to solve \eqref{Large-y-system} - \eqref{last-y-large} iteratively with data
$$ \mathbf{W}^{\pm}(y) = (W^{(\pm)}_{n, n}, \dots, W^{(\pm)}_{n, 0})^T,~~ \mathbf{F}^{\pm}(t,y) = (F^{(\pm)}_{n, n}, \dots, F^{(\pm)}_{n, 0})^T.$$
~~\\
\emph{Homogeneous system}. For solutions of $ (\mathcal{L}_S + \mu )f = 0$ at $ y \gg1 $, we refer to the standard description given for instance in \cite[Theorem 2.1]{Olver} and which we include  in the appendix \ref{sec:SingODE}.\\[3pt]
In particular, it is not difficult to infer that the ansatz 
\[
(\mathcal{L}_S + \mu)\big (e^{- \frac{i}{4} y^2} f^{(d)}( \tfrac{i}{4} y^2) \big) =  0,~~ d = 3,4
\]
leads the profile $f^{(d)}( z)$  to satisfy  Kummer's equation
\begin{align}
	z \cdot \partial_z^2 f^{(d) } + \big( \frac{d}{2} -z \big) \partial_z f^{(d) } - \big( \frac{d + 2}{4} + i \mu\big)f^{(d) } = 0,
\end{align}
in $ z = \f{i}{4} y^2$. 
By the application of Lemma \ref{Olvers-Lemma}, we obtain two well-known independent solutions
\begin{align}
	&f^{(d)}_+(z) = z^{- i \mu - \frac{d+2}{4}} \big(a_0^+ + O(z^{-1})\big),~~z \to \infty,\\ 
	&f^{(d)}_-(z) = e^{ z}(-z)^{ i \mu - \frac{d-2}{4}} \big(a_0^- + O(z^{-1})\big)~~ z \to \infty,~~
\end{align}
with an expansion as in Lemma \ref{Olvers-Lemma} depending smoothly on $\mu$ 
. We now sum this up in the following Lemma.
\begin{Lemma}\label{FS-Inner-infty} Let  $ \mu \in \C $. Then $ \mathcal{L}_S + \mu$ has fundamental solutions  $ \phi_{\infty}^{(d)}(\mu,y),~ \psi_{\infty}^{(d)}(\mu,y)$ for $ y \in (0, \infty)$  smooth in $ \mu,y$ and for any compact subset $ K \subset \C,\;k \in \Z_{\geq 0},~ n \in \Z_+$ there are positive constants  $ C > 0, y_0> 0$ (depending on $K,k,n$) with
	\begin{align}\label{one}
		&\sup_{\mu \in K}\big | \partial_y^{l}\big(y^{- 2 i \mu} \phi_{\infty}^{(d)}(\mu,y) - \sum_{m = 0}^{n-1} c^{(1)}_m y^{-\frac{d-2}{2} - 2m}\big)\big | \leq C y^{-\frac{d-2}{2}-2n-l},~~ y \geq y_0,\\ \label{two}
		&\sup_{\mu \in K}\big | \partial_y^{l}\big( y^{2 i \mu} e^{ \frac{i y^2}{4}} \psi_{\infty}^{(d)}(\mu,y) - \sum_{m = 0}^{n-1} c^{(2)}_my^{-\frac{d+2}{2} - 2m}\big)\big | \leq C y^{-\frac{d+ 2 }{2} -2n-l},~~y \geq y_0,
	\end{align} 
	for $  l = 0,\dots, k$.  We have $ c_0^{(i)} = 1 $ and $ (c^{(i)}_m)_{m \geq 1},~ i = 1,2$ only depend (smoothly) on $ \mu,d$ through the relation \eqref{third-here} .
\end{Lemma}
\begin{Rem}
	(i)~We will write 
	\begin{align}
		&\phi_{\infty}^{(d)}(\mu,y) =  y^{2i \mu -\frac{d-2}{2}} \big(1 + O(y^{-2})\big),\\
		& \psi_{\infty}^{(d)}(\mu,y) = e^{- \frac{i}{4}y^2 } y^{-2 i\mu-\frac{d+2}{2}} \big(1 + O(y^{-2})\big),
	\end{align}
	with an expansion as in Lemma \ref{Olvers-Lemma}. \\
	(ii) ~For the remaining part of this section, we will use similar notation. For $\beta \in \R$ we write
	$
	g(y,\mu) = f(y,\mu) \cdot  O(y^{-\beta})
	$
	if  $f(y,\mu)$ is a smooth function in $y \in \R_+ $, $\mu \in U \subset \C$ is an open subset of parameters and the remainder is as in Lemma \ref{Olvers-Lemma}, meaning
	\begin{align*}
		&g(y,\mu)\cdot y^{\beta} = f(y,\mu)\big(\sum_{k = 0}^{N-1} y^{-2k} c_k  + R_N(y)\big),~~~R_N(y) = O_l(y^{-2N}),~~ y \gg 1
	\end{align*}
	where  $ R_N(y)$ is locally uniformly (wrt $\mu$) estimated as in \eqref{one}, \eqref{two}.
\end{Rem}
~~\\ The Green's function is given by
$$G(y,s) = \phi_{\infty}^{(d)}(y) \psi_{\infty}^{(d)}(s) -  \phi_{\infty}^{(d)}(s) \psi_{\infty}^{(d)}(y).$$
Thus particular solutions of the linear equation $ (\mathcal{L}_S + \mu)u = f$  have the form
\begin{align}
	&u(y) = \beta^{(1)} \phi_{\infty}^{(d)}(y) + \beta^{(2)} \psi_{\infty}^{(d)}(y)  + \int_y^{\infty} G(y,s) w^{-1}(s) f(s)~ds,
\end{align} 
with $ w^{-1}(s)  \sim s^{2}e^{ i \frac{s^2}{4}}$ and in case, say $ f(y) \sim y^{- \frac{ d + 2}{2}-} (1 +O(y^{-2}))$ as  $ y \to \infty$. Otherwise we have to fix $ y_0 > 0$  and  consider 
\[
\tilde{\beta}^{(1)} \phi_{\infty}^{(d)}(y) + \tilde{\beta}^{(2)} \psi_{\infty}^{(d)}(y)  +  \int_{y_0}^{y} G(y,s) w^{-1}(s) f(s)~ds.
\]
The asymptotic behaviour of $ A_{N}^*(t,y)$ as $ y \to \infty$ is essentially provided by a fundamental base for $\mathbf{L}_S^{\gg1}(y) $, i.e. we set $ F(t, y) = 0$ in \eqref{Large-y-system} - \eqref{last-y-large}. We refer to the proof of \cite[Lemma 2.6]{Perelman} for instance. 
\begin{Lemma}\label{basis-large-y}Let us set $ m = m_n$ in dimension $d =3$ in the following.
	The  homogeneous 
	system of \eqref{sys-gen-large-y-lim}, i.e. 
	\begin{align*}
		(i (T-t) \partial_t + \mathcal{L}_S + \alpha_0) \bigg((T-t)^{\nu(n + \frac{4-d}{2})}\sum_{j = 0}^m (\log(y)  \pm \f12 \log(T -t))^j  g^{\pm}_{n,j ,h}(y) \bigg) = 0,
	\end{align*}
	has a basis of solutions $ \{\mathbf{g}^{\pm}_{n,h} \}_{h = 0}^{m}$, i.e. such that
	\begin{align*}
		&\mathbf{g}^{\pm}_{n,h}  = (g^{\pm}_{n,m ,h }~g^{\pm}_{n,m-1,h}~ \dots g^{\pm}_{n,0,h}  )^T, ~~h = 0,1,\dots, m,
	\end{align*}
	and further
	\begin{align}\label{this-second}
		&  g^{+}_{n,j, h }(y) =y^{2 i \mu_n -\frac{1}{2}} \cdot  O(y^{-2(m - j -h)}),~~~~ j \leq m -h\\[2pt]\label{this-third}
		& g^{-}_{n,j, h }(y)= e^{- \frac{i}{4}y^2 } y^{- 2 i \mu_n -\frac{5}{2}} \cdot  O(y^{-2(m -j-h)}),~~~~ j \leq m -h\\[2pt]
		&  g^{\pm}_{n,j, h }(y) = 0,~~~~ j > m - h.
	\end{align}
\end{Lemma}
The proof follows successively by using Lemma \ref{FS-Inner-infty} in the first step for $ j = m -h$. In particular if we calculate
\begin{align*}
	\big( i (T-t)\partial_t +  (\mathcal{L}_S + \alpha_0) \big) \bigg( (T-t)^{\nu(n + \f12)} \sum_{j \leq m} (\log(y)  \pm \f12  \log(T-t))^j g^{\pm}_{n,j,h}(y)\ \bigg) = 0,
\end{align*}
as mentioned above, we  infer that $ g^{\pm}_{n,j,h}$ must satisfy
\begin{align}\label{new-first}
	(\mathcal{L}_S + \mu_n)g^{\pm}_{n,m,h}(y) &=\; 0\\[4pt] \label{new-second}
	(\mathcal{L}_S +\mu_n)g^{\pm}_{n,m-1,h}(y)&= \; - m_n D^{\pm}_y g^{\pm}_{n,m,h}(y)\\[4pt]
	(\mathcal{L}_S + \mu_n)g^{\pm}_{n,j,h}(y) &=\;  - (j +1) D^{\pm}_y g^{\pm}_{n,j+1,h}(y) - (j+2)(j+1) \tfrac{1}{y^2} g^{\pm}_{n,j+2,h}(y)\\[4pt] \nonumber
	0 \leq j \leq m -2,~~~~~&
\end{align}
for any $ h = 0, \dots, m$ and $ D^{\pm}_y =  \big( i(\frac{1}{2} \mp \f12) + \frac{1}{y^2}\big) +  \frac{2}{y} \partial_y$. We may assume that $ h = 0$. If  indeed we set $ h > 0$, then we choose $ g^{\pm}_{n,j,h} = 0$ if $ j > m - h$ and start the same procedure with $ m$ being replaced by $ m -h$. We observe that the constructed solutions (running over  $\pm$ and $ h = 0, \dots, h$) are linearly independent for $\mathbf{L}_S^{\gg1}$. Now $ (\mathcal{L}_S + \mu_n)g^{\pm}_{n,m,0}(y) =  0$ is solved by  $ g^{+}_{n,m,0} = \phi_{\infty}^{(d)}(\mu_n, \cdot) ,~ g^{-}_{n,m,0} = \psi_{\infty}^{(d)}(\mu_n, \cdot)$. For the $ y \geq y_0  \gg1 $  asymptotics of $g^{\pm}_{n,j,m}(y) $ we inductively calculate
\begin{align*}
	&\int_{y_0}^{y}G(y,s) s^{2} e^{\frac{i}{4}s^2} D^{\pm}_y g^{\pm}_{n,j+1,m}(s)~ds + (j+2) \int_{y_0}^{y}G(y,s)  e^{\frac{i}{4}s^2} g^{\pm}_{n,j+2,m}(s)~ds,
\end{align*}
~~\\
with $0 \leq j \leq m -1 $ and $ g^{\pm}_{n,m+1,m} = 0$. Hence  adding linear combinations
\[
\tilde{\beta}^{(1)}_{n,j}  \phi_{\infty}^{(d)}(\mu_n, \cdot) +  \tilde{\beta}^{(2)}_{n,j}  \psi_{\infty}^{(d)}(\mu_n, \cdot),
\]
for suitable $  \tilde{\beta}^{(1)}_{n,j} ,  \tilde{\beta}^{(2)}_{n,j}  \in \C$ implies the required asymptotic expression. We spare further details here and instead proceed to state the following.
\begin{Corollary} \label{Cor-usef} Let $ y_0 \in (0, \infty)$ and $ f_{n, \ell}, g_{n, \ell} \in \C$ be fixed. We also assume $F_{n, \ell}(y)$ are smooth functions on $\R_+$. Then the problem
	\begin{align} \label{dises-sys}
		&(i (T -t) \partial_t + \mathcal{L}_S + \alpha_0)  \big(\sum_{n = 0}^{N_2} (T-t)^{\nu(n +  \frac{1}{2})} \sum_{\ell = 0}^{m}  (\log(y) - \nu \log(T-t))^{\ell} A_{n,  \ell}(y)\big)\\ \nonumber
		& =  \sum_{n \geq 0} (T-t)^{\nu(n +  \frac{1}{2})} \sum_{\ell = 0}^m (\log(y) - \nu \log(T-t))^{\ell}  F_{n,  \ell}(y).
	\end{align}
	has a unique solution $\{ A_{n, \ell} \}$ with data 
	\begin{align}\label{ini-dat}
		A_{n, \ell}(y_0) = f_{n, \ell},~~ \partial_yA_{n, \ell}(y_0)  = g_{n, \ell}.
	\end{align}
	Further we have $ A_{n, \ell}(y) = A^+_{n, \ell}(y)  + A^-_{n, \ell}(y) + A^{\text{inhom}}_{n, \ell}(y)  $ where
	\[
	\sum_{j = 0}^{m}  (\log(y) - \nu \log(T-t))^{j} A^{\pm}_{n, j}(y) = \sum_{j= 0}^{m}  (\log(y) \pm \f12 \log(T-t))^{j} \hat{W}^{\pm}_{n,  j}(y)
	\]
	and for uniquely determined $\beta_{n,j}^{\pm} \in \C,~ 0 \leq j \leq m$ there holds
	\begin{align}
		&\hat{W}^{\pm}_{n,  j}(y) = \sum_{\ell = 0}^{m-j} \beta_{n,\ell}^{\pm} g^{\pm}_{n,j,\ell}(y),~~A^{\text{inhom}}_{n, \ell}(y) = \int_{y_0}^{y} G(y,s) w^{-1}(s) F_{n, \ell}(s) ~ds.
	\end{align}
\end{Corollary}\label{ein-koroll}
\begin{Rem} (i) We assume $N_2 \in \Z_+$ is taken large enough to make sense of \eqref{dises-sys}. Also, the choice of $\beta^{\pm}_{n, j}$ corresponds to satisfying \eqref{ini-dat}. This is seen by passing to a basis for the homogeneous system via 
	\[
	\sum_{j = 0}^{m}  (\log(y) - \nu \log(T-t))^{j} \hat{g}^{\pm}_{n, j, h}(y) = \sum_{j= 0}^{m}  (\log(y) \pm \f12 \log(T-t))^{j} g^{\pm}_{n, h, j}(y),
	\]
	for $  h = 0, 1, \dots, m, $. More precisely the functions
	\begin{align*}
		&	\hat{g}^{\pm}_{n, m, h}(y) = c_{m,0}^{\pm} g^{\pm}_{n, m, h}(y),\\[3pt]
		&	\hat{g}^{\pm}_{n, m-1, h}(y) = c_{m-1,0}^{\pm} g^{\pm}_{n, m-1, h}(y) + c_{m-1,1}^{\pm} \log(y)g^{\pm}_{n, m, h}(y),\\[3pt]
		&	\hat{g}^{\pm}_{n, m-2, h}(y) = c_{m-2,0}^{\pm} g^{\pm}_{n, m-2, h}(y) + c_{m-2,1}^{\pm} \log(y)g^{\pm}_{n, m-1, h}(y) + c_{m-2,2}^{\pm} \log^2(y)g^{\pm}_{n, m, h}(y),\\[2pt]
		& \hat{g}^{\pm}_{n, j, h}(y) = \sum_{\ell = 0}^{m -j}c_{j,\ell}^{\pm} \log^{\ell}(y)g^{\pm}_{n, j, h}(y),
	\end{align*}
	form a basis of  the linear system \eqref{dises-sys} where $ c_{j,\ell}^{\pm} = (1 \pm \tfrac{1}{2\nu})^j (\pm \tfrac{1}{2\nu})^{j + \ell} \binom{j + \ell}{\ell}.$\\[2pt]
	(ii)~~  We may also simply write 
	\begin{align*}
		&\sum_{j = 0}^{m}  (\log(y) - \nu \log(T-t))^{j} A_{n, j}(y) = \sum_{\pm}\sum_{j= 0}^{m}  (\log(y) \pm \f12 \log(T-t))^{j} W^{\pm}_{n,  j}(y),\\
		& W^{\pm}_{n,j}(y) = \sum_{\ell = 0}^{m-j} \beta_{n,\ell}^{\pm} g^{\pm}_{n,j,\ell}(y) + \sum_{\ell = 0}^{m-j} \tilde{c}_{j,\ell}^{\pm} \int_{y_0}^{y} G(y,s) w^{-1}(s) \log^{\ell}(s) F_{n, j + \ell}(s)~ds,
	\end{align*}
	where  $ \{W^{\pm}_{n,j}\}$ are particular solutions of 
	\eqref{sys-gen-large-y-lim} with 
	$
	F^{\pm}_{n,j}(y) = \sum_{\ell = 0}^{m-j} \tilde{c}_{j,\ell}^{\pm} \log^{\ell}(y) F_{n, j + \ell}(y),
	$
	and $\tilde{c}_{j,\ell}^{\pm} = (1 \pm 2 \nu)^j (\pm 2\nu)^{j + \ell} \binom{j + \ell}{\ell} 2^{- j - \ell}.$
\end{Rem}
~~\\
Especially, the solution $\{ A_{n,j}\}$ of Lemma \ref{Lemma-solut-inner1} is of the form given in Corollary \ref{ein-koroll}.
Therefore we need to calculate the $y$-asymptotics of the interactions terms  via the following  Lemma 
, c.f.  \cite[Lemma 2.6]{Perelman}.

		\begin{Lemma}\label{Lemma-solut-inner1-infty} Let   $ \{ A_{n,j}(y)\}$ be the unique solution of  \eqref{three-first} - \eqref{three-last} in Lemma \ref{Lemma-solut-inner1}. Then there exist unique $ a^{\pm}_{n,j} \in \C,~0 \leq j \leq   m$ such that
			\begin{align}\label{decompo1}
				&A_{n,0}(y) = A_{n,0}^{(+)}(y) + A_{n,0}^{(-)}(y),~~ n = 0,1,~~A_{2,1}(y) = A_{2,1}^{(+)}(y) + A_{2,1}^{(-)}(y),\\
				&A_{n,j}(y) = A_{n,j}^{(+)}(y) + A_{n,j}^{(-)}(y) + A^{p}_{n,j}(y),~~ 0 \leq j \leq m,
			\end{align}
			where
			\begin{align} \nonumber
				&\sum_{j \leq \frac{n}{2}} (\log(y) - \nu \log(T-t))^j A^{(\pm)}_{n, j}(y) = \sum_{j \leq \frac{n}{2}} (\log(y)  \pm\f12 \log(T-t))^j \widehat{W}^{(\pm)}_{n, j}(y),\\
				&\widehat{W}^{(+)}_{n,j}(y)  = e^{ 2 i \mu_n  } y^{-\frac{1}{2}} \big(a^{+}_{n,j} + O(y^{-2})\big),~~y \gg1,\\
				&	\widehat{W}^{(-)}_{n,j}(y)  = e^{-2 i \mu_n - \frac{i}{4}y^2 } y^{-\frac{5}{2}} \big(a^{-}_{n,j} + O(y^{-2})\big),~~y \gg1.
			\end{align}
			and there holds
			\begin{align}
				&A^{p}_{n,j}(y)  = \sum_{k = -n-1}^{n}  e^{i k \frac{y^2}{4} } y^{2 i \alpha_0 (2k +1)}  A_{n,j,k}(y).
			\end{align}
			Further we  have the following expansion at $ y \gg1$ 
			\begin{align}
				&A_{n,j,k}(y) = \sum_{\ell = 0 }^{m-j} \sum_{\substack{r = 0\\ \text{r~even}}}^{2(n-k)} w^{\ell, k}_{n,j,r} \log^{\ell}(y) y^{\nu(2(n-k-r) +1)} \cdot O(y^{- \f52 - 2(k+1)}),~~1 \leq k \leq n,\\[3pt] \nonumber
				&A_{n,j,k}(y) = \sum_{\ell = 0 }^{m-j} \sum_{\substack{r = -2k\\ \text{r~even}}}^{2n +2} w^{\ell, k}_{n,j,r} \log^{\ell}(y) y^{\nu(2(n-k-r) +1)} \cdot O(y^{- \f52 - 2(1-k)}) ,~~-n-1 \leq k \leq -2,\\[3pt]
				&A_{n,j,0}(y) = \sum_{\ell = 0 }^{m-j} \sum_{\substack{r = 2\\ \text{r~even}}}^{2n} w^{\ell, 0}_{n,j,r} \log^{\ell}(y) y^{\nu(2(n-r) +1)} \cdot O(y^{- \f52 } ),~~ k =0\\[3pt]
				&A_{n,j,-1}(y) = \sum_{\ell = 0 }^{m-j} \sum_{\substack{r = 1\\ \text{r~odd}}}^{2n-1} w^{\ell, -1}_{n,j,r} \log^{\ell}(y) y^{\nu(2(n-r) +1)} \cdot O(y^{- \f52 } ),~~ k =-1.
			\end{align}
			Here the asymptotics is allowed to depend on $\alpha_0, \nu, n,j,k,\ell, r$.
		\end{Lemma}
		The following representation is a direct consequence of Lemma \ref{Lemma-solut-inner1-infty}.
		\begin{Corollary}\label{Final-form}
			Let $a^{\pm}_{n,j} \in \C,~ 0 \leq j \leq m$ and $\{ A_{n,j}\}$ as in Lemma \ref{Lemma-solut-inner1-infty}, then in particular 
			\begin{align}
				&\sum_{j \leq \frac{n}{2}} (\log(y) - \nu \log(T-t))^j A_{n, j}(y) = \sum_{\pm}\sum_{j \leq \frac{n}{2}} (\log(y)  \pm \f12 \log(T-t))^j W^{(\pm)}_{n, j}(y),\\\label{decomp}
				&W^{(\pm)}_{n,j}(y)  =   \widehat{W}^{(\pm)}_{n,j}(y) +  \sum_{k = -n-1}^{n}  e^{i k \frac{y^2}{4} } y^{2 i \alpha_0 (2k +1)}  W^{\pm}_{n,j,k}(y),
			\end{align}
			with expansion for the interaction part
			\begin{align}
				&W^{\pm}_{n,j,k}(y) = \sum_{\ell = 0 }^{m-j} \sum_{\substack{r = 0\\ \text{r~even}}}^{2(n-k)} w^{\ell, k, \pm}_{n,j,r} \log^{\ell}(y) y^{\nu(2(n-k-r) +1)} \cdot O(y^{- \f52 - 2(k+1)} ),~~1 \leq k \leq n,\\[3pt] \nonumber
				&W^{\pm}_{n,j,k}(y) = \sum_{\ell = 0 }^{m-j} \sum_{\substack{r = -2k\\ \text{r~even}}}^{2n +2} w^{\ell, k,\pm}_{n,j,r} \log^{\ell}(y) y^{\nu(2(n-k-r) +1)} \cdot O(y^{- \f52 - 2(1-k)}), ~~-n-1 \leq k \leq -2,
			\end{align}
			and 
			\begin{align}
				&W^{\pm}_{n,j,0}(y) = \sum_{\ell = 0 }^{m-j} \sum_{\substack{r = 2\\ \text{r~even}}}^{2n} w^{\ell, 0, \pm}_{n,j,r} \log^{\ell}(y) y^{\nu(2(n-r) +1)} \cdot O(y^{- \f52 } ) ,~~ k =0\\[3pt]
				&W^{\pm}_{n,j,-1}(y) = \sum_{\ell = 0 }^{m-j} \sum_{\substack{r = 1\\ \text{r~odd}}}^{2n-1} w^{\ell, -1, \pm}_{n,j,r} \log^{\ell}(y) y^{\nu(2(n-r) +1)} \cdot O(y^{- \f52 }),~~ k =-1.
			\end{align}
			
		\end{Corollary}
		\begin{proof}[Proof of Lemma \ref{Lemma-solut-inner1-infty}]
			We start with solving $( \mathcal{L}_S + \mu_n) A_{n,0} = ( \mathcal{L}_S + \mu_2) A_{2,1} = 0$ for $ n = 0,1$, i.e. we set
			\begin{align}
				&	\widehat{W}^{(+)}_{0,0}(y)  =  \beta_{0,0}^+ \cdot g^{+}_{0,0,0}(y),~~ 	\widehat{W}^{(-)}_{0,0}(y) = \beta_{0,0}^- \cdot g^{-}_{0,0,0}(y),\\
				&	\widehat{W}^{(+)}_{1,0}(y)  =  \beta_{1,0}^+\cdot g^{+}_{1,0,0}(y),~~ 	\widehat{W}^{(-)}_{1,0}(y) = \beta_{1,0}^-\cdot g^{-}_{1,0,0}(y),\\
				&	\widehat{W}^{(+)}_{2,1}(y)  =  \beta_{2,1}^+\cdot g^{+}_{2,1,1}(y),~~ 	\widehat{W}^{(-)}_{2,1}(y) = \beta_{2,1}^-\cdot g^{-}_{2,1,1}(y),
			\end{align}
			where the $ g^{\pm}_{n,j,h}$ are  fundamental solutions of $(\mathcal{L}_S + \mu_n)g =0$. Hence we set
			$
			\beta_{n,0}^{\pm} = a_{n,0}^{\pm} ,~  \beta_{2,1}^{\pm} =  a_{2,1}^{\pm} $ where $n = 0,1.
			$
			Now for $n = 2, j= 0$ we have 
			\[
			(\mathcal{L}_S + \mu_2)A_{2,0}(y) = - D_y A_{2,1}(y) - |A_{0,0}(y)|^4 A_{0,0}(y),
			\]
			i.e. for the  'preliminary' homogeneous solutions $\tilde{A}_{2,0}^{(\pm)}$ we solve
			\begin{align*}
				&(\mathcal{L}_S + \mu_2 ) \widehat{W}^{(\pm)}_{2,0}(y) = - D^{\pm}_y\widehat{W}^{(\pm)}_{2,1}(y),~~\widehat{W}^{(\pm)}_{2,0}(y) = \beta^{\pm}_{2,0} \cdot g^{\pm}_{2,0,0}(y) + \beta_{2,1}^{\pm}\cdot g_{2,0,1}^{\pm}(y),\\
				&\tilde{A}_{2,0}^{(\pm)}(y) = \tilde{c}^{\pm}_{0,0} \widehat{W}^{(\pm)}_{2,0}(y) +  \tilde{c}^{\pm}_{0,1} \log(y) \widehat{W}^{(\pm)}_{2,1}(y).
			\end{align*}
			Further we infer by direct calculation in \eqref{three-first} - \eqref{three-last}
			\begin{align*}
				&|A_{0,0}(y)|^4 A_{0,0}(y) = e^0_{1,1}(y) = \sum_{k = -3}^2 e^{i k \frac{y^2}{4}}y^{2 i \alpha_0(2k+1)} e^{0,k}_{1,1}(y),
			\end{align*}
			where for $y \to \infty$ we have the expansion
			\begin{align*}
				&e^{0,1}_{1,1}(y) = \hat{w}^{1}_{2,0,0} y^{3 \nu} \cdot O(y^{- \f52 -2}) +  \hat{w}^{1}_{2,0,2} y^{-\nu} \cdot O(y^{- \f52 -2}),\\
				&e^{0,2}_{1,1}(y) = \hat{w}^{2}_{2,0,0} y^{  \nu} \cdot O(y^{- \f52 -4}),\\
				&e^{0,-3}_{1,1}(y) = \hat{w}^{-3}_{2,0,6} y^{- \nu} \cdot O(y^{- \f52 -6}),\\ 
				&e^{0,-2}_{1,1}(y) = \hat{w}^{-2}_{2,0,4} y^{ \nu} \cdot O(y^{- \f52 -4}) +  \hat{w}^{-2}_{2,0,6} y^{-3\nu} \cdot O(y^{- \f52 -4}),\\[4pt]
				&e^{0,0}_{1,1}(y) = \hat{w}^{0}_{2,0,4} y^{-3 \nu} \cdot O(y^{- \f52 }) +  \hat{w}^{0}_{2,0,2} y^{\nu} \cdot O(y^{- \f52 }) +  \hat{w}^{0}_{2,0,0} y^{5\nu} \cdot O(y^{- \f52 }),\\
				&e^{0,-1}_{1,1}(y) = \hat{w}^{-1}_{2,0,3} y^{- \nu} \cdot O(y^{- \f52 }) +  \hat{w}^{-1}_{2,0,1} y^{3\nu} \cdot O(y^{- \f52 }) +  \hat{w}^{-1}_{2,0,5} y^{-5\nu} \cdot O(y^{- \f52 }).
			\end{align*}
			Hence  integrating $(\mathcal{L}_S + \mu_2) \tilde{A}_{n,j}(y) = - e^{0,0}_{1,1}(y) $  we obtain the solution
			\[
			\tilde{A}_{n,j}(y) =   \tilde{\beta}_{2,0}^{+}\phi_{\infty}^{(3)}(\mu_2, y)  + \tilde{\beta}_{2,0}^{-}\psi_{\infty}^{(3)}(\mu_2, y)    + \sum_{k = -3}^2 e^{i k \frac{y^2}{4}}y^{2 i \alpha_0(2k+1)} A_{2,0,k}(y),
			\]
			where $ (\mathcal{L}_S + \mu_2)(e^{ik \frac{y^2}{4}} y^{2 i \alpha_0(2k +1)}A_{2,0,k}(y)) = e^{ik \frac{y^2}{4}} y^{2 i \alpha_0(2k +1)} e^{0,k}_{1,1}(y)$ and for $ y \gg1$ we have the representation
			\begin{align*}
				& A_{2,0,1}(y) = \sum_{\ell = 0,1}\log^{\ell}(y)\big(w^{\ell, 1}_{2,0,0} y^{3 \nu} \cdot O(y^{- \f52 -4}) +  w^{\ell,1}_{2,0,2} y^{-\nu} \cdot O(y^{- \f52 -4})\big),\\
				&A_{2,0,2}(y)  = \sum_{\ell = 0,1}\log^{\ell}(y)\big(w^{\ell, 2}_{2,0,0} y^{  \nu} \cdot O(y^{- \f52 -6}) \big),\\
				&A_{2,0,-3}(y)  = \sum_{\ell = 0,1}\log^{\ell}(y)\big(w^{\ell, -5}_{2,0,6} y^{- \nu} \cdot O(y^{- \f52 -8})\big),\\ 
				&A_{2,0,-2}(y)  = \sum_{\ell = 0,1}\log^{\ell}(y)\big( w^{\ell -2}_{2,0,4} y^{  \nu} \cdot O(y^{- \f52 -6}) +  w^{\ell, -2}_{2,0,6} y^{-3\nu} \cdot O(y^{- \f52 -6})\big).
			\end{align*} 
			Further we set 
			\begin{align*}
				&e^{0,0,+}_{1,1}(y) := \hat{w}^{0}_{2,0,0} y^{5\nu} \cdot O(y^{- \f52 }),~ e^{0,-1, +}_{1,1}(y): = \hat{w}^{-1}_{2,0,5} y^{-5\nu} \cdot O(y^{- \f52 }),\\
				&e^{0,0,-}_{1,1}(y) =  e^{0,0}_{1,1}(y)- e^{0,0,+}_{1,1}(y),~e^{0,-1,-}_{1,1}(y) =  e^{0,-1}_{1,1}(y)- e^{0,-1,+}_{1,1}(y)
			\end{align*}
			and 
			write $ A_{2,0,0}(y) = A_{2,0,0}^+(y)  + A_{2,0,0}^-(y),~ A_{2,0,-1}(y) = A_{2,0,-1}^+(y)  + A_{2,0,-1}^-(y)$ where 
			\begin{align*}
				&(\mathcal{L}_S + \mu_2)(e^{- i \frac{y^2}{4}} y^{-2i \alpha_0} A_{2,0,-1}^{\pm}(y)) = e^{- i \frac{y^2}{4}} y^{-2i \alpha_0}e^{0,0,\pm}_{1,1}(y),\\
				&(\mathcal{L}_S + \mu_2)( y^{2i \alpha_0} A_{2,0,0}^{\pm}(y)) = y^{2i \alpha_0}e^{0,0,\pm}_{1,1}(y).
			\end{align*}
			In particular we have the representation
			\begin{align*}
				&A_{2,0,0}^+(y) =  \sum_{\ell = 0,1} \log^{\ell}(y)w^{\ell, 0}_{2,0,0} y^{2\nu(2 + \f12)} \cdot O(y^{- \f52 }),\\
				&A_{2,0,0}^-(y) =  \sum_{\ell = 0,1}\log^{\ell}(y) \big(  w^{\ell, 0}_{2,0,4} y^{-3 \nu} \cdot O(y^{- \f52 }) +  w^{\ell, 0}_{2,0,2} y^{\nu} \cdot O(y^{- \f52 })\big),\\
				&A_{2,0,-1}^+(y) =  \sum_{\ell = 0,1}\log^{\ell}(y) w^{\ell, -1}_{2,0,0} y^{-2\nu(2 + \f12)} \cdot O(y^{- \f52 }),\\
				&A_{2,0,-1}^-(y) =  \sum_{\ell = 0,1}\log^{\ell}(y) \big( w^{\ell,-1}_{2,0,3} y^{ -\nu} \cdot O(y^{- \f52 }) +  w^{\ell,-1}_{2,0,1} y^{3\nu} \cdot O(y^{- \f52 })\big).
			\end{align*}
			We can now set 
			\begin{align*}
				&A_{2,0}^{(+)}(y) = \tilde{A}_{2,0}^{(+)}(y) +  \tilde{\beta}_{2,0}^{+}\phi_{\infty}^{(3)}(\mu_2, y)  + y^{2i \alpha_0} A_{2,0,0}^{+}(y),\\
				&A_{2,0}^{(-)}(y) = \tilde{A}_{2,0}^{(-)}(y) +  \tilde{\beta}_{2,0}^{-}\psi_{\infty}^{(3)}(\mu_2, y)  + e^{- i \frac{y^2}{4}} y^{-2i \alpha_0} A_{2,0,-1}^{+}(y).
			\end{align*}
			We note here the latter terms $A_{2,0,-1}^{+},~A_{2,0,0}^{+}$ correspond to exactly one interaction term and are of the required form for $A_{2,0}^{(+)},~A_{2,0}^{(-)}$. Especially transforming $ W_{2,0}^{(\pm)}(y) = c^{\pm}_{0,0} A^{(\pm)}_{2,0}(y) + \log(y) c^{\pm}_{0,1} A^{(\pm)}_{2,1}(y) $ implicitly defines the $a^{\pm}_{2,0}$ coefficients.\\[3pt]
			\emph{Induction step}. Assume we have the decomposition of $A_{k,j}(y)$ for $ 0 \leq j \leq \lfloor k \slash 2 \rfloor,~ 2 \leq k \leq n-1$.  Then we observe for the interaction in \eqref{three-first} - \eqref{three-last}
			\[
			e^0_{n-1, j}(y) = \sum_{k = - n-1}^n e^{i k \frac{y^2}{4}} y^{2i \alpha_0(2k+1)} e^{0, k}_{n-1,j}(y),~ 0 \leq j \leq m,
			\]
			where
			\begin{align*}
				&e^{0,k}_{n-1,j}(y) =  \sum_{\ell = 0 }^{m-j} \sum_{\substack{r = 0\\ \text{r~even}}}^{2(n-k)} \hat{w}^{\ell, k}_{n,j,r} \log^{\ell}(y) y^{\nu(2(n-k-r) +1)} \cdot O(y^{- \f52 - 2k}),~~1 \leq k \leq n,\\[3pt] \nonumber
				&e^{0,k}_{n-1,j}(y) = \sum_{\ell = 0 }^{m-j} \sum_{\substack{r = -2k\\ \text{r~even}}}^{2n +2} \hat{w}^{\ell, k}_{n,j,r} \log^{\ell}(y) y^{\nu(2(n-k-r) +1)} \cdot O(y^{- \f52 + 2k }) ,~~-n-1 \leq k \leq -2,\\[3pt]
				&e^{0,0}_{n-1,j}(y) = \sum_{\ell = 0 }^{m-j} \sum_{\substack{r = 0\\ \text{r~even}}}^{2n} \hat{w}^{\ell, 0}_{n,j,r} \log^{\ell}(y) y^{\nu(2(n-r) +1)} \cdot O(y^{- \f52 } ),~~ k =0\\[3pt]
				&e^{0,-1}_{n-1,j}(y) = \sum_{\ell = 0 }^{m-j} \sum_{\substack{r = 1\\ \text{r~odd}}}^{2n+1} \hat{w}^{\ell, -1}_{n,j,r} \log^{\ell}(y) y^{\nu(2(n-r) +1)} \cdot O(y^{- \f52 } ),~~ k =-1.
			\end{align*}
			The latter two lines we split into 
			\begin{align*}
				&e^{0,0,+}_{n-1,j}(y) = \sum_{\ell = 0 }^{m-j}  \hat{w}^{\ell, 0}_{n,j,0} \log^{\ell}(y) y^{2\nu(n + \f12)} \cdot O(y^{- \f52 } ),\\
				&e^{0,-1,+}_{n-1,j}(y) =  \sum_{\ell = 0 }^{m-j}  \hat{w}^{\ell, -1}_{n,j, 2n+1} \log^{\ell}(y) y^{-2 \nu(n + \f12)} \cdot O(y^{- \f52 } ),\\
				&e^{0,0,-}_{n-1,j}(y)  = \sum_{\ell = 0 }^{m-j} \sum_{\substack{r = 2\\ \text{r~even}}}^{2n} \hat{w}^{\ell, 0}_{n,j,r} \log^{\ell}(y) y^{\nu(2(n-r) +1)} \cdot O(y^{- \f52 } ),\\
				& e^{0,-1,-}_{n-1,j}(y) = \sum_{\ell = 0 }^{m-j} \sum_{\substack{r = 1\\ \text{r~odd}}}^{2n-1} \hat{w}^{\ell, -1}_{n,j,r} \log^{\ell}(y) y^{\nu(2(n-r) +1)} \cdot O(y^{- \f52 } ).
			\end{align*}
			Now integrating \eqref{three-first} - \eqref{three-last} gives
			\begin{align*}
				&A_{n,j}(y) =   \tilde{A}^{(+)}_{n,j}(y) + \tilde{A}^{(-)}_{n,j}(y)  + \sum_{k = -n-1}^n e^{i k \frac{y^2}{4}} y^{2 i \alpha_0(2k+1)}A_{n,j,k}(y),\\
				& \sum_{j \leq m} (\log(y) - \nu \log(T-t))^j\tilde{A}^{(\pm)}_{n,j}(y)  = \sum_{j \leq m}(\log(y) \pm \tfrac{1}{2} \log(T-t))^j \widetilde{W}^{(\pm)}_{n,j},
			\end{align*}
			where
			\begin{align*}
				& W_{n,j}^{(+)}(y) = \sum_{h = 0}^m \beta^{+}_{n,j} g^{+}_{n,j,h}(y) + \tilde{\beta}^{+}_{n,j} \phi^{(3)}_{\infty}(y),~~W_{n,j}^{(-)}(y) = \sum_{h = 0}^m \beta^{-}_{n,j} g^{-}_{n,j,h}(y) + \tilde{\beta}^{-}_{n,j} \psi^{(3)}_{\infty}(y).
			\end{align*}
			We have  $(\mathcal{L}_S + \mu_n)A_{n,m}(y) = 0$ and thus we take $A_{n,m, k}(y) = 0$. Further
			$e^{i k \frac{y^2}{4}} y^{2 i \alpha_0(2k+1)}A_{n,j,k}(y)$  for $ 0 \leq j \leq m-1$ solves \eqref{three-first} - \eqref{three-last} with data $ e^{i k \frac{y^2}{4}} y^{2 i \alpha_0(2k+1)}e^{0,k}_{n-1,j}(y)$ and where we set $ A_{n,m}(y) = 0$. Hence $A_{n,j,k}(y) $ has the following expansion as $ y \to \infty$
			\begin{align*}
				&A_{n,j,k}(y) =  \sum_{\ell = 0 }^{m-j} \sum_{\substack{r = 0\\ \text{r~even}}}^{2(n-k)} w^{\ell, k}_{n,j,r} \log^{\ell}(y) y^{\nu(2(n-k-r) +1)} \cdot O(y^{- \f52 - 2(k+1)}),~~1 \leq k \leq n,\\ \nonumber
				&A_{n,j,k}(y) = \sum_{\ell = 0 }^{m-j} \sum_{\substack{r = -2k\\ \text{r~even}}}^{2n +2} w^{\ell, k}_{n,j,r} \log^{\ell}(y) y^{\nu(2(n-k-r) +1)} \cdot O(y^{- \f52 - 2(1-k)}) ,~~-n-1 \leq k \leq -2.
			\end{align*}
			Now we express $ e^{-i \frac{y^2}{4}} y^{-2i \alpha_0} A_{n, j ,-1}^{\pm}(y)$ and $  y^{2i \alpha_0} A_{n, j ,0}^{\pm}(y)$ as solutions of \eqref{three-first} - \eqref{three-last} with respective data 
			$
			e^{-i \frac{y^2}{4}} y^{-2i \alpha_0}e^{0,-1,\pm}_{n-1,j}(y)$, $   y^{2i \alpha_0}e^{0,0,\pm}_{n-1,j}(y),$  and where $A_{n,m}(y) = 0$. Hence as $ y \to \infty$ we infer
			\begin{align*}
				&A_{n,j,0}^+(y) = \sum_{\ell = 0 }^{m-j}  w^{\ell, 0}_{n,j,0} \log^{\ell}(y) y^{2\nu(n + \f12)} \cdot O(y^{- \f52 } ),\\
				&A_{nj,-1}^+(y) =  \sum_{\ell = 0 }^{m-j}  w^{\ell, -1}_{n,j, 2n+1} \log^{\ell}(y) y^{-2 \nu(n + \f12)} \cdot O(y^{- \f52 } ),\\
				&A_{n,j,0}^-(y)  = \sum_{\ell = 0 }^{m-j} \sum_{\substack{r = 2\\ \text{r~even}}}^{2n} w^{\ell, 0}_{n,j,r} \log^{\ell}(y) y^{\nu(2(n-r) +1)} \cdot O(y^{- \f52 } ),\\
				&A_{n,j,-1}^-(y) = \sum_{\ell = 0 }^{m-j} \sum_{\substack{r = 1\\ \text{r~odd}}}^{2n-1} w^{\ell, -1}_{n,j,r} \log^{\ell}(y) y^{\nu(2(n-r) +1)} \cdot O(y^{- \f52 } ).
			\end{align*}
			As before we find 
			\begin{align*}
				&A_{n,j}^{(+)}(y) =  \tilde{A}_{n,j}^{(+)}(y) + y^{2i \alpha_0} A_{n, j ,0}^{\pm}(y),~~A_{n,j}^{(-)}(y) =  \tilde{A}_{n,j}^{(-)}(y) + e^{-i \frac{y^2}{4}} y^{-2i \alpha_0} A_{n, j ,-1}^{\pm}(y),
			\end{align*}
			and calculating $ W^{(\pm)}_{n,j}(y) = \sum_{\ell = 0}^{m-j} c_{j, \ell}^{\pm} \log^{\ell}(y) A_{n, j + \ell}^{(\pm)}(y)$ defines $ a_{n,j}^{\pm},~ 0 \leq j \leq m$ successively. Here we may note that coefficients in the expansion of $W^{(\pm)}_{n,j}(y)$ are determined by the solutions  $A_{n,j}^{(\pm)}(y)$ of \eqref{three-first} - \eqref{three-last} with data $ y^{2i \alpha_0}e^{0,0,\pm}_{n-1,j}(y)$ and $ e^{-i \frac{y^2}{4}} y^{-2i \alpha_0}e^{0,-1,\pm}_{n-1,j}(y)$  depending only on $ A_{k,j}^{(+)}(y)$ and $ A_{k,j}^{(-)}(y)$ for $ k \leq n-1$ respectively.
		\end{proof}
		~~\\
		We restrict to $d =3 $ dimensions from now on. Recall $\tilde{u}^{N_2}_S$ is the self-similar approximation and with $ \lambda(t)^{-\f12}(T-t)^{- \f14} = (T-t)^{\frac{\nu}{2}}$ we set 
		\begin{align*}
			&u_S^{N_2}(t,R ) : = (T-t)^{\frac{\nu}{2}} \tilde{u}^{N_2}_S(t, (T-t)^{\nu}R) = (T-t)^{\frac{\nu}{2}}\sum_{n = 0}^{N_2} A_n(t,(T-t)^{\nu}R),\\
			&A_n(t,y) = \sum_{j \leq m}(\log(y) - \nu \log(T-t))^j A_{n,j}(y),
		\end{align*}with error $ e^{N_2}_{S}(t,(T-t)^{\nu}R)$ where
		\begin{align*}
			& e^{N_2}_{S}(t,y) =  i (T-t) \partial_t \tilde{u}^{N_2}_{S}(t,y )+ (\partial_{yy} + (d-1) \tfrac{1}{y}\partial_y) \tilde{u}^{N_2}_{S}(t,y)\\[1pt] \nonumber
			&~~~~~~~~~~~~~~~~~~~~~ + [i (\tfrac{d-2}{4} + \tfrac12 y \partial_y) + \alpha_0] \tilde{u}^{N_2}_{S}(t,y) + |  \tilde{u}^{N_2}_{S}(t,y) |^{\frac{4}{d-2}} \tilde{u}^{N_2}_{S}(t,y).
		\end{align*} 
		For $ T > 0$ and $ t \in [0, T)$, we set 
		\[
		\chi_{S}(t,R) = \begin{cases}
			1 & \frac{1}{\tilde{C}} (T-t)^{\epsilon_1 - \nu} \leq R \leq \tilde{C} (T -t)^{- \epsilon_2 - \nu} \\
			0 & \text{otherwise}.
		\end{cases}
		\]
		The following is a direct consequence of Lemma \ref{Lemma-solut-inner1} and Lemma \ref{Lemma-solut-inner1-infty} (and its proof). We fix $ \epsilon_1 : = \frac{\nu}{2}$ and let $ 0 < \epsilon_2 < \f12$ be fixed later. (We use fixing $\epsilon_1$ especially in \eqref{self-est3}. Surely we may choose $ 2\epsilon_1 < \nu$ to make the estimate work.)
		\begin{Corollary}\label{estimates-self} For $ \alpha_0 \in \R,~ \epsilon_1 : =\frac{\nu}{2},~ \f12 > \epsilon_2 > 0$  there exists $  0  <  T_0(\alpha_0, \nu, N_2) \leq 1 $, such that for all $ 0<T \leq T_0$ the functions $ u^{N_2}_{S}, e^{N_2}_{S}$  satisfy for all $ t \in [0,T)$
			\begin{align}\label{self-est1}
				&\|  \chi_{S}(t,R) \cdot (W - u_S^{N_2})  \|_{L^{\infty}_R} \leq C_{\nu, |\alpha_0|} (T - t )^{\nu},\\[2pt]\label{self-est2}
				&\|  \chi_{S}(t,R) \cdot R^{-j}\partial_R^{i}(W - u_S^{N_2})  \|_{L^{\infty}_R} \leq C_{\nu, |\alpha_0|} (T - t )^{2\nu},~~~1 \leq j + i \leq 2,\\[2pt]\label{self-est3}
				&\|  \chi_{S}(t,R) \cdot R^{-j}\partial_R^{i}(W - u_S^{N_2}) \|_{L^{2}(R^{2}dR)} \leq C_{\nu, |\alpha_0|} (T - t )^{\frac{\nu}{2}(1-2 \epsilon_2)},~~~0 \leq j + i,\\[2pt]\label{self-est4}
				&\|  \chi_{S}(t,R) \cdot R^{-j}\partial_R^{i}e^{N_2}_{S}\|_{L^{2}(R^{2}dR)} \leq C_{\nu, |\alpha_0|} (T - t )^{ \nu N_2(1 - 2 \epsilon_2) - \f32(\nu + \epsilon_2)- \f52\epsilon_1  },~~~0 \leq j + i ,\\
				&\|  \chi_{S}(t,R) \cdot R^{-j}\partial_R^{i}e^{N_2}_{S}\|_{L^{\infty}_R} \leq C_{\nu, |\alpha_0|} (T - t )^{ \nu N_2(1 - 2 \epsilon_2) - \f52\epsilon_1  },~~~0 \leq j + i.
			\end{align}
			Further by Lemma \ref{consist} we obtain for $(t,R) $ with $ (t,r) \in I \cap S$ for $ N \in \Z_+$ 
			\begin{align}\label{self-est5}
				\big | \partial_R^m& \big[ u^{N_2}_S(t, (T-t)^{\nu}R) - u^{N_1}_{In}(t, R)\big]\big | \leq C_{m , N_1, N_2} R^{-m} (T-t)^{N \nu -1},~~0 \leq m,
			\end{align}
			where $ N_1= N_1(N), N_2 = N_2(N_1)$ are large enough and $ T = T(\alpha_0, \nu, N_1,N_2)$.
		\end{Corollary}
		\subsection{Remote region $ (T -t)^{\f12 - \epsilon_2  }  \lesssim r  $}\label{subsec:RR}~For some constant $  0 < C_0 <\tilde{C} $ and  the previous choice of $ 0 < \epsilon_2 < \f12 $  we restrict to the region
		\begin{align}\label{conditionRemo}
			R = \big\{ (t,r) ~|~C_0 (T -t)^{  - \epsilon_2} \leq r (T-t)^{- \f12}  \big\}.
		\end{align}
		Thus if $(t,r) \in S \cap R$, then $ y \sim (T-t)^{- \epsilon_2}, r \sim (T-t)^{\f12 - \epsilon_2}$. For $ 0 < T-t \ll1$ we may consider the $ y \to \infty$ asymptotics of $ u_S^N(t, y) = u_S^N(t, (T-t)^{- \f12 }r)$, which corresponds to taking $ r \ll1$. In particular 
		by Lemma \ref{Lemma-solut-inner1-infty} and Corollary \ref{Final-form} we have 
		\begin{align*}
			\tilde{u}_S^N(t, y) &=	\sum_{n = 0}^N(T-t)^{\nu(n + \f12)}\sum_{j \leq \frac{n}{2}} (\log(y) - \nu \log(T-t))^j A_{n, j}(y)\\
			&= \sum_{n = 0}^N(T-t)^{\nu(n + \f12)}\sum_{j \leq \frac{n}{2}} \log^j(r)W^{(+)}_{n, j}((T-t)^{- \f12}r),\\
			&~~~ + \sum_{n = 0}^N(T-t)^{\nu(n + \f12)}\sum_{j \leq \frac{n}{2}} (\log(r\slash (T-t))^j W^{(-)}_{n, j}((T-t)^{- \f12}r).
		\end{align*}
		with expansions
		\begin{align}\label{erste-radia}
			&\sum_{n = 0}^N(T-t)^{\nu(n + \f12)}\sum_{j \leq \frac{n}{2}} \log^j(r)W^{(+)}_{n, j}((T-t)^{- \f12}r)\\ \nonumber
			&~~~ =  \tilde{W}^+_{\text{high}}(t,r) + (T-t)^{\f14 - i \alpha_0}\sum_{n = 0}^N\sum_{j \leq \frac{n}{2}} \log^j(r) r^{2i \alpha_0 +\nu(2n + 1) - \f12 }\\ \nonumber
			&~~~~~~~~~~~~~~~~~\hspace{3cm}~~~\times  \big( \sum_{k = 0}^{\tilde{N}}a_{n,j,0}^k(T-t)^{k}\slash r^{2k} + O\big( \frac{(T-t)^{\tilde{N}+1}}{r^{2\tilde{N} + 2}}\big)\big),\\ \label{zweite-oscil}
			& \sum_{n = 0}^N(T-t)^{\nu(n + \f12)}\sum_{j \leq \frac{n}{2}} (\log(r\slash (T-t))^j W^{(-)}_{n, j}((T-t)^{- \f12}r)\\ \nonumber
			& =  \tilde{W}^-_{\text{high}}(t,r)  +  (T-t)^{-i \alpha_0 - \frac{5}{4}}e^{- i \frac{r^2}{4(T-t)}}\sum_{n = 0}^N\sum_{j \leq \frac{n}{2}} \log^j(r\slash (T-t)) (r \slash (T-t))^{-2i \alpha_0 - \nu(2n + 1) - \frac{5}{2} }\\ \nonumber
			&~~~~~~~~~~~~~~~~~\hspace{3cm}~~~\times \big( \sum_{k = 0}^{\tilde{N}}a_{n,j,-1}^k(T-t)^{k}\slash r^{2k} + O\big( \frac{(T-t)^{\tilde{N}+1}}{r^{2\tilde{N} + 2}}\big)\big). 
		\end{align}
		Some remarks are in order: First, we note in the remote region $(t,r) \in  R$ we have
		\[
		\frac{(T-t)^{\f12}}{r} \lesssim (T-t)^{\epsilon_2},~~~~ \frac{(T-t)}{r} \lesssim (T-t)^{ \f12 + \epsilon_2},~~ t \in [0,T)
		\]
		and hence contributions such as the second term on the right of  \eqref{zweite-oscil} will be substantially small for $ T - t \ll1$. Secondly, for the full approximation, we multiply both sides of \eqref{erste-radia} and \eqref{zweite-oscil} with 
		\[    (T-t)^{-\f14} e^{i \alpha_0 \log(T-t)} = (T-t)^{-\f14 +i \alpha_0},
		\]
		thus the second term on the right of \eqref{erste-radia} is time-independent to leading order ($k=0$).\\[3pt] Third, the terms   $\tilde{W}^+_{\text{high}}(t,r), \tilde{W}^-_{\text{high}}(t,r) $ are higher order corrections as described in Corollary \ref{Final-form}. Different to the second term on the right of \eqref{erste-radia}, we consider these  \emph{decaying terms}  as $ t \to T^-$ in the $(t,r) $ coordinate frame, with factors of order $O(y^{\pm2\nu (m + \f12) - \f52 -\tilde{m}}),~ m , \tilde{m} \in \Z,~ \tilde{m} \geq0$ in the expansion. Especially, the $m \in \Z$ is 'better' than   in the above lower order expansion, since there we exactly cancel $ (T-t)^{\nu(n + \f12)}$.\\[3pt] 
		Thus, the aymptotics provides  a time-independent \emph{radiation} profile
		\[
		\tilde{f}(r): = \sum_{n = 0}^N\sum_{j \leq \frac{n}{2}} \log^j(r) r^{2i \alpha_0 +\nu(2n + 1) - \f12 } \beta_{n,j},~~~ \beta_{n,j} = a^0_{n,j,0},
		\]
		to the lowest order expansion ($k=0$) in \eqref{erste-radia}. The low order term in \eqref{zweite-oscil}, however, is rapidly oscillating in the $(t,r)$ frame (and decays sufficiently). Now we set 
		\begin{align}\label{time-indep}
			&f_0(r) = f_0^N(r) : =  \Theta(\delta^{-1} x) \sum_{n = 0}^N\sum_{j \leq \frac{n}{2}} \log^j(r) r^{2i \alpha_0 +\nu(2n + 1) - \f12 } \beta_{n,j},\\
			&\delta> 0,~ N \in \Z_+,~ N  \gg1,~ r = |x|,~ x \in \R^3,
		\end{align}
		where $\Theta \in C^{\infty}(\R^3)$ is a radial positive cut-off function with $ \Theta(x) = 1 $ if $ |x| \leq 1$ and  $\Theta(x) = 0$ if $ |x| \geq 2$. We write $ \Theta(x) = \tilde{\Theta}(|x|)$. Especially, we obtain $ e^{\theta } f_0(r) \in H^{1 + \nu-}(\R^3)$ for $ \theta \in \R$ and 
		\[
		\|  e^{i \theta} f_0  \|_{\dot{H}^{s}} \leq C \delta^{1 + \nu - s},~ 0 \leq s < 1 + \nu.
		\]
		We seek  to obtain a solution of \eqref{NLSequation} via a perturbation of \eqref{time-indep}, c.f. \cite[Section 2.4]{Perelman}. In particular we set 
		\[
		w(t,r) = f_0^N(r) + \eta(t,r),~~\eta(t,r) = e^{i \alpha_0 \log(T-t)} (T-t)^{- \f14} \tilde{\eta}(t,r),
		\]
		where  $w(t,r)$ solving \eqref{NLSequation} is  equivalent to
		\begin{align}\label{this-foreta-remot}
			i \partial_t \eta + \Delta \eta + V_0(f, \bar{f_0}) + V_1(f, \bar{f_0}) \eta + V_2(f, \bar{f_0}) \bar{\eta} + \mathcal{N}(f, \bar{f_0}, \eta, \bar{\eta}) = 0,
		\end{align}
		with
		\begin{align}
			&V_0(f, \bar{f_0})  = \Delta f_0 + |f_0|^4 f_0,~~V_1(f, \bar{f_0}) =   3 |f_0|^4,~~~V_2(f, \bar{f_0}) = 2 |f_0|^2 f_0^2\\[5pt]
			& \mathcal{N} =  N_0 + N_1 + N_2,\\[3pt]
			&N_0(f_0, \bar{f_0}, \eta, \bar{\eta}) = 4 f_0^2 \bar{f_0} \eta \bar{\eta} + f_0^3 \bar{\eta}^2 + (f_0^3 + 2 f_0 \bar{f_0}^2)\eta^2,\\[3pt]
			&N_1(f_0, \bar{f_0}, \eta, \bar{\eta}) = 3 f_0^2 \bar{\eta}^2 \eta + 6 f_0 \bar{f_0} \eta^2 \bar{\eta} + \bar{f_0}^2 \eta^3,\\[3pt]
			& N_2(f_0, \bar{f_0}, \eta, \bar{\eta}) = 3 f_0 \eta^2 \bar{\eta}^2 + 2 \bar{f_0}\eta^3 \bar{\eta} + \eta^3 \bar{\eta}^2.
		\end{align}
		~~\\
		Then, according to Lemma \ref{Lemma-solut-inner1-infty} and Corollary \ref{Final-form}  we should choose $ \eta(t,r)$ of the form
		\begin{align*}
			\eta(t,r) = \sum_{n = 0}^{N} \sum_{j \geq 1} \sum_{q = 0}^{2n+1} (T-t)^{\nu q + j} \sum_{\substack{k = - \min\{j,q \} \\ q-k ~\text{even}}}^{\min\{(j-2)_+,q\}} \sum_{\ell = 0}^{\lfloor \frac{n}{2} \rfloor} e^{ i k\Phi}(\log(r) - \log(T-t))^{\ell} g_{q,j,k,\ell}^{n}(r),
		\end{align*}
		where  $ \sum_{j \geq 1}$ is a finite sum and we set 
		\[
		\Phi(r,t) := - 2 \alpha_0 \log(T-t)  +  \frac{r^2}{4 (T-t)} + \phi(r)
		\]
		with $\phi(r)$ to be chosen. For the sake of simplicity we may rewrite this as follows
		\begin{align} \label{form-remot}
			\eta(t,r) = \sum_{\substack{q \geq 0, \\j \geq 1}}   (T-t)^{\nu q + j} \sum_{\substack{- \Omega_{j,q} \leq k \leq \Omega_{(j-2)_+,q} \\ q-k ~\text{even}}} \sum_{\ell \geq 0} e^{ i k\Phi}(\log(r) - \log(T-t))^{\ell} g_{q,j,k,\ell}(r),
		\end{align}
		where $ \sum_{q \geq 0, j \geq1}, \sum_{\ell \geq 0}$ are finite sums and $ \Omega_{j,q} : = \min\{j,q\}$. Here we find $ m = m(q) \in \N_0$ (increasing in $q$), such that  $g_{q,j,k,\ell} = 0$ if $ \ell > m$. We consider this form \eqref{form-remot} in \eqref{this-foreta-remot}, i.e. we start by setting 
		\[ D\eta := i \partial_t \eta + \Delta \eta + V_1 \eta + V_2 \bar{\eta}.
		\]
		We note the contributions of order $O((T-t)^{-2})$  in $ D\eta$ are exactly
		\[i \partial_t e^{ i k \frac{r^2}{(T-t)4}} = - k \frac{r^2}{4 (T-t)^2} e^{ i k \frac{r^2}{(T-t)4}},\hspace{1cm}- k^2 \frac{r^2}{(T-t)^24}e^{ i k \frac{r^2}{(T-t)4}}, 
		\]
		where the latter appears in $ \Delta \eta(t,r)$. Thus, the sum of these contributions is absent if $  k = -1, k =0$ and we therefore write
		\begin{align*}
			&D \eta(t,r) = \sum_{q\geq 0, j \geq 2}   (T-t)^{\nu q + j -2} \sum_{\substack{- \Omega_{j,q} \leq k \leq \Omega_{(j-2)_+,q} \\ q-k ~\text{even}}} \sum_{\ell \geq 0} e^{ i k\Phi}(\log(r) - \log(T-t))^{\ell} v_{q,j, k ,\ell}^{\text{lin}}(r),
		\end{align*}
		where we have 
		\begin{align*}
			v_{q,j, k ,\ell}^{\text{lin}}(r) &= - k(1+k) \frac{r^2}{4} g_{q,j,k,\ell} +  v^{\text{lin},1}_{q,j, k ,\ell}(g_{q,j-1,k,\tilde{\ell}}~|~\tilde{\ell} = \ell, \ell +1) + v^{\text{lin},2}_{q,j, k ,\ell}(g_{q,j-2,k,\ell'}~|~ \ell' = \ell, \ell+1, \ell+2),\\[2pt]
			v^{\text{lin},1}_{q,j, k ,\ell} &= - i(\nu q + j -1  - \tfrac{3}{2}k - 2i k \alpha_0) g_{q,j-1, k , \ell}(r)  + i(k+1) (\ell +1) g_{q, j-1, k, \ell +1}(r)\\
			&~~ + i k r( \partial_r  + ik\phi'(r)) g_{q, j-1,k,\ell}(r),\\[2pt]
			v^{\text{lin},2}_{q,j, k ,\ell} &= e^{-ik \phi} \Delta( e^{ik \phi} g_{q,j-2, k , \ell}(r))  + \frac{2(\ell +1)}{r}  e^{-ik \phi} \partial_r (e^{ik \phi}g_{q, j-2, k, \ell +1}(r))\\
			&~~~ + \frac{(\ell+1)(\ell +2)}{r^2} g_{q, j-2,k,\ell+2} + V_1g_{q,j-2, k , \ell}(r) + V_2 \bar{g}_{q,j-2, - k , \ell}(r).
		\end{align*}
		In the notation of \cite[Section 2.3]{Perelman},  we may set $ g_{q,j,k,\ell}(r) = 0$ in the above relation if $ (q,j,k,\ell) \notin \Omega$, where
		\[
		\Omega : = \{ (q,j,k,\ell)~|~ j \geq 1, q \geq 0, q - k ~\text{even}, 0 \leq \ell \leq m, - \Omega_{j,q} \leq k \leq \Omega_{(j-2)_+,q}   \}.
		\]
		Now for the interaction terms in \eqref{this-foreta-remot}, we write
		\begin{align*}
			& N_0(f_0, \bar{f_0}, \eta, \bar{\eta}) =  \sum_{q\geq 1, j \geq 4}   (T-t)^{\nu q + j -2} \sum_{\substack{- \Omega_{j,q} \leq k \leq \Omega_{(j-2)_+,q} \\ q-k ~\text{even}}} \sum_{\ell \geq 0} e^{ i k\Phi}(\log(r) - \log(T-t))^{\ell} v_{q,j, k ,\ell}^{\text{nl},0}(r),\\
			&N_1(f_0, \bar{f_0}, \eta, \bar{\eta}) = \sum_{q\geq 1, j \geq 5}   (T-t)^{\nu q + j -2} \sum_{\substack{- \Omega_{j,q} \leq k \leq \Omega_{(j-2)_+,q} \\ q-k ~\text{even}}} \sum_{\ell \geq 0} e^{ i k\Phi}(\log(r) - \log(T-t))^{\ell} v_{q,j, k ,\ell}^{\text{nl},1}(r),\\
			&N_2(f_0, \bar{f_0}, \eta, \bar{\eta}) = \sum_{q\geq 1, j \geq 6}   (T-t)^{\nu q + j -2} \sum_{\substack{- \Omega_{j,q} \leq k \leq \Omega_{(j-2)_+,q} \\ q-k ~\text{even}}} \sum_{\ell \geq 0} e^{ i k\Phi}(\log(r) - \log(T-t))^{\ell} v_{q,j, k ,\ell}^{\text{nl},2}(r).
		\end{align*}
		The terms $ v_{q,j, k ,\ell}^{\text{nl},0}(r), v_{q,j, k ,\ell}^{\text{nl},1}(r), v_{q,j, k ,\ell}^{\text{nl},2}(r) $ only depend on $ g_{\tilde{q}, \tilde{j}, \tilde{k}, \tilde{\ell}}$ with $ (\tilde{q}, \tilde{j}, \tilde{k}, \tilde{\ell}) \in \Omega$ and $ \tilde{j } \leq j -3$, to be precise
		\begin{align*}
			&v_{q,j, k ,\ell}^{\text{nl},0}(r) = v_{q,j, k ,\ell}^{\text{nl},0}(g_{\tilde{q}, \tilde{j}, \tilde{k}, \tilde{\ell}}~|~ \tilde{j} \leq j-3 ), \\
			&v_{q,j, k ,\ell}^{\text{nl},1}(r) = v_{q,j, k ,\ell}^{\text{nl},1}(g_{\tilde{q}, \tilde{j}, \tilde{k}, \tilde{\ell}}~|~ \tilde{j} \leq j-4),\\
			&v_{q,j, k ,\ell}^{\text{nl},2}(r) = v_{q,j, k ,\ell}^{\text{nl},2}( g_{\tilde{q}, \tilde{j}, \tilde{k}, \tilde{\ell}}~|~ \tilde{j} \leq j-5).
		\end{align*}
		We intend, see \cite[Section 2.3]{Perelman}, to solve \eqref{this-foreta-remot} via the equivalent system for 
		\begin{align}
			\begin{cases}
				v_{0,2, 0,0}^{\text{lin}}(r)  = - V_0(r), & ~~\\
				v_{q,j, k ,\ell}^{\text{lin}}(r) + v_{q,j, k ,\ell}^{\text{nl}}(r)  = 0,&  (q,j,k,\ell) \in \Omega,~ (q,j,k,\ell) \neq (0,2,0,0),
			\end{cases}
		\end{align}
		where we set 
		\[ v_{q,j, k ,\ell}^{\text{nl}}(r)  = v_{q,j, k ,\ell}^{\text{nl},0}(r)  + v_{q,j, k ,\ell}^{\text{nl},1}(r)  + v_{q,j, k ,\ell}^{\text{nl},2}(r).
		\]
		This is may be solved via  the following reccurent system
		\begin{align} \label{reccurent-remo2}
			\begin{cases}
				v_{0,2, 0,0}^{\text{lin}}(r)  = - V_0(r), & ~~\\[2pt]
				v_{2s,2, 0,\ell}^{\text{lin}}(r)   = 0, & (s, \ell) \neq (0,0)\\[2pt]
				v_{2s+1,2, -1,\ell}^{\text{lin}}(r)  = 0, & (s, \ell) \neq (0,0),\\[6pt]
				v_{q,j+1, k ,\ell}^{\text{lin}}(r) + v_{q,j+1, k ,\ell}^{\text{nl}}(r)  = 0,&  (q,j,k,\ell) \in \Omega,~  j \geq2,~ k = 0,-1,\\[3pt]
				v_{q,j, k ,\ell}^{\text{lin}}(r) + v_{q,j, k ,\ell}^{\text{nl}}(r)  = 0,&  (q,j,k,\ell) \in \Omega,~ j \geq2,~ k \neq 0,-1.
			\end{cases}
		\end{align}
		At this point we set $ \phi(r) = 0$ and write the first three lines of \eqref{reccurent-remo2} (the linear part) into 
		\begin{align} \label{reccurent-remo-lin-neu}
			\begin{cases}
				(2 \nu s + 1) g_{2s,1,0,\ell}(r) - (\ell +1) g_{2s,1,0,\ell +1}(r)  = 0, & ~~(s,\ell) \neq (0,0),\\[2pt]
				g_{0,1,0,0}(r) = - i V_0(r), & ~~~\\[2pt]
				r \partial_r  g_{2s+1,1,-1,\ell}(r) + (\nu(2s +1) + \f52 + 2 i \alpha_0) g_{2s+1,1,-1,\ell}(r) = 0.
			\end{cases}
		\end{align}
		We now fit this to the expansion \eqref{zweite-oscil}, i.e. we set
		\begin{align}  \label{sol-lin-final-remo}
			\begin{cases}
				g_{2n,1,0,\ell}(r)   = 0, & ~~(n,\ell) \neq (0,0),\\[2pt]
				g_{0,1,0,0}(r) = - i V_0(r), & ~~~\\[2pt]
				g_{2n+1,1,-1,\ell}(r) = \tilde{\beta}_{n,j} r^{- 2\alpha_0i - \nu(2n +1) - \f52}, & 0 \leq \ell \leq n \slash 2,~ 0 \leq n \leq N,\\[2pt]
				g_{2n+1,1,-1,\ell}(r) = 0 &~~ n > N,
			\end{cases}
		\end{align}
		where we set $ \tilde{\beta}_{n,j} : = a_{n,j,-1}^0$.
		\begin{Rem} The function $ \phi(r)$ is not necessary here. It has a non-trivial choice in \cite[Section 2.3]{Perelman} in order to correct the linearized  quadratic differential in the nonlinearity of the (stereographic) Schr\"odinger maps flow. To be precise a term of the form $ r( \partial_r  - ik\phi'(r) - 2 \bar{f_0}\partial_r f_0 (1 + |f_0|^2)^{-1}) g_{q, j-1,k,\ell}(r)$ replaces the respective part in t $\tilde{v}^{\text{lin}}_{q,j,k,\ell}$. Then the choice 
			\[
			\phi(r) = -i \int_0^r \frac{\partial_sf_0 \bar{f_0} - f_0\partial_s \bar{f_0} }{(1 + |f_0|^2)}~ds
			\]
			leads to  $  i\phi'(r) - 2f_0\partial_r \bar{f_0} (1 + |f_0|^2)^{-1} = -(\log(1 + |f_0|^2))'$ in the oscillatory  case $  k = -1$. 
		\end{Rem}
		The following definition is the (slightly adapted) notation of \cite[Section 2.3]{Perelman}.
		\begin{Def}(A)~For $m\in \Z$ we let $\mathcal{A}_m$ be the space of continuous functions $a : \R_{\geq 0}\to \C$ such that
			\begin{itemize}
				\item[(i)] There holds $ a \in C^{\infty}(\R_{> 0})$ and  {\upshape supp}$(a) \subset \{  x \leq 2 \delta\}$. 
				\item[(ii)] We require in an absolute sense
				\[
				a(r) = \sum_{\substack{n \geq K(m),\\ n - m -1 \text{even}}} \sum_{ \ell \geq 0} \alpha(n, \ell)\log^{\ell}(r) r^{\nu n},~~ 0 \leq r < \delta.
				\]
				Here $ K(m) = m+1$ if $ m \geq 0$ and $K(m) = |m| -1$ if $ m \leq -1$ Further $\sum_{\ell \geq 0}$ are finite sums (The number of summands is assumed to be increasing in $n$).
			\end{itemize}
			(B)~ For $ k \in \Z_+$ we further let $\mathcal{B}_k$ be the space of continuous functions $ b : \R_{> 0} \to \C$ such that
			\begin{itemize}
				\item[(i)] There holds $ b \in C^{\infty}(\R_{> 0})$ and for $ r > 2\delta $ the function $b(r)$ is a polynomial of order $ k-1$.
				\item[(i)] We require in an absolute sense 
				\[
				b(r) = \sum_{ n = 0}^{\infty} \sum_{ \ell \geq 0} \beta(n, \ell)\log^{\ell}(r) r^{2 \nu n},~~0 \leq r < \delta,
				\]
				where again $\sum_{\ell \geq 0}$ are finite sums.
			\end{itemize}
			Finally we let $ \mathcal{B}_k^0 = \{ b \in \mathcal{B}_k~|~ b(0) = 0\}$.
		\end{Def}
		
		We now check the following
		\begin{Lemma} There holds for $ m, m_1, m_2 \in \Z,k, k_1,k_2 \in \Z_+,~m_1-m_2 ~\text{even},$ 
			\begin{align}
				&r \partial_r \mathcal{A}_m \subset \mathcal{A}_m,~r \partial_r \mathcal{B}_k \subset \mathcal{B}_k ,~~\mathcal{B}_k \cdot \mathcal{A}_m \subset \mathcal{A}_m,\\
				&~\mathcal{B}_{k_1} \cdot \mathcal{B}_{k_2} \subset \mathcal{B}_{k_1 + k_2 -1}, ~\mathcal{A}_{m_1} \cdot \mathcal{A}_{m_2} \subset \mathcal{B}_1.
			\end{align}
		\end{Lemma}
		\begin{Corollary} \label{Das-Cor-remot}(i)~There holds
			\begin{align*}
				&f_0(r) \in r^{2 i \alpha_0 - \f12} \mathcal{A}_0,~~g_{0,1,0,0}(r) \in r^{2 i \alpha_0 - \f52} \mathcal{A}_0,\\
				& g_{2n+1,1,-1,\ell}(r) \in  r^{-2\alpha_0i - \nu(2n +1) - \f52} \mathcal{B}_1.
			\end{align*}
			(ii)~ Further we have
			\begin{align*}
				&v^{\text{lin},1}_{q,j,k,\ell},~v^{\text{lin},2}_{q,j,k,\ell},~v^{\text{nl},0}_{q,j,k,\ell},~v^{\text{nl},1}_{q,j,k,\ell},~v^{\text{nl},2}_{q,j,k,\ell} \in r^{2i\alpha_0(2k+1) - \nu q -2(j-1) -\f12} \mathcal{A}_k,~ k \neq -1,\\
				&v^{\text{lin},2}_{q,j,-1,\ell},~v^{\text{nl},0}_{q,j,-1,\ell},~v^{\text{nl},1}_{q,j,-1,\ell},~v^{\text{nl},2}_{q,j,-1,\ell} \in r^{-2i\alpha_0 - \nu q -2(j-1) -\f12}\mathcal{B}_{j-2},~ j \geq 3.
			\end{align*}
			if for all $(\tilde{q}, \tilde{j}, \tilde{k}, \tilde{\ell}) \in \Omega$ necessary to define the above $v's$ , we have $g_{\tilde{q}, \tilde{j}, \tilde{k}, \tilde{\ell}} \in r^{2i \alpha_0(2\tilde{k}+1) - \nu \tilde{q} - 2 \tilde{j} -\f12 }\mathcal{A}_{\tilde{k}}$ if $ \tilde{k} \neq -1$ and $ g_{\tilde{q}, \tilde{j}, -1, \tilde{\ell}} \in r^{-2i \alpha_0- \nu \tilde{q} - 2 \tilde{j} -\f12 }\mathcal{B}_{\tilde{j}}$.
		\end{Corollary}
		We now consider the latter two lines (the interaction part) of the system \eqref{reccurent-remo2}. With the definition of $ v^{\text{lin},1}_{q,j,k,\ell},~v^{\text{lin},2}_{q,j,k,\ell}$ we write this schematically into
		\begin{align}\label{system-final-remot}
			\begin{cases}
				~~- \frac{k (k+1)}{4}r^2 g_{q,j,k,\ell}(r) = - v^{\text{lin},1}_{q,j,k,\ell}(r) - v^{\text{lin},2}_{q,j,k,\ell}(r) - v^{\text{nl}}_{q,j,k,\ell}(r),~~k \neq 0,-1,&~~\\[5pt]
				~~r \partial_r g_{q,j,-1,\ell}(r) + (\nu q + j + \f32 + 2i \alpha_0)g_{q,j,-1,\ell}(r) =  - iv^{\text{lin},2}_{q,j,-1,\ell}(r)  - i v^{\text{nl}}_{q,j,-1,\ell}(r), &~~\\[5pt]
				~~(\nu q + j) g_{q,j,0,\ell}(r) - (\ell +1)g_{q,j,0,\ell+1}(r) =  - iv^{\text{lin},2}_{q,j,0,\ell}(r)  - i v^{\text{nl}}_{q,j,0,\ell}(r). 
			\end{cases}
		\end{align}
		Let us define (c.f. \cite[Section 2.3]{Perelman})  the functions 
		\begin{align}
			&B_{q,j,k,\ell} :=  - v^{\text{lin},1}_{q,j,k,\ell}(r) - v^{\text{lin},2}_{q,j,k,\ell}(r) - v^{\text{nl}}_{q,j,k,\ell}(r),~ k \neq 0,-1,\\
			&C_{q,j,k,\ell} := - iv^{\text{lin},2}_{q,j+1,k,\ell}(r)  - i v^{\text{nl}}_{q,j+1,k,\ell}(r),~ k = 0,-1.
		\end{align}
		for the right side of \eqref{system-final-remot}. Clearly we infer $ B_{q,j,k,\ell}, C_{q,j,k,\ell}$ depend only on $ g_{\tilde{q},\tilde{j},\tilde{k}, \tilde{\ell}}$ with  $ \tilde{j} \leq j-1$, i.e. we write
		\begin{align*}
			B_{q,j,k,\ell} =& ~B_{q,j,k,\ell}(r;  g_{\tilde{q},\tilde{j},\tilde{k}, \tilde{\ell}}~|~(\tilde{q},\tilde{j},\tilde{k}, \tilde{\ell}) \in \Omega,~ \tilde{j} \leq j-1), \\
			C_{q,j,k,\ell} =&~ C_{q,j,k,\ell}(r;  g_{\tilde{q},\tilde{j},\tilde{k}, \tilde{\ell}}~|~(\tilde{q},\tilde{j},\tilde{k}, \tilde{\ell})  \in \Omega,~ \tilde{j} \leq j-1).
		\end{align*}
		\begin{Prop} \label{sol-final-remot} Given  \eqref{sol-lin-final-remo}, there exists a unique solution $ (g_{q,j,k,\ell})_{j \geq 2}$ of \eqref{system-final-remot}  with parameters  $(q,j,k,\ell) \in \Omega $  and such that 
			\begin{align}
				&g_{q,j,k,\ell} \in r^{2 i \alpha_0(2k+1) - \nu q - 2j - \f12} \mathcal{A}_{k},~~ k \neq -1,\\[3pt]
				&g_{q,j,-1,\ell} \in r^{- 2 i \alpha_0 - \nu q - 2j - \f12} \mathcal{B}_j.
			\end{align}
			Further for all $ j \geq 2$ there exist $M(j), \tilde{M}(j) \in \Z_+$ (increasing in $ j$) such  that $ g_{q,j,k,\ell}  = 0$ if $ q > (N + \f12)M(j),~k \neq 0, -1$ and $ g_{q,j,k,\ell}  = 0$ if $ q > (N + \f12)\tilde{M}(j)~ k = 0,-1$.
		\end{Prop}
		\begin{proof}
			We check that there exist sequences $M(j), \tilde{M}(j),~ j \geq2$ such that $ B_{q,j,k,\ell}  = 0$ if $ q > (N + \f12)M(j),~k \neq 0, -1$ and $ C_{q,j,k,\ell}  = 0$ if $ q > (N + \f12)\tilde{M}(j),~ k = 0,-1$ holds true, if the same is verified for all $ g_{\tilde{q},\tilde{j},\tilde{k}, \tilde{\ell}}, ~g_{\tilde{q},\tilde{j},0, \tilde{\ell}} $ and $g_{\tilde{q},\tilde{j},-1, \tilde{\ell}}$ necessary to define $ B_{q,j,k,\ell} , C_{q,j,k,\ell} $. Now if we start with $ j =2$, then \eqref{system-final-remot} reads
			\begin{align*}
				\begin{cases}
					~~ \frac{1}{2}r^2 g_{2s,2,-2,\ell}(r) = - B_{2s, 2,-2,\ell} ,~~s \geq 1&~~\\[5pt]
					~~r \partial_r g_{2s+1,2,-1,\ell}(r) + (\nu (2s+1) + \f72 + 2i \alpha_0)g_{2s+1,2,-1,\ell}(r) =  C_{2s+1, 2, -1,\ell},~ s \geq 0 &~~\\[5pt]
					~~(\nu 2s + 2) g_{2s,2,0,\ell}(r) - (\ell +1)g_{2s,2,0,\ell+1}(r) = C_{2s,2,0,\ell},~~ s \geq 0. 
				\end{cases}
			\end{align*}
			The right sides depend only on  $ g_{\tilde{q},\tilde{j},0, \tilde{\ell}} $ with $ (\tilde{q},\tilde{j},0, \tilde{\ell}) \in \Omega$ and $ \tilde{j}=1$. Hence they are determined by \eqref{sol-lin-final-remo}. By definition and Corollary \ref{Das-Cor-remot}, we have
			\begin{align}
				&B_{q,2,-2,\ell} \in r^{- 6 i \alpha_0 - \nu q - \f52} \mathcal{A}_{-2},~ C_{q,2,0,\ell} \in r^{2 i \alpha_0 - \nu q - \f92} \mathcal{A}_{0},\\
				& C_{q,2,-1,\ell} \in r^{-2 i \alpha_0 - \nu q - \f92} \mathcal{B}_{1}.
			\end{align}
			Further $ B_{q,2,-2,\ell} = 0$ if $ q > (N+ \tfrac12)M(2)$ and  $C_{q, 2,k ,\ell} = 0$ if $ q > (N+ \tfrac12)\tilde{M}(2)$ where $ k = 0,-1$.  The above system of equations, i.e. \eqref{system-final-remot} for $ j=2$, then implies
			\begin{align}
				&g_{2s,2,-2,\ell}(r) = - \frac{2}{r^2}B_{2s, 2,-2,\ell}  \in r^{- 6 i \alpha_0 - 2s \nu - \f92} \mathcal{A}_{-2},~~s \geq 1,~~\\[5pt]
				&g_{2s,2,0,m}(r)  = \frac{1}{(\nu 2s + 2) } C_{2s,2,0,m}  \in r^{2i \alpha_0 - 2s \nu - \f92}\mathcal{A}_0,\\ \nonumber
				&g_{2s,2,0,\ell}(r)  = \frac{1}{(\nu 2s + 2) } C_{2s,2,0,\ell} + \frac{\ell +1}{(\nu 2s + 2) }g_{2s,2,0, \ell +1} \in  r^{2i \alpha_0 - 2s \nu - \f92}\mathcal{A}_0 ,~~ 0 \leq s,~ \ell \leq m -1,
			\end{align}
			where  we recall $ m \in \N_0$ is the top logarithmic power, i.e. $g_{q,j,k,\ell} = 0$ if $ \ell > m,~ m = m(q)$ with parameters in $\Omega$. This number may also be set $ N \slash 2$, however, we keep this implicit as it does not change the asymptotics. Next, we consider
			\[
			r \partial_r g_{2s+1,2,-1,\ell}(r) + (\nu (2s+1) + \f72 + 2i \alpha_0)g_{2s+1,2,-1,\ell}(r) =  C_{2s+1, 2, -1,\ell},~ s \geq 0.
			\] 
			Therefore, let $ g_{2s+1,2,-1,\ell}(r) = r^{- 2\alpha_0i - \nu(2s+1) - \f72} \tilde{g}_{2s+1,2,-1,\ell}(r)$. Then it is easy to see 
			\[
			\partial_r \tilde{g}_{2s+1,2,-1,\ell}(r) =  r^{-2}\cdot r^{ 2\alpha_0i + \nu(2s+1) + \f92} C_{2s+1,2,-1,\ell}(r) \in r^{-2} \mathcal{B}_1.
			\]
			Now setting $ g(r) : =  r^{ 2\alpha_0i + \nu(2s+1) + \f92} C_{2s+1,2,-1,\ell}(r)$, we infer the function $ \tilde{g}_{2s+1,2,-1,\ell}(r) $ with $ \tilde{g}= O(r^{-1})$ as $ r \to 0$ is given by
			\begin{align*}
				&g(r) = \sum_{n, \ell} \beta(n,\ell) \log^{\ell}(r) r^{2\nu n},~~ 0 \leq r < \delta,\\
				&\tilde{g}_{2s+1,2,-1,\ell}(r) =   - \beta(0,0)r^{-1}  + \int_0^r   \big(g(\sigma) - \beta(0,0)\big)\frac{d\sigma}{\sigma^2} \in r^{-1}\mathcal{B}_2.
			\end{align*}
			We  also infer $ g_{2s+1,2,-1,\ell} = 0$ for $ s \geq (N + \tfrac12)\tilde{M}(2)$ since  $C_{q,2,-1,\ell} = 0$ if $ q \geq (N + \tfrac12)\tilde{M}(2)$.\\[2cm]
			\emph{Induction step}.~~Suppose there exists a solution $(g_{q,j,k,\ell})_{j \leq J-1} $ for $ j = 2, \dots, J -1$ with $ J \geq 3$ and the Propositions holds true. Then, $B_{q,L,k,\ell},~C_{q,L,-1,\ell}, C_{q,L,0,\ell}, $ for $ k \neq 0,-1$ are known and determined by $ (g_{q,j,k,\ell})_{j \leq J-1}$. By induction and Corollary \ref{Das-Cor-remot}, we have
			\begin{align*}
				&\f14 k(1+k) g_{q,L,k,\ell}(r) = - B_{q,L,k,\ell}(r),~~k \neq 0,-1,\\
				&B_{q,L,k,\ell}(r) \in r^{2i \alpha_0(2k+1) - \nu q - 2(L-1) - \f12}\mathcal{A}_k,\\
				&B_{q,L,k,\ell} = 0,~~ q > (N+\tfrac12)M(L).
			\end{align*}
			Thus 
			\begin{align*}
				&g_{q,L,k,\ell}(r) = 	-\frac{4}{k(1+k)}B_{q,L,k,\ell}(r) \in r^{2i \alpha_0(2k+1) - \nu q - 2L - \f12}\mathcal{A}_k,\\
				&g_{q,L,k,\ell}(r)  = 0,~~q > (N+ \tfrac12 )M(L).
			\end{align*}
			Next, we consider 
			\[
			(\nu 2s + j) g_{2s,L,0,\ell}(r) - (\ell +1)g_{2s,L,0,\ell+1}(r)  = C_{2s,L,0,\ell}(r) \in r^{2i \alpha_0 - 2s\nu - 2L - \f12} \mathcal{A}_0 ,~~ 0 \leq s ,~0 \leq \ell \leq m,
			\] 
			where $ C_{2s, L,0,\ell} = 0$ if $ 2s > (N+ \tfrac12) \tilde{M}(L)$. Therefore, as above for $ k =0$, we infer
			\begin{align*}
				&	g_{2s,L,0,\ell}(r) \in r^{2i \alpha_0 - 2\nu s - 2L- \f12} \mathcal{A}_k,~0 \leq s,~0 \leq j \leq m,\\
				&	g_{2s,L,0,\ell} = 0,~~ q > (N+\tfrac12) \tilde{M}(L).
			\end{align*}
			Then, in the remaining case of  $ k =-1$, we consider
			\begin{align*}
				&r \partial_r g_{2s+1,L,-1,\ell}(r) + (\nu (2s+1) + L + \f32 + 2i \alpha_0)g_{2s+1,L,-1,\ell}(r) =  C_{2s+1, L, -1,\ell}(r),~ s \geq 0,\\
				& C_{2s+1, L, -1,\ell}(r) \in r^{-2\alpha_0i - \nu(2s+1) - 2L} \mathcal{B}_{L-1},\\
				& C_{2s+1, L, -1,\ell} = 0,~ 2s+1 > (N + \tfrac12) \tilde{M}(L).
			\end{align*}
			As above, we set 
			\begin{align*}
				&g_{2s+1,L,-1,\ell}(r) = r^{-\nu (2s+1) - L - \f32 - 2i \alpha_0}\tilde{g}_{2s+1,L,-1,\ell}(r),\\
				&g(r) : = r^{\nu (2s+1) + 2L + \f32 - 2i \alpha_0} C_{2s+1,L,-1,\ell}(r),\\
				& g(r) = \sum_{n,l} \beta(n,l) r^{2\nu n } \log^{l}(r),~~ 0 \leq r < \delta.
			\end{align*}
			Thus we may solve 
			\begin{align*}
				\tilde{g}_{2s+1,L,-1,\ell}(r) =&~ - \sum_{n: 2\nu n \leq L-1 }  \sum_{l \geq 0 }\int_{r}^{\infty} \beta(n,l) \sigma^{2\nu n - L} \log^l(\sigma) ~d \sigma\\
				& + \int_0^r \big(  g(\sigma) - \sum_{n: 2\nu n \leq L-1 }  \sum_{l \geq 0 }  \beta(n,l)\sigma^{2\nu n} \log^l(\sigma) \big) \sigma^{-L}~d \sigma.
			\end{align*}
			Then $ g_{2s+1,L,-1,\ell}(r) \in r^{-\nu (2s+1) - 2L - \f12 - 2i \alpha_0} \mathcal{B}_L$ and $ g_{2s+1,L,-1,\ell} = 0$ if $ q > (N + \tfrac12)\tilde{M}(L)$.
		\end{proof}
		
		\begin{Rem} \emph{The $N  \in \Z_+$ implicitly given in \eqref{form-remot} and Proposition \ref{sol-final-remot} is now set to be $ N_2 \in \Z_+,~ N_2 \gg1$ taken from the definition of $ u^{N_2}_{S}(t,y)$. In order to define the $R$-approximation, we } allow to further  iterate over $ j \geq 1$  in \eqref{form-remot} according to Proposition \ref{sol-final-remot} to any finite number\emph{ of terms}.
		\end{Rem} 
		We define the approximation in the remote region. 
		\begin{Def} Let $ N_3 \in \Z_+,~ N_3 \gg1 $, the we set
			\begin{align}
				&\tilde{u}^{N_3}_{Re}(t,r) : = f_0(r) + \sum_{j= 1}^{N_3} \sum_{\substack{q,k,\ell: \\ (q,j,k,\ell) \in \Omega}} (T-t)^{\nu q + j} e^{i k \Phi}\big(\log(r) - \log(T-t)\big)^{\ell} g_{q,j,k,\ell}(r),\\
				& e^{N_3}_{Re}(t,r) : = i \partial_t\tilde{u}^{N_3}(t,r) + \Delta \tilde{u}^{N_3}(t,r) + | \tilde{u}^{N_3}|^4 \tilde{u}^{N_3}(t,r),\\
				& u^{N_3}_{Re}(t,y) : = e^{- i \alpha(t)} (T-t)^{\f14}\tilde{u}^{N_3}(t, (T-t)^{\f12 }y),~~ C_0(T-t)^{-\epsilon_2} \leq y.
			\end{align}
		\end{Def}
		For $ T > 0$ and $ t \in [0, T)$, we set 
		\[
		\chi_{Re}(t,r) = \begin{cases}
			1 &C_0(T-t)^{\f12 -\epsilon_2 } \leq r, \\
			0 & \text{otherwise}.
		\end{cases}
		\]
		Then we directly infer the following from Proposition \ref{sol-final-remot} and Corollary \ref{Final-form}. We recall $ 0 < \delta \ll1$ was fixed in the beginning through the definition of $f_0(r)$.
		\begin{Corollary}\label{estimates-remot} Let $ \alpha_0 \in \R$ and we fix $ \f12 > \epsilon_2 > \f38$.  Then there exists $  0 <  T_0(\alpha_0, \nu, N_3,\delta) \leq 1 $ such that  for all $ 0 < T \leq T_0$ the functions $  \tilde{u}^{N_3}_{Re}, e^{N_3}_{Re}$  satisfy the following for all $ t \in [0,T)$.
			\begin{align}\label{remot-est1}
				&\|  \chi_{Re}(t,R) \cdot r^{-l}\partial_r^k(f_0(r) -  \tilde{u}_{Re}^{N_3}(t,r))  \|_{L^{\infty}_R} \leq C_{\nu, |\alpha_0|} (T - t )^{\f52 \epsilon_2 - \f14 - (k+l)(\f12 - \epsilon_2) },\\ \nonumber
				&0 \leq k+l \leq 4 ,\\[2pt]\label{remot-est2}
				&\|  \chi_{Re}(t,R) \cdot r^{-l}\partial_r^k( \tilde{u}_{Re}^{N_3}(t,r))  \|_{L^{\infty}_R} \leq C_{\nu, |\alpha_0|} ( \delta^{\nu - k-l - \f12} +  (T - t )^{\f12(\nu - k -l) + \eta}),\\ \nonumber
				& 0 < \eta \ll1,\\[2pt]\label{remot-est3}
				&\|  \chi_{Re}(t,R) \cdot r^{-l}\partial_r^k(f_0(r) -  \tilde{u}_{Re}^{N_3}(t,r))  \|_{L^{2}(r^2 dr)} \leq C_{\nu, |\alpha_0|} (T - t )^{2\epsilon_2 + \f12 - (k+l)(\f12 - \epsilon_2) }\\ \nonumber 
				&0 \leq k + l \leq 3,\\[2pt]\label{remot-est4}
				&\|  \chi_{Re}(t,R) \cdot r^{k}\partial_r^k\tilde{u}_{Re}^{N_3}(t,r)  \|_{L^{\infty}_r} \lesssim \delta^{\nu - \f12},~~\|  \chi_{Re}(t,R) \cdot r^{k}\partial_r^k\tilde{u}_{Re}^{N_3}(t,r) \|_{L^{2}(r^2 dr)} \lesssim \delta^{\nu },\\[2pt]\label{remot-est5}
				&\|  \chi_{Re}(t,r) \cdot r^{-l}\partial_r^{k}e^{N_3}_{Re}\|_{L^{2}(r^2 dr)} \leq C_{\nu, |\alpha_0|} (T - t )^{ N_3\tilde{\eta}  },\\ \nonumber
				&0 \leq l + k \leq 2,~  0 <  \tilde{\eta} = \tilde{\eta}(\nu, \epsilon_2) \ll1,~ N_3 \gg1.
			\end{align}
			Note we may take $ l \in \Z$ in the above estimates
		\end{Corollary}
		\begin{Rem} (i)~We note the extra decay of order $ O(r^{-\f12})$ for the $ g_{q,j,k,\ell}$'s leads to fixing $\epsilon_2 \in (0, \f12)$ 'close' to $\tilde{\epsilon} =  \f12$   due to the first non-stationary correction the form $ - (T-t) i V_0(r) $ in $ \tilde{u}^{N_3}_{Re}(t,r)$, c.f. \cite[Section 2.3]{OP} where $ \f38 < \epsilon_2 < \f12$ is fixed.\\[4pt]
			(ii)~In fact in \eqref{remot-est1} we obtain the bound if $ k + l \geq 0$  and $ \epsilon_2 > ((k+l)\f12 + \f14)(k+l + \f52)^{-1} $ and if we fix  $ \f38 < \epsilon_2 < \f12$ we may take $(T-t)^{\frac{3}{16}}$ on the right for $0 \leq k + l \leq 4$. The respective estimates in the setting of Schr\"odinger maps in \cite[Lemma 2.10]{Perelman} allow to take $0 < \epsilon_2 \ll1$. We also obtain $O((T-t)^{\nu \eta})$ for some $\eta = \eta(\epsilon_2) > 0$ and any $ \epsilon_2 \in (0,\f12)$ if  $ \nu > \f52$. Further we obtain such a bound for all $ \nu, \epsilon_2$, if we replace $ f_0(r) -  \tilde{u}_{Re}^{N_3}(t,r)$ by 
			\[
			f_0(r) - i (T-t) V_0(r) -  \tilde{u}_{Re}^{N_3}(t,r),~~ V_0 = \Delta f_0 + |f_0|^4 f_0.
			\]
			(iii)~In \eqref{remot-est3} we can have $ O((T - t )^{  (\nu +1)(\epsilon_2 + \f12) - (k+l)(\f12 - \epsilon_2) })$ if we again replace $ f_0(r) -  \tilde{u}_{Re}^{N_3}(t,r)$ by $ f_0(r) - i (T-t) V_0(r) -  \tilde{u}_{Re}^{N_3}(t,r)$.
		\end{Rem}
		
		Further we have the following from Proposition \ref{sol-final-remot} and Corollary \ref{Final-form}.
		\begin{Lemma}[Consistency in $S \cap R$]\label{cons-remot} We obtain for $(t,y) $ with $ (t,r) \in S \cap R$, i.e. all $(t,y)$ with $ C_0 (T-t)^{- \epsilon_2 } \leq y \leq \tilde{C}(T-t)^{- \epsilon_2} $, for $ N \in \Z_+$
			\begin{align}\label{remot- est6}
				\big | y^{-l}\partial_y^m& \big[ \tilde{u}^{N_2}_S(t,y)  - u^{N_3}_{Re}(t, y)\big]\big | \leq C_{m,l , N_2, N_3} (T-t)^{c N \epsilon_2 - \tilde{\eta}},~~0 \leq l + m \leq 2,~
			\end{align}
			where $ c>0,~ \tilde{\eta}= \tilde{\eta}(\nu, \epsilon) $ are fixed numbers, $ N_2, N_3$ are large, depending on $N, \epsilon_2$, and $ T = T(\alpha_0, \nu, N_2,N_3, \delta)$.
		\end{Lemma}
		\subsection{The final approximation} Here we provide the proof of Proposition \ref{main-prop-approx}.\\[3pt]
		First, we fix $\alpha_0, \nu$ and $ \frac{\nu}{2} \geq \epsilon_1 > 0,~ \f12  > \epsilon_2 > \f38$. Then for $N \in \Z_+$, (let us take this number large compared to $\epsilon_1, \epsilon_2$, say $ N > 100 (1 + \lfloor \nu \rfloor)$), we choose $ N_1, N_2, N_3 \in \Z_+$ successively, such that all estimates in  Corollary \ref{estimates}, Corollary \ref{estimates-self} and Corollary \ref{estimates-remot} are satisfied for some $  0 < T \leq 1$ depending on $ (\alpha_0, \nu, N_1, N_2, N_3, \delta)$.\\[5pt]
		In particular  the error terms $e^{N_1}_{In}, e^{N_2}_S, e^{N_3}_{Re}$  should be estimated of order $ \sim_{\nu, \epsilon} N$. Furthermore, the same should hold true for the upper bounds in the 'consistency' Lemma \ref{consist} and Lemma \ref{cons-remot}. Now we define  in the variables $ (t,R),~R = \lambda(t) r,$
		~~\\
		\begin{align}\label{Def-apprx}
			\tilde{u}_{\text{app}}^N(t,R) : =&~ \Theta(  (T-t)^{-\nu - \epsilon_1} R)  u^{N_1}_{In}(t,R)\\[3pt] \nonumber
			&~~ + (1 - \Theta( (T-t)^{\nu - \epsilon_1}R)) \Theta( (T-t)^{\epsilon_2 - \nu} R) u^{N_2}_S(t, (T-t)^{\nu}R)\\[3pt] \nonumber
			&~~ +  (1 - \Theta((T-t)^{\epsilon_2 + \nu}R)) e^{- i \alpha(t)} \underset{= \lambda(t)^{-\f12}}{\underbrace{ (T-t)^{\f14 + \frac{\nu}{2}}} } ~\tilde{u}^{N_3}_{Re}(t, (T-t)^{\f12 + \nu} R).
		\end{align}
		Further we set for $x \in \R^3,~0 \leq t < T$,
		\begin{align}\label{Def-apprx-fin}
			u_{\text{app}}^N(t,x) : =& ~e^{i \alpha(t)} \lambda(t)^{\f12} \tilde{u}^{\text{app}}_N(t,\lambda(t) |x|),\\ \nonumber
			= &~ e^{i \alpha_0 \log(T-t)} (T-t)^{- \f14 - \frac{\nu}{2}}~ \tilde{u}^{\text{app}}_N(t, (T-t)^{- \f12- \nu} |x|).
		\end{align}
		\begin{Rem} We may write the constants $ 2C^{-1}, 2\tilde{C}, 2\tilde{C}^{-1}, 2C_0$, respectively into  the cut-off function (from left to right) in order to use the above Corollary and  Lemma exactly as stated. However, this is not relevant, since the estimates scale in these constants.
		\end{Rem}
		~~\\
		By Corollary \ref{estimates}, Corollary \ref{estimates-self} and Corollary \ref{estimates-remot}, we obtain a radial $C^{\infty}$- function $u^{\text{app}}_N : \R^2 \times \R_{\geq 0} \to \C $ of the form required in Proposition \ref{main-prop-approx}, including the estimates for $z^N$ in part (a). For $\zeta^N_*$ we choose
		\[
		\zeta^N_*(x) : = f_0^{N_2}(|x|),~ x \in \R^3.
		\]
		The convergence in $\dot{H}^1 \cap \dot{H}^2$ follows directly from the estimates in Corollary  \ref{estimates-remot}. In particular, we note as $ t \to T^-$, the remote region and thus $ u^{N_3}_{Re}(t,r)$ are dominant.
		Further, considering in addition Lemma \ref{consist} and Lemma \ref{cons-remot}, we conclude an  error estimate of the form  
		\[
		\|  e_N(t) \|_{H^2} + \|  \langle x \rangle e^N(t)  \|_{L^2} \leq C (T-t)^{N \eta_1 - \eta_2},
		\] 
		where $ \eta_1, \eta_2 > $ are fixed constants. Rechoosing $ N_1, N_2, N_3 $ and replacing $ N$ by $\tilde{N}$ (and taking $T> 0 $ smaller), we simply need to chose $N_j,~j=1,2,3$ large enough, so that $ \tilde{N}\eta_1 - \eta_2 > N $. Then we rename $ e^{\tilde{N}}, u_{\text{app}}^{\tilde{N}}, \tilde{u}_{\text{app}}^{\tilde{N}}, \zeta_*^{\tilde{N}}$ in order to obtain part (b).
		\section{Completion to exact solutions: Final iteration}\label{sec:linearized}
		We seek a radial solution of \eqref{NLSequation} of the form 
		\[
		u(t,x) =  u_{\text{app}}^N(t,x) + \varepsilon(t,x),~~ x \in \R^3,~~ 0 \leq t < T,
		\]
		where we take $N \gg1 $ large and $ u_{\text{app}}^N(t,x)$ is  the approximate solution given by Proposition \ref{main-prop-approx} with  $x \in \R^3,~0 \leq t < T \ll1$ . We now rewrite this into
		\begin{align}\label{ansatz-final-final}
			&u(t, x) = e^{i \alpha(t)} \lambda^{\f12}(t) v(\tau, y) =   e^{i \alpha(t)} \lambda^{\f12}(t) ( v_{\text{app}}^N(\tau, y) + f(\tau, y)),\\
			&  v_{\text{app}}^N(\tau, y) = e^{- i \alpha(t(\tau))} \lambda(t(\tau))^{-\f12} u_{\text{app}}^N(t(\tau),x(y)),\\
			& f(\tau, y) = e^{- i \alpha(t(\tau))} \lambda(t(\tau))^{-\f12} \varepsilon(t(\tau),x(y)),
		\end{align}
		where we changed into the variables
		\[
		y = \lambda(t) x = (T-t)^{- \f12 - \nu} x,~~ \tau = \int_t^{\infty} \lambda^2(- s)~ds  = (2\nu)^{-1} (T-t)^{- 2 \nu}.
		\]
		We note in particular $\tau'(t) = - \lambda^2(t),~ \tau \in [\tau_0, \infty)$ where we set $ \tau_0 = (2\nu)^{-1} T^{- 2 \nu}$. Plugging \eqref{ansatz-final-final} into  \eqref{NLSequation}, the equation for $f(\tau, y)$ can now be calculated similar as above  in Section \ref{sec:approx}.
		\begin{Rem} We obtain  the following expressions
			\begin{align}
				&t(\tau) = T  - (\tau 2\nu)^{-\frac{1}{2 \nu}},\\
				&\alpha(t(\tau)) = - \frac{\alpha_0}{2 \nu}(\log(\tau) + \log(2\nu)),~\lambda(t(\tau)) = (2 \nu \tau)^{\frac{1}{4 \nu} + \f12}.
			\end{align}
		\end{Rem}
		\subsection{The final system} We calculate the following with $i \partial_t = - \lambda^2(\tau) i \partial_{\tau}$ 
		\begin{align*}
			e^{-i \alpha(t)} \lambda^{-\f12}(t) 	i \partial_tu(t, x)  =&~~  \lambda^2(\tau)\dot{\alpha}(\tau) v(\tau, y) - i \f12 \lambda(\tau) \dot{\lambda}(\tau)v(\tau, y) \\
			&~~  - i \lambda^{2} i \partial_{\tau} v(\tau, y)  - i \lambda(\tau) \dot{\lambda}(\tau) y \cdot \nabla_y v(\tau, y),
		\end{align*}
		hence with $ \Delta u = \lambda^2(\tau) \Delta_y v,~ e^{- i \alpha(t)} \lambda^{-\f12} |u|^4u =  \lambda^2(\tau) |v|^4 v $  the equation \eqref{NLSequation} reads ($\dot{\lambda}, \dot{\alpha}$ means in $ \tau$ here)
		\begin{align*}
			&\dot{\alpha} v(\tau, y) - i \f12 \dot{\lambda}\lambda^{-1}v(\tau, y)-  i \partial_{\tau} v(\tau, y)\\
			&  - i  \dot{\lambda} \lambda^{-1} y \cdot \nabla_y v(\tau, y) + \Delta_y v(\tau, y)  + |v|^4 v = 0.
		\end{align*}
		Now since $ \dot{\lambda}\lambda^{-1} = \frac{1 + 2\nu}{4 \nu \tau}$, $ \dot{\alpha} = - \frac{\alpha_0}{2 \nu}$, we write this into the following system 
		~~\\
		\begin{align}\label{system-final-final-y}
			&	- i \partial_{\tau} \eta = H\eta + D\eta + N_1(\eta) + N_2(\eta) + \mathbf{e},\\[6pt]
			&H = \begin{pmatrix}
				- \Delta_y - 3 W^4 & - 2 W^4\\
				2 W^4 & \Delta_y + 3 W^4
			\end{pmatrix},~~ D = \tau^{-1}\big( \frac{\alpha_0}{2\nu} \sigma_3 + i \frac{1+2\nu}{4 \nu}(\f12 + y \cdot \nabla_y)\big),\\
			& \eta = \begin{pmatrix}
				f\\
				\bar{f}
			\end{pmatrix},~~N_j = \begin{pmatrix}
				\mathcal{N}_j\\
				- \overline{\mathcal{N}_j}
			\end{pmatrix},~ j =1,2,~~  \mathbf{e} = \begin{pmatrix}
				\tilde{e}^N\\
				- \overline{\tilde{e}^N}
			\end{pmatrix},
		\end{align}
		where we set 
		\begin{align}
			& \mathcal{N}_1 = 3(W^4 - |v_{\text{app}}^N|^4) \eta + 2(W^4 - (v_{\text{app}}^N)^2 |v_{\text{app}}^N|^2)\overline{\eta},\\[5pt]
			&\mathcal{N}_2 = - | v_{\text{app}}^N + f|^4 (v_{\text{app}}^N + f) + |v_{\text{app}}^N|^4v_{\text{app}}^N + 3 |v_{\text{app}}^N|^4f + 2 (v_{\text{app}}^N)^2 |v_{\text{app}}^N|^2 \bar{f},\\[5pt]
			& \tilde{e}^N(\tau, y ) = e^{- i \alpha(\tau)}\lambda^{- \f52}(\tau) e^N(t(\tau), x(y)).
		\end{align}
		We seek a solution of \eqref{system-final-final-y} in $ f$ with vanishing boundary at $ \tau = \infty$.  Note in particular that $ \tau \geq \tau_0$ and $ \tau_0$ may be taken arbitrarily large by choosing $ 0 < T \leq 1 $ small. The system  \eqref{system-final-final-y} further  reduces to the radial flow in the variable $ |y| = \lambda(\tau) r = R$. Hence, using $
		y \cdot \nabla_y = R \partial_R,$ and $ R^2 \partial_R(\cdot) = R \partial_R(R \cdot) - R \cdot $, we obtain the $1$D-Schr\"odinger system
		~~\\
		\begin{align} \label{system-final-final-R}
			&-	i \partial_{\tau} \tilde{\eta}  =~\mathcal{H}\tilde{\eta} + D_0\tilde{\eta} + N_1(\tilde{\eta}) + R N_2( R^{-1} \tilde{\eta}) + R ~\mathbf{e}(\tau, R),\\[6pt] \label{system-final-final-R-operator}
			&\mathcal{H} =  \begin{pmatrix}
				- \partial_R^2 - 3 W^4(R) & - 2 W^4(R)\\
				2 W^4(R) &  \partial_R^2 + 3 W^4(R)
			\end{pmatrix},~~ \tilde{\eta}(\tau, R) = R \cdot \eta(\tau, R),\\[4pt]
			& D_0 =  \tilde{\beta}(\tau)  \sigma_3  + i \beta(\tau) ( \f12  + R \partial_R),\\
			&\tilde{\beta}(\tau) = \tau^{-1} \frac{\alpha_0}{2\nu} = -\dot{\alpha}(\tau),~~ \beta(\tau) =  \tau^{-1} \frac{1+2\nu}{4 \nu} = \dot{\lambda} \lambda^{-1}.
		\end{align}
		\begin{Rem}
			(i)~We note since $\eta(\tau, y)$ is radial, the function $\tilde{\eta}(\tau,R)$ has an odd extension to $(- \infty,0)$ and  $\tilde{\eta}(\tau, 0)=0$. All quantities  in \eqref{system-final-final-R} have either odd or even (regular) extensions to $(- \infty,0)$. (ii)~ The spectral properties of $\mathcal{H}$ are studied in Section \ref{sec:Scat} and Section \ref{sec:spec}. In particular, we have $\sigma(\mathcal{H}) = \{ \pm i \kappa\} \cup \R,~ \sigma_{\text{ess}}(\mathcal{H}) = \R$ and $ \lambda_0 = 0$ is a (threshold) resonance. In Proposition \ref{Fourier-inv-fin}, we establish a spectral representation for $ \mathcal{H}$, corresponding to inversion $ \mathcal{F}^{-1} \mathcal{F} = \mathit{I} $ of the associated Fourier transform.
		\end{Rem}
		Following the approach in  \cite[Section 6]{KST-slow}, we proceed by calculating  the Fourier coefficients of \eqref{system-final-final-R} with the Fourier transform $\mathcal{F}$ defined  in Section \ref{subsec:ft}. We use $ \xi $ for the spectral parameter here instead of $ \lambda $ as in Section \ref{sec:Scat} in order to avoid confusion. First note
		\[
		\mathcal{F}(D_0) = \tilde{\beta}(\tau) [\mathcal{F}(\sigma_3 \cdot \mathcal{F}^{-1})] \mathcal{F}  +  i \beta(\tau)( \f12 - \xi \partial_{\xi}) \mathcal{F} + \beta(\tau)\mathcal{K} \mathcal{F}, 
		\]
		where $ \mathcal{K}$ is the transference operator from Section \ref{subsec:tranf} and $  \xi \partial_{\xi}\hat{f}^{\pm}(\xi)$ only acts on the essential spectrum. 
		We recall the Fourier inversion for $\tilde{\eta}$ has, see Proposition \ref{Fourier-inv-fin} for the dual version, the form
		\begin{align*}
			\tilde{\eta}(\tau, R) =  x^{+}_0(\tau)& \phi^+_d(R) + x^{-}_0(\tau) \phi^-_d(R)  + \frac{1}{2 \pi}\int_{- \infty}^{\infty} x^{+}(\tau,\xi) e_+(R, \xi)~d \xi \\[2pt]
			+~&  \frac{1}{2 \pi}\int_{- \infty}^{\infty} x^{-}(\tau,\xi) e_-(R, \xi)~d \xi,
		\end{align*}
		where we set Fourier coefficients to be
		\begin{align}
			X^t(\tau, \xi) : = ( x^{+}_0(\tau), x^{-}_0(\tau), x^{+}(\tau, \xi), x^{-}(\tau, \xi)).
		\end{align}
		Separating the  discrete spectrum,  the system \eqref{system-final-final-R} is rewritten into 
		\begin{align} \label{system-final-final-X}
			& \begin{pmatrix}
				T_1 & 0_{2\times 2} \\
				0_{2\times 2} & T_2
			\end{pmatrix} X = (i \beta(\tau) \f12 - \tilde{\beta} \mathcal{F}(\sigma_3 \cdot \mathcal{F}^{-1}))X- i\beta(\tau)\mathcal{K}X\\ \nonumber
			& \hspace{3cm}~~-  \mathcal{F}N_1( \mathcal{F}^{-1}X) -  \mathcal{F} R N_2( R^{-1} \mathcal{F}^{-1}X)\\ \nonumber
			& \hspace{3cm}~~-  ~\mathcal{F} R\mathbf{e}(\tau, R),\\[8pt]
			&T_1 = \begin{pmatrix}
				i(\partial_{\tau} + \kappa) & 0\\
				0 & i(\partial_{\tau} - \kappa)
			\end{pmatrix},~~  T_2 = \begin{pmatrix}
				i\partial_{\tau}  - i\beta \xi \partial_{\xi}  + \xi^2& 0\\
				0 &  i\partial_{\tau}  - i\beta \xi\partial_{\xi}  - \xi^2
			\end{pmatrix}.
		\end{align}
		We intend to  solve \eqref{system-final-final-X} perturbatively for the operator on the right, hence we rewrite this into the following  elliptic-transport form
		\begin{align}\label{system-final-final-ellip-tra}
			\begin{cases}
				~~i(\partial_{\tau} + \kappa) x^+_d(\tau) = b_d^+(\tau),&~~\\[3pt]
				~~i(\partial_{\tau} - \kappa) x^-_d(\tau) = b_d^-(\tau),&~~\\[3pt] 
				~~(i \partial_{\tau} - i \beta(\tau) \xi\partial_{\xi} + \xi^2)x^+(\tau, \xi) = b^+(\tau, \xi),&~~\\[3pt]
				~~(i \partial_{\tau} - i \beta(\tau) \xi \partial_{\xi} - \xi^2)x^+(\tau,\xi) = b^-(\tau, \xi).&~~
			\end{cases}
		\end{align}
		Considering the first two lines in \eqref{system-final-final-ellip-tra}, we can solve these equations, c.f. \cite[Section 3]{OP}, via
		\[
		x^+_d(\tau) = -i \int_{\tau_1}^{\tau} e^{- \kappa(\tau -s)} b_d^+(s)~ds,~~x^-_d(\tau) = i \int_{\tau}^{\infty} e^{ \kappa(\tau -s)} b_d^-(s)~ds,
		\]
		where we need to fix $\tau _1 >  \tau_0\gg1 $  later. Further, for the latter two lines in \eqref{system-final-final-ellip-tra}, we  obtain the transport equations 
		\begin{align}\label{F-final-glei-glei1}
			~~( \partial_{\tau} -  \beta(\tau) \xi\partial_{\xi} - i \xi^2)x^+(\tau, \xi) = - ib^+(\tau, \xi),&~~\\[3pt] \label{F-final-glei-glei2}
			~~( \partial_{\tau} -  \beta(\tau) \xi \partial_{\xi}  + i \xi^2)x^+(\tau,\xi) = - ib^-(\tau, \xi).&~~
		\end{align}
		We consider the  backwards operators $K_{\pm}$ with kernels $ K_{\pm}(\tau, \sigma)$, i.e. we solve via
		\[
		x^{\pm}(\tau) = \int_{\tau}^{\infty} K_{\pm}(\tau, \sigma) b(\sigma)~d\sigma.
		\]
		By calculating the characteristics of \eqref{F-final-glei-glei1}, \eqref{F-final-glei-glei2}, we infer the homogeneous solutions 
		\[
		x^{\pm}(\tau, \xi) = e^{\pm i 2 \nu \tau \xi^2}h((2 \nu \tau)^{\frac{1}{4 \nu} + \f12}\xi),
		\]
		from which we calculate $K_{\pm}(\tau, \sigma,\xi)$ with $ K_{\pm}(\tau, \tau,1) = 1$ setting $h =1$. Alternatively, we may also reduce \eqref{F-final-glei-glei1} \eqref{F-final-glei-glei2} to, c.f. \cite[chapter 7]{KST1}, the equations
		\[
		(L^{\pm}_{\tau, \xi }x^{\pm})(\tau, \lambda^{-1}(\tau)\xi)  : = (i \partial_{\tau}  \mp i \lambda^{-1}(\tau)\xi^2)x^{\pm}(\tau, \lambda^{-1}(\tau)\xi) = - ib^{\pm}(\tau, \lambda^{-1}(\tau)\xi),
		\]
		hence the homogeneous solutions are given by
		\begin{align}
			x^{\pm}(\tau, \lambda^{-1}(\tau)\xi) = e^{\pm i (2\nu \tau)^{- \frac{1}{2\nu} } \xi^2} h(\xi),
		\end{align}
		where we choose $ h(\xi) =1 $ such that the kernel for $ L_{\tau, \xi}$ satisfies
		\[
		S(\tau, \sigma , \xi) = e^{\pm i (2\nu \tau)^{- \frac{1}{2\nu} } \xi^2} e^{\mp i (2\nu \sigma)^{- \frac{1}{2\nu} } \xi^2} ,~~ 1 \lesssim \tau < \sigma,~~S(\tau, \tau , \xi) = 1,
		\]
		and we note the scaling $ S(\tau, \sigma , \xi)  = S(\xi^{- 4 \nu}\tau, \sigma , 1)$. 
		Now we state the following Lemma, for which we recall the norm on the essential spectrum
		\[
		\|  f \|_{L^{2, \alpha}_{\xi}}^2 = \int |f(\xi)|^2 \langle \xi \rangle^{2 \alpha}~d \xi.
		\]
		\begin{Prop} \label{prop-decay-final}Let $ \alpha \geq 0$, the there exists $ C = C(\alpha) > 0$ large such that 
			\[
			\| K_{\pm}(\tau, \sigma) \|_{L^{2, \alpha}_{\xi} \to L^{2, \alpha }_{\xi} }  \lesssim \big( \frac{\sigma}{\tau}\big)^{C},~~ 1 \lesssim \tau < \sigma.
			\]
		\end{Prop}
		This is obtained by the above form of $ S(\tau, \sigma, \xi)$, respectively $ K(\tau, \sigma)$, and the fact 
		\[
		\frac{(1 + \lambda^{-1}(\tau) |\xi|)^{\alpha}}{(1 + \lambda^{-1}(\sigma) |\xi|)^{\alpha}}  \lesssim  \big(  \frac{\sigma}{\tau}\big)^{C},
		\]
		if $ C > 0$ is large enough (depending on $\alpha$).
		We now further define the space $ L^{\infty, N}_{\tau} L^{2, \alpha}_{\xi}$  via the following norm
		\[
		\|  f\|_{L^{\infty, N}_{\tau} L^{2, \alpha}_{\xi}} : = \sup_{\tau \geq \tau_0}\big( \tau^N \|  f(\tau)  \|_{ L^{2, \alpha}_{\xi}}\big),~~ 1 \lesssim \tau_0,~~ N \in \Z_+.
		\]
		\emph{Linear and nonlinear estimates}.~ The following is implied by Proposition  \ref{prop-decay-final}, where we denote by $ H^{\pm}_0 $ the integral operator of the elliptic equations in \eqref{system-final-final-ellip-tra} (first $+$ and second $-$, respectively).
		\begin{Prop} \label{prop-lin-final-new} Given $ \alpha \geq 0$ and $ N \gg1 $ large  enough. Then 
			\begin{align}
				&	\|  K_{\pm}f \|_{L^{\infty, N-1}_{\tau} L^{2, \alpha}_{\xi}} \lesssim_{\alpha} N^{-1} \|    f   \|_{L^{\infty, N}_{\tau} L^{2, \alpha}_{\xi}},\\
				&	\| H^{-}_0 b_d^{-} \|_{L^{\infty,N}_{\tau}} \lesssim_N \tau_0^{-N+1} \| b_d^{-}  \|_{L^{\infty,N}_{\tau}},\\
				&	\| H^{+}_0 b_d^{+} \|_{L^{\infty,N}_{\tau}} \lesssim_N  \max\{\tau_1^{-N+1}, \tau_0^{-N+1} - \tau_1^{-N+1}\}  \| b_d^{+}  \|_{L^{\infty,N}_{\tau}}. 
			\end{align}
		\end{Prop}
		In order to have small constants for the latter two lines, we choose $ \tau_1 > \tau_0 \gg1  $ large.  We further define 
		\[
		T = \text{diag}(H_0^+, H_0^-, K_+, K_-),~ b^t = (b_d^+, b_d^-, b^+, b^-).
		\]
		Then we have 
		\begin{Corollary}\label{prop-lin-final-new-new} Let $\alpha \geq 0$ and $ N \gg1 $ be large. Then we choose $ \tau_1 > \tau_0 > 0$ large enough (depending on $ N$), such that
			\[
			\| Tb \|_{L^{\infty, N-1}_{\tau} L^{2, \alpha }_{\xi}}  \lesssim _{\alpha} N^{-1}  \| b\|_{L^{\infty, N}_{\tau} L^{2, \alpha }_{\xi}},
			\]
			where we use a suitable adaption of $L^{\infty, N}_{\tau} L^{2, \alpha }_{\xi}$ to the vector setting.
		\end{Corollary}
		Before we continue, note the following. 
		\begin{Lemma} For $\alpha \geq 0$ we have (we restrict to the space of even functions)
			\[
			\|  f \|_{L^{2,2\alpha}_{\xi}} \sim \| R^{-1} \mathcal{F}^{-1}f \|_{H^{2\alpha}_{\text{rad}}(\R^3)}.
			\]
		\end{Lemma}
		This Lemma follows with $ \alpha = k \in \Z_+$ by the calculation 
		\begin{align*}
			\|  f \|_{L^{2,2\alpha}}  \sim \sum_{j = 0}^k \| \mathcal{H}^j \mathcal{F}^{-1}f \|_{L^2} &= \sum_{j \leq k} \|  (R^{-1} \mathcal{H} R)^j R^{-1} \mathcal{F}^{-1} f   \|_{L^2(\R^3)}\\
			&=  \sum_{j \leq k} \|  H^j R^{-1} \mathcal{F}^{-1} f   \|_{L^2(\R^3)} 
		\end{align*}
		where  we then use that $W$ and its derivatives are bounded. In fractional cases for $ \alpha > 0$, we have to interpolate. See \cite[Section 6]{K-S-stable}.\\[3pt]
		Now we essentially have two analogous possibilities to proceed. \textbf{(1)}~First, with $T$, we may formulate \eqref{system-final-final-ellip-tra} into the fixed point equation
		\begin{align}  \label{system-final-final-fix}
			X(\tau) = T\bigg( (i &\beta(\tau) \f12 - \tilde{\beta} \mathcal{F}(\sigma_3 \cdot \mathcal{F}^{-1}))X- i\beta(\tau)\mathcal{K}X\\ \nonumber
			&~~-    \mathcal{F} R \tilde{N}( R^{-1} \mathcal{F}^{-1}X)  - \mathcal{F} R\mathbf{e}(\tau, R) \bigg)
		\end{align}
		on the space $ L^{\infty, N-1}_{\tau} L^{2, \alpha }_{\xi}$ and where $\tilde{N} = N_1 + N_2$. Here we like to solve for $ X$ whence $ \eta = R^{-1}\mathcal{F}^{-1}X$.
		By Proposition \ref{Propo-boundedne-K},  the transference operator $\mathcal{K}$ is bounded on $L^{2 , \alpha}_{\xi}$. Hence the first line of \eqref{system-final-final-fix} is clearly a Lipschitz map from $ L^{\infty, N -1}_{\tau} L^{2, \alpha }_{\xi}$ into $ L^{\infty, N}_{\tau} L^{2, \alpha }_{\xi}$ since $ \beta(\tau) \sim \tau^{-1},~ \tilde{\beta}(\tau) \sim \tau^{-1}$.\\[3pt]
		\textbf{(2)}~ Alternatively we can multiply \eqref{system-final-final-X} by $ R^{-1} \mathcal{F}^{-1}$ from the left. The fixed point is then formulated  for $ R^{-1}\mathcal{F} X =  \eta$, cf. \cite[Section 3.2]{OP} for a similar setting, in the space $L^{\infty,N-1}H^{\alpha}_{\text{rad}}$. 
		In particular, the fixed point operator would then be given by
		\begin{align*}
			&	\mathcal{T} = \mathcal{T}_0 + \mathcal{T}_+ +\mathcal{T}_-,\\
			&	 \mathcal{T}_0\eta  = \int_{\tau}^{\infty} U(\tau, s) P_c\big( i \beta(s) \f12\eta  - \tilde{\beta}(s) \sigma_3 \eta -  i\beta(s)\tilde{\mathcal{K}}\eta -    \tilde{N}( \eta) -  \mathbf{e}(s) \big) ~d s\\
			& \mathcal{T}_+\eta  = \int_{\tau}^{\infty} e^{ \kappa(\tau -s) } P_+\big( i \beta(s) \f12\eta  - \tilde{\beta}(s) \sigma_3 \eta -  i\beta(s)\tilde{\mathcal{K}}\eta -    \tilde{N}( \eta) -  \mathbf{e}(s) \big) ~d s,\\
			& \mathcal{T}_-\eta  = \int_{\tau_1}^{\tau}e^{- \kappa(\tau - s)} P_-\big( i \beta(s) \f12\eta  - \tilde{\beta}(s) \sigma_3 \eta -  i\beta(s)\tilde{\mathcal{K}}\eta -    \tilde{N}( \eta) -  \mathbf{e}(s) \big) ~d s,
		\end{align*}
		where $ \tilde{\mathcal{K}} : H^{\alpha} \to H^{\alpha} $ is bounded and $ P_c, P_+, P_-$ are projections to the essential part and the discrete part, respectively. We then showed above that for the $ \mathcal{T}_j,~ j = 0,1,2$, the integral kernel in $(\tau, s)$ maps from $ L^{\infty,N}H^{\alpha}_{\text{rad}} $ to $ L^{\infty,N-1}H^{\alpha}_{\text{rad}}$. Hence in both cases \textbf{(1)} and \textbf{(2)} we need to prove the same Lipschitz bound.
		\begin{Lemma} For $ 2 \leq \alpha < 1 + \nu $, the map $ \eta \mapsto \tilde{N}(\eta)$ is Lipschitz as a map from $ L^{\infty,N-1}H^{\alpha} $ to $L^{\infty,N}H^{\alpha}$.
		\end{Lemma}
		\begin{Corollary} For $ 2 \leq \alpha < 1 + \nu $, the map $ X \mapsto \mathcal{F}R\tilde{N}(R^{-1} \mathcal{F}^{-1}X)$ is Lipschitz as a map from $ L^{\infty,N-1}L^{2, \alpha}_{\xi} $ to $L^{\infty,N}L^{2, \alpha}_{\xi}$.
		\end{Corollary}
		\begin{proof}
			We sketch the proof for $ \alpha =2$. For $ \alpha > 2$, we need corresponding improved estimates in the approximation, for which we refer to the remark below Proposition \ref{main-prop-approx}. We recall 
			\begin{align*}
				& \mathcal{N}_1 = 3(W^4 - |v_{\text{app}}^N|^4) \eta + 2(W^4 - (v_{\text{app}}^N)^2 |v_{\text{app}}^N|^2)\overline{\eta},\\[5pt]
				&\mathcal{N}_2 = - | v_{\text{app}}^N + f|^4 (v_{\text{app}}^N + f) + |v_{\text{app}}^N|^4v_{\text{app}}^N + 3 |v_{\text{app}}^N|^4f + 2 (v_{\text{app}}^N)^2 |v_{\text{app}}^N|^2 \bar{f}.
			\end{align*}
			For $ N_1 $ we need to gain $O( \tau^{-1})$ from the factors 
			\[
			\mathcal{V}_1 = W^4 - |v_{\text{app}}^N|^4,~~\mathcal{V}_2 = W^4 - (v_{\text{app}}^N)^2|v_{\text{app}}^N|^2,
			\]
			which can be written into the sum of terms of the form
			\[
			W^3(W - v_{\text{app}}^N),~ W^2 (W - v_{\text{app}}^N)^2,~ W ( W - v_{\text{app}}^N)^3,~ (W - v_{\text{app}}^N)^4.
			\]
			Now we use $W(R) \sim R^{-1}$ as $R \to \infty$ and  by Proposition \ref{main-prop-approx} we have 
			\begin{align*}
				&	\|   W - v_{\text{app}}^N  \|_{L^{\infty}} = O(\tau^{- \f14}),~	\|   R^{-1}(W - v_{\text{app}}^N)  \|_{L^{\infty}} = O(\tau^{-\f12}),\\
				&	\|   R^{-2}(W - v_{\text{app}}^N)  \|_{L^{\infty}} = O(\tau^{- 1}),~ 	\|   R^{-3}(W - v_{\text{app}}^N)  \|_{L^{\infty}} = O(\tau^{-1}).
			\end{align*}
			For $N_2$ we simply use $ \| v_{\text{app}}^N \|_{W^{2, \infty}} \leq C$, which then suffices by definition of the norm of $L^{\infty,N}_{\tau}H^{\alpha}$ since $f$ (counting also $ \bar{f}$)  appears at least quadratically.
		\end{proof}
		Now we are in position to prove Theorem \ref{main}.\\[4pt]
		\emph{Proof of Theorem \ref{main}}.~We choose $ \delta>0$ small as in the theorem and hence obtain the above approximation $u^N_{\text{app}}(t)$ of Proposition \ref{main-prop-approx} on $[0, T)$ for any small enough $ 0 < T \ll1$.
		By the previous analysis of the Section and the following stated in Proposition \ref{main-prop-approx}
		\begin{align}
			\|  \mathbf{e}(t) \|_{H^2} \leq C_0 \tau^{- 1 - \frac{1}{2\nu}\tilde{N}},~0 \leq t < T,~~ \tilde{N} \gtrsim 1,
		\end{align}
		the fixed point operator $J(\eta)$ (or likewise $J(X)$ on the Fourier side) satisfies for $N \gg1$ and $\tau_1 \geq \tau_0 \gg1 $ large enough
		\begin{align*}
			&\| J(\eta) \|_{L^{\infty, N}_{\tau}H^{2}} \leq \frac{C}{N} \|  \eta\|_{L^{\infty, N}_{\tau}H^{2}}( 1 + \| \eta \|_{L^{\infty, N}_{\tau}H^{2}}) + \tau_0^{- \tilde{\tilde{N}}}C_0,~~ \\\
			& \| J(\eta_1) -  J(\eta_2)\|_{L^{\infty, N}_{\tau}H^{2}} \leq \frac{C}{N} (\|  \eta_1\|_{L^{\infty, N}_{\tau}H^{2}} + \|  \eta_2\|_{L^{\infty, N}_{\tau}H^{2}} +1) \| \eta_1 - \eta_2 \|_{L^{\infty, N}_{\tau}H^{2}}.
		\end{align*}
		Hence $ J$ is a contractive map on the unit ball for large $N, \tau_0$. (The same works in principle without large $N$ if we factor off powers of $\tau_0^{-1}$ and $ \tau_0 $ is large). We note, in particular, $ \tau_0 \sim T^{- \frac{1}{2 \nu}}$ and thus we simply set  the $ T_0> 0$  to be sufficiently small. By construction, the solutions are of the form  
		\begin{align}
			u(t,x) = e^{i \alpha(t) } \lambda^{\f12}(t)( W (\lambda(t)x) + \zeta^N(t, \lambda(t)x) + f(\tau(t), \lambda(t)x ).
		\end{align}
		Then we finally observe  $\varepsilon(t,x) = e^{i \alpha(t) } \lambda^{\f12}(t)f(\tau(t), \lambda(t)x )$ and all the remaining statements in Theorem \ref{main} follow from Proposition \ref{main-prop-approx}.
		\section[Scattering theory  and the distorted Fourier side for Schr\"odinger systems]{Scattering theory  and the distorted Fourier side for Schr\"odinger systems}
		\label{sec:Scat}
		In this section, we present the calculation of the scattering matrix, the spectral representation and properties of the distorted Fourier transform for Schr\"odinger systems associated to the Matrix  operators of the type in \eqref{system-final-final-R-operator} in Section \ref{sec:linearized}.
		\subsection{Scattering Theory} \label{subsec: Scat} We 
		start by calculating Jost solutions for the operator  \eqref{system-final-final-R-operator} and include the standard construction here (see \cite{OP}, \cite{Busl-Per}, \cite{K-S-stable}) in order to make it self contained. This part is similar as in \cite[Section 4.1]{OP},  however we assume polynomial decay of arbitrary order for $V(r)$, which might be of independent interest.\\[2pt]
		The calculation is well presented in a similar context for the NLS in  \cite[Section 5]{K-S-stable} and originates in the work of 
		Buslaev-Perelman \cite[Section 2]{Busl-Per}. We also refer to the scattering theory for scalar $1$D operators in \cite[chapters 4 \& 5]{Yafa}. \\[2pt]
		\emph{Hamiltonian}. Recall the operator in \eqref{system-final-final-y} for radial functions on $\R^3$ takes the form
		
		\begin{align}\label{op}
			&H_0 = - \Delta \sigma_3 =   \begin{pmatrix}
				- \Delta  & 0\\
				0& \Delta 
			\end{pmatrix},~~~~~ V(r) = \begin{pmatrix}
				V_1(r) & V_2(r)\\
				- V_2(r) & -V_1(r)
			\end{pmatrix},~~~~ 
			\\[8pt] \nonumber
			& H = H_0 + V(r),~~~D(H) = H^2_{rad}(\R^3, \C^2) \subset L^2_{rad}(\R^3, \C^2),
		\end{align}
		where 
		\begin{align*}
			&V_1(r) = -3 W^4(r) = -\frac{3}{(1 + \frac{r^2}{3})^2},~~~V_2(r) = -2 W^4(r) = - \frac{2}{(1 + \frac{r^2}{3})^2},~ r = |x|,~~ x \in \R^3.
		\end{align*}
		It is convenient to introduce the Pauli matrices 
		\[
		\sigma_1 = \begin{mmatrix}
			0 & 1\\
			1 & 0
			~\end{mmatrix}, ~ \sigma_2 = \begin{mmatrix}
			0 & -i\\
			i & 0
		\end{mmatrix},~ \sigma_3 = \begin{mmatrix}
			1 & 0\\
			0 & -1
		\end{mmatrix},\]
		~~\\
		for later reference. We consider the generalized eigenvalue problem
		\begin{align} \label{EVp}\lambda^2 f(r) =  H f(r) =  - \Delta \sigma_3 f(r) + V(r) f(r),~~\lambda \geq  0.
		\end{align} 
		Now by setting $ \tilde{f}(r) = rf(r)$, we further simplify the above eigenvalue problem for $ H$ to
		\begin{align}
			\lambda^2 f(x) = \mathcal{H} f(x) =  - \partial_x^2 \sigma_3 f + V(x) f ,~~ x \in \R.
		\end{align}
		In order to recover \eqref{EVp}, we might restrict to the subspace of odd functions.
	\begin{Def} \label{Def-on} For $ \mu \geq 0$ we let $ \mathcal{H} = \mathcal{H}_0 + V(x) $ be  defined by the operators
		
		\begin{align}\label{op}
			&\mathcal{H}_0  =  \begin{pmatrix}
				-\partial_x^2 + \mu & 0\\
				0& \partial_x^2 - \mu
			\end{pmatrix},~ V(x) = \begin{pmatrix}
				V_1(x) & V_2(x)\\
				- V_2(x) & -V_1(x)
			\end{pmatrix},
		\end{align}
		where we assume $V_1,V_2 \in C^{\infty}(\R) $ are real, positive and even with
		\begin{align}
			|\partial_x^l V_j(x)|\lesssim_{l} \langle x\rangle^{-m-l},~~j = 1,2 , ~ l \geq 0
		\end{align} 
		for some $ m \in \Z_+,~ m \geq 2$.
	\end{Def}
	Note from the above Definition we have the identities
	$ \sigma_1 \mathcal{H}\sigma_1 = - \mathcal{H},~\sigma_3 \mathcal{H}^* \sigma_3 = \mathcal{H}$
	with spectral implications
	and the eigenvalue problem reads 
	\begin{align}
		\mathcal{H} f(x) = ( - \partial_x^2 + \mu) \sigma_3 f(x) = (\lambda^2 + \mu) f(x)   ,~~ x \in \R.
	\end{align}
	We now follow the reference \cite{K-S-stable} as far as possible (note $ \mathcal{H}$ has a threshold resonance for $ \mu = 0$  and $V$ does not decay exponentially). 
	For the case $ \mu = 0$ we essentially recover \cite[Section 5]{OP}. From now on we set $ \gamma := \sqrt{\lambda^2 + 2 \mu}$.
	\begin{Lemma}
		\label{LemmaJost-decay}
		For $\lambda \in \R$ there exists a real solution $ f_3(x, \lambda) $ of 
		\[
		\mathcal{H}f_3(\cdot, \lambda) = (\lambda^2 + \mu) f_3(\cdot, \lambda),
		\]
		such that $ (x, \lambda) \mapsto f_3 (x, \lambda),~ (x, \lambda) \in \R \times\R$ is smooth and $ f_3(x, \lambda) \sim e^{-\gamma x}\binom{0}{1}$  as $ x\to \infty $. In fact $ \sup_{\lambda \in \R}\sup_{x \in \R} | e^{\gamma x}f_3(x, \lambda)| \leq C(V) $ and for all $ x \geq 0$
		\begin{align}\label{first-est-decaying-part}
			&\bigg| \partial^l_{\lambda} \partial^k_x \bigg[   e^{\gamma x}f_3(x, \lambda) - \begin{pmatrix}0\\[1pt] 1	\end{pmatrix} \bigg] \bigg| \lesssim_{k,m}  \langle x \rangle^{2 -m - k + l} ( 1 + | \lambda| \langle  x\rangle )^{-1 -l},~~ k \geq 0,~l + 2 < m,\\ \label{first-est-decaying-part2}
			&\bigg| \partial^{m -2}_{\lambda} \partial^k_x \bigg[   e^{\gamma x}f_3(x, \lambda) - \begin{pmatrix}0\\[1pt] 1	\end{pmatrix} \bigg] \bigg| \lesssim_{k,m} \langle x \rangle^{-k} ( 1 + |\lambda| \langle  x\rangle )^{1 -m}\log(( |\lambda| \langle x\rangle)^{-1} + 2),~~k \geq 0,
		\end{align}
		
	\end{Lemma}
	\begin{Rem} For $ \mu > 0$  and the proof of Lemma \ref{LemmaJost-decay}  (similar to \cite{K-S-stable}) there holds
		\begin{align}\label{coarser-dec}
			&\bigg| \partial^l_{\lambda} \partial^k_x \bigg[   e^{\gamma x}f_3(x, \lambda) - \begin{pmatrix}0\\[1pt] 1	\end{pmatrix} \bigg] \bigg| \lesssim \gamma^{-1 -l} \langle x \rangle^{1 -m - k } .
		\end{align}
		for all $ k , l \geq 0$ with an implicit constant depending on $ k, l,m, \mu$.	\end{Rem}
	\begin{proof}
		Provided $f_3(x, \lambda )$ exists as desired, we require
		\[
		f_3(x, \lambda) = e^{- \gamma x}\begin{psmallmatrix}0\\[1pt] 1	\end{psmallmatrix} + \int_x^{\infty} \begin{psmallmatrix}\frac{\sin(\lambda(x-y))}{\lambda}& 0\\  
			0 & - \frac{\sinh(\gamma(x-y))}{\gamma}	\end{psmallmatrix}  V(y) f_3(y, \lambda)~dy.
		\]
		Factoring off $ e^{-\gamma x}$ (resp. $e^{-\gamma y} $ in the integral) we hence formulate the following integral equation 
		\begin{align}\label{int}
			&\chi(x, \lambda) = \begin{psmallmatrix}0\\[1pt] 1	\end{psmallmatrix} + \int_x^{\infty} K(x,y,\lambda) V(y) \chi(y, \lambda)~dy\\ \nonumber
			& K(x,y,\lambda) = \begin{pmatrix}
				\frac{\sin(\lambda(y-x))}{\lambda}& 0\\ 
				0 & - \frac{\sinh(\gamma(y-x))}{\gamma}
			\end{pmatrix}e^{\gamma(x-y)},
		\end{align}
		For a solution $ \chi(x, \lambda) $ of  \eqref{int} we then set  $ f_3(x,\lambda) : = e^{- \gamma x} \chi(x, \lambda)$. Note further that 
		\begin{align}
			& |  \partial_{\lambda}^l K(x,y, \lambda) |  \lesssim_l  ~ \frac{ (y-x)^{l+1}}{\langle \lambda(y-x) \rangle^{l +1}},~~l \geq 0,~ y > x,
		\end{align}
		directly implied in the regions $ \gamma(y-x) \lesssim1 $ and $ \gamma (y-x) \gtrsim 1 $. Likewise we may check the bound
		\begin{align}
			&\sup_{y > x} |  \partial_{\lambda}^l  K(x,y,\lambda) | \lesssim_l \gamma^{-l-1},~~ l \geq 0, ~\lambda \geq 0,~~~~~  V(x) \lesssim \langle x \rangle^{-m},~
		\end{align}
		which implies \eqref{int} has  a solution $ \chi(x, \lambda) $  by the standard  Volterra iteration and such that \eqref{first-est-decaying-part} in case $ l = k = 0$ follows from estimating 
		\[
		\int_x^{\infty} | K(x,y, \lambda) | |V(y) | ~dy.
		\]
		Using further  $ \partial_x  K(x,y,\lambda) = - \partial_y K(x,y,\lambda) $ we infer for higher derivatives 
		\begin{align}
			&\partial_{x}^k  \chi(x, \lambda)  = \delta^k_0 \begin{psmallmatrix}0\\[1pt] 1	\end{psmallmatrix} + \sum_{j = 0}^k \begin{psmallmatrix}k\\[2pt] j	\end{psmallmatrix}  \int_x^{\infty} K(x,y,\lambda) V^{(k -j)}(y) (\partial_{y}^j \chi(y, \lambda))~dy,\\
			&\partial_{\lambda}^l \chi(x, \lambda)  = \delta^l_0 \begin{psmallmatrix}0\\[1pt] 1	\end{psmallmatrix} + \sum_{j = 0}^l \begin{psmallmatrix}l\\[2pt] j	\end{psmallmatrix}  \int_x^{\infty} \partial_{\lambda}^{l -j}K(x,y,\lambda) V(y) (\partial_{\lambda}^j \chi(y, \lambda))~dy.
		\end{align}
		Again since 
		by the assumption 
		$ |\partial_x^l V(x)|\lesssim_{l} \langle x\rangle^{-m-l} $ we obtain \eqref{first-est-decaying-part} and \eqref{first-est-decaying-part2} inductively from estimating integrals of the form
		\[
		\int_x^{\infty} \frac{(y -x)^{l +1}}{\langle  \lambda (y -x) \rangle^{l+1}} \langle y \rangle^{-m - k} f_{l, k }(y, \lambda)~dy.
		\]
	\end{proof}
	\begin{Lemma}
		\label{LemmaJost-oscillating}
		For $\lambda \in \R$ there exists a pair of solutions $ f_j(x, \lambda),~ j = 1,2$ of 
		\[
		\mathcal{H}f_j(\cdot, \lambda) = (\lambda^2 + \mu) f_j(\cdot, \lambda),
		\]
		such that $ (x, \lambda) \mapsto f_j (x, \lambda),~ (x, \lambda) \in \R \times\R$ is smooth, $ f_2(\cdot, \lambda) = \overline{f_1(\cdot, \lambda)} $ and 
		\[
		f_1(x, \lambda) = e^{i \lambda x} \begin{psmallmatrix}
			1\\0 \end{psmallmatrix}  + e^{i \lambda x} O( \langle x \rangle^{2 -m} ) + O(e^{-\gamma x}),~~~x\to \infty,
		\]
		where the $O$-bound is uniform in $ \lambda$.  In fact for $ \lambda \in \R$ there exists $ x_0 \geq 0$ 
		such that for $ x \geq x_0$, and $l -k + 2 < m$
		\begin{align}\label{diese-esti}
			&\bigg| \partial^l_{\lambda} \partial^k_x \bigg[   e^{- i \lambda x}f_1(x, \lambda) - \begin{pmatrix}1\\[1pt] 0	\end{pmatrix} \bigg] \bigg| \lesssim_k  
			\langle \lambda \rangle^{-1 +k}\langle x \rangle^{2 -m - k + l}   +  \gamma^{-1 +k}  x^{l}e^{- \gamma x},\\
			&\bigg| \partial^{m-2}_{\lambda} \bigg[   e^{- i \lambda x}f_1(x, \lambda) - \begin{pmatrix}1\\[1pt] 0	\end{pmatrix} \bigg] \bigg| \lesssim    \frac{\log(|\lambda|^{-1} +1)}{1 + |\lambda|}(C_{V, x_0} + \langle x \rangle^{2 -m }\log(x +2)).
		\end{align}
		where $x_0$ only depends on $V$ if $ \mu > 0$ and further, if $ |\lambda | \geq \lambda _0 \gtrsim 1 $ is large, we can take $ x _0 = 0$. Moreover for $ \mu = 0$ and $ 0 < |\lambda| \ll1,~x \geq 0 $ we obtain
		\begin{align} \label{jene-est}
			&\bigg| \partial^l_{\lambda} \partial^k_x \bigg[   e^{- i \lambda x}f_1(x, \lambda) - \begin{pmatrix}1\\[1pt] 0	\end{pmatrix} \bigg] \bigg| \lesssim_k   \langle x \rangle^{2 -m - k + l}   +  |\lambda|^{k+1 - l}\langle \lambda x\rangle^{l}e^{- |\lambda| x},~~l -k + 2 < m,
		\end{align}
		
	\end{Lemma}
	\begin{Rem} From the proof it is clear we may choose $ x \geq x_0$ only depending on $V$ in order to obtain the following. In \eqref{diese-esti} for $l = k = 0$  replace 
		$O(\gamma^{-1 }  e^{- \gamma x})$  on the right by  $ O(\langle x \rangle)$. This upper bound is $O(1)_{\lambda}$ as $ \lambda \to 0$ and uniform with respect to $ \mu \to 0$.
	\end{Rem}
	
	\begin{proof}
		We prove \eqref{diese-esti} and for \eqref{jene-est} refer to \cite[Section 4.1, Lemma 4.2]{OP}. We write
		$$ f_1(x,\lambda) = u(x, \lambda) f_3(x,\lambda) + v(x,\lambda)\begin{pmatrix}1\\ 0	\end{pmatrix}$$
		where $ v(x,\lambda) \sim e^{i \lambda x}$ and $ f_3(x, \lambda) = \begin{psmallmatrix}f_3^{(1)}(x, \lambda)\\ f^{(2)}_3(x, \lambda)	\end{psmallmatrix}$. 
		We then conclude (using $g' = \partial_x g$)
		\[
		\mathcal{H} (u f_3)= u (\lambda^2 + \mu) f_3 + \begin{pmatrix} -u'' f_3^{(1)}- 2 u' \partial_x f_3^{(1)}\\  u'' f_3^{(2)} + 2 u' \partial_xf_3^{(2)}\end{pmatrix},	
		\]
		and thus 
		\begin{align} \label{Red}
			0 = (\mathcal{H}- (\lambda^2 + \mu))f_1 = \begin{pmatrix} (- \partial_{xx} - \lambda^2 + V_{1}) v \\ - V_{2 } v \end{pmatrix} + \begin{pmatrix} -u'' f_{3}^{(1)} - 2 u' \partial_x f_{3}^{(1)}\\  u'' f_{3}^{(2)} + 2 u' \partial_x f_{3}^{(2)}\end{pmatrix}
		\end{align}
		The standard procedure of solving \eqref{Red} is e.g. well described  in Lemma 5.3 in \cite[Section 5.3]{K-S-stable} (except with more decay for $ V_{1}, V_{2}$) and hence using
		\[
		y'' f_3 + 2 y' f_3' = 0,~~~\text{if}~~ y'(x, \lambda) = C (f_3^{(2)}(x, \lambda))^{-2},~~ x \geq x_0
		\]
		(with $x_0$ independent of $\lambda$), we choose $ u$ by the variations of constants formula (in the second line of \eqref{Red})
		\begin{align}\label{COnst-var}
			u'(x, \lambda) = - [f_3^{(2)}(x, \lambda)]^{-2}\int_x^{\infty}f_3^{(2)}(y, \lambda) V_{2}(y) v(y, \lambda)~dy.
		\end{align}
		Clearly we also infer  from the first line of \eqref{Red}
		\begin{align}
			v(x, \lambda) = e^{i x \lambda }- \int_x^{\infty}\frac{\sin(\lambda( y -x))}{\lambda}\big( u''(y) f_3^{(1)}(y) + 2 u'(y) f_3^{(1)}(y) - V_1(y) v(y)\big)~dy,
		\end{align}
		and hence plugging in \eqref{COnst-var} we may write
		\begin{align}
			&v(x, \lambda) = e^{i x \lambda } +  \int_x^{\infty}K_1(x,y;\lambda) v(y)~dy  + \int_x^{\infty}K_2(x,y;\lambda) v(y)~dy,\\
			&K_1(x,y;\lambda) = \frac{\sin(\lambda(y-x))}{\lambda} V_a(y, \lambda) := \frac{\sin(\lambda(y-x))}{\lambda} \big(V_1(y) - \frac{f_3^{(1)}}{f_3^{(2)}}(y, \lambda) V_2(y)\big),\\
			&K_2(x,y;\lambda) = 2 \int_y^x\frac{\sin(\lambda(z-x))}{\lambda} V_b(z, \lambda) ~dz~ \cdot f_3^{(2)}(y, \lambda) V_2(y),\\
			& V_b(z, \lambda) : =  \big(- \frac{f_3^{(2)}(z, \lambda)'}{f_3^{(2)}(z, \lambda)}f_3^{(1)}(z, \lambda) + f_3^{(1)}(z, \lambda)' \big)[f_3^{(2)}(z, \lambda)]^{-2}.
		\end{align}
		In particular by Lemma \ref{LemmaJost-decay} for $ y \geq x \geq x_0$ there holds
		\begin{align}
			| \partial_{y}^l V_a(y, \lambda)| &\leq C_l \;\langle y \rangle^{-m -l },~~ l \geq 0,\\[3pt]
			| \partial_{y}^l \partial_{\lambda}^kV_a(y, \lambda)| &\leq C_l \;\langle y \rangle^{2-2m -l +k} \langle \lambda y \rangle^{-1-k},~~ l \geq 0,~k > 0,~ k +2  < m,\\[3pt]
			| \partial_{y}^l \partial_{\lambda}^{m-2}V_a(y, \lambda)| &\leq C_l  \;\langle y \rangle^{-m -l} \langle \lambda y \rangle^{1-m} \log((|\lambda| y)^{-1} + 2),\\[3pt]
			| \partial_{y}^l \partial_{\lambda}^kV_b(y, \lambda)| &\leq C_l\;  e^{\gamma y}\langle y \rangle^{1-m -l +k} \langle \lambda y \rangle^{-1-k},~~ l \geq 0,~k + 2 < m,\\[3pt]
			| \partial_{y}^l \partial_{\lambda}^{m-2}V_b(y, \lambda)| &\leq C_l \; e^{\gamma y} \langle y \rangle^{-1 -l} \langle \lambda y \rangle^{1-m} \log((|\lambda| y)^{-1} + 2),
		\end{align}
		with a constant $C_l  > 0$. Therefore 
		\begin{align*}
			|K_1(y, x;\lambda) | &\leq \frac{(y-x)}{\langle \lambda (y-x)\rangle} \langle y \rangle^{-m},\\
			|K_2(y, x;\lambda) | &\leq \int_x^y \frac{(z-x)}{\langle \lambda (z-x)\rangle} e^{\gamma z} \langle z \rangle^{1-m } \langle \lambda z \rangle^{-1}~dz~ e^{- \gamma y} \langle y \rangle^{-m}\\
			&~\leq C \frac{(y-x)}{\langle \lambda (y-x)\rangle} \langle y \rangle^{-m} \langle x \rangle^{2-m}.
		\end{align*}
		Hence we have a solution $v(x, \lambda)$ by Volterra iteration on $[x_0, \infty) $ with 
		\begin{align*}
			|v(x, \lambda) - e^{ i \lambda x}| \leq C \int_x^{\infty}  \frac{(y-x)}{\langle \lambda (y-x)\rangle} \langle y \rangle^{-m}~dy \leq&~ C \langle x \rangle^{2-m} \int_0^{\infty} \frac{u}{1 + |\lambda| u} \langle u \rangle^{-2}~du\\
			\lesssim&~ (1 + |\lambda|)^{-1} \langle x\rangle^{2-m},
		\end{align*}
		for $ x \geq x_0$. This also implies
		\begin{align}\label{u-est}
			|u'(x, \lambda)| &\leq ~C e^{2 \gamma x} \langle \lambda \rangle^{-1}\int_x^{\infty} e^{- \gamma y} \langle y \rangle^{2 - 2m}~dy  +  C e^{2 \gamma x} \int_x^{\infty} e^{- \gamma y} \langle y \rangle^{ - m}~dy\\
			&\leq ~C e^{\gamma x} \langle \lambda \rangle^{-1}\langle x \rangle^{3 -2m} + e^{\gamma x}  \langle x \rangle^{1-m} ,~~~ x \geq x_0,~ \lambda \in \R.
		\end{align}
		Further, integrating by parts for the first term, 
		\begin{align*}
			\int_{x_0}^x e^{\gamma y} \langle y \rangle^{3-2m}~dy =&~ \gamma^{-1}\big(e^{\gamma x} \langle x \rangle^{3-2m} + c_{x_0}\big) +  \frac{2m-3}{\gamma}\int_{x_0}^x e^{\gamma y} \langle y \rangle^{2-2m}~dy\\
			\leq&~ \gamma^{-1}\big(e^{\gamma x} \langle x \rangle^{3-2m} + c_{x_0}\big) +  \frac{2m-3}{\gamma (1 + x_0)} \int_{x_0}^x e^{\gamma y} \langle y \rangle^{3-2m}~dy,
		\end{align*}
		we obtain for $ x_0 = x_0(\lambda) \gg1 $
		\begin{align}
			|\int_{x_0}^x e^{\gamma y} \langle y \rangle^{3-2m}~dy| \leq \frac{1 + x_0}{\gamma(1 + x_0) - (2m-3)} \big(e^{\gamma x} \langle x \rangle^{3-2m} + c_{x_0}\big) 
		\end{align}
		By taking $ u(x_0, \lambda) = 0$ we use \eqref{u-est} to bound $|u(x, \lambda)|$ as desired. We extend the solution $f_1 = u f_3 + v \begin{psmallmatrix}
			1\\[1pt]0
		\end{psmallmatrix}$  by local uniqueness. Further, say for $\partial_{\lambda}^l$ derivatives,  setting 
		\[
		\tilde{K}_j(x, y; \lambda) : = e^{- i \lambda(x-y)}K_j(x,y;\lambda) ,~ j = 1,2,~~ \tilde{v}(x, \lambda) :=e^{- i \lambda x}v(x, \lambda),
		\]
		we need to inductively estimate
		\begin{align} \label{ind}
			\partial_{\lambda}^l \tilde{v}(x, \lambda) = \sum_{i = 1,2}\sum_{j = 0}^{\ell} \begin{psmallmatrix}
				l\\[2pt]j
			\end{psmallmatrix} \int_x^{\infty} \partial_{\lambda}^{j}\tilde{K}_i(x,y; \lambda) \partial_{\lambda}^{l -j}\tilde{v}(y, \lambda)~ dy.
		\end{align}
		First for $ y \geq x \geq x_0 \gg1 $ large using $ |\partial_{\lambda}\ f_3(x, \lambda)| \leq C \langle x \rangle^{3-m}(1 + \lambda \langle x\rangle)^{-2}$, we see
		\begin{align*}
			&|\partial_{\lambda} \tilde{K}_1(x, y;\lambda)| \leq C  \frac{(y-x)^2}{\langle \lambda (y-x) \rangle } \langle y \rangle^{-m}+ C (1+|\lambda|)^{-1} \frac{(y-x)}{\langle \lambda (y-x) \rangle } \langle y\rangle^{3 -2m},\\
			&|\partial_{\lambda} \tilde{K}_2(x, y;\lambda)| \leq C \int_x^y \frac{(z-x)^2}{\langle \lambda(z-x)\rangle } e^{\gamma z} \langle z\rangle^{1-m}\langle \lambda z\rangle^{-1}~dz~ e^{- \gamma y} \langle y \rangle^{-m}\\ \nonumber
			&~~~~~~~~~~~~~~~~~~~~~+ C\int_x^y \frac{(z-x)}{\langle \lambda (z-x)\rangle }e^{\gamma z} \langle z \rangle^{2 -m} \langle \lambda z\rangle^{-2}~dz ~e^{- \gamma y} \langle y \rangle^{-m}\\
			&~~~~~~~~~~~\leq C ( 1 + |\lambda|)^{-1} \frac{(y-x)^2}{\langle \lambda (y-x) \rangle } \langle y \rangle^{ 2 -2m} + ( 1 + |\lambda|)^{-1} \frac{(y-x)}{\langle \lambda (y-x) \rangle } \langle y \rangle^{3- 2m} 
		\end{align*}
		By 	\eqref{ind} this implies  $ | \partial_{\lambda}\tilde{v}(x, \lambda)| \leq C (1 + |\lambda|)^{-1} \langle x\rangle^{3-m}$.  Higher derivatives for $ l >1$ and $\partial_x^k$ derivatives follow similarly. For the former we note
		\begin{align}
			|\partial_{\lambda}^l \tilde{K}(x, y; \lambda)| &\leq C_l  \frac{(y-x)^{l+1}}{1 + |\lambda| (y-x)} \langle y\rangle^{-m} +  C_l\frac{(y-x)}{1 + |\lambda| (y-x)}  \langle y\rangle^{2 -2m +l},\\
			|\partial_{\lambda}^{m-2} \tilde{K}(x, y; \lambda)| &\leq C_l\frac{(y-x) \langle y\rangle^{-m}}{1 + |\lambda| (y-x)}  \log(y+2)\log(|\lambda|^{-1} +2)\\ \nonumber
			\nonumber
			&~~~+  C_l  \frac{(y-x)^{m-1}}{1 + |\lambda| (y-x)} \langle y\rangle^{-m}.
		\end{align}
		which we use inductively in \eqref{ind}  for $\partial_{\lambda}^lv(x, \lambda)$. This is then inserted in \eqref{COnst-var} in order to bound $ \partial_{\lambda}^lu'(x, \lambda)$ and eventually 
		\begin{align*}
			\partial_{\lambda}^l f_1(x, \lambda)  = \partial_{\lambda}^l v(x, \lambda) \begin{pmatrix}
				1\\0
			\end{pmatrix} + \sum_{j = 0}^l \begin{pmatrix}
				j\\l
			\end{pmatrix} \partial_{\lambda}^j f_3(x, \lambda) \partial_{\lambda}^{j-l} u(x, \lambda).
		\end{align*}
		We spare details for the remaining part of the proof.
	\end{proof}
	Next, we consider exponentially growing solutions as $ x \to \infty$, i.e. we have the following.
	\begin{Lemma}
		\label{LemmaJost-exponential-growth}
		For $\lambda \in \R_*$ there exist a solution $ f_4(x, \lambda)$ of 
		\[
		\mathcal{H}f_4(\cdot, \lambda) = (\lambda^2 + \mu)f_4(\cdot, \lambda),
		\]
		such that $f_4(x, \lambda) = e^{\gamma x}\begin{psmallmatrix}0\\[1pt] 1	\end{psmallmatrix} + e^{\gamma x} \cdot O(x^{1-m})$ as $ x \to \infty$ with $O$-bound uniform in $ \lambda $.  In particular for $\lambda \in \R_* $ there exists $x_1 \geq 0$ such that there holds
		\begin{align}\label{dec-exp-growi}
			&\bigg| \partial^k_x \bigg[   e^{-  \lambda x}f_4(x, \lambda) - \begin{pmatrix}0\\[1pt] 1	\end{pmatrix} \bigg] \bigg| \lesssim_k  (1 + |\lambda|)^{-1 -k}\langle x \rangle^{1 -m -k},~~~ x \geq x_1,
		\end{align}
		where $ x_1$ only depends on $V$ if $ \mu > 0$. Further we take $ x_1 = 0$ if $ |\lambda | \geq \lambda_1 \gtrsim1$ is large. 
	\end{Lemma}
	\begin{proof} The function $ \chi( x, \lambda) = e^{- \gamma x} f_4(x, \lambda) $ is found with the contraction principle for the integral equation 
		\begin{align}\label{int-exp-growing}
			\chi(x, \lambda) =& \begin{psmallmatrix}0\\[1pt] 1	\end{psmallmatrix} + \int_x^{\infty} \begin{psmallmatrix}0& 0\\  
				0 &  \frac{1}{2\gamma}	\end{psmallmatrix}  V(y) \chi(y, \lambda)~dy\\[2pt] \nonumber
			&+ \int_{x_1}^{x} \begin{psmallmatrix}\frac{\sin(\lambda(x-y))}{\lambda}e^{\gamma(y-x)}& 0\\
				0 &  \frac{e^{2 \gamma (y-x)}}{2\gamma}	\end{psmallmatrix}  V(y) \chi(y, \lambda)~dy, 
		\end{align}
		which is well defined on functions $ \chi(y, \lambda) \lesssim 1$ and where $ x_1 = x_1(\lambda) \geq 0$ will be fixed and large. Denoting the linear expression on the right of \eqref{int-exp-growing} by $ T(\chi)$ we infer 
		\begin{align}
			| T(\chi) - T(\tilde{\chi})| \lesssim~~& (1 + |\lambda|)^{-1} \int_x^{\infty} \langle y \rangle^{-m} |\chi(y, \lambda) - \tilde{\chi}(y, \lambda)|~dy\\[2pt] \nonumber
			& + (1 + |\lambda|)^{-1} \int_{x_1}^x (\gamma (x-y) e^{- \gamma(x-y)} + e^{-2 \gamma(x-y)})\langle y \rangle^{-m} |\chi(y, \lambda) - \tilde{\chi}(y, \lambda)|~dy.
		\end{align}
		We thus find $ x_1 \geq 0$  (potentially large depending on the size of $\lambda > 0$) such that $ T$ is a contractive self-map on the ball $ B_X = \{ f \in  X ~|~ \| f\|_X \leq 2\}$ where
		\[
		X := BUC([x_1, \infty)) = \{f \in C_b([x_1, \infty))~|~ f~\text{unif. cont.}\}
		\]
		with the supremum norm. Especially we have \eqref{dec-exp-growi} for $ k =0$ in the fixed map. Higher regularity is bootstrapped by \eqref{int-exp-growing} and the bounds for the derivatives are obtained inductively by integrating by parts in \eqref{int-exp-growing}. 
	\end{proof}
	\begin{Rem} If $ \mu > 0 $ we extend the solution  $f_4(x, \lambda)$ smoothly  to $ \lambda = 0$. 
	\end{Rem}
	Now at $ \lambda = 0$ or $ \lambda \gg1 $ we state the following observations which will be useful below. In particular the following is useful when calculating the derivative $ \partial_{\lambda}$ of a Wronskian at $ \lambda = 0$.
	\begin{Corollary} \label{Corr} Let $ \mu = 0$, then there holds for some $C  > 0$
		\begin{align}
			&\left|\partial_{\lambda}f_1(x, 0) - \begin{psmallmatrix} i x\\[1pt] 0	\end{psmallmatrix} \right| \leq C, \hspace{10pt} \left|\partial_{\lambda}\partial_xf_1(x, 0) - \begin{psmallmatrix} i \\[1pt] 0	\end{psmallmatrix} \right| \leq C \langle x \rangle^{2-m},\\[6pt]
			& \left|\partial_{\lambda}f_3(x, 0) + \begin{psmallmatrix} 0\\[1pt] x	\end{psmallmatrix} \right| \leq C \langle x \rangle^{3-m},\hspace{10pt} \left|\partial_{\lambda}\partial_xf_3(x, 0) - \begin{psmallmatrix} 0 \\[1pt] 1	\end{psmallmatrix} \right| \leq C \langle x \rangle^{2-m}.
		\end{align}
	\end{Corollary}
	\begin{proof} The bounds follow directly from Lemma \ref{LemmaJost-decay} and Lemma \ref{LemmaJost-oscillating}. For the left bound in the first line for instance we write
		\[
		e^{-i \lambda x} (\partial_{\lambda}f_1(x, \lambda)-  e^{i \lambda x}\begin{psmallmatrix} i x\\[1pt] 0	\end{psmallmatrix}  ) = \partial_{\lambda}( e^{-i \lambda x}f_1(x, \lambda) - \begin{psmallmatrix} 1\\[1pt] 0	\end{psmallmatrix}) +  i x ( e^{-i \lambda x} f_1(x, \lambda) - \begin{psmallmatrix} 1\\[1pt] 0	\end{psmallmatrix} ) ,
		\]
		and estimate the right side at $\lambda = 0$ by Lemma \ref{LemmaJost-oscillating}. For the other estimates we proceed similarly.
	\end{proof}
	The following Corollaries are useful when calculating the large $ \lambda \gg1$ asymptotic of a Wronskian, for which we will conveniently choose $ x=0$.
	\begin{Corollary} \label{Corr2}For $|\lambda|\geq \lambda_0$ with $\lambda_0  \gg 1$ large, we have for all $ x \geq 0$ 
		\begin{align*}
			&\left|\partial_x f_1(x, \lambda) - i \lambda\begin{pmatrix} e^{i \lambda x}\\[1pt] 0	\end{pmatrix} \right| \leq C_{\lambda_0, V}, \hspace{10pt} \left| \partial_xf_3(x, \lambda) + \gamma \begin{pmatrix} e^{- \gamma x} \\[1pt] 0	\end{pmatrix} \right| \leq C_{\lambda_0, V},\\
			&\left|\partial_{\lambda} f_1(x, \lambda) - i x\begin{pmatrix} e^{i \lambda x}\\[1pt] 0	\end{pmatrix} \right| \leq C_{\lambda_0, V}|\lambda|^{-1}, \hspace{10pt} \left| \partial_{\lambda}f_3(x, \lambda) + \lambda \gamma^{-1} x \begin{pmatrix} e^{- \gamma x} \\[1pt] 0	\end{pmatrix} \right| \leq C_{\lambda_0, V} |\lambda|^{-1}.
		\end{align*}
	\end{Corollary}
	\begin{proof}The  bounds for $\partial_xf_1$ and $ \partial_{\lambda}f_1$  are obtained by inspecting the proof of Lemma \ref{LemmaJost-oscillating}. In particular we use that $x_0 \geq 0$ is allowed to be chosen $ x_0 = 0$ if $ \lambda \gg 1 $ is large enough, see also \cite[Remark 5.4]{K-S-stable}. The bounds for $f_3$ are obtained as in the proof of Corollary \ref{Corr}.
	\end{proof}
	Similarly we observe the following for second order derivatives in the above sense.
	\begin{align*}
		&\partial_{\lambda}\partial_x f_1(x, \lambda) = ( i - x\lambda)\begin{pmatrix} e^{i \lambda x}\\[1pt] 0	\end{pmatrix} + O(1)_{\lambda_0, V},\\
		&\partial_{\lambda} \partial_xf_3(x, \lambda)  = - ( 1 - \gamma x)\lambda \gamma^{-1} \begin{pmatrix} e^{- \gamma x} \\[1pt] 0	\end{pmatrix} + O(1)_{\lambda_0, V}.
	\end{align*}
	\begin{Corollary}For $|\lambda|\geq \lambda_0$ with $\lambda_0  \gg 1$ large, we have for all $ x \geq 0$ 
		\begin{align*}
			&\left|\partial_x f_4(x, \lambda) - \gamma\begin{pmatrix} 0\\[1pt]e^{\gamma x}	\end{pmatrix} \right| \leq C_{\lambda_0, V}, \hspace{10pt}\left|\partial_{\lambda} f_4(x, \lambda) - \lambda\gamma^{-1} x\begin{pmatrix} 0\\[1pt]e^{\gamma x}	\end{pmatrix} \right| \leq C_{\lambda_0, V}\frac{1}{|\lambda|},\\
			&\left|\partial_{\lambda}\partial_x f_4(x, \lambda) - (1 + \gamma x)\lambda\gamma^{-1} \begin{pmatrix} 0\\[1pt]e^{\gamma x}	\end{pmatrix} \right| \leq C_{\lambda_0, V}.
		\end{align*}
	\end{Corollary}
	~~\\
	{\bf\emph{Properties of} $\{f_j \}_{j = 1}^4$}. Let us note Lemma \ref{LemmaJost-decay} and  Lemma \ref{LemmaJost-oscillating} state the solutions 
	\[
	f_3(x, \lambda)\sim e^{- \gamma x} ,~ f_{1}(x, \lambda )\sim e^{i\lambda x},~f_{2}(x, \lambda) \sim e^{-i\lambda x},~~x \to \infty
	\]  exist smoothly for $ \lambda \in \R$ with  upper bounds in $\lambda = 0$ if $ \mu > 0$ and  $ f_4(\cdot, \lambda) \sim e^{\gamma x}$  in  Lemma  \ref{LemmaJost-exponential-growth} has a smooth extension to $\lambda = 0$ if $ \mu > 0$.
	
	\begin{Lemma} \label{Prp} There holds for $ \lambda \in \R$ (resp. $\lambda \in \R \backslash \{0\}$ if $\mu =0$)
		\begin{align*}
			&f_1(\cdot , -\lambda) = \overline{f_1(\cdot , \lambda) } = f_2(\cdot , \lambda),~~~	f_2(\cdot , -\lambda) = \overline{f_2(\cdot , \lambda) } = f_1(\cdot , \lambda)\\
			&f_3(\cdot , -\lambda) = \overline{f_3(\cdot , \lambda) } = f_3(\cdot , \lambda),~~~f_4(\cdot , -\lambda) = \overline{f_4(\cdot , \lambda) } = f_4(\cdot , \lambda).
		\end{align*}
		Also $f_1(\cdot , 0) = f_2(\cdot , 0)$ and $ g_j(x,\lambda) := f_j(- x, \lambda)$ are solutions of $ (\mathcal{H} - \lambda^2 - \mu)f = 0$.
	\end{Lemma}
	The proof follows, see e.g. \cite{K-S-stable}, from the fact that $V_1, V_2$ are real valued, even and the constructions in the above Lemma are conjugation invariant up to $ \lambda \mapsto - \lambda$ in the oscillating terms.
	\begin{Lemma} \label{Wronsk-Lemma}
		The Wronskian for  differentiable functions $f,g: \R \to \C^2$ is defined by
		\[
		w(f,g) = \langle f' , g \rangle - \langle f, g' \rangle,
		\]
		where $\langle \cdot , \cdot \rangle$ is the euclidean scalar product. 
		Then $ w(f,g)$ is constant if 
		$(\mathcal{H} - (\lambda^2 + \mu))f = 0 $ and $(\mathcal{H} - (\lambda^2 + \mu))g = 0$.
		In particular
		\begin{align}
			w(f_1, f_2) = 2 i \lambda ,~ w(f_1 , f_3) = w(f_2, f_3) = 0,~ w(f_3, f_4) = - 2 \gamma.
		\end{align}
	\end{Lemma}
	For proof we refer to \cite[Lemma 5.8]{K-S-stable} in a similar context and also \cite[Section 4.1]{OP} which includes the statement of the Lemma for  \eqref{op-prob}.
	\begin{Rem} We may replace $ f_4(x, \lambda)$ as in \cite[Lemma 5,8]{K-S-stable} by 
		\[
		\tilde{f}_4(\cdot, \lambda) = f_4(\cdot, \lambda) - c_1(\lambda)f_1(\cdot, \lambda) - c_2(\lambda)f_2(\cdot, \lambda),
		\] 
		such that  $ c_j(\lambda) = O(\lambda^{-1})$ and $ w(f_1, \tilde{f}_4) = w(f_2, \tilde{f}_4) = 0$.
		It follows directly from Lemma \ref{LemmaJost-exponential-growth},  \ref{LemmaJost-oscillating} and \ref{Prp} that this is possible such that the estimate in Lemma \ref{LemmaJost-exponential-growth} and Lemma \ref{Prp} still hold with $\tilde{f}_4$ replacing $f_4$.
	\end{Rem}
	Let us now introduce the Wronskian for matrix-valued functions.
	\begin{Lemma}\emph{(\cite[Lemma 5.10]{K-S-stable})}\label{KS-Lemma} Let $ F,G : \R \to M(2\times 2, \C)$ be differentiable  with values in the space of complex $2\times 2$ matrices. The matrix Wronskian
		\begin{align}
			\mathcal{W}(F,G) = (F')^t G - F^t G'
		\end{align} 
		is constant if $ (\mathcal{H} - (\lambda^2 + \mu))F = (\mathcal{H} - (\lambda^2 + \mu))G = 0$. Further if $\mathcal{W}(F,F) = 0$ or $ \mathcal{W}(G,G) = 0  $, then 
		\[
		\det\mathcal{W}(F,G) = 0,
		\]
		if and only if $ F(x) a + G(x) b = 0$ for some $ a,b \in \C^2$ with either $ a \neq 0 $ or $b \neq 0 $.
	\end{Lemma}
	We note that for differentiable functions $F,G$ as in Lemma \ref{KS-Lemma} and  constant $2\times 2 $ matrices $C_1, C_2$,  there clearly holds
	\begin{align}\label{ruls}
		&\mathcal{W}(FC_1,GC_2) = C_1^t \cdot\mathcal{W}(F,G)\cdot C_2,~~\mathcal{W}(\overline{F},G) = \overline{\mathcal{W}(F,\overline{G})},\\ \nonumber
		&\mathcal{W}(F,G)	= - \mathcal{W}(F,G)^t.
	\end{align}
	Lemma \ref{KS-Lemma} and Lemma \ref{Prp} imply the following Corollary.
	\begin{Cor}\emph{(\cite[Lemma 5.18]{K-S-stable})}
		For $ \lambda \in \R$ and the $2 \times 2$ matrices
		\begin{align*}
			&F_1(\cdot, \lambda) := (f_1(\cdot, \lambda), f_3(\cdot, \lambda)),~~F_2(\cdot, \lambda) := (f_2(\cdot, \lambda), f_4(\cdot, \lambda)),\\
			&G_1(\cdot, \lambda) := (g_2(\cdot, \lambda), g_4(\cdot, \lambda)),~~G_2(\cdot, \lambda) := (g_1(\cdot, \lambda), g_3(\cdot, \lambda)),
		\end{align*}
		there holds
		\begin{align*}
			&G_1(x, \lambda) = F_2(-x, \lambda),~~ 	G_2(x, \lambda) = F_1(-x, \lambda),\\
			&\overline{F_1(\cdot, \lambda)} = F_1(\cdot, - \lambda),~~\overline{F_2(\cdot, \lambda)} = F_2(\cdot, - \lambda),\\
			&\overline{G_1(\cdot, \lambda)} = G_1(\cdot, - \lambda),~~\overline{G_2(\cdot, \lambda)} = G_2(\cdot, - \lambda).
		\end{align*}
	\end{Cor}
	~~\\
	In particular the columns of $G_1, G_2$ and likewise $ F_1, F_2$ are fundamental solutions for the operator $\mathcal{H} - (\lambda^2 + \mu)$ if $ \lambda \neq 0$. Now direct calculation and using Lemma \ref{Wronsk-Lemma} shows the identities
	\begin{align}\label{des23}
		&\mathcal{W}(\bar{F_1},F_1) = - \mathcal{W}(\bar{G_2},G_2) =	\mathcal{W}(\bar{G_1},G_1) =  - 2 i \lambda p,\\ \nonumber
		&\mathcal{W}(\bar{G_1},G_2) = - \mathcal{W}(\bar{G_2},G_1) = - 2 \gamma q,\\\nonumber
		&\mathcal{W}(F_1,F_1) = \mathcal{W}(F_2,F_2) = 0,\\\nonumber
		&\mathcal{W}(F_1,F_2) = - \mathcal{W}(F_2,F_1) = 2i \lambda p - 2 \gamma q,
	\end{align}
	where
	\[
	p = \begin{pmatrix}1 & 0\\[1pt] 0 & 0	\end{pmatrix},~~q = \begin{pmatrix}0 & 0\\[1pt] 0 & 1	\end{pmatrix}.
	\]
	~~\\
	The following two Lemma from \cite{K-S-stable} establish analogue identities of reflection and transmission coefficients in the scalar scattering theory for the matrix system in this Section. 
	\begin{Lemma} \label{AantB}For $\lambda \neq 0$ we have the linear combination
		\begin{align}\label{lin}
			F_1(\cdot, \lambda) = G_1(\cdot, \lambda) A(\lambda) + G_2(\cdot, \lambda)B(\lambda),
		\end{align}
		where $A(\lambda), B(\lambda)$ are complex $2\times 2$ matrices depending smoothly on $\lambda \in \R_*$ and such that
		\begin{align}
			&A(- \lambda) = \overline{A(\lambda)},~B(- \lambda) = \overline{B(\lambda)},\\
			&\lambda p A(\lambda), q A(\lambda), \lambda p B(\lambda), q B(\lambda) \in  C^{\infty}(\R, M(2\times2, \C))\\ \label{ye}
			& A(\lambda) = I + O(\lambda^{-1}),~~B(\lambda) = O(\lambda^{-1}),~~ \lambda \to \infty.
		\end{align}
		We have the identities
		\begin{align}\label{das}
			&\mathcal{W}(F_1(\cdot, \lambda), G_2(\cdot, \lambda)) = A(\lambda)^t(2 i \lambda  p - 2 \gamma q),\\\label{das23}
			&\mathcal{W}(F_1(\cdot, \lambda), G_1(\cdot, \lambda)) =  - B(\lambda)^t(2i \lambda   p - 2 \gamma q),\\ \label{des}
			&G_2(\cdot, \lambda) = F_2(\cdot, \lambda)A(\lambda) + F_1(\cdot, \lambda)B(\lambda).
		\end{align}
	\end{Lemma}
	\begin{proof}[Sketch of proof]
		The existence of $A(\lambda), B(\lambda)$ is immediate since the columns of $G_1(\cdot, \lambda), G_2(\cdot, \lambda)$ form a fundamental base for $\mathcal{H} - \lambda^2- \mu $ if $ \lambda \neq 0$. The lines \eqref{das} and \eqref{das23} are calculated with \eqref{ruls} using \eqref{lin} in the Wronskian and Lemma \ref{Wronsk-Lemma} for the components. The regularity follows directly form the Wronskians in \eqref{das}, \eqref{das23} and the symmetry properties of $A(\lambda), B(\lambda)$ are due to conjugation symmetry of $F_1(\cdot, \lambda)$.  Moreover, the third line \eqref{ye}  follows from a direct calculation of  
		\[
		\mathcal{W}(F_1(0, \lambda), G_2(0, \lambda)),~\mathcal{W}(F_1(0, \lambda), G_1(0, \lambda))
		\]
		where we take $ \lambda \gg 1 $ large and use Corollary \ref{Corr2} (see \cite[Corollary 5.15]{K-S-stable}). 
	\end{proof}
	\begin{Lemma} \label{transmi} There holds for $\lambda \neq 0$ and $A(\lambda), B(\lambda)$
		\begin{align}
			A^*(\lambda)( 2 i\lambda   p A(\lambda) + 2 \gamma q B(\lambda)) - 2i \lambda p &=   B^*(\lambda) ( 2   i \lambda p B(\lambda) + 2 \gamma q A(\lambda)),\\
			B^t(\lambda)(2i\lambda p - 2 \gamma q)A(\lambda) &= A^t(\lambda)(2i \lambda p - 2 \gamma q)B(\lambda) ,\\
			2i \lambda ( B^*(\lambda) p +   p B(\lambda)) &= 2 \gamma (A^*(\lambda) q  - qA(\lambda)).
		\end{align}
	\end{Lemma}
	For the proof we refer to \cite{K-S-stable}. However one essentially only needs to rewrite
	\[
	\mathcal{W}(\bar{F}_1, F_1),~ \mathcal{W}(F_1, F_1),~ \mathcal{W}(\bar{F}_1, G_2), 
	\]
	via \eqref{lin} and \eqref{ruls} in the first two Wronskians and for the third one makes two separate calculations using   \eqref{lin}  and \eqref{des}. This is then compared to the explicit formulas provided by \eqref{des23}.\\[5pt]
	\emph{The threshold}. For the 1D-operator $ \mathcal{H}$ defined in \ref{Def-on} we call $ z = \pm \mu$ a \emph{resonance} if there exists   $f \in L^{\infty}\backslash L^2(\R)$ such that $ \mathcal{H}f = \pm \mu f$ holds in the distributional sense. We note that, since $ f_1(\cdot, 0) = f_2(\cdot, 0)$, a fundamental base for $(\mathcal{H} - \mu)f = 0$  with $ \mu > 0$ is given by
	\begin{align}\label{base1}
		f_1(x, 0), ~f_3(x, 0), ~\partial_{\lambda}f_1(x,0),~ f_4(x,0).	
	\end{align}
	In particular since $ \sigma_1 \mathcal{H} = - \mathcal{H}\sigma_1$, an analogue base for $ (\mathcal{H} + \mu)f = 0 $  is given by transforming \eqref{base1} via $ g \mapsto \sigma_1g$. In  case $\mu = 0 $ we have to choose
	\begin{align}\label{base}
		f_1(x, 0), ~f_3(x, 0), ~\partial_{\lambda}f_1(x,0),~ \partial_{\lambda}f_3(x,0).	
	\end{align}
	Thus  for any solution of $ \mathcal{H}f = 0$, there exists some fixed $ v \in \C^2$ with either
	\[
	\lim_{x \to \infty}f(x) = \begin{pmatrix}
		v_1\\v_2
	\end{pmatrix},~~ \text{or}~~f(x)  =  \begin{pmatrix}
		v_1\\v_2
	\end{pmatrix}\cdot O(x),~~\text{as}~ x \to \infty,
	\]
	and likewise for $ x \to - \infty$.
	~~\\
	We collect some facts obtained for $\mathcal{H}$ so far (cf. \cite[Lemma 5.17 - 5.20]{K-S-stable}).
	\begin{Lemma} \label{as2} Any solution of $ \mathcal{H}f = 0$ with $\mu = 0$ and $f \in L^{\infty}\backslash L^2$ has the form
		\begin{align}\label{as}
			&f(x) = \begin{pmatrix}a_{\pm} \\[1pt] b_{\pm} \end{pmatrix} + O(\langle x \rangle^{2-m}),~~~ x \to \pm \infty,~ a_{\pm}, b_{\pm} \in \C. 	
		\end{align}
		where either $ a_{\pm} \neq 0 $ or $ b_{\pm} \neq 0$. Any solution of $ (\mathcal{H} \mp \mu)f = 0$ with $\mu > 0$  and $f \in L^{\infty}\backslash L^2$ has the form 
		\begin{align}\label{as3}
			&f(x) = \begin{pmatrix}a_{\pm} \\[1pt] 0 \end{pmatrix} + O(\langle x \rangle^{2-m}),~~~ x \to \pm \infty,~ a_{\pm} \in \C,~~	\text{in the (-) case},\\
			&f(x) = \begin{pmatrix} 0 \\[1pt] b_{\pm} \end{pmatrix} + O(\langle x \rangle^{2-m}),~~~ x \to \pm \infty,~ b_{\pm} \in \C,~~	\text{in the (+) case},
		\end{align}
		where both $ a_{\pm} \neq 0 $ (and $ b_{\pm} \neq 0$ respectively).
	\end{Lemma}
	\begin{proof}
		The statement follows directly by inspecting  the $x \to \pm \infty$ asymptotics of \eqref{base1} and  \eqref{base}.
	\end{proof}
	\begin{Lemma}\label{A-LEm}
		There holds $ 0 = \det(A(\lambda)) $ and $ 0 =\det \mathcal{W}(F_1(\cdot, \lambda), G_1(\cdot, \lambda))$ for any $\lambda \neq 0$ if and only if $\lambda^2 + \mu$ is an eigenvalue for $\mathcal{H}$. Further 
		\begin{align} \label{di}
			0 =\det \mathcal{W}(F_1(\cdot, 0), G_1(\cdot, 0))
		\end{align} 
		if and only if $ \lambda = \mu$ (and thus also $ \lambda = - \mu$) is either  an eigenvalue or a resonance for $ \mathcal{H}$. If $ \mu = 0$, then $ \lambda = 0$ is not an eigenvalue, hence \eqref{di} holds if and only if $ \lambda = 0$ is a resonance.
	\end{Lemma}
	\begin{proof}[Sketch of Proof] We recall from the prior observation in \eqref{das}
		$$   4\lambda^2|\det(A(\lambda))| = |\det\mathcal{W}(F_1(\cdot, \lambda), G_1(\cdot, \lambda))|.$$
		The fact that both sides vanish if and only if $ \lambda \neq 0$ is an eigenvalue follows by contradiction. In particular if $A(\lambda)v = 0$ for some $ v \neq 0$, then testing the first identity of Lemma \ref{transmi} with $v$ implies $ \lambda^2  \neq 0 $ is an eigenvalue. The statement \eqref{as} at $ \lambda = \pm \mu$ follows from  the asymptotics of  fundamental base \eqref{base} stated in Corollary \ref{Corr} and direct calculation of the Wronskian $\mathcal{W}(F_1, G_2)$ at $ \lambda = 0$. For more details we refer to \cite[Lemma 5.17]{K-S-stable}.
	\end{proof}
	~~\\
	{\bf \emph{Scattering solutions and resolvents}}. We define  the smooth matrix valued function
	\begin{align}\label{D}
		D(\lambda) := \mathcal{W}(F_1(\cdot, \lambda), G_2(\cdot, \lambda)),~~ \lambda \in \R \backslash \{ 0\}.
	\end{align}
	Then from the previous observations we have
	\begin{align} \label{D2}
		&D(\lambda) =  A^t(\lambda)(2 i \lambda p - 2 \gamma q) =  A^t(\lambda)\begin{pmatrix}
			2 i \lambda & 0\\
			0 & - 2 \gamma 
		\end{pmatrix},\\ \nonumber
		&D(\lambda)^t = D(\lambda),~~\overline{D(\lambda)} = D(- \lambda).
	\end{align}
	Directly implied by the definition we infer the following. Let $ \mathcal{H}$ in \eqref{op} have no imbedded eigenvalues, then  $D(\lambda)^{-1}$ via  Lemma \ref{AantB} and Lemma \ref{A-LEm} with
	$$  A(\lambda)^{-1} =  D(\lambda)^{-1} ( 2i\lambda p - 2 \gamma q) $$
	is well defined, smooth on  $\R \backslash \{ 0\}$. Further if $ \pm \mu$ is not a resonance, then $ \lambda \mapsto D(\lambda)^{-1}  $ is smooth on  $\R $.  Therefore the \emph{absence of imbedded eigenvalues} is a standing assumption for this Subsection.
	\begin{Def} We set the  function 
		\begin{align}
			k(\lambda) : = 2 i \lambda D(\lambda)^{-1}, ~~ \lambda \in \R\backslash\{0\},
		\end{align}
		In particular $ k(\lambda)^t = k(\lambda)$ and $  k(- \lambda) = \overline{ k(\lambda) }$. In case  the threshold $ \pm \mu$ is not a resonance, $ k(\lambda )$  has a smooth extension to $\lambda = 0$ with $k(\lambda) = O(\lambda)$ as $ \lambda \to 0$.
	\end{Def}
	\begin{Lemma} In any case we have $k(\lambda) = O(1),~k'(\lambda) = O(\lambda^{-1})$ as $\lambda \to \infty$.
	\end{Lemma}
	\begin{proof}
		The asymptotics for $ \lambda \gg1 $ is directly implied by \eqref{ye}, \eqref{das} of Lemma \ref{AantB}. Further we calculate $D'(\lambda)$ at $x = 0$ as in the proof of Lemma \ref{AantB} (see \cite{K-S-stable}) and Corollary \ref{Corr2} in the identity
		$
		[\lambda D(\lambda)^{-1}]' = D(\lambda)^{-1} - \lambda D(\lambda)^{-1} D'(\lambda)D(\lambda)^{-1}
		$.
	\end{proof}
	Let us restrict to the case of $m =4$, i.e. $V_j(x) \lesssim \langle x \rangle^{-4}$ from now on. If $\mu = 0$ we infer  (cf. \cite[Section 4.1]{OP}) by Corollary \ref{Corr}, Lemma \ref{LemmaJost-decay} and Lemma \ref{LemmaJost-oscillating} 
	\begin{align} \label{calcul}
		& \partial_{\lambda}D(0) = 2(q - ip),  \\[2pt] \label{calllcul}
		&| \partial_{\lambda}^2 D(\lambda) | \lesssim \log(\lambda^{-1}+1),~~ 0 < \lambda \lesssim1.
	\end{align}
	In particular, for the linearized operator \eqref{op-prob} and its 1-D reduction in Definition \ref{Def-on},  the Lemma \ref{FS-inner-re} in Section \ref{subsec:inner}  provides  an explicit fundamental base of $ \mathcal{H}f = 0$ given by the odd extensions
	$ x \cdot\Theta_{\pm}(x), ~ x \cdot\Phi_{\pm}(x),~x \in \R$. Thus we find
	\begin{align*}
		&\mathcal{W}_0(x) = \f{1}{\sqrt{3}}W(x) \begin{psmallmatrix}
			1\\[1pt]-1
		\end{psmallmatrix},~~~ \mathcal{W}_1(x) = - \f{2}{\sqrt{3}}W_1(x)\begin{psmallmatrix}
			1\\[1pt]1
		\end{psmallmatrix},\\
		&f_1(x, 0) =  x\big(\mathcal{W}_0(x) +  \mathcal{W}_1(x)\big),~~~f_3(x, 0) =  x \big(\mathcal{W}_1(x) -  \mathcal{W}_0(x)\big),
	\end{align*}
	and hence $ D(0) = 0$. From e.g. \eqref{calcul} and  \eqref{calllcul} we obtain for instance
	\begin{align*}
		| D(\lambda)^{-1}| \gtrsim  (\lambda | \log(\lambda)|)^{-1},~~~ 0 < \lambda \ll 1.
	\end{align*}	
	and $D(\lambda)^{-1}$ has no $C^0$ extension to $ \lambda = 0$. More precisely in this situation we obtain the following.
	\begin{Lemma} \label{LemmafD}  Let  $ \mu = 0$ and $ D(0) = 0$. Then the function $ k(\lambda) = 2 i \lambda D(\lambda)^{- 1}  $ is continuously extended to $\lambda = 0$ and 
		\begin{align}
			\lim_{\lambda \to 0^+} k(\lambda)  &=  i q - p = 2i \big[\partial_{\lambda}D(0)\big]^{-1},\\ \label{sen}
			| \partial_{\lambda} k(\lambda) | &\lesssim |\log(\lambda)|,~~ 0 < \lambda \ll1,\\ \label{sen22}
			| \partial^2_{\lambda} k(\lambda) | &\lesssim | \lambda|^{-1}|\log(\lambda)|,~~ 0 < \lambda \ll1.
		\end{align}
	\end{Lemma} 
	\begin{proof}
		The continuous extension to $ \lambda = 0$ is directly implied by \eqref{calcul} - \eqref{calllcul} and  a second order Taylor expansion of $\lambda D(\lambda)$. For \eqref{sen} we calculate 
		\[
		\partial_{\lambda} k(\lambda) =k(\lambda) \lambda^{-2}\big( D(\lambda) - \lambda \partial_{\lambda}D(\lambda) \big)k(\lambda) 
		\]
		and use \eqref{calllcul} in a Taylor expansion for the middle part on the right. Alternatively we use $ D(\lambda)^{-1} \sim \lambda^{-1}(iq - p)$ at $ \lambda = 0$ and differentiate $D(\lambda) \cdot D(\lambda)^{-1}$.  For  \eqref{sen22}, we differentiate the latter twice in $ \lambda$ and use \eqref{sen}. 
	\end{proof}
	For the remaining part of this Subsection we make the following \emph{additional assumption}: If $ \mu > 0$, we assume $\pm \mu $ is not a resonance and  if $ \mu = 0$, we let $ \lim_{\lambda \to 0 } \lambda k(\lambda)^{-1} = 0$ (This holds true if $f_1(\cdot, 0), f_3(\cdot, 0) $ are two odd resonance functions).\\[5pt]
	We define
	\begin{align}
		&k(\lambda)\underline{e} =: s(\lambda) \underline{e}+ \tilde{s}(\lambda) \begin{psmallmatrix}
			0\\1
		\end{psmallmatrix},\\
		&B(\lambda)	k(\lambda)\underline{e} =: r(\lambda) \underline{e}+ \tilde{r}(\lambda) \begin{psmallmatrix}
			0\\1
		\end{psmallmatrix}.
	\end{align} 
	Then by a simple calculation using Lemma \ref{AantB} we have
	\begin{align*}
		\mathcal{F}(\cdot, \lambda) =& ~s(\lambda)f_1(\cdot, \lambda) + \tilde{s}(\lambda) f_3(\cdot, \lambda)\\
		=& ~g_2(\cdot, \lambda) + r(\lambda) g_1(\cdot, \lambda) + \tilde{r}(\lambda) g_3(\cdot, \lambda),\\
		\mathcal{G}(\cdot, \lambda) =&~ f_2(\cdot, \lambda) + r(\lambda) f_1(\cdot, \lambda) + \tilde{r}(\lambda) f_3(\cdot, \lambda).
	\end{align*}
	The following is proved as in \cite{K-S-stable} in combination with Corollary  \ref{Corr2}.
	\begin{Lemma} \label{Das-asym-fomu}Let $ \mu > 0$. The functions $ s(\lambda), \tilde{s}(\lambda), r(\lambda), \tilde{r}(\lambda)$ are smooth on $ \R$ with $ r(0) = -1, \tilde{s}(0) =  s(0) = \tilde{r}(0) = 0$ . Further $ s(\lambda), r(\lambda), \lambda s'(\lambda), \lambda r'(\lambda) = O(1)$ as $\lambda \to \infty $ and uniformly in $ \mu \to 0^+$.
	\end{Lemma}
	\begin{Lemma} \label{Das-asym} Let $ \mu = 0$. The functions $ s(\lambda), \tilde{s}(\lambda), r(\lambda), \tilde{r}(\lambda)$ are smooth on $ \R\backslash\{0\}$, continuously  extended to $ s(0) = -1, \tilde{s}(0) =  r(0) = \tilde{r}(0) = 0$  and in the space
		$$ X_0 = \{ f \in C^{0}(\R)\cap C^{1}(\R \backslash \{0\})~|~ \partial_{\lambda}^kf(\lambda)  = O(|\lambda|^{1-k}|\log(|\lambda|)|) \text{ as} ~ \lambda \to 0^{\pm},~ k = 1,2\}.$$
		Further $ s(\lambda), r(\lambda), \lambda s'(\lambda), \lambda r'(\lambda) = O(1)$ as $\lambda \to \infty$ and there holds
		\begin{align}\label{ide1}
			2 i \lambda r(\lambda) = w(g_2, f_3)\tilde{s}(\lambda)  + w(g_2, f_1)s(\lambda),\\ \label{ide2}
			2  \lambda \tilde{r}(\lambda) = w(g_4, f_3)\tilde{s}(\lambda)  + w(g_4, f_1)s(\lambda).
		\end{align}
	\end{Lemma}
	\begin{proof}  The identities \eqref{ide1} and \eqref{ide2} follow from Lemma \ref{AantB}, i.e.
		\begin{align*}
			&2 \lambda  q B(\lambda)p =  q \mathcal{W}(G_1(\cdot, \lambda), F_1(\cdot, \lambda))p,~~2 \lambda q B(\lambda)q =  q \mathcal{W}(G_1(\cdot, \lambda), F_1(\cdot, \lambda))q,\\
			&2 \lambda i p B(\lambda)p =  p \mathcal{W}(G_1(\cdot, \lambda), F_1(\cdot, \lambda))p,~~2 \lambda i p B(\lambda)q =  p \mathcal{W}(G_1(\cdot, \lambda), F_1(\cdot, \lambda))q.
		\end{align*}
		The smoothness and  continuous extendability follow by Lemma \ref{LemmafD} and Lemma \ref{AantB}.  The limits  $ \lim_{\lambda \to  0}r(\lambda), \lim_{\lambda \to 0}\tilde{r}(\lambda)$ are calculated  by \eqref{ide1}, \eqref{ide2}. Likewise the asymptotics for the derivative is implied by Lemma \ref{LemmafD} and \eqref{ide1}, \eqref{ide2}.
	\end{proof}
	\begin{Lemma}\label{Das-Lem}We set $ \underline{e} = \begin{psmallmatrix}
			1 \\[1pt] 0
		\end{psmallmatrix}$ and further
		\begin{align*}
			\mathcal{F}(x, \lambda) : = F_1(x, \lambda) k(\lambda)\underline{e},\\
			\mathcal{G}(x, \lambda) : = G_1(x, \lambda)  k(\lambda)\underline{e}.
		\end{align*}
		Then   	$\mathcal{F}(\cdot, \lambda), \mathcal{G}(\cdot, \lambda)  $ are solutions of $ (\mathcal{H}- (\lambda^2 + \mu)) f = 0$ and with 
		\[
		s(\lambda)\underline{e}:= p k(\lambda)\underline{e},~~~r(\lambda)\underline{e} := p B(\lambda) k(\lambda)\underline{e},
		\]
		there holds
		\begin{align}
			&\mathcal{F}(x, \lambda) = s(\lambda)\big[e^{i x \lambda} \underline{e} + O(\langle \lambda\rangle^{-1}\langle x\rangle^{-2}) + O(|\lambda|\langle \lambda \rangle^{-1} e^{- \gamma x}) \big] +  O(  e^{- \gamma x})
			,~~ x \to \infty,\\[2pt]
			&\mathcal{F}(x, \lambda) = \big[e^{i x \lambda} + r(\lambda) e^{- i \lambda x}\big]\underline{e}  + O(\langle \lambda \rangle^{-1}\langle x\rangle^{-2}) + O(\langle \lambda \rangle^{-1} e^{ \gamma x}),~~ x \to - \infty,\\[2pt]
			&\mathcal{G}(x, \lambda) = s(\lambda)\big[e^{- i x \lambda} \underline{e} + O(\langle \lambda\rangle^{-1}\langle x\rangle^{-2}) + O(|\lambda|\langle \lambda \rangle^{-1} e^{ \gamma x})  \big] + O(  e^{ \gamma x})
			),~~ x \to - \infty,\\[2pt]
			&\mathcal{G}(x, \lambda) = \big[e^{-i x \lambda} + r(\lambda) e^{i \lambda x}\big]\underline{e} + + O(\langle \lambda \rangle^{-1}\langle x\rangle^{-2}) + O(\langle \lambda \rangle^{-1} e^{ -\gamma x}),~~ x \to  \infty.
		\end{align}
		Further there holds
		\begin{align}
			&|r(\lambda)|^2 + |s(\lambda)|^2 = 1,~~ s(\lambda)\overline{r(\lambda)} + r(\lambda) \overline{s(\lambda)} = 0,\\
			&s(- \lambda) = \overline{s(\lambda)},~r(- \lambda) = \overline{r(\lambda)}
		\end{align}
	\end{Lemma}
	\begin{Rem}In particular the Scattering matrix  
		\begin{align} \label{Scat}
			S(\lambda ) := \begin{pmatrix}
				s(\lambda) &r(\lambda)\\
				r(\lambda) & s(\lambda)
			\end{pmatrix},~~~ \lambda \in \R,
		\end{align}
		is unitary, i.e.
		$
		S(\lambda)^* = S(\lambda)^{-1} = S(\lambda)
		$.
	\end{Rem}
	\begin{Rem} The $O$-asymptotic in the above expansions are differentiable according to Lemma \ref{LemmaJost-decay}, Lemma \ref{LemmaJost-oscillating} and Lemma \ref{LemmaJost-exponential-growth}. Further: $\bold{(1)}$~ the terms $O(e^{\pm \gamma x}), ~O(\langle \lambda \rangle^{-1} e^{\pm \gamma x})$ are multiplied by $ O(|\lambda|\langle \lambda \rangle^{-1}) $ if $ \mu > 0$,  which holds true since then  $k(\lambda) = 2 i \lambda D(\lambda)^{-1} = O(\lambda)$ as $ \lambda \to 0$. $\bold{(2)}$ ~the terms $O(e^{\pm \gamma x}), ~O(\langle \lambda \rangle^{-1} e^{\pm \gamma x})$ are multiplied by $ O(|\lambda|\langle \lambda \rangle^{-1} |\log(2 |\lambda| \langle \lambda \rangle^{-1})|) $ if $ \mu =0$, which holds true by the asymptotic description of Lemma \ref{Das-asym} essentially implied by Lemma \ref{LemmafD} and the asymptotics of the Jost solutions and their Wronskians.
	\end{Rem}
	For the proof of Lemma \ref{Das-Lem} we refer to the proof of  \cite[Lemma 6.4]{K-S-stable} combined with Lemma \ref{LemmaJost-decay}, Lemma \ref{LemmaJost-exponential-growth}, Lemma \ref{LemmaJost-oscillating}. 
	\begin{Corollary}   All solutions $ f \in L^{\infty}(\R)$ of $ (\mathcal{H} - (\lambda^2 + \mu))f = 0$ are linear combinations of $\mathcal{F}(\cdot, \lambda), \mathcal{G}(\cdot, \lambda)$.
	\end{Corollary}
	\begin{proof}Clearly by the proof of Lemma \ref{as2} and the remark before this Lemma, all solutions $ f \in L^{\infty}(\R)$ must be linear combinations of  the columns of $ F_1, G_2$.  Further by Lemma \ref{AantB} 
		\begin{align}
			&F_1(x, \lambda ) v  = G_1(x, \lambda) A(\lambda)v + G_2(x, \lambda)B(\lambda )v,\\
			&G_2(x, \lambda)w  = F_1(x, \lambda) B(\lambda)w + F_2(x, \lambda)A(\lambda )w,
		\end{align}
		for all vectors $v,w \in \C^2$. Considering the asymptotic expression as $ x \to - \infty$  in the first line and $ x \to  \infty $ in the second line,  we see that $ \{\underline{e}, A(\lambda)v\}$,~ $ \{\underline{e}, A(\lambda)w\}$ must be proportional respectively (with constants depending on $ \lambda $).
	\end{proof}
	~~\\
	\emph{Resolvents}. We now turn to the (absolute) spectral density  of $\mathcal{H} = \mathcal{H}_0 + V$. We denote by
	\[
	(\mathcal{H} - z I )^{-1},~~ z \in \C \backslash \sigma(\mathcal{H}),
	\]
	the resolvent map of $\mathcal{H}$.  The resolvent identity (see below) and  the kernel $ (\mathcal{H}_0 - z I )^{-1}(x,y)$ then imply
	\begin{align}
		\lim_{\varepsilon \to 0^+}(\mathcal{H} - (\lambda \pm i \varepsilon))^{-1} &= \lim_{\varepsilon  \to 0^+}   (I  + (\mathcal{H}_0 - (\lambda \pm i 0))^{-1}V)^{-1}(\mathcal{H}_0 - (\lambda \pm i 0))^{-1}\\ \nonumber
		& = (\mathcal{H} - (\lambda \pm i 0))^{-1},~ ~\lambda > \mu
	\end{align}
	exists  at least distributionally.  Recall that if $ V = 0$ the kernel of the free resolvent
	\begin{align} \label{kern}
		(\mathcal{H}_0 - (\lambda^2 + \mu \pm i 0))^{-1}(x,y) = \begin{pmatrix}
			\mp	\frac{e^{\pm i \lambda|x-y|}}{2 i \lambda} & 0\\
			0 & \frac{e^{- \gamma |x-y|}}{-2 \gamma }
		\end{pmatrix},~~ \lambda > 0,
	\end{align}
	
	\begin{Lemma}\emph{(\cite[Lemma 6.5]{K-S-stable})}\label{densit}
		For $ \lambda > 0$ there holds
		\begin{align}
			&(\mathcal{H} - (\lambda^2 + \mu + i 0))^{-1}(x,y) = \begin{cases}
				- \frac{1}{  2i\lambda} F_1(x,  \lambda) k(\lambda) G_2(y,  \lambda)^t \sigma_3 & x \geq y\\[4pt]
				- \frac{1}{ 2i\lambda} G_2(x, \lambda) k(\lambda) F_1(y,  \lambda)^t \sigma_3 & x \leq y.
			\end{cases}\\[4pt]
			&(\mathcal{H} - (\lambda^2 + \mu - i 0))^{-1}(x,y) = \begin{cases}
				\frac{1}{  2i\lambda} F_1(x, - \lambda) k(- \lambda) G_2(y, - \lambda)^t \sigma_3 & x \geq y\\[4pt]
				\frac{1}{ 2i\lambda} G_2(x, - \lambda) k(- \lambda) F_1(y, - \lambda)^t \sigma_3 & x \leq y.
			\end{cases}
		\end{align}
	\end{Lemma}
	
	\begin{Lemma}\emph{(\cite[Lemma 6,7]{K-S-stable})} \label{densit2}
		For $ \lambda > 0$ there holds
		\begin{align}
			(\mathcal{H} - (\lambda^2 + \mu + i 0))^{-1}(x,y)  - (\mathcal{H} - (\lambda^2 + \mu - i 0))^{-1}(x,y)  = - \frac{1}{2 i \lambda}\mathcal{E}(x, \lambda) \mathcal{E}^{*}(y, \lambda) \sigma_3
		\end{align}
		where we define
		\begin{align}
			\mathcal{E}(x, \lambda) = \big[\mathcal{F}(x, \lambda), \mathcal{G}(x, \lambda) \big],~~ \lambda \in \R.
		\end{align}
	\end{Lemma}

	\subsection{Spectral representation and distorted Fourier transform} \label{subsec:ft}~~\\[3pt]
	Based on the previous subsection, we now provide a spectral representation for the (non-unitary) flow $e^{ i t \mathcal{H}}$ of the operator
	$$\mathcal{H} = ( - \partial_x^2 + \mu)\sigma_3 + V$$
	as in Definition \ref{Def-on}. We start with a rather general argument following the lines of \cite{Erd-Schlag}, \cite{Schlag2} and \cite{K-S-stable}.  The desired expansion reads
	
	\begin{align}\label{formular}
		e^{i t \mathcal{H}}u =&~  \frac{1}{2 \pi i}\int_{\{|\lambda| > \mu\}}e^{i t \lambda}\big[ (\mathcal{H}- (\lambda + i 0))^{-1} -(\mathcal{H} - (\lambda - i 0))^{-1} \big]~d \lambda + \sum_{\zeta} e^{i t H}P_{\zeta} u
	\end{align} 
	~~\\
	where $ \zeta \in \C$ are isolated eigenvalues in the discrete spectrum and $ P_{\zeta}$ are corresponding Riesz projections. Note here that for $ \mu = 0$ the threshold $ \lambda = 0$ can not be assumed to be regular due to a resonance. 
	Let us first note the following. 
	\begin{Lemma} We have 
		$ \sigma(\mathcal{H}) = - \sigma(\mathcal{H}) = \overline{\sigma(\mathcal{H})} = \sigma(\mathcal{H}^*)$
		and in fact the essential spectrum $ \sigma_{\text{ess}}(\mathcal{H}) = ( - \infty, -\mu] \cup [ \mu , \infty)$. The discrete spectrum $\sigma_d(\mathcal{H})$ consists of eigenvalues $\{ \zeta_j\}_{j =1}^N$,~$\zeta_j \in \C \backslash \sigma_{\text{ess}}(\mathcal{H})$,~ $ 1 \leq N \leq \infty$ of finite multiplicity with closed range $\Ran(  \mathcal{H} - \zeta_j)$ and which are poles of the meromorphic resolvent $( \mathcal{H} -z)^{-1}$, $z \in \C \backslash \sigma_{\text{ess}}(\mathcal{H})$.  The order of the pole $k_j = k_j(\zeta_j)$ is the the minimal integer with $ \ker(\mathcal{H} - \zeta_j)^{k_j} =  \ker(\mathcal{H} - \zeta_j)^{k_j+1}$.  
	\end{Lemma}
	The proof follows directly  from the symmetry of $V$ in Def. \ref{Def-on} and the Weyl criterion (for the latter see for instance \cite{K-S-stable}). We further like to remark the discrete spectrum $\sigma_d(\mathcal{H}) = \{\zeta_j\}$ is always assumed to be a finite set under the assumption of  Def. \ref{Def-on}. \footnote{We note this will be a standing assumption for all operators in this Section. When applying this Section to small $ \mu$-perturbations of the linearized NLS  operator, the finiteness of $\sigma_d(\mathcal{H})$ holds true. 
	This will be considered below.} However algebraic and geometric multiplicity might differ for each isolated eigenvalue, that is $k_j > 1$. The following notation of admissibility will be useful.
	\begin{Def} \label{admiss} Let $ \mathcal{H} = (- \partial_x^2 + \mu)\sigma_3 + V(x)$ for some $ \mu > 0$  be as in Definition \ref{Def-on}. We call $\mathcal{H}$ {\bf admissible}, if the following holds.
		\begin{itemize}\setlength\itemsep{.3em}
			\item[(i)] The operator $ \mathcal{H} $ has no eigenvalues in $( - \infty, - \mu] \cup [ \mu, \infty)$  
			\item[(ii)] The  edges $\pm \mu $ are no resonances for $ \mathcal{H} $.
		\end{itemize}
	\end{Def}
	Recall from the previous section we say $\pm \mu$ is a \emph{resonance} for $ \mathcal{H} $ if  there exists a solution  $ f \in L^{\infty}\backslash L^2(\R)$ of 
	$ \mathcal{H}f = \pm \mu f $.
	~~\\[5pt]
	\emph{Limiting absorption principle}.~Let $( \mathcal{H}_0 - z)^{-1}$ be the free resolvent on $\C \backslash \sigma_{\text{ess}}(\mathcal{H}_0)$. Then we have
	\begin{align} \label{kern2-scat}
		(\mathcal{H}_0 -  z  )^{-1} =   \begin{pmatrix}
			R_0(z - \mu) & 0\\
			0 &  - R_0(- z - \mu)
		\end{pmatrix},~~ z  \notin ( - \infty, - \mu] \cup [\mu, \infty),
	\end{align}
	where $ R_0(z) = (- \Delta - z)^{-1}$ with kernel 
	in $d =1$
	\begin{align}
		R_0(z)(x,y) =  \frac{ e^{ i \sqrt{z}|x-y|}}{-2i\sqrt{z}},~~ z \in \C \backslash [0, \infty).
	\end{align}
	For the following limits, as already stated above, we thus obtain
	\begin{align} \label{kern3}
		&(\mathcal{H}_0 -  (\lambda^2 + \mu \pm i 0)  )^{-1}(x,y) =   \begin{pmatrix}
			\frac{ e^{\pm i \lambda|x-y|}}{\mp2i \lambda} & 0\\
			0 &  \frac{e^{- \gamma|x-y|}}{-2\gamma} 
		\end{pmatrix},\\[2pt] \nonumber
		&\gamma = \sqrt{\lambda^2 + 2 \mu},~~ \lambda > 0.
	\end{align}
	Further set $X_{\sigma} := L^{2, \sigma}(\R)  \times L^{2, \sigma}(\R) $ where $L^{2, \sigma}(\R) = \langle x \rangle^{- \sigma} L^2$ with $ \sigma > 0$. For $ \lambda_0 > \mu$ and $ \sigma > \f12 $ we have the limiting absorption principle
	\begin{align}\label{LAP1}
		\sup_{|\lambda| > \lambda_0,~ \varepsilon > 0} |\lambda|^{\f12}\| (\mathcal{H}_0 - (\lambda \pm i \varepsilon))^{-1}   \|_{X_{\sigma} \to X_{- \sigma}} < \infty,
	\end{align}
	implied by \eqref{kern2-scat} and we refer to Agmon's work \cite{Agmon} for general scalar theory in arbitrary dimension.
	Then also the limit $ \lim_{\varepsilon \to 0^+}  (\mathcal{H}_0 - (\lambda \pm i \varepsilon ))^{-1} $ exists in $X_{\sigma}$  and \eqref{LAP1} holds for $(\mathcal{H}_0 - (\lambda \pm i 0))^{-1}$. \\[2pt]
	The following is proved similar to \cite{K-S-stable} (cf also \cite{Erd-Schlag})  based on the limiting absorption principle \eqref{LAP1}. 
	\begin{Lemma} \label{LAP} Let $ \mathcal{H} = (- \partial_x^2 + \mu)\sigma_3 +V$ be as in Def. \ref{Def-on} with no embedded eigenvalue in $(- \infty,  -\mu] \cup [\mu, \infty)$. Then for any $ |\lambda| > \mu$,~$ \sigma > \f12 $  the operator 
		\begin{align}\label{LAP2}
			I  + (\mathcal{H}_0 - (\lambda \pm i 0))^{-1}V : X_{- \sigma} \to X_{ - \sigma}
		\end{align}
		is invertible and for $ \sigma, \tilde{\sigma} > \f12 $ there holds
		\begin{align}\label{LAP3}
			\sup_{|\lambda| > \lambda_0,~ \varepsilon > 0} |\lambda|^{\f12}\| (\mathcal{H} - (\lambda \pm i \varepsilon))^{-1}   \|_{X_{\sigma} \to X_{- \tilde{\sigma}}} < \infty.
		\end{align}
		where $ \lambda_0 > \mu$. Further for any $ \sigma, \tilde{\sigma} > \f12 $  the limit
		\begin{align}\label{LAP22}
			& (\mathcal{H} - (\lambda \pm i 0))^{-1} =   (I  + (\mathcal{H}_0 - (\lambda \pm i 0))^{-1}V)^{-1}(\mathcal{H}_0 - (\lambda \pm i 0))^{-1} 
		\end{align}
		exists in the norm of $ X_{\sigma} \to X_{- \tilde{\sigma}}$ and \eqref{LAP3} holds for $ \varepsilon = 0$. 
	\end{Lemma}
		%
		\begin{proof}[Sketch of Proof] 
			The invertability of 
			$$ 	I  + (\mathcal{H}_0 - (\lambda \pm i 0))^{-1}V $$
			in Lemma \ref{LAP} follows by the absence of embedded eigenvalues as in \cite{Erd-Schlag} by Agmon's bootstrapping Lemma. Further, considering $ \lambda > \lambda_0,~ z = \lambda \pm i \epsilon$ with   $0 < \epsilon \ll1 $ small, we have
			$$ (\mathcal{H} - z)^{-1} = (I  + (\mathcal{H}_0 - z)^{-1}V)^{-1}(\mathcal{H}_0 - z)^{-1}.$$
			The inverse of 
			$ I  + (\mathcal{H}_0 - (\lambda \pm i 0))^{-1}V $ has a Neumann expression if
			$$ 	\sup_{|\lambda| > \tilde{\lambda}}\| \mathcal{H}_0 - (\lambda \pm i 0))^{-1}V\| < \f12 $$
			and for which we take $ \tilde{\lambda}\gg 1 $ large. From \eqref{LAP1} we thus have $ \| (\mathcal{H} - z)^{-1}    \| \lesssim  |z|^{-\f12}$ with $ z = \lambda \pm i \epsilon$  if $| z | \gg1 $ is large, say $ |z| > R$. Since also from \eqref{kern3} (in the norm of $ X_{\sigma}, X_{- \tilde{\sigma}}$)
			\[
			\sup_{\Re(z) > \lambda_0}\|  (\mathcal{H}_0 - z)^{-1} \| < \infty,
			\]
			we show by contradiction the failure of \eqref{LAP3} would imply that 	$I  + (\mathcal{H}_0 - (\lambda \pm i 0))^{-1}V$ has no inverse for some $ \lambda \in [\lambda_0, R]$. We refer to \cite{Erd-Schlag} for more details (though for $3$D operators).
		\end{proof}
		\begin{Rem} Assume $\mathcal{H}$ in Lemma \ref{LAP} is admissible for some $ \mu > 0$, that is $\pm \mu$ is not a resonance.  We then obtain  the limiting absorption \eqref{LAP3} holds for $\lambda_0 = \mu$.
		\end{Rem}
		The Riesz projection $P_d $  onto the discrete spectrum of $\mathcal{H}$ is defined as
		\begin{align}\label{Riesz}
			P_d = \frac{1}{2 \pi i}\oint_{\Gamma} ( \mathcal{H} - z   )^{-1}~dz,
		\end{align} 
		where $\Gamma$ is a simple, closed contour enclosing all eigenvalues in $\C \backslash \sigma_{\text{ess}}(\mathcal{H})$. Let $P_{\Lambda}$ for a finite set $ \Lambda \subset \C \backslash \sigma_{\text{ess}}(\mathcal{H})$ be a sum of operators defined as in \eqref{Riesz} with small closed contours around each eigenvalue in $\Lambda$ and $ P_{\zeta_j} := P_{\{\zeta_j\}}$. Note that
		\begin{align}
			&\Ran(P_{\zeta_j}) = \bigcup_{k \geq 0} \ker(\mathcal{H} - \zeta_j  )^{k} ,\\
			&L^2(\R) \times L^2(\R) = \Ran(P_{\zeta_j}) \oplus \Ran(I - P_{\zeta_j}),
		\end{align}
		where both spaces in the latter direct sum are closed (see \cite{Hisl-Sig}). In case $ \mu = 0$, isolated eigenvalues are separated by the real axis and we for instance sum \eqref{Riesz} over $\Gamma^{\pm}$ enclosing all eigenvalues in $ \{  z ~|~ \pm\Im(z) > 0\}$. 
		We further set the projection
		$$ P := I - P_d $$ 
		to be analogue to the continuous spectrum in the scalar case. (For $ \mu> 0$, this is typically called $P_s$ for \emph{stable spectrum} (see \cite{Schlag2}, \cite{K-S-stable}) and not to be confused with $P_c$, for which we would, on the real axis, only subtract $ z=0$. In case $ \mu =0$ this notation reduces to $P_c$.)\\[8pt]
		The following representation of $ \langle e^{i t \mathcal{H}} \phi,\psi \rangle $ is  analogous to the spectral theorem for s.-a. Schr\"odinger operators with  asymptotic completeness.
		\begin{Lemma}\label{spectral-repr} Let  $\mathcal{H}$  be as in Definition \ref{Def-on} with $ \mu > 0$ and admissible. Then there holds
			\begin{align}\label{final-abs-spect}
				e^{i t \mathcal{H}} =  \frac{1}{2 \pi i}\int_{ |\lambda| \geq \mu}e^{i t \lambda}\big[ ( \mathcal{H} - (\lambda + i 0))^{-1} -(\mathcal{H}- (\lambda - i 0))^{-1} \big]~d \lambda + \sum_{j} e^{i t \mathcal{H}}P_{\zeta_j} 
			\end{align}
			where the sum runs over the finite discrete spectrum $\{ \zeta_j\}$. In particular  \eqref{final-abs-spect} is to be understood in the weak sense, i.e. for all $ \phi, \psi \in  W^{2,2}(\R) \times W^{2,2}(\R)$ there holds
			\begin{align}\label{final-abs-spect2}
				\langle e^{i t \mathcal{H}}\phi, \psi\rangle =  \lim_{R \to \infty}\frac{1}{2 \pi i}\int_{\mu \leq |\lambda| < R}&e^{i t \lambda} \langle \big[ (\mathcal{H} - (\lambda + i 0))^{-1} -(\mathcal{H} - (\lambda - i 0))^{-1}  \big] \phi, \psi\rangle~d \lambda\\ \nonumber
				& + \sum_{j} \langle e^{i t \mathcal{H}}P_{\zeta_j}\phi, \psi \rangle. 
			\end{align} 
			The integral in \eqref{final-abs-spect2} is well-defined by Lemma \ref{LAP}.
		\end{Lemma}
		\begin{proof}  The proof is a standard application of Hille-Yoshida's theorem  as provided in \cite{Erd-Schlag}, \cite{K-S-stable}. We give some details.
			The operator $ i \mathcal{H}$ is (up to a shift) the generator of a contractive semigroup, i.e. for $ a > 0$ large enough there holds
			\begin{align}
				( 0, \infty) \subset \C \backslash \sigma (i \mathcal{H} - a),~~~\sup_{\lambda > 0}\big( \lambda \| ( i \mathcal{H} - a - \lambda)^{-1}   \|_{2 \to 2}\big) \leq 1,
			\end{align}
			This is implied (for large  $ a \gg1$)  by 
			\[
			\|( i \mathcal{H}_0 - a - \lambda)^{-1}\|_{2 \to 2} \leq (\lambda + a)^{-1}
			\]
			and an absolute bound of the  Neumann expansion
			\begin{align}
				( i \mathcal{H} - a - \lambda)^{-1} =&~ ( i \mathcal{H}_0 - a - \lambda)^{-1} (  I + i V(i \mathcal{H}_0 - a - \lambda)^{-1} )^{-1}\\[2pt] \nonumber
				= &~ \sum_{k \geq 0} (-i)^k V^k ( i \mathcal{H}_0 - a - \lambda)^{-k-1}.
			\end{align}
			Thus the Laplace transform 
			\begin{align}
				(i\mathcal{H} - z)^{-1} = - \int_0^{\infty} e^{- t z} e^{i t \mathcal{H}}~dt,~~\Re(z) > a,
			\end{align}
			converges in norm topology since  $\| e^{i t \mathcal{H}} \|_{2 \to 2} \lesssim  e^{a |t|} $ for $ t \in \R $. Let $ \phi, \psi \in W^{2,2}\times W^{2,2}$, then  we have 
			\begin{align}
				- \frac{1}{2 \pi i} \int_{ b - iR}^{b + iR} e^{tz} \langle (i\mathcal{H} - z)^{-1} \phi, \psi \rangle~dz &=  \frac{1}{2 \pi i} \int_{ b - iR}^{b + iR} e^{tz} \int_0^{\infty} e^{- s z} \langle e^{i s \mathcal{H}} \phi, \psi \rangle~ds~dz\\ \nonumber
				&= \frac{1}{ \pi } \int_0^{\infty} e^{(t-s)b} \frac{\sin((t-s)R)}{(t-s)}  \langle e^{i s \mathcal{H}} \phi, \psi \rangle~ds.
			\end{align}
			The latter is a convolution product with the Dirichlet kernel  and hence  converges for $ t > 0$ as $ R \to \infty$ . In fact we have  
			\begin{align}
				- \frac{1}{2 \pi i} 	\lim_{R \to \infty}\int_{ b - iR}^{b + iR} e^{tz} \langle (i\mathcal{H} - z)^{-1} \phi, \psi \rangle~dz  = \begin{cases}
					\langle e^{i t \mathcal{H}}\phi, \psi\rangle & t > 0\\
					0 & t < 0.
				\end{cases}
			\end{align}
			Hence also
			\begin{align} \label{so}
				\langle &e^{i t \mathcal{H}}\phi, \psi\rangle\\ \nonumber
				&= - \frac{1}{2 \pi i}\lim_{R \to \infty} \int_{ b - iR}^{b + iR} e^{tz} \langle (i\mathcal{H} - z)^{-1} \phi, \psi \rangle~dz  +  \frac{1}{2 \pi i}\lim_{R \to \infty} \int_{ - b - iR}^{-b + iR} e^{tz} \langle (i\mathcal{H} - z)^{-1} \phi, \psi \rangle~dz\\ \nonumber
				& =  \frac{1}{2 \pi i}\lim_{R \to \infty} \int_{ -R}^{R} e^{i t\lambda} \langle e^{- bt} \big( (\mathcal{H} - (\lambda + i b))^{-1 } - (\mathcal{H} - (\lambda - i b))^{-1 } \big) \phi, \psi \rangle~d\lambda
			\end{align}
			Now we consider the rectangular contour $ \Gamma^+_{R, \delta}$  shown in figure \ref{fig:cont} below connecting the points $ \pm R + i b,\pm R + i0$ (by the latter we actually mean $\pm R + i\varepsilon$ and take the limit $ \varepsilon \to 0^+$ in the contour integral). Here we have to take out small semi circles of radius say $r_{+, R, \delta}\sim \delta$  centered at each isolated  eigenvalue in $(- \mu, \mu) $ where $ (\mathcal{H} -z)^{-1}$ has pole singularities. Reflecting the contour $ \Gamma^+_{R, \delta}$   at the real axis gives us the analogue $ \Gamma^-_{R, \delta}$   on the lower half plane and thus yields  full circles around each of such eigenvalues (and hence Riesz projections for the contour integral). The residue theorem states
			\[
			\frac{1}{2 \pi i}\oint_{\Gamma_{R, \delta}^{\pm}} e^{tz}(  \mathcal{H}-z)^{-1}~dz  = \sum_j \frac{1}{2 \pi i} \oint_{\gamma_j}e^{tz}( \mathcal{H} - z)^{-1}~dz,
			\]
			where $\gamma_j$ are small closed, simple contours enclosing $\zeta_j$ located in the rectangles. There now holds, see e.g. the argument in the proof of  \cite[Lemma 12]{Erd-Schlag}, 
			\[
			\sum_j\oint_{\gamma_j}e^{tz}(  \mathcal{H} - z)^{-1}~dz = \sum_je^{it \mathcal{H}}P_{\zeta_j}.
			\]
			Then the identity \eqref{final-abs-spect2} follows from letting $ R \to \infty$,i.e. we calculate the limit \eqref{so} and the corresponding limit at $ b = 0$. The latter provides the absolute part on the right side of \eqref{final-abs-spect}, where we integrate the jump cut for $ |\lambda|> \mu$ and the integral exactly cancels in between the circles on $(-\mu, \mu)$.  Finally, the integral over the vertical contours vanishes as $ R \to \infty$ seen from the decay provided by the limiting absorption principle \eqref{LAP3} in Lemma \ref{LAP}. 
			\begin{figure}[h!]
				\centering
				\begin{tikzpicture}[scale=1.2,extended line/.style={shorten >=-#1,shorten <=-#1},]
					\draw [->](0,0)--(0,2) node[right]{$$};
					\draw [<->](-3,0)--(3,0) node[right]{$$};
					\begin{scope}
						\clip (-.5,.05cm) rectangle (.5,.55 cm );
						\draw (0,.05 cm) circle(.1);
					\end{scope}
					\begin{scope}
						\clip (-1,.05cm) rectangle (0,.55 cm );
						\draw (-.5,.05 cm) circle(.1);
					\end{scope}
					\begin{scope}
						\clip (1,.05cm) rectangle (0,.55 cm );
						\draw (.5,.05 cm) circle(.1);
					\end{scope}
					
					\draw (-1.6, .05 cm)-- (-.6, 0. 05cm );
					\draw (-.4, .05 cm)-- (-.1, 0. 05cm );
					\draw (.1, .05 cm)-- (.4, 0. 05cm );
					\draw (.6, .05 cm)-- (1.6, 0. 05cm );
					\draw (-1.6, .05 cm)-- (-1.6, 1.2);
					\draw (-1.6, 1.2)-- (1.6, 1.2);
					\draw (1.6, .05 cm)-- (1.6, 1.2);
					\draw[fill] (1.6, .05 cm) circle(.4pt)node[below]{{ \tiny$i 0 + R$}};
					\draw[fill] (-1.6, .05 cm) circle(.4pt)node[below]{ \tiny $i 0 -R$};
					\draw[fill] (1.6,1.2) circle(.4pt)node[above]{\tiny$i b + R$};
					\draw[fill] (-1.6,1.2) circle(.4pt)node[above]{\tiny $ i b -R$};
					
					\draw[fill] (-.8, 0 cm) circle(.4pt)node[below]{{ \tiny$-\mu$}};
					\draw[fill] (.8, 0 cm) circle(.4pt)node[below]{{ \tiny$\mu$}};
					
				\end{tikzpicture}
				\label{fig:cont}
				\caption{The contour $\Gamma_{R, \delta}^+$}
			\end{figure}
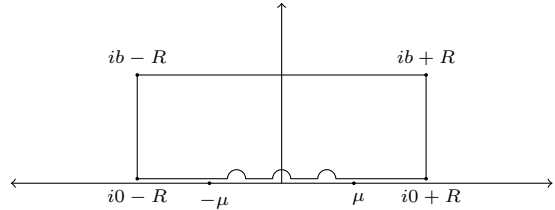
		\end{proof}
		Let us recall the potentials $V$ in Definition \ref{Def-on} satisfy  $\sigma_1 V \sigma_1 = -V$ and in particular $\sigma_1 \mathcal{H} \sigma_1 = - \mathcal{H}$. 
		\begin{Def}[Fourier base] Let $\mathcal{F}_+(x, \lambda), \mathcal{G}_+(x, \lambda)$ be as in Lemma \ref{Das-Lem}. Then we set 
			\begin{align}
				&\mathcal{F}_-(x, \lambda ) := \sigma_1 	\mathcal{F}_+(x, \lambda ) ,\\
				&\mathcal{G}_-(x, \lambda ) := \sigma_1 	\mathcal{G}_+(x,  \lambda ).
			\end{align}
			Moreover we define
			\begin{align}
				e_{\pm}(x, \lambda) = \begin{cases}
					\mathcal{F}_{\pm}(x, \lambda ) & \lambda \geq 0\\
					\mathcal{G}_{\pm}(x, -\lambda ) & \lambda < 0
				\end{cases}.\\ \nonumber
			\end{align}
		\end{Def}
		\begin{Rem}
			We note from the Definition there holds
			\begin{align}\label{inter}
				&e_{\pm}(-x, \lambda) = e_{\pm}(x, - \lambda),~~ x \in \R,~\lambda \in \R\backslash \{0\}.
			\end{align}
			Further in case $V = 0$ we have\\ 
				$$ e_{+}(x, \lambda) = e^{i \lambda x} \underline{e} = \begin{pmatrix} e^{i \lambda x}\\0 \end{pmatrix},~~e_{-}(x, \lambda) = e^{i \lambda x} \sigma_1 \underline{e} = \begin{pmatrix} 0\\ e^{i \lambda x} \end{pmatrix}.$$
			\end{Rem}
			\begin{Lemma} \label{Fourier-inv} Let  $\mathcal{H}$  be as in Definition \ref{Def-on} with $ \mu > 0$ and admissible.  Then we  have for $ \phi, \psi \in  \mathcal{S}(\R) \times \mathcal{S}(\R)$ 
				
				\begin{align}
					\langle P \phi, \psi \rangle  =& ~ \frac{1}{2 \pi}\int_{- \infty}^{\infty} \langle \phi, \sigma_3 e_+(\cdot, \lambda) \rangle \overline{\langle \psi, e_+(\cdot, \lambda) \rangle }~ d \lambda\\ \nonumber
					& + \frac{1}{2 \pi}\int_{- \infty}^{\infty}  \langle \phi, \sigma_3 e_-(\cdot, \lambda) \rangle \overline{\langle \psi, e_-(\cdot, \lambda) \rangle }~ d \lambda.
				\end{align}
				The integrals on the right side converge absolutely.
			\end{Lemma}
			\begin{proof} 	The proof is as in \cite[Prop. 6.9]{K-S-stable}. We start with the following expansion of the stable spectrum in the principal value sense.
				\begin{align}\label{final-abs-spect3}
					\langle P \phi, \psi\rangle =&~~  \frac{1}{2 \pi i}\bigg(\int_{\mu}^{\infty} + \int_{ - \infty}^{- \mu}\bigg) \langle \big[ (\mathcal{H} - (\lambda + i 0))^{-1} -(\mathcal{H} - (\lambda - i 0))^{-1}  \big] \phi, \psi\rangle~d \lambda\\ \nonumber
					=&~~\frac{1}{2 \pi i}\int_{0}^{\infty}  2 \lambda \langle \big[ (\mathcal{H} - (\lambda^2 + \mu + i 0))^{-1} -(\mathcal{H} - (\lambda^2 + \mu - i 0))^{-1}  \big] \phi, \psi\rangle~d \lambda\\ \nonumber
					&~+~\frac{1}{2 \pi i}\int_{0}^{\infty}  2 \lambda \langle \big[ (\mathcal{H} - (-\lambda^2 - \mu + i 0))^{-1} -(\mathcal{H} - (-\lambda^2 - \mu - i 0))^{-1}  \big] \phi, \psi\rangle~d \lambda
				\end{align}
				In fact, this is obtained  as in the proof of Lemma \ref{spectral-repr} by calculation of
				\begin{align}
					\f{1}{2 i\pi } \oint_{\Gamma^{\pm}_R} ( \mathcal{H} - z)^{-1}~dz.
				\end{align}
				The only difference is the contour line connecting $ \pm i b -R $ and $ \pm ib +R $ respectively, where for $ b \gg1 $ large ($ b > a $ suffices)  we use
				\begin{align}
					\lim_{R \to \infty}\frac{1}{\pi}\int_0^{\infty} e^{-s b}~ \frac{\sin(s R)}{s} \langle e^{i s \mathcal{H} }\phi, \psi \rangle ~ds = \frac{1}{2}
				\end{align}
				by dominated convergence. Lemma \ref{densit2} implies (simply by definition)
				\begin{align}\nonumber
					&(\mathcal{H} - (\lambda^2 + \mu + i 0))^{-1}(x,y) -(\mathcal{H} - (\lambda^2 + \mu - i 0))^{-1}(x,y)\\ \label{des1}
					&~~= - \frac{1}{2 \lambda i} [e_+(x,\lambda), e_+(x, - \lambda)] \cdot [e_+(y,\lambda), e_+(y, - \lambda)]^* \sigma_3,\\[5pt] \nonumber
					&(\mathcal{H} - (-\lambda^2 - \mu + i 0))^{-1}(x,y) -(\mathcal{H} - (-\lambda^2 - \mu - i 0))^{-1}(x,y)\\ \label{des2}
					&~~= - \frac{1}{2 \lambda i} [e_-(x,\lambda), e_-(x, - \lambda)] \cdot [e_-(y,\lambda), e_-(y, - \lambda)]^* \sigma_3,  
				\end{align}
				where \eqref{des1}, \eqref{des2} are expressed with
				$ \mathcal{E}_+(x,\lambda) = \mathcal{E}(x,\lambda),~\mathcal{E}_-(x,\lambda) = \sigma_1\mathcal{E}_+(x,\lambda)$
				in the notation of Lemma \ref{densit2}. 
				Then the integrals  in  \eqref{final-abs-spect3} are written into
				\begin{align*}
					&\frac{1}{2 \pi} \int_0^{\infty} \langle \mathcal{E}_+^*(y, \lambda)\sigma_3 \phi(y), \mathcal{E}_+^*(x, \lambda) \psi(x) \rangle d \lambda\\
					&~~~~~~~~~~~~~~+ \frac{1}{2 \pi} \int_0^{\infty} \langle \mathcal{E}_-^*(y, \lambda)\sigma_3 \phi(y), \mathcal{E}_-^*(x, \lambda) \psi(x) \rangle~ d \lambda\\
					&= \frac{1}{2 \pi} \int_{- \infty}^{\infty} \langle  \phi, \sigma_3 e_+(\cdot , \lambda)\rangle \overline{ \langle  \psi, e_+(\cdot , \lambda) \rangle } d \lambda + \frac{1}{2 \pi} \int_{- \infty}^{\infty} \langle  \phi, \sigma_3 e_-(\cdot , \lambda)\rangle \overline{ \langle  \psi, e_-(\cdot , \lambda) \rangle }~ d \lambda.
				\end{align*}
				
			\end{proof}	
			~~\\
			By the proof of Lemma \ref{spectral-repr} and Lemma \ref{Fourier-inv} likewise imply the following  for $ e^{i t \mathcal{H}}$. 
			\begin{Corollary}\label{Fourier-inv-cor} Let $ \mathcal{H}$ be as in Lemma \ref{Fourier-inv}. Then there  holds for $ \phi, \psi \in \mathcal{S}(\R)$
				
				\begin{align}
					\langle e^{ i t\mathcal{H}}P \phi, \psi \rangle  = & ~ \frac{e^{i t \mu}}{2 \pi}\int_{- \infty}^{\infty} e^{it\lambda^2} \langle \phi, \sigma_3 e_+(\cdot, \lambda) \rangle \overline{\langle \psi, e_+(\cdot, \lambda) \rangle }~ d \lambda\\ \nonumber
					& + \frac{e^{- i t \mu}}{2 \pi}\int_{- \infty}^{\infty} e^{-it\lambda^2} \langle \phi, \sigma_3 e_-(\cdot, \lambda) \rangle \overline{\langle \psi, e_-(\cdot, \lambda) \rangle }~ d \lambda.\\ \nonumber
				\end{align}
			\end{Corollary}
			\begin{Rem}
				Clearly the expressions in Lemma \ref{Fourier-inv} and Corollary \ref{Fourier-inv-cor} are well defined. In fact  under reasonable assumptions on $f : \R \to \R$, e.g. $f(x)$ continuous with polynomial upper bounds as $ |x| \to \infty$, we have 
				
				\begin{align*}
					(\phi, \psi) \mapsto \frac{1}{2 \pi}\int_{- \infty}^{\infty} f(\pm\lambda^2 \pm \mu) \langle \phi, \sigma_3 e_{\pm}(\cdot, \lambda) \rangle \overline{\langle \psi, e_{\pm}(\cdot, \lambda) \rangle }~ d \lambda,~ \phi, \psi \in \mathcal{S}(\R)
				\end{align*} 
				are well defined functionals. Recall $ \sigma_3 \mathcal{H} = \mathcal{H}^* \sigma_3 $, then for instance
				~~\\
				\begin{align*}
					\langle \phi, \sigma_3 e_{\pm}(\cdot, \lambda) \rangle (\mu + \lambda^2)^{m} = \langle \phi, \sigma_3 \mathcal{H}^me_{\pm}(\cdot, \lambda) \rangle = \langle \mathcal{H}^m \phi, \sigma_3 e_{\pm}(\cdot, \lambda) \rangle.\\ 
				\end{align*}
			\end{Rem}
			\begin{Rem} \label{Das-Rema}(a)~ Let $ X \subset L^2(\R) $ be a proper closed $ \mathcal{H}$-invariant subspace, that is $ \mathcal{H} (X) \subset X$ with $ \sigma_{\text{ess}}(\mathcal{H}_{\vert X}) = ( - \infty, - \mu] \cap [\mu, \infty)$. Then the statement of Lemma \ref{spectral-repr} holds true for  $ \mathcal{H}_{\vert X}$ via \eqref{final-abs-spect2} with $ \psi, \phi  \in W^{2,2}(\R) \cap X$ if $ \mathcal{H}_{\vert X}$ is admissible. We apply this in the following for $ X= L^2_{\text{odd}}(\R) = \overline{\mathcal{S}_{odd}(\R)}$ the closure of odd Schwartz functions.\\[5pt]
				(b)~If in Lemma \ref{Fourier-inv} and Corollary \ref{Fourier-inv-cor} the operator $ \mathcal{H}_{\vert X}$ with $ X = L^2_{\text{odd}}(\R)$ is admissible, then both expressions hold for all $ \phi, \psi \in \mathcal{S}_{\text{odd}}(\R)$. This is seen from writing
				\begin{align*}
					(\mathcal{H}_{\vert X} - (\lambda^2 + \mu + i 0))^{-1}(x,y)  - (\mathcal{H}_{\vert X} - (\lambda^2 + \mu - i 0))^{-1}(x,y)  = - \frac{1}{2 i \lambda}\tilde{\mathcal{E}}(x, \lambda) \mathcal{E}^{*}(y, \lambda) \sigma_3
				\end{align*}
				for $ \lambda \geq 0$ and where
				\begin{align}
					\tilde{\mathcal{E}}(x, \lambda) = 	\f12(\mathcal{E}(x, \lambda) - 	\mathcal{E}(- x, \lambda)),~ ~x \in \R,~~ \lambda \in \R.
				\end{align}
				Following the proof of Lemma \ref{Fourier-inv} we obtain the desired formula in an absolute sense.
			\end{Rem}
			~~\\
			\emph{The linearized NLS operator}. We now want to extend the Fourier inversion in  Lemma \ref{Fourier-inv} and Corollary \ref{Fourier-inv-cor} to the operator \eqref{system-final-final-R-operator}. Let  $V(x)$ as Definition \ref{Def-on}, i.e. we consider 
			\begin{align} \label{lin-po}
				&H_0 = - \partial_x^2 \sigma_3 =   \begin{pmatrix}
					- \partial_x^2 & 0\\
					0& - \partial_x^2 
				\end{pmatrix},~~~~~ V(x) = \begin{pmatrix}
					V_1(x) & V_2(x)\\
					- V_2(x) & -V_1(x)
				\end{pmatrix},~~~~ 
				\\[8pt] \nonumber
				& \mathcal{H} = \mathcal{H}_0 + V(x),~~~D(\mathcal{H}) = W^{2,2}_{\text{odd}}(\R, \C^2) \subset L^2_{\text{odd}}(\R, \C^2),
			\end{align}
			where
			\begin{align}\label{poli}
				V_1(x) = -3 W^4(x),~~ V_2(x) = -2 W^4(x),~~ W(x) = (1 + \frac{x^2}{3})^{-\f12}.
			\end{align}
			We need to use the the following Lemma, which we prove in the subsequent Section.
			\begin{Lemma} \label{Ass} Let $ V(x)$ be as above. Then  $ \mathcal{H} = - \partial_x^2 \sigma_3 + V(x)$ has no real eigenvalues and for small $ 0 < \mu \ll1 $ the operators $ \mathcal{H}_{\mu} = (- \partial_x^2 + \mu)\sigma_3 + V(x)$ are admissible with finite discrete spectrum. All operators are restricted to the subspace of odd functions $ L^2_{\text{odd}}(\R)$.
			\end{Lemma}
			~~\\
			Let us first collect properties of the Fourier base implied by  Lemma \ref{LemmaJost-decay} - \ref{LemmaJost-oscillating} and Lemma \ref{Das-asym}. Therefore let
			\[
			\underline{e}^+ = \underline{e} = \begin{pmatrix}
				1\\[1pt] 0
			\end{pmatrix},~ \underline{e}^- = \sigma_1 \underline{e} = \begin{pmatrix}
				0\\[1pt] 1
			\end{pmatrix}.
			\]
			We observe  (cf. \cite[Section 4.2]{OP})
			\begin{Corollary} \label{Das-Cor}  For $ \lambda \geq 0 $ we have 
				
				\begin{align} e_{\pm}(y, \lambda) = 
					\begin{cases}
						s(\lambda) e^{i \lambda y} \underline{e}^{\pm} + e^{\infty}_{\pm}(y, \lambda), &  y> 0,\\[5pt]
						[r(\lambda) e^{-i \lambda y} + e^{i \lambda y}] \underline{e}^{\pm} + e^{-\infty}_{\pm}(y, \lambda), & y < 0,
					\end{cases}
				\end{align}
				
				such that there holds
				\begin{align}
					& |e^{\pm \infty}(y, \lambda)| \leq C (\langle \lambda \rangle^{-1}\langle y \rangle^{-2} + 
					|\lambda|\langle \lambda \rangle^{-1} |\log(2\lambda \langle \lambda \rangle^{-1})| e^{\mp |\lambda| y }),\\
					&  |  \partial_y e^{\pm \infty}(y \lambda) | \leq C( \langle y \rangle^{-3} +|\lambda|\langle \lambda \rangle^{-1} |\log(2\lambda \langle \lambda \rangle^{-1})| e^{\mp |\lambda| y }(|\lambda| + \langle  y \rangle^{-3})),\\
					&|\partial_{\lambda}e^{\pm \infty}(y, \lambda)| \leq C \langle \lambda \rangle^{-1} |\log(2\lambda \langle \lambda \rangle^{-1})|(\langle y \rangle^{-1} + e^{\mp |\lambda| \frac{y}{2} }),\\[3pt]
					& |\partial_{\lambda}^2e^{\pm \infty}(y, \lambda)| \leq C |\lambda|^{-1}|\log(|\lambda|)|,~~|\lambda|  \lesssim 1.
				\end{align}
				Further 
				\begin{align}
					& \|e^{\pm \infty}(y, \lambda)\|_{L^2_y} \leq C ,~~~
					\|\partial_{\lambda}e^{\pm \infty}(y, \lambda)\|_{L^2_y} \leq C 
					|\lambda|^{- \f12} |\log(|\lambda|)|.
				\end{align}
			\end{Corollary}
			\begin{proof} We inspect the asymptotic expansion in Lemma \ref{Das-Lem} and use Lemma \ref{LemmaJost-decay}, Lemma \ref{LemmaJost-oscillating}. Further we need the asymptotics of $s(\lambda), r(\lambda)$ in Lemma \ref{Das-asym}. Then, in particular for $ |\lambda| \geq \lambda_0 \gtrsim 1 $ large, we obtain the estimates directly for  all $ y \geq 0$ and some constant $C > 0$.  Now for $|\lambda| \leq \lambda_0$ we may need to choose $y \geq x_0(\lambda) \gg1$, however the exceptional set is then compact in $(y, \lambda)$ and we only need to enlarge $C > 0$.
			\end{proof}
			\begin{Rem} \label{subs-Rem}
				In particular for $ \lambda < 0$  we have by \eqref{inter} 
				
				\begin{align} e_{\pm}(y, \lambda) = 
					\begin{cases}
						[\overline{r(\lambda) } e^{-i \lambda y} + e^{i \lambda y}]  \underline{e}^{\pm} + e^{-\infty}_{\pm}(-y,- \lambda), & y > 0,\\[5pt]
						\overline{s(\lambda)} e^{i \lambda y} \underline{e}^{\pm} + e^{\infty}_{\pm}(-y, -\lambda), & y < 0,
					\end{cases}
				\end{align}
				
			\end{Rem}
			~~\\
			\begin{Corollary} \label{Das-Cor2} For $ \lambda \geq 0 $ there holds
				
				\begin{align} &( \lambda \partial_{\lambda} - y \partial_y)e_{\pm}(y, \lambda) = 
					\begin{cases}
						[\lambda \partial_{\lambda}s(\lambda) ]e^{i \lambda y} \underline{e}^{\pm} + ( \lambda \partial_{\lambda} - y \partial_y)e^{\infty}_{\pm}(y, \lambda), &  y> 0,\\[5pt]
						[\lambda \partial_{\lambda}r(\lambda) ]e^{-i \lambda y} \underline{e}^{\pm} + ( \lambda \partial_{\lambda} - y \partial_y)e^{-\infty}_{\pm}(y, \lambda), & y < 0,
					\end{cases}
				\end{align}
				
				and for $\tilde{e}^{\pm\infty} = ( \lambda \partial_{\lambda} - y \partial_y)e^{\pm \infty}$ the estimate 
				\begin{align}
					& |\tilde{e}^{\pm \infty}(y, \lambda)| \leq C (\langle y \rangle^{-1} + |\lambda|\langle \lambda \rangle^{-1} |\log(2\lambda \langle \lambda \rangle^{-1})|e^{\mp |\lambda| \frac{y}{2} }),\\
					& |\partial_y \tilde{e}^{\pm \infty}(y, \lambda)| \leq C (\langle y \rangle^{-2} + |\lambda|\langle \lambda \rangle^{-1} |\log(2\lambda \langle \lambda \rangle^{-1})| e^{\mp |\lambda| \frac{y}{2} }).
				\end{align}
				
			\end{Corollary}
			\begin{Rem}
				For  $\lambda \geq 0$ there holds for second order derivatives
				
				\begin{align*}
					&( \lambda^2\partial_{\lambda}^2 - y^2\partial_y^2)e_{\pm}(y, \lambda) = 
					\begin{cases}
						e^{i \lambda y} [\lambda^2 \partial_{\lambda}^2+2 \lambda^2 \partial_{\lambda} i y] s(\lambda) \underline{e}^{\pm} + \tilde{e}^{\infty}_{\pm}(y, \lambda), &  y> 0,\\[5pt]
						e^{-i \lambda y}[\lambda^2 \partial_{\lambda}^2- 2 \lambda^2 \partial_{\lambda} i y] r(\lambda)  \underline{e}^{\pm} + \tilde{e}^{-\infty}_{\pm}(y, \lambda), & y < 0,
					\end{cases},
				\end{align*}
				
				where $ \tilde{e}_{\pm} = ( \lambda^2 \partial_{\lambda}^2 - y^2 \partial_y^2)e_{\pm}$ and 
				\begin{align*}
					& | \tilde{e}^{\pm \infty}(y, \lambda)| \leq C |\lambda \log(|\lambda|)|,~~0< |\lambda| \ll1.
				\end{align*}
			\end{Rem}
			Finally we need to state the following Corollary.
			\begin{Corollary} \label{proj-deponm}The operators $ \mathcal{H} = (- \partial_x^2 + \mu)\sigma_3 +V$ from Lemma \ref{Ass}  have discrete spectrum $\sigma_d(\mathcal{H})$ with continuous $\mu$ dependence and the corresponding Riesz projections 
				\[
				P_{\zeta(\mu)} \to P_{\zeta(0)},~~~ \zeta(\mu) \in \sigma_d(\mathcal{H}),~~\text{as}~ \mu \to 0^+
				\] 
				at least in the strong sense. Further $ e^{it\mathcal{H}}$ converges in the strong sense locally uniform in $t$.
			\end{Corollary}
			For the proof we note finite systems of eigenvalues have continuous dependence on $\mu$ in the perturbation $\mathcal{H} = \tilde{\mathcal{H}} + \mu \sigma_3$, see e.g. \cite[Chapter 7, I.3 ]{Kato}. The convergence of $P_{\zeta}$ follows from \eqref{Riesz} and the strong resolvent convergence 
			In fact the isolated eigenvalues are stable if $0 < \mu \ll1 $ is small enough (and thus $P_{\zeta}$ converge in norm). The convergence of $ e^{it\mathcal{H}}$ is likewise obtained from strong resolvent convergence in a contour integral as in the proof of Lemma \ref{spectral-repr}.
	\begin{Prop}\label{Fourier-inv-fin} Let $ \mathcal{H} = - \partial_x^2 \sigma_3 + V(x)$ with $V$ as in \eqref{lin-po}, \eqref{poli}  and $ P_c = I - P_d$ the associated projection. Then for $ \phi, \psi \in \mathcal{S}_{\text{odd}}(\R)$
		\begin{align}\label{FI-fin}
			\langle P_c \phi, \psi \rangle  =& ~ \frac{1}{2 \pi}\int_{- \infty}^{\infty} \langle \phi, \sigma_3 e_+(\cdot, \lambda) \rangle \overline{\langle \psi, e_+(\cdot, \lambda) \rangle }~ d \lambda\\ \nonumber
			& + \frac{1}{2 \pi}\int_{- \infty}^{\infty}  \langle \phi, \sigma_3 e_-(\cdot, \lambda) \rangle \overline{\langle \psi, e_-(\cdot, \lambda) \rangle }~ d \lambda.
		\end{align}
		Further there holds	
		\begin{align}\label{spec-fin}
			\langle e^{ i t\mathcal{H}}P_c \phi, \psi \rangle  = & ~ \frac{1}{2 \pi}\int_{- \infty}^{\infty} e^{it\lambda^2} \langle \phi, \sigma_3 e_+(\cdot, \lambda) \rangle \overline{\langle \psi, e_+(\cdot, \lambda) \rangle }~ d \lambda\\ \nonumber
			& + \frac{1}{2 \pi}\int_{- \infty}^{\infty} e^{-it\lambda^2} \langle \phi, \sigma_3 e_-(\cdot, \lambda) \rangle \overline{\langle \psi, e_-(\cdot, \lambda) \rangle }~ d \lambda. 
		\end{align}
		Here all integrals are well defined and $e_{\pm}(x, \lambda)$ are defined as above for $ \mu = 0$.
	\end{Prop}
	\begin{proof} 
		By Lemma \ref{Ass}, Lemma \ref{Fourier-inv} and Corollary \ref{Fourier-inv-cor} (note the Remark \ref{Das-Rema}),  we obtain \eqref{FI-fin} and  \eqref{spec-fin} for the respective admissible operators $\mathcal{H}^{\mu} = (- \partial_x^2 + \mu) \sigma_3 + V(x) $ where $ 0 < \mu \ll1 $ is taken small enough. Further,  the $\mu$-dependence of  associated Jost solutions $ f_j(y, \lambda, \mu)$ in the Lemma \ref{LemmaJost-decay}, \ref{LemmaJost-oscillating}, \ref{LemmaJost-exponential-growth}, of $2i \lambda D(\lambda)^{-1}$ and in particular of the coefficients $s(\lambda), r(\lambda)$ is smooth for fixed $ (y, \lambda) \in \R \times \R \backslash \{ 0\}$. 
		The left side of \eqref{FI-fin}, \eqref{spec-fin}  converges as $ \mu \to 0^+$  by Corollary \ref{proj-deponm} and we  evaluate the right side by dominated convergence (first in $ y$, then in $\lambda$). For \eqref{FI-fin}, say,  we split the integrand into
		\begin{align*}
			\overline{\langle \psi, e_{\pm}(\cdot, \lambda) \rangle}_{L^2_y}  =& \int_0^{\infty} e^{ i \lambda y}  \langle \overline{\psi}(y), \underline{e}^{\pm} \rangle_{\C^2} dy~ s(\lambda) + \int_0^{\infty}  \overline{\psi}(y) \cdot e^{\infty}_{\pm}(y, \lambda)~dy\\
			&+  \int_{-\infty}^0  \langle \overline{\psi}(y), \underline{e}^{\pm} \rangle_{\C^2}[e^{- i \lambda y} r(\lambda) + e^{i \lambda y}] dy + \int_{-\infty}^0  \overline{\psi}(y) \cdot e^{-\infty}_{\pm}(y, \lambda)~dy\\
			\langle \phi, \sigma_3e_{\pm}(\cdot, \lambda) \rangle_{L^2_y}  =& \int_0^{\infty} e^{ -i \lambda y}  \langle \phi(y), \sigma_3 \underline{e}^{\pm} \rangle_{\C^2} dy~ \overline{s(\lambda)} + \int_0^{\infty}  \phi(y) \cdot \sigma_3\overline{ e^{\infty}_{\pm}(y, \lambda)}~dy\\
			&+  \int_{-\infty}^0  \langle \phi(y), \sigma_3\underline{e}^{\pm} \rangle_{\C^2}[e^{i \lambda y} \overline{r(\lambda)} + e^{-i \lambda y}] dy + \int_{-\infty}^0  \phi(y) \cdot \sigma_3 \overline{e^{-\infty}_{\pm}(y, \lambda)}.
			~dy
		\end{align*}
		For fixed $ \lambda \in \R \backslash \{ 0 \}$, the integrals of the form
		\[
		\int_0^{\infty} f(y) \cdot e^{\infty}_{\pm}(y, \lambda)~dy,~ \int_{-\infty}^0 f(y) \cdot e^{-\infty}_{\pm}(y, \lambda)~dy,
		\]
		converge as $\mu \to 0^+$ by Corollary \ref{Das-Cor} (and Remark \ref{subs-Rem}). Now  again with Corollary \ref{Das-Cor}, integration by parts and Lemma \ref{Das-asym-fomu} all of the above integrals are in $ O (\langle \lambda \rangle^{-1})$ as $ |\lambda| \to \infty$. Hence we obtain (uniformly in $ \mu$)
		\[
		\overline{\langle \psi, e_{\pm}(\cdot, \lambda) \rangle}_{L^2_y},~ \langle \phi, \sigma_3e_{\pm}(\cdot, \lambda) \rangle_{L^2_y}  \in O(\langle \lambda \rangle^{-2}),~~|\lambda| \to \infty,
		\]
		by which we conclude the result. The proof of \eqref{spec-fin} follows similarly.
	\end{proof}
	\begin{Corollary}[$L^2$ stability]  Let $V$ and $\mathcal{H}$ be as in Proposition \ref{Fourier-inv-fin}. Then there holds
		\begin{align}
			\sup_{t \in \R} \| e^{i t \mathcal{H}}  P_c \|_{2 \to 2} \leq C.
		\end{align} 
	\end{Corollary}
	\begin{proof} This follows as in the proof  \cite[Lemma 6.11 ]{K-S-stable} immediately from Proposition \ref{Fourier-inv-fin} by means of Corollary \ref{Das-Cor} and
		\begin{align*}
			| \langle  e^{i t \mathcal{H}}\phi, \psi \rangle | \leq \max_{\pm}\big( \int_{- \infty}^{\infty} | \langle \phi, \sigma_3 e_{\pm}(\cdot, \lambda)\rangle_{L^2}|^2~d \lambda\big)^{\f12} \cdot \max_{\pm} \big( \int_{- \infty}^{\infty} | \langle \psi, \sigma_3 e_{\pm}(\cdot, \lambda)\rangle_{L^2}|^2~d \lambda\big)^{\f12}.
		\end{align*}
	\end{proof}
	\subsection{The transference identity}\label{subsec:tranf}
	We now clarify the explicit framework of the distorted Fourier side for $ \mu = 0$ in the identity  \eqref{FI-fin} of Proposition \ref{Fourier-inv-fin}. In this section we further use the following Lemma, which is directly implied by the spectral properties of the $3$D radial operator discussed in Section \ref{sec:spec}.
	\begin{Lemma}
		The discrete spectrum  $ \sigma_d(\mathcal{H}) = \{ \pm i \kappa \}$  for $ \kappa > 0 $ consists of two simple eigenvalues with odd eigenfunctions $\{ \phi^{\pm}_d(y)\}$, i.e.
		$
		\mathcal{H}\phi^{\pm}_d = \pm i \kappa \phi^{\pm}_d
		$. 
		Note in particular $ \phi_d^{\pm}(y)$ are Schwartz functions and $ \phi^+_d(y) = \overline{\phi^-_d(y)} = \sigma_1 \phi^-_d(y) $ since we require $ \| \phi_d^{\pm} \|_{L^2} =1$.
	\end{Lemma}
	~~\\ Let us now define the Fourier transform associated to the operator $\mathcal{H} $.
	\begin{Def}\label{Lemma-FT}
		We let the distorted Fourier transform on sufficiently decaying functions be defined via the map
		\begin{align}\label{FT}
			&\hat{f}( i \kappa) : = \int_{\R} f(y) \overline{ \phi_d^{+}(y)}~dy,~~\hat{f}(- i \kappa) : = \int_{\R} f(y) \overline{ \phi_d^{-}(y)}~dy\\[1pt] \nonumber
			& \hat{f}(\lambda) : =(\hat{f}^+(\lambda), \hat{f}^-(\lambda)) := \big( \int_{\R} f(y) \sigma_3 \overline{ e_+(y, \lambda)}~dy, \int_{\R} f(y) \sigma_3 \overline{ e_-(y, \lambda)}~dy\big),~\lambda \in \R.
		\end{align}
		Further we set 
		\begin{align} \label{inv-FT}
			&\mathcal{F}^{-1}(g)(y) := ~~\phi^+_d(y) g(i \kappa) +   \phi^-_d(y) g(-i \kappa)\\ \nonumber
			&~~~~~~~~~~~~~~~ + \lim_{R \to \infty} \frac{1}{2 \pi}\int_{|\lambda| \leq R } g^+(\lambda) e_+(y, \lambda) ~d \lambda +  \lim_{R \to \infty} \frac{1}{2 \pi}\int_{|\lambda| \leq R } g^-(\lambda) e_-(y, \lambda) ~d \lambda,
		\end{align}
		where $ g (\lambda) = (g^+(\lambda), g^{-}(\lambda))$ is $\C^2$ valued  and the above limits in \eqref{inv-FT} are to 
		be understood in an absolute sense. 
	\end{Def}
	By Corollary \ref{Das-Cor} we may view  $\hat{f}(\lambda),  \mathcal{F}^{-1}(g)$   up to error terms as a free Fourier multiplier, with symbols depending only  on $ s(\lambda), r(\lambda)$. Further we note $\hat{f}^{\pm}(\lambda)$ are discontinuous at $ \lambda = 0$ unless $ f $ is an odd function.\\[2pt]
	In fact we state  the following simple Lemma on the decay of the Fourier transform.
	\begin{Lemma} \label{Fourier-decay}
		$(a)$ The Fourier transform decays rapidly, i.e. $\hat{f}(\lambda)$ satisfies  
		$$ |\partial_{\lambda}^k\hat{f}^{\pm}(\lambda)| \leq C_N \langle \lambda \rangle^{-N},~~~ |\lambda| \gg1 $$
		for $ k = 0,1$ and all $ N \in \Z_+$,  
		if $ f \in \mathcal{S}(\R, \C^2)$ is in the Schwartz class. Further if $ f$ is an odd function, then $\hat{f}^{\pm}(\lambda)$ is  even and $\hat{f}^{\pm}(0) = 0$ . There holds
		\begin{align}
			& |\partial_{\lambda} \hat{f}^{\pm}(\lambda)| \lesssim |\log(\lambda)|,~~~ 0< \lambda \ll1,\\
			&|\partial_{\lambda}^2 \hat{f}^{\pm}(\lambda)| \lesssim \lambda^{-1} |\log(\lambda)|,~~~ 0< \lambda \ll1.
		\end{align}
		$(b)$ Let $ g \in  C^1_0(\R \backslash \{0\}, \C)$ be sufficiently decaying, say
		\begin{align*}
			&|\partial^k_{\lambda}g(\lambda)| \leq C_N \langle  \lambda \rangle^{-2-k},~~|\lambda | \gg1 ,~ k = 0,1
		\end{align*}
		and such that $ g_{\pm}(0) = \lim_{ \lambda \to 0^{\pm} } g(\lambda) $ exist.
	Then 
	$$ \mathcal{F}^{-1}(g)^{\pm}(y) = \int_{- \infty}^{\infty} g(\lambda) e_{\pm}(y, \lambda)~ d \lambda$$
	satisfies $ \mathcal{F}^{-1}(g)^{\pm}(y)= O(\langle y\rangle^{-1} )$ as $ |y| \to \infty$.
	If in addition $g \in  C^2_c(\R \backslash \{0\}, \C)$ and $g(\lambda) = O_2(\lambda) $ 
	as $ \lambda  \to 0$, then  $ \mathcal{F}^{-1}(g)^{\pm}(y) = O(\langle y\rangle^{-2} )$ as $ |y| \to \infty$. Further if $g(\lambda)$ is an even function, then $ \mathcal{F}^{-1}(g)^{\pm}(y) $ are odd.
\end{Lemma}

\begin{proof}
	For part $(a)$ we use  the duality  $\sigma_3\mathcal{H} = \mathcal{H}^* \sigma_3$ pointwise, i.e.
	\begin{align}\label{pw}
		f(y) \sigma_3 \overline{ \mathcal{H}e_{\pm}(y, \lambda) } =&~ \mathcal{H}f(y) \sigma_3 \overline{ e_{\pm}(y, \lambda) } - \partial_y \big( f(y) \cdot \sigma_3 \partial_y\overline{ e_{\pm}(y, \lambda) } \big)\\ \nonumber
		& +  \partial_y \big( \partial_yf(y) \cdot \sigma_3 \overline{ e_{\pm}(y, \lambda) } \big).
	\end{align}
	Integrating \eqref{pw} we infer
	\begin{align*}
		\int f(y) \sigma_3 \overline{ e_{\pm}(y, \lambda) } ~dy = \frac{1}{\pm \lambda^2} \int f(y) \sigma_3 \overline{ \mathcal{H}e_{\pm}(y, \lambda) } ~dy   = \frac{1}{\pm \lambda^2} \int  \mathcal{H}f(y) \sigma_3 \overline{e_{\pm}(y, \lambda) } ~dy,
	\end{align*}
	and thus iterating this step,
	\begin{align*}
		\int f(y) \sigma_3 \overline{ e_{\pm}(y, \lambda) } ~dy =& ~ (\pm 1)^k\lambda^{-2k}\int \mathcal{H}^k f(y) \sigma_3 \overline{ e_{\pm}(y, \lambda) } ~dy.
	\end{align*}
	Further for the derivatives, we use
	\begin{align*}
		\lambda^{2k +1}\partial_{\lambda} \hat{f}^{\pm}(\lambda) =&~ (\pm 1)^k \int \mathcal{H}^k f (y) \sigma_3 y \partial_y \overline{ e_{\pm}(y , \lambda)}~ dy + (\pm 1)^k \int \mathcal{H}^k f (y) \sigma_3 ( \lambda \partial_{\lambda}  - y \partial_y ) \overline{ e_{\pm}(y , \lambda)}~ dy\\
		&~ - 2k  (\pm 1)^k \int \mathcal{H}^k f (y) \sigma_3 \overline{ e_{\pm}(y , \lambda)}~ dy.
	\end{align*}
	For the first term, we thus integrate by parts and by Corollary \ref{Das-Cor} in the second integral we infer the claim. 
	In order to verify the claim in $(b)$ we split
	\[
	\int_{-\infty}^{\infty} g(\lambda) e_{\pm}(y, \lambda)~d\lambda = \int_{0}^{\infty} g(\lambda) e_{\pm}(y, \lambda)~d\lambda + \int_{- \infty}^{0} g(\lambda) e_{\pm}(y, \lambda)~d\lambda
	\]
	For $\lambda > 0$, we additionally take wlog  $ y \gg1 $ large (the case of $ - y \gg1 $ is implied analogously using Corollary \ref{Das-Cor}).
	Now by Corollary \ref{Das-Cor}  we  further split
	\[
	\int_{0}^{\infty} g(\lambda) e_{\pm}(y, \lambda)~d\lambda = \underline{e}^{\pm}\int_{-\infty}^{\infty} \chi(\lambda > 0 )g(\lambda) s(\lambda) e^{i \lambda y}~d\lambda + \int_{0}^{\infty} g(\lambda) e^{\infty}_{\pm}(y, \lambda)~d\lambda =: \underline{e}^{\pm}I_1(y) + I_2(y).	
	\]
	and again by Corollary \ref{Das-Cor}, we have  $ I_2(y) = O(y^{-2})$ as $ y \to \infty$. Since $|s'(\lambda)| \lesssim |\log(\lambda)|$ for $ 0 < \lambda \ll1 $ we may integrate by parts using $s(0) = -1$
	\begin{align*}
		I_1(y) =  - (iy)^{-1}\big[ g_+(0) + \int_{0}^{\infty} \big(g s\big)'(\lambda) e^{i \lambda y}d\lambda\big].
	\end{align*}
	Thus in case $ g(\lambda) = O_2(\lambda) $ we iterate this for the latter integral.
	For  $\lambda < 0$ and $ y \gg1 $ we likewise split the integral
	\[
	\int_{-\infty}^{0} g(\lambda) e_{\pm}(y, \lambda)~d\lambda = \underline{e}^{\pm}\int_{-\infty}^{0}g(\lambda) \big[\overline{r(\lambda)} e^{i \lambda y} + e^{- i \lambda y}\big]~d\lambda + \int_{- \infty}^{0} g(\lambda) e^{-\infty}_{\pm}(-y, -\lambda)~d\lambda.	
	\]
	
\end{proof}

\begin{Corollary} \label{Cor-Gourier-decay}Let $ g \in C^2_c(\R, \C)$.  
	Then the functions
	\begin{align}\label{decyy1}
		&\lambda \mapsto  \langle \int_{-\infty}^{\infty} g(\tilde{\lambda}) e_{\pm }(y, \tilde{\lambda})~d \tilde{\lambda}, \sigma_3 e_{+}(y, \lambda) \rangle_{L^2_y},\\ \label{decyy2}
		&\lambda \mapsto  \langle \int_{-\infty}^{\infty} g(\tilde{\lambda}) e_{\pm }(y, \tilde{\lambda})~d \tilde{\lambda}, \sigma_3 e_{-}(y, \lambda) \rangle_{L^2_y},~~\lambda \in \R,
	\end{align}
	are well defined and decay of arbitrary order, i.e.  
	$O(|\lambda|^{-N } )$ with $ N \in \Z_+$ as $ |\lambda| \to \infty$. Moreover 
	$$  \lambda \mapsto \langle \int_{-\infty}^{\infty} g(\tilde{\lambda}) y \partial_ye_{\pm }(y, \tilde{\lambda})~d \tilde{\lambda}, \sigma_3 e_{\pm}(y, \lambda) \rangle_{L^2_y} = O(|\lambda|^{-2}),$$
	as $ |\lambda| \to \infty$.
\end{Corollary}
\begin{proof} As before, we  use duality
	\begin{align*}
		\langle \int_{\infty}^{\infty} g(\tilde{\lambda}) e_{\pm }(y, \tilde{\lambda})~d \tilde{\lambda}, \sigma_3 e_{+}(y, \lambda) \rangle_{L^2_y} & = \frac{1}{\pm \lambda^2}\langle \int_{\infty}^{\infty} g(\tilde{\lambda}) e_{\pm }(y, \tilde{\lambda})~d \tilde{\lambda}, \sigma_3 \mathcal{H} e_{+}(y, \lambda) \rangle_{L^2_y}\\
		& = \frac{1}{\pm \lambda^2}\langle \int_{\infty}^{\infty} g(\tilde{\lambda}) \tilde{\lambda}^2e_{\pm }(y, \tilde{\lambda})~d \tilde{\lambda}, \sigma_3  e_{+}(y, \lambda) \rangle_{L^2_y},
	\end{align*}
	where in the second line   $\int_{\infty}^{\infty} g(\tilde{\lambda}) \tilde{\lambda}^2e_{\pm }(y, \tilde{\lambda})~d \tilde{\lambda} = O(\langle y\rangle^{-2})$ by Lemma \ref{Fourier-decay} $(b)$ and further
	$$\mathcal{H} \int_{\infty}^{\infty} g(\tilde{\lambda}) e_{\pm }(y, \tilde{\lambda})~d \tilde{\lambda} = \int_{\infty}^{\infty} g(\tilde{\lambda}) \tilde{\lambda}^2e_{\pm }(y, \tilde{\lambda})~d \tilde{\lambda}.$$
	Again we may actually integrate the  pointwise identity  in order to conclude convergence. For higher order decay we iterate this step as long as Lemma \ref{Fourier-decay} $(b)$ applies. For the integrals 
	\[
	\lambda \mapsto \langle \int_{\infty}^{\infty} g(\tilde{\lambda}) y \partial_ye_{\pm }(y, \tilde{\lambda})~d \tilde{\lambda}, \sigma_3 e_{\pm}(y, \lambda) \rangle_{L^2_y}
	\]
	we proceed alike, however note the additional terms using
	$$ [\mathcal{H}, y \partial_y] = 2 \mathcal{H} - 2V + y (\partial_y V) =  2 \mathcal{H}  + O(\langle y \rangle^{-4})$$
\end{proof}
\begin{Rem} The above results for the operators defined in  \eqref{decyy1}, \eqref{decyy2} are preliminary in order to  calculate with these Fourier coefficients rigorously (enough regularity provided). In fact, in the proof of Proposition \ref{Kerne-prop}, we show the operators are $\delta $-functions.
\end{Rem}
~~\\
Let us set the space $L^2(\R \cup \{\pm i \kappa\})$  such that the elements map into $\C^2$ on $ \R$ and into $\C$ on $\pm i \mu$. Further with $ L^2_{\text{e}}$ we denote the $L^2$ closure of even Schwartz functions.
\begin{Lemma}
	The maps $ \mathcal{F}, \mathcal{F}^{-1}$  extend to bounded operators	\begin{align*}
		&\mathcal{F}: L^2_{\text{odd}}(\R) \to L^2_{\text{e}}(\R \cup \{\pm i \kappa\}) ,~\mathcal{F}: f \to \hat{f},\\
		&\mathcal{F}^{-1}: L^2_{\text{e}}(\R \cup \{\pm i \kappa\}) \to L^2_{\text{odd}}(\R),
	\end{align*}
	and there holds
	\begin{align}
		&\mathcal{F}^{-1}(\hat{f}) = f,~~\widehat{\mathcal{H}f}^{\pm}(\lambda) = \pm \lambda^2 \hat{f}^{\pm}(\lambda),~~ \lambda \in \R. \label{ide}
	\end{align}
\end{Lemma}
\begin{proof} By Lemma \ref{Fourier-decay}, the operator $ \mathcal{F}, \mathcal{F}^{-1}$ are well defined if we restrict to a dense domain $ \mathcal{D}(\mathcal{F}), \mathcal{D}(\mathcal{F}^{-1})$ of rapidly decaying functions (at least as required in Lemma \ref{Fourier-decay} for instance). Then from the dual identity in Proposition  \ref{Fourier-inv-fin}, we find in $L^2_y(\R)$
	\[
	P_c f = \int_{- \infty}^{\infty} \hat{f}^{+}(\lambda) e_+(\cdot , \lambda) ~d \lambda + \int_{- \infty}^{\infty} \hat{f}^{-}(\lambda) e_-(\cdot , \lambda) ~d \lambda, ~~f \in \mathcal{D}(\mathcal{F}),
	\]
	In particular we extend $\mathcal{F}$ to $ L^2(\R)$ wit $\| \mathcal{F} \|_{2 \to 2} \leq 1$. We set the projection $ P_c$ on the Fourier side $ P_c g(\lambda) = (g^{+}(\lambda), g^{-}(\lambda))$ and thus $ \mathcal{F}^{-1}(P_c g) = P_c \mathcal{F}^{-1}(g)$. Further we obtain  by duality
	\begin{align*}
		\|  \mathcal{F}^{-1}(P_c g ) \|_{L^2_y}^2 &= \langle \mathcal{F}(\sigma_3 \mathcal{F}^{-1}(P_cg)), P_cg \rangle_{L^2_{\lambda}}\\
		&\leq \| \mathcal{F}^{-1}(P_cg)  \|_{L^2_y} \| g \|_{L^2 },
	\end{align*}
	where $ \| g \|_{L^2 }^2 = \| P_c g\|_{L^2_{\lambda}}^2 + | g(i \kappa)|^2 + g(- i \kappa)|^2$. Hence we extend $\mathcal{F}^{-1}$ to $L^2$ with $ \| \mathcal{F}^{-1}\|_{2 \to 2}\leq 1$ and thus $ \| \mathcal{F} \|_{2 \to 2} = 1$. By approximation  \eqref{ide} holds and  follows as a direct consequence of the dual $ \mathcal{H}^* = \sigma_3 \mathcal{H} \sigma_3$ where we recall $ \mathcal{H}e_{\pm} = \pm \lambda^2 e_{\pm}$. 
\end{proof}
~~\\
In Section \ref{sec:linearized}, we consider \emph{Fourier coefficients}, i.e. we use the expansion
\begin{align}\label{F-Coeff}
	\tilde{u}(\tau, y) =&~~ \hat{u}(\tau, i \kappa) \phi^{+}_d(y) + \hat{u}(\tau, -i \kappa) \phi^{-}_d(y)\\ \nonumber
	& + ~\frac{1}{2 \pi}\int_{- \infty}^{\infty} \hat{u}^+(\tau, \lambda) e_+(y, \lambda) ~d \lambda + \frac{1}{2 \pi} \int_{- \infty}^{\infty} \hat{u}^-(\tau, \lambda) e_-(y, \lambda) ~d \lambda
\end{align}
with unknowns $ (\hat{u}(\tau, i \kappa), \hat{u}(\tau,- i \kappa),  \hat{u}^+(\tau, \lambda),  \hat{u}^-(\tau, \lambda))$. In order to expand the perturb system in Section \ref{sec:linearized}, we need to calculate $ \widehat{y \partial_y u}$ which compares to $ - \lambda \partial_{\lambda} \hat{u}$. However in the distorted case, there is an additional transport error caused by off-diagonal interaction in the Fourier representation, i.e. we set
\begin{align}
	\mathcal{K}\hat{u} := &~\widehat{y \partial_y u} + \mathcal{D} \hat{u},\\
	\mathcal{D}\hat{u}(\lambda)^{\pm} := &~\langle  \int_{- \infty}^{\infty}  [ \tilde{\lambda} \partial_{\tilde{\lambda}}\hat{u}^+(\tilde{\lambda}) ] e_+(y , \tilde{\lambda})~ d \tilde{\lambda}, \sigma_3 e_{\pm}(y, \lambda) \rangle_{L^2_y} \\ \nonumber 
	&~+ ~ \langle  \int_{- \infty}^{\infty} [ \tilde{\lambda} \partial_{\tilde{\lambda}}\hat{u}^-(\tilde{\lambda}) ] e_-(y , \tilde{\lambda})~ d \tilde{\lambda}, \sigma_3 e_{\pm}(y, \lambda) \rangle_{L^2_y}.
\end{align}
where the operator $\mathcal{D}$ only acts the continuous spectrum. 
\begin{Rem} It follows from the proof of Proposition \ref{Kerne-prop} that the kernel of $\mathcal{D}$ is  the $\delta$-function applied to $ \lambda \partial_{\lambda}$ and in fact $ \mathcal{D}f^{\pm}(\lambda ) =  \lambda \partial_{\lambda} f^{\pm}(\lambda)$.
\end{Rem}
The calculation of the operator $\mathcal{K}$ is usually referred to as transference identity, see e.g.  \cite[Section 5]{KST-slow}. We first observe using the Fourier inversion in Lemma \ref{Lemma-FT}
\begin{align} \label{first-ca}
	\widehat{y \partial_y u}(\pm i\kappa) =& ~ \hat{u}(i \kappa)\langle y \partial_y \phi^+_d(y), \phi^{\pm}_d(y) \rangle_{L^2_y} \; + \; \hat{u}(- i \kappa) \langle y \partial_y \phi^-_d(y), \phi^{\pm}_d(y) \rangle_{L^2_y}\\ \nonumber
	& ~+ ~\langle  \int_{- \infty}^{\infty} \hat{u}^+(\lambda) y \partial_y e_+(y , \lambda)~ d \lambda, \phi^{\pm}_d(y) \rangle_{L^2_y}\\ \nonumber
	& ~+ ~\langle  \int_{- \infty}^{\infty} \hat{u}^-(\lambda) y \partial_y e_-(y , \lambda)~ d \lambda, \phi^{\pm}_d(y) \rangle_{L^2_y},
\end{align}
and similar for $\lambda \in \R$
\begin{align}\label{seco-ca}
	&~~	\widehat{y \partial_y u}(\lambda)^{\pm} =~~ \hat{u}(i \kappa)\langle y \partial_y \phi^+_d(y), \sigma_3 e_{\pm}(y, \lambda)  \rangle_{L^2_y} \; + \; \hat{u}(- i \kappa) \langle y \partial_y \phi^-_d(y), \sigma_3 e_{\pm}(y, \lambda)\rangle_{L^2_y}\\ \nonumber
	& ~~~~~~~~~~~~~~~~~+ ~\langle  \int_{- \infty}^{\infty} \hat{u}^+(\tilde{\lambda}) y \partial_y e_+(y , \tilde{\lambda})~ d \tilde{\lambda}, \sigma_3 e_{\pm}(y, \lambda) \rangle_{L^2_y}\\ \nonumber 
	&~~~~~~~~~~~~~~~~~+ ~\langle  \int_{- \infty}^{\infty} \hat{u}^-(\tilde{\lambda}) y \partial_y e_-(y , \tilde{\lambda})~ d \tilde{\lambda}, \sigma_3 e_{\pm}(y, \lambda) \rangle_{L^2_y}.
\end{align}
We can interpret the Fourier transform as a map
\begin{align}\label{form-FT}
	u \mapsto \hat{u} = \begin{pmatrix}
		\hat{u}(\pm i \kappa)\\[2pt]
		\hat{u}_c
	\end{pmatrix} = \begin{pmatrix}
		(\hat{u}(i \kappa), \hat{u}(-i \kappa))^t \\[2pt]
		(\hat{u}^+, \hat{u}^-)^t
	\end{pmatrix},
\end{align}
where the lower component represents the continuous part of the spectrum. In the following understand function spaces, i.e. the space $C^{2}_0(\R \cup \{ \pm i \kappa\})$, consisting of $\C^2$ valued elements of the form in \eqref{form-FT}. 
For instance for $f \in C^{2}_0$ we write
\[
f = \begin{pmatrix}
	f(\pm i \kappa)\\[2pt]
	f_c
\end{pmatrix} = \begin{pmatrix}
	(f(i \kappa), f(-i \kappa))^t \\[2pt]
	(f^+, f^-)^t
\end{pmatrix}.
\]
Thus, using \eqref{first-ca} and \eqref{seco-ca}, we rewrite $\mathcal{K}$ acting on the essential$\slash $discrete  spectrum via
\begin{align*}
	&\mathcal{K} = \begin{pmatrix}
		\mathcal{K}_{dd} & \mathcal{K}_{d c}\\
		\mathcal{K}_{cd} & \mathcal{K}_{cc}
	\end{pmatrix} : C^{2}_0(\R \cup \{ \pm i \kappa\}) \to C^{2}_0(\R \cup \{ \pm i \kappa\}),\\[3pt]
	&\mathcal{K}f(\pm i \kappa ) =  \mathcal{K}_{dd} \cdot f(\pm i \kappa) + \mathcal{K}_{dc}f_c,\\
	&[\mathcal{K}f]_c(\lambda) = \mathcal{K}_{cd}(\lambda) \cdot f(\pm i \kappa) + \mathcal{K}_{cc} f_c(\lambda),
\end{align*}
~~\\
where for the domain $\mathcal{D}(\mathcal{K}) \subset C^2_0$ we require the functions to have compact support (as in Corollary \ref{Cor-Gourier-decay}) and 
\begin{align*}
	&\mathcal{K}_{dd} : = \begin{pmatrix}
		\langle y \partial_y \phi^+_d(y), \phi^{+}_d(y) \rangle_{L^2_y} & \langle y \partial_y \phi^-_d(y), \phi^{+}_d(y) \rangle_{L^2_y}\\[5pt]
		\langle y \partial_y \phi^+_d(y), \phi^{-}_d(y) \rangle_{L^2_y} &   \langle y \partial_y \phi^-_d(y), \phi^{-}_d(y) \rangle_{L^2_y}
	\end{pmatrix},\\[8pt]
	&\mathcal{K}_{dc}f_c : = \begin{pmatrix}
		\langle  \int_{- \infty}^{\infty} f^+(\lambda) y \partial_y e_+(y , \lambda)~ d \lambda, \phi^{+}_d(y) \rangle_{L^2_y}\\[5pt]
		\langle  \int_{- \infty}^{\infty} f^+(\lambda) y \partial_y e_+(y , \lambda)~ d \lambda, \phi^{-}_d(y) \rangle_{L^2_y}
	\end{pmatrix}
	+ 
	\begin{pmatrix}
		\langle  \int_{- \infty}^{\infty} f^-(\lambda) y \partial_y e_-(y , \lambda)~ d \lambda, \phi^{+}_d(y) \rangle_{L^2_y}\\[5pt]
		\langle  \int_{- \infty}^{\infty} f^-(\lambda) y \partial_y e_-(y , \lambda)~ d \lambda, \phi^{-}_d(y) \rangle_{L^2_y}
	\end{pmatrix},\\[8pt]
	&\mathcal{K}_{cd}(\lambda) : = \begin{pmatrix}
		\langle y \partial_y \phi^+_d(y), \sigma_3 e_+(y, \lambda)  \rangle_{L^2_y} &  \langle y \partial_y \phi^-_d(y), \sigma_3 e_+(y, \lambda)\rangle_{L^2_y}\\[5pt]
		\langle y \partial_y \phi^+_d(y), \sigma_3 e_-(y, \lambda)  \rangle_{L^2_y} & \langle y \partial_y \phi^-_d(y), \sigma_3 e_-(y, \lambda)\rangle_{L^2_y}
	\end{pmatrix},\\[8pt]
	& \mathcal{K}_{cc} f_c(\lambda) :=  	\begin{pmatrix}
		[\mathcal{K}_{cc}f_c]_{\text{++}}(\lambda)\\[4pt]
		[\mathcal{K}_{cc}f_c]_{+-}(\lambda)
	\end{pmatrix}	+ 
	\begin{pmatrix}
		[\mathcal{K}_{cc}f_c]_{-+}(\lambda)\\[4pt]
		[\mathcal{K}_{cc}f_c]_{--}(\lambda)
	\end{pmatrix},
\end{align*}
for which  we set
\begin{align} \label{c1}
	&[\mathcal{K}_{cc}f_c]_{+\pm}(\lambda) =   	\langle  \int_{- \infty}^{\infty} f^+(\tilde{\lambda}) y \partial_y e_+(y , \tilde{\lambda})~ d \tilde{\lambda}, \sigma_3 e_{\pm}(y, \lambda) \rangle_{L^2_y}\\ \nonumber
	&~~~~~~~~~~~~~~~~~~~~~~~~~~ + \langle  \int_{- \infty}^{\infty}  [\tilde{\lambda} \partial_{\tilde{\lambda}}f^+(\tilde{\lambda}) ] e_+(y , \tilde{\lambda})~ d \tilde{\lambda}, \sigma_3 e_{\pm}(y, \lambda) \rangle_{L^2_y},\\ \label{c23}
	&[\mathcal{K}_{cc}f_c]_{-\pm}(\lambda) =   	\langle  \int_{- \infty}^{\infty} f^-(\tilde{\lambda}) y \partial_y e_-(y , \tilde{\lambda})~ d \tilde{\lambda}, \sigma_3 e_{\pm}(y, \lambda) \rangle_{L^2_y}\\ \nonumber
	&~~~~~~~~~~~~~~~~~~~~~~~~~~ + \langle  \int_{- \infty}^{\infty}  [\tilde{\lambda} \partial_{\tilde{\lambda}}f^-(\tilde{\lambda}) ] e_-(y , \tilde{\lambda})~ d \tilde{\lambda}, \sigma_3 e_{\pm}(y, \lambda) \rangle_{L^2_y}.
\end{align}
Let us remark Corollary \ref{Cor-Gourier-decay}   implies  the terms  \eqref{c1} and \eqref{c23} are well defined and further $ \mathcal{K}_{dd}, \mathcal{K}_{c d}(\lambda)$ act as linear maps on $ \C^2$.\\[2pt]
For  $ \mathcal{K}_{dd}$ we note 
\begin{align}
\langle y \partial_y \phi^{\pm}_d(y), \phi^{\mp}_d(y) \rangle_{L^2_y} &= \overline{\langle y \partial_y \phi^{\mp}_d(y), \phi^{\pm}_d(y) \rangle}_{L^2_y}= - \overline{ \langle \phi^{\mp}_d(y),  y \partial_y\phi^{\pm}_d(y) \rangle}_{L^2_y}\\  \nonumber
&~~~~~~~~~~~~~~~~~~~~~~~~~~~~~~= - \langle  y \partial_y\phi^{\pm}_d(y), \phi^{\mp}_d(y) \rangle_{L^2_y},\\[3pt]
\langle y \partial_y \phi^{\pm}_d(y), \phi^{\pm}_d(y) \rangle_{L^2_y} &= - \underset{= 1}{\underbrace{\| \phi^{\pm}_d \|_{L^2_y}^2}} - \langle y \partial_y \phi^{\mp}_d(y), \phi^{\mp}_d(y) \rangle_{L^2_y}.
\end{align}
Since also 
\begin{align*}
\langle y \partial_y \phi^{\pm}_d(y), \phi^{\pm}_d(y) \rangle_{L^2_y}  = 	\langle y \partial_y \sigma_1 \phi^{\pm}_d(y), \sigma_1 \phi^{\pm}_d(y) \rangle_{L^2_y} = \langle y \partial_y  \phi^{\mp}_d(y),  \phi^{\mp}_d(y) \rangle_{L^2_y},
\end{align*}
we have $ \mathcal{K}_{dd} = - \f12 I_{2 \times 2}$. For $ \mathcal{K}_{cd}(\lambda)$ we note  
\begin{align*}
\langle \phi^{\pm}_d(y), \sigma_3 e_{+}(y, \lambda)\rangle_{L^2_y}& = \lambda^{-2} 	\langle \phi^{\pm}_d(y), \sigma_3 \mathcal{H}e_{+}(y, \lambda)\rangle_{L^2_y}\\
&= \lambda^{-2} 	\langle \mathcal{H} \phi^{\pm}_d(y), \sigma_3 e_{+}(y, \lambda)\rangle_{L^2_y}\\
&= \pm i \lambda^{-2} \kappa	\langle \phi^{\pm}_d(y), \sigma_3 e_{+}(y, \lambda)\rangle_{L^2_y}.
\end{align*}
and the analogue identity if we replace $ \sigma_3 e_{+}(y, \lambda) $ by $ \sigma_3 e_{-}(y, \lambda)$.
Thus we have 
$$\langle \phi^{\pm}_d(y), \sigma_3 e_{+}(y, \lambda)\rangle = \langle \phi^{\pm}_d(y), \sigma_3 e_{-}(y, \lambda)\rangle = 0$$
if $ \lambda \in \R \backslash \{0\}$ and hence
\begin{align}
\mathcal{K}_{cd}(\lambda) = - \begin{pmatrix}
	\langle \phi^+_d(y), \sigma_3  y \partial_y e_+(y, \lambda)  \rangle_{L^2_y} &  \langle  \phi^-_d(y), \sigma_3 y \partial_y e_+(y, \lambda)\rangle_{L^2_y}\\[5pt]
	\langle \phi^+_d(y), \sigma_3 y \partial_y e_-(y, \lambda)  \rangle_{L^2_y} & \langle  \phi^-_d(y), \sigma_3 y \partial_y e_-(y, \lambda)\rangle_{L^2_y}
\end{pmatrix}
\end{align}
Integrating by parts with respect to $\tilde{\lambda}$ in \eqref{c1}, \eqref{c23} yields
\begin{align}
&[\mathcal{K}_{cc}f_c]_{+\pm}(\lambda) =   	\langle  \int_{- \infty}^{\infty} f^+(\tilde{\lambda}) [y \partial_y - \tilde{\lambda}\partial_{\tilde{\lambda}}]e_+(y , \tilde{\lambda})~ d \tilde{\lambda}, \sigma_3 e_{\pm}(y, \lambda) \rangle_{L^2_y}\\ \nonumber
&~~~~~~~~~~~~~~~~~~~~~~~~~~~~~~~- \langle  \int_{- \infty}^{\infty}  f^+(\tilde{\lambda})  e_+(y , \tilde{\lambda})~ d \tilde{\lambda}, \sigma_3 e_{\pm}(y, \lambda) \rangle_{L^2_y},\\ 
&[\mathcal{K}_{cc}f_c]_{-\pm}(\lambda) =   	\langle  \int_{- \infty}^{\infty} f^-(\tilde{\lambda}) [y \partial_y - \tilde{\lambda}\partial_{\tilde{\lambda}}] e_-(y , \tilde{\lambda})~ d \tilde{\lambda}, \sigma_3 e_{\pm}(y, \lambda) \rangle_{L^2_y}\\ \nonumber
&~~~~~~~~~~~~~~~~~~~~~~~~~~~~~~~ - \langle  \int_{- \infty}^{\infty}  f^-(\tilde{\lambda}) e_-(y , \tilde{\lambda})~ d \tilde{\lambda}, \sigma_3 e_{\pm}(y, \lambda) \rangle_{L^2_y}.
\end{align}
\begin{Rem}In the above calculations we omitted the factor $(2\pi)^{-1}$ using $\mathcal{F}^{-1}$.  In the definition of $ \mathcal{K}$, this surely has to be corrected.
\end{Rem}

\begin{Prop}\label{Kerne-prop}
(a)	~We have $ \mathcal{K}_{cc}f = \mathcal{K}_0f +\tilde{\mathcal{K}}_{cc}f$ where
\begin{align*}
	&\tilde{	\mathcal{K}}_{cc} = \begin{pmatrix}
		\tilde{\mathcal{K}}_{++} & \tilde{\mathcal{K}}_{+-} \\
		\tilde{\mathcal{K}}_{-+} & \tilde{\mathcal{K}}_{--} 
	\end{pmatrix},~~\mathcal{K}_0 = \begin{pmatrix}
		[\mathcal{K}_0]_{++} & [\mathcal{K}_0]_{+-} \\
		[\mathcal{K}_0]_{-+} & [\mathcal{K}_0]_{--}
	\end{pmatrix},
\end{align*}
the operator $\tilde{\mathcal{K}}_{cc}$ is a $\delta$-function
\begin{align}
	&\tilde{\mathcal{K}}_{cc}(\lambda, \tilde{\lambda})= a(\lambda) I_{2\times 2} \cdot  \delta(\tilde{\lambda} - \lambda),\\[3pt]
	& a(\lambda) :=	\f12\lambda ( s'(\lambda) \overline{s(\lambda)}+ r'(\lambda) \overline{r(\lambda)}) - 1,~~ \lambda \geq 0,\\[3pt]
	& a(\lambda) := a(- \lambda) = \overline{a(\lambda)},~~ \lambda < 0.
\end{align}
and $\mathcal{K}_0$ has the kernel 
\begin{align}
	&[K_0]_{ \pm + }(\lambda, \tilde{\lambda}) = \frac{1}{(\lambda^2 \mp \tilde{\lambda}^2)} F_{\pm +}(\lambda, \tilde{\lambda}),\\
	&[K_0]_{ \pm - }(\lambda, \tilde{\lambda}) = \frac{1}{(\lambda^2 \pm \tilde{\lambda}^2)} F_{\pm -}(\lambda, \tilde{\lambda}).
\end{align}
Here $  F_{\pm +}(\lambda, \tilde{\lambda}),  F_{\pm -}(\lambda, \tilde{\lambda})$ are $C^2$ functions  with 
\[ F_{\pm +}(\lambda, \tilde{\lambda}) = \overline{F_{+ \pm}(\tilde{\lambda}, \lambda)} ,~~ F_{\pm -}(\lambda, \tilde{\lambda}) = \overline{F_{- \pm}(\tilde{\lambda}, \lambda) }
\]
and there holds (similar for $ F_{\pm, -}(\lambda, \tilde{\lambda}) $ replacing $\mp$ on the right side by $\pm$ ) for any number $ N \in \Z_+$
\begin{align} \label{firsto-ord1}
	&|F_{\pm, +}(\lambda, \tilde{\lambda}) |  \lesssim \begin{cases}
		|\lambda| + |\tilde{\lambda}|  & \lambda^2 + \tilde{\lambda}^2 \lesssim 1\\[3pt]
		(1 + \big| |\lambda| \mp |\tilde{\lambda}| \big|)^{-N} & \lambda^2 + \tilde{\lambda}^2 \gtrsim 1,
	\end{cases}\\[3pt] \label{firsto-ord2}
	& |\partial_{\lambda}F_{\pm, +}(\lambda, \tilde{\lambda})| \lesssim   \begin{cases}
		(|\lambda| + |\tilde{\lambda}|)| \log(|\lambda|)|  + |\lambda|^{\f12} + |\tilde{\lambda}|^{\f12} & \lambda^2+ \tilde{\lambda}^2 \lesssim 1\\[3pt]
		(1 + \big| |\lambda| \mp |\tilde{\lambda}| \big|)^{-N} | \log(2|\lambda|\langle \lambda \rangle^{-1})| 
		& \lambda^2 + \tilde{\lambda}^2 \gtrsim 1,
	\end{cases}\\[3pt]  \label{firsto-ord3}
	& |\partial_{\tilde{\lambda}}F_{\pm, +}(\lambda, \tilde{\lambda})| \lesssim   \begin{cases}
		(|\lambda| + |\tilde{\lambda}|)| \log(|\tilde{\lambda}|)|  + |\lambda|^{\f12} + |\tilde{\lambda}|^{\f12} & \lambda^2+ \tilde{\lambda}^2 \lesssim 1\\[3pt]
		(1 + \big| |\lambda| \mp |\tilde{\lambda}| \big|)^{-N}  \log(2|\tilde{\lambda}|\langle \tilde{\lambda} \rangle^{-1})| 
		& \lambda^2 + \tilde{\lambda}^2 \gtrsim 1,
	\end{cases}\\[3pt] \label{seco-ord1}
	& |\partial_{\lambda}\partial_{\tilde{\lambda}}F_{\pm, +}(\lambda, \tilde{\lambda})| \lesssim   \begin{cases}
		| \log(|\tilde{\lambda}|) \log(|\lambda|)|& \lambda^2+ \tilde{\lambda}^2 \lesssim 1\\[3pt]
		(1 + \big| |\lambda| \mp |\tilde{\lambda}| \big|)^{-N} | \log(2|\tilde{\lambda}|\langle \tilde{\lambda} \rangle^{-1}) \log(2\lambda|\langle \lambda \rangle^{-1})| & \lambda^2 + \tilde{\lambda}^2 \gtrsim 1,
	\end{cases}\\[3pt] \label{seco-ord2}
	& |\partial_{\lambda}^2F_{\pm, +}(\lambda, \tilde{\lambda})| \lesssim   \begin{cases}
		| \lambda^{-1}\log(|\lambda|)| & \lambda^2+ \tilde{\lambda}^2 \lesssim 1\\[3pt]
		|\tilde\lambda|^{-N} |\lambda|^{-1}| \log(|\lambda|) & \lambda^2 \ll1 \lesssim  \tilde{\lambda}^2 
	\end{cases}
\end{align}
(b)~ The discrete part $\mathcal{K}_{cd} $
with 
\begin{align*}
	[\mathcal{K}_{cd}]_{\pm +} (\lambda) = \langle y \partial_y \phi^{\pm}_d(y), \sigma_3 e_{+}(y, \lambda) \rangle_{L^2_y},~~ [\mathcal{K}_{cd}]_{\pm -} (\lambda) = \langle y \partial_y \phi^{\pm}_d(y), \sigma_3 e_{-}(y, \lambda) \rangle_{L^2_y},
\end{align*}
as well as the kernel functions
\begin{align*}
	[K_{dc}]_{\pm +} (\lambda) = \langle y \partial_y  e_{\pm}(y, \lambda), \phi^{+}_d(y) \rangle_{L^2_y},~~ [\mathcal{K}_{dc}]_{\pm -} (\lambda) = \langle y \partial_y  e_{\pm}(y, \lambda), \phi^{-}_d(y) \rangle_{L^2_y},
\end{align*}
for $ \mathcal{K}_{dc}g = \int K_{dc}(\lambda) g(\lambda) ~d \lambda $ decay rapidly as $ |\lambda| \to \infty$.
\end{Prop}
\begin{proof} Part (b) follows directly from duality $ \mathcal{H} = \sigma_3 \mathcal{H}^*\sigma_3 $ as  above and the calculation of $ [ y \partial_y , \mathcal{H}]$ (see below).
For part (a), we start with defining the functions
\begin{align*}
	&u_{+}(y):= \int_{-\infty}^{\infty} f^{+}(\tilde{\lambda})[ y \partial_y - \tilde{\lambda}\partial_{\tilde{\lambda}}] e_{+}(y, \tilde{\lambda})~d \tilde{\lambda},\\
	&u_{-}(y):= \int_{-\infty}^{\infty} f^{-}(\tilde{\lambda})[ y \partial_y - \tilde{\lambda}\partial_{\tilde{\lambda}}] e_{-}(y, \tilde{\lambda})~d \tilde{\lambda}.
\end{align*}
Then, as in the proof of Corollary \ref{Cor-Gourier-decay}, we justify the following integration by parts
\begin{align*}
	&\pm \lambda^2 \langle u_{+}, \sigma_3e_{\pm}(\cdot, \lambda) \rangle_{L^2_y} =  \langle u_{+}, \sigma_3\mathcal{H}e_{\pm}(\cdot, \lambda) \rangle_{L^2_y}  = \langle \mathcal{H} u_{+}, \sigma_3e_{\pm}(\cdot, \lambda) \rangle_{L^2_y}, \\
	&\pm\lambda^2 \langle u_{-}, \sigma_3 e_{\pm}(\cdot, \lambda) \rangle_{L^2_y}  =  \langle u_{-}, \sigma_3\mathcal{H}e_{\pm}(\cdot, \lambda) \rangle_{L^2_y}  = \langle \mathcal{H} u_{-}, \sigma_3e_{\pm}(\cdot, \lambda) \rangle_{L^2_y} .
\end{align*}
Moreover we obtain
\begin{align*}
	\mathcal{H} u_{\pm} =&~ \int_{ - \infty}^{\infty} f^{\pm}(\tilde{\lambda})[\mathcal{H}, y \partial_y] e_{\pm}(y, \tilde{\lambda} )~d \tilde{\lambda} + \int_{ - \infty}^{\infty} f^{\pm}(\tilde{\lambda})[ y \partial_y - \tilde{\lambda}\partial_{\tilde{\lambda}}] \mathcal{H}e_{\pm}(y, \tilde{\lambda} )~d \tilde{\lambda}\\
	=&~ \int_{ - \infty}^{\infty} f^{\pm}(\tilde{\lambda})[\mathcal{H}, y \partial_y] e_{\pm}(y, \tilde{\lambda} )~d \tilde{\lambda} \pm \int_{ - \infty}^{\infty} \tilde{\lambda}^2f^{\pm}(\tilde{\lambda})[ y \partial_y - \tilde{\lambda}\partial_{\tilde{\lambda}}]e_{\pm}(y, \tilde{\lambda} )~d \tilde{\lambda}\\
	& \mp 2\int_{ - \infty}^{\infty} \tilde{\lambda}^2f^{\pm}(\tilde{\lambda}) e_{\pm}(y, \tilde{\lambda} )~d \tilde{\lambda} 
\end{align*}
where, cf. the proof of Corollary \ref{Cor-Gourier-decay}, we have
$$ [\mathcal{H}, y \partial_y] = 2 \mathcal{H} - 2 V + y (\partial_yV) =: 2 \mathcal{H} + U(y).$$
Thus 
\begin{align*}
	\mathcal{H} u_{\pm} = \int_{ - \infty}^{\infty} f^{\pm}(\tilde{\lambda})U(y) e_{\pm}(y, \tilde{\lambda} )~d \tilde{\lambda} \pm \int_{ - \infty}^{\infty} \tilde{\lambda}^2f^{\pm}(\tilde{\lambda})[ y \partial_y - \tilde{\lambda}\partial_{\tilde{\lambda}}]e_{\pm}(y, \tilde{\lambda} )~d \tilde{\lambda},
\end{align*}
and 
\begin{align*}
	&\lambda^2 [\mathcal{K}_{cc}f]_{\pm +}(\lambda) = \pm [\mathcal{K}_{cc}(\lambda^2 \cdot f)]_{\pm +}(\lambda) + \langle   \int_{ - \infty}^{\infty} f^{\pm}(\tilde{\lambda})U(y) e_{\pm}(y, \tilde{\lambda} )~d \tilde{\lambda}, \sigma_3 e_{+}(y, \lambda) \rangle_{L^2_y},\\
	&\lambda^2 [\mathcal{K}_{cc}f]_{\pm -}(\lambda) =  \mp [\mathcal{K}_{cc}(\lambda^2 \cdot f)]_{\pm -}(\lambda) -  \langle   \int_{ - \infty}^{\infty} f^{\pm}(\tilde{\lambda})U(y) e_{\pm}(y, \tilde{\lambda} )~d \tilde{\lambda}, \sigma_3 e_{-}(y, \lambda) \rangle_{L^2_y}.
\end{align*}
We recall $U(y) = O(\langle y \rangle^{-4}) $ and therefore, using absolute convergence on the right,  we read off the following
\begin{align}
	&(\lambda^2 \mp \tilde{\lambda}^2)[K_0]_{\pm, +}(\lambda, \tilde{\lambda}) = \langle U(y) e_{\pm}(y, \tilde{\lambda}), \sigma_3 e_+(y, \lambda) \rangle_{L^2_y},\\
	&(\lambda^2 \pm \tilde{\lambda}^2)[K_0]_{\pm, -}(\lambda, \tilde{\lambda}) = - \langle U(y) e_{\pm}(y, \tilde{\lambda}), \sigma_3 e_-(y, \lambda) \rangle_{L^2_y}.
\end{align}
We hence set 
\begin{align*}
	&F_{\pm, +}(\lambda, \tilde{\lambda}) =  \langle U(y) e_{\pm}(y, \tilde{\lambda}), \sigma_3 e_+(y, \lambda) \rangle_{L^2_y},\\
	&F_{\pm, -}(\lambda, \tilde{\lambda}) =   - \langle U(y) e_{\pm}(y, \tilde{\lambda}), \sigma_3 e_-(y, \lambda) \rangle_{L^2_y}.
\end{align*}
From Corollary \ref{Das-Cor} we recall the following bounds for a constant $ C > 0$
\begin{align}
	&|e_{\pm}(y, \lambda)   | \leq C,\\
	&|\partial_{\lambda}e_{\pm}(y, \lambda)   | \leq C (  \langle y \rangle +  \langle \lambda \rangle^{-1} ),~~  | \lambda| \gg 1,\\
	&|\partial_{\lambda}e_{\pm}(y, \lambda)   | \leq C |\log(|\lambda|)| \langle y \rangle ,~~  | \lambda| \lesssim 1,\\
	&|\partial_{\lambda}^2e_{\pm}(y, \lambda)   | \leq C |\lambda^{-1} \log(|\lambda|)|   \langle y \rangle^{2},~~|\lambda| \lesssim 1.
\end{align}
Therefore we directly obtain the estimates (the same holds for $F_{\pm -}(\lambda, \tilde{\lambda})$)
\begin{align}\label{einn}
	&|F_{\pm +}(\lambda, \tilde{\lambda}) |  \leq C,\\ \label{diesr}
	& |\partial_{\lambda} F_{\pm +}(\lambda, \tilde{\lambda})| \leq C | \log(2|\lambda| \langle \lambda \rangle^{-1})|,~~~|\partial_{\tilde{\lambda}} F_{\pm +}(\lambda, \tilde{\lambda})| \leq C |\log(2|\tilde{\lambda}| \langle \tilde{\lambda} \rangle^{-1})|,\\ \label{diser3}
	& |\partial_{\tilde{\lambda}} \partial_{\lambda} F_{\pm +}(\lambda, \tilde{\lambda})| \leq C  |\log(2|\tilde{\lambda}| \langle \tilde{\lambda} \rangle^{-1})\log(2|\lambda| \langle \lambda \rangle^{-1})|,\\ \label{disr4}
	&  | \partial_{\lambda}^2 F_{\pm +}(\lambda, \tilde{\lambda})| \leq C  |\lambda^{-1}  \log(|\lambda|)|,~~ |\lambda | \lesssim 1.
\end{align}
from the above definition of $ F_{\pm +}, F_{\pm -}$. However, we show how to improve these bounds \eqref{einn} - \eqref{disr4}. Let us first consider $F_{\pm +}(\lambda, \tilde{\lambda}), F_{\pm -}(\lambda, \tilde{\lambda})$.\\[4pt]
\emph{Case: High frequencies}~$\max\{|\lambda|,|\tilde{\lambda}| \} \gtrsim 1$:~ Here we may proceed as follows (c.f.  \cite[Theorem 5.1]{KST-slow}). 
\begin{align}\label{eq1}
	\lambda^2 F_{\pm +}(\lambda, \tilde{\lambda}) &= ~\langle U(y) e_{\pm}(y, \tilde{\lambda}), \sigma_3 \mathcal{H} e_{+}(y, \lambda) \rangle_{L^2_y}\\ \nonumber
	& = \langle [ \mathcal{H}, U(y)] e_{\pm}(y, \tilde{\lambda}), \sigma_3 e_{+}(y, \lambda) \rangle_{L^2_y}  \pm \tilde{\lambda}^2 F_{\pm +}(\lambda, \tilde{\lambda}),\\[4pt] \label{eq12}
	\lambda^2 F_{\pm -}(\lambda, \tilde{\lambda}) &= ~ - \langle U(y) e_{\pm}(y, \tilde{\lambda}), \sigma_3 \mathcal{H} e_{-}(y, \lambda) \rangle_{L^2_y}\\ \nonumber
	& = - \langle [ \mathcal{H}, U(y)] e_{\pm}(y, \tilde{\lambda}), \sigma_3 e_{-}(y, \lambda) \rangle_{L^2_y}  \mp \tilde{\lambda}^2 F_{\pm -}(\lambda, \tilde{\lambda}),
\end{align}
where by definition of $U(y)$ and $V(y)$ we have 
\begin{align} \label{commut}
	[\mathcal{H}, U(y)] =& - \sigma_3 U''(y) - 2 \sigma_3 U'(y) \partial_y + 2 y  \sigma_1 (V_2'(y) V_1(y) - V_2(y) V_1'(y) )\\ \nonumber
	=& - \sigma_3 U''(y) - 2 \sigma_3 U'(y) \partial_y.
\end{align}
We note generally for $V, \tilde{V}$ as in Definition \ref{Def-on} there holds $[V, \tilde{V}] = 2 \sigma_1 (\tilde{V}_2 V_1 - V_2 \tilde{V}_1)$. 
Now we repeat this argument and obtain
\begin{align*}
	&(\lambda^2 \mp \tilde{\lambda}^2)^2 F_{\pm +}(\lambda, \tilde{\lambda}) = \langle [ \mathcal{H}, [ \mathcal{H}, U(y)] ] e_{\pm}(y, \tilde{\lambda}), \sigma_3 e_{+}(y, \lambda) \rangle_{L^2_y} ,\\
	&(\lambda^2 \pm \tilde{\lambda}^2)^2 F_{\pm -}(\lambda, \tilde{\lambda}) = \langle [ \mathcal{H}, [ \mathcal{H}, U(y)] ] e_{\pm}(y, \tilde{\lambda}), \sigma_3 e_{-}(y, \lambda) \rangle_{L^2_y} .
\end{align*}
Here we calculate with \eqref{commut}
\begin{align*}
	[ \mathcal{H}, [ \mathcal{H}, U(y)] ]  =&~ \partial_y^{(4)}U(y) + 4 U'''(y) \partial_y + 4 U''(y) \sigma_3 \mathcal{H}  - [V,  \sigma_3 U''(y)]\\
	&~  - 2 V \sigma_3 U'(y) \partial_y + 2 \sigma_3 U'(y) \partial_y V - 4 U''(y) \sigma_3 V(y)\\
	= & ~U^{\text{even}}_1(y) + U^{\text{odd}}(y) \partial_y + U^{\text{even}}_2(y) \mathcal{H},
\end{align*}
where $U^{\text{even}}_{1,2}, U^{\text{odd}}$ are even and odd functions of order
\[
U^{\text{even}}_{1}(y) = O(\langle  y \rangle^{-8}),~~U^{\text{even}}_{2}(y) = O(\langle  y \rangle^{-6}), ~~U^{\text{odd}}_{1}(y) = O(\langle  y \rangle^{-7}).
\]
By induction we derive the existence of odd and even $ M(2\times 2, \R)$ valued smooth rational   functions with 
\begin{align*}
	&U^{\text{even}}_{k, j}(y) = O(\langle y \rangle^{-4-2k}),~ U^{\text{odd}}_{k,\ell}(y) = O(\langle y \rangle^{-5-2k}),\\
	& j = 1, \dots, k,~~\ell = 1, \dots, k-1,
\end{align*}
such that for $ k \in \Z_+$
\begin{align*}
	(\lambda^2 \mp \tilde{\lambda}^2)^{2k} F_{\pm +}(\lambda, \tilde{\lambda}) =&~\big\langle  \bigg[ \sum_{j = 0}^{k} \tilde{\lambda}^{2j} U^{\text{even}}_{k, j}(y)  + \sum_{\ell = 0}^{k-1} \tilde{\lambda}^{2\ell} U^{\text{odd}}_{k,\ell}(y) \partial_y \bigg] e_{\pm}(y, \tilde{\lambda}), \sigma_3 e_{+}(y, \lambda) \big \rangle_{L^2_y},\\[2pt]
	(\lambda^2 \pm \tilde{\lambda}^2)^2 F_{\pm -}(\lambda, \tilde{\lambda}) =&~\big\langle  \bigg[ \sum_{j = 0}^{k} \tilde{\lambda}^{2j} U^{\text{even}}_{k, j}(y)  + \sum_{\ell = 0}^{k-1} \tilde{\lambda}^{2\ell} U^{\text{odd}}_{k,\ell}(y) \partial_y \bigg] e_{\pm}(y, \tilde{\lambda}), \sigma_3 e_{-}(y, \lambda) \big \rangle_{L^2_y}.
\end{align*}
\begin{Rem}We note that in the above formulas we actually suppress
	\[
	U^{\text{even}}_{k, j} = U^{\text{even}}_{k, j, \pm},~ U^{\text{odd}}_{k,\ell} = U^{\text{odd}}_{k,\ell, \pm}
	\]
	for the sake of notation and since it is not significant. Further the asymtotics holds clearly smoothly, i.e. we may write $ O_m(\langle y\rangle^{4 - 2k})$ for any $ m \in \Z_+$.
\end{Rem}
~\\
In conclusion there holds
\begin{align}
	| F_{\pm +}(\lambda, \tilde{\lambda}) |  \lesssim \langle \tilde{ \lambda} \rangle^{2k} (\lambda^2 \mp \tilde{\lambda}^2)^{-2k},~~| F_{\pm -}(\lambda, \tilde{\lambda}) |  \lesssim \langle \tilde{ \lambda} \rangle^{2k} (\lambda^2 \pm \tilde{\lambda}^2)^{-2k}.
\end{align}
and thus  for $ |\lambda| + |\tilde{\lambda}| \ge 1$ and arbitrary $ N \in \Z_+$
\begin{align}
	&| F_{ ++}(\lambda, \tilde{\lambda}) |  + | F_{ - -}(\lambda, \tilde{\lambda}) | \leq C   ( 1 + \big| |\lambda|  - | \tilde{\lambda}| \big|)^{-N},\\
	&| F_{ -+}(\lambda, \tilde{\lambda}) | + | F_{ + -}(\lambda, \tilde{\lambda})|  \leq C  ( 1 +  |\lambda|  + | \tilde{\lambda}| )^{-N}.
\end{align}
~~\\
\emph{Case: Low frequencies}~$\max\{|\lambda|,|\tilde{\lambda}| \} \lesssim 1$:~ Here we observe  $F_{\pm +}(0,0) = F_{\pm -}(0,0) = 0$ by using the definition and calculating  similar to \eqref{eq1} and \eqref{eq12}.  In fact write $  \chi_{\lesssim \tilde{\lambda}^{-\f12}} , \chi_{\gtrsim \tilde{\lambda}^{-\f12}} $ for $ \chi(\cdot \tilde{\lambda}^{\f12}),  1 -  \chi(\cdot \tilde{\lambda}^{\f12})$, where $ \chi \in C_c^{\infty}(\R)$ is an suitable cut-off with $ \chi = 1$ around zero. Then 
\[
F_{\pm +}(\lambda, \tilde{\lambda}) = \langle U(y) e_{\pm}(y, \tilde{\lambda}) \chi_{\lesssim \tilde{\lambda}^{-\f12}}(y), \sigma_3 e_{+}(y, \lambda) \rangle + \langle U(y) e_{\pm}(y, \tilde{\lambda})  \chi_{ \gtrsim\tilde{\lambda}^{-\f12}}(y), \sigma_3 e_{+}(y, \lambda) \rangle
\]
and  using $ U(y) = O(y^{-4})$ we directly infer
\[
| \langle U(y) e_{\pm}(y, \tilde{\lambda}) \chi_{ \gtrsim\tilde{\lambda}^{-\f12}}(y), \sigma_3 e_{+}(y, \lambda) \rangle| \lesssim |\tilde{\lambda}|^{\frac{3}{2}}.
\]
Further since also $ U(y) = [\mathcal{H}, y \partial_y] - 2 \mathcal{H}$ we write
\begin{align*}
	\langle U(y) e_{\pm}(y, \tilde{\lambda})  \chi_{ \lesssim\tilde{\lambda}^{-\f12}}(y), \sigma_3 e_{+}(y, \lambda) \rangle =& -2  \langle \mathcal{H} e_{\pm}(y, \tilde{\lambda})   \chi_{ \lesssim\tilde{\lambda}^{-\f12}}(y), \sigma_3 e_{+}(y, \lambda) \rangle\\
	&~ +  \langle [\mathcal{H}, y \partial_y] e_{\pm}(y, \tilde{\lambda})   \chi_{ \lesssim\tilde{\lambda}^{-\f12}}(y), \sigma_3 e_{+}(y, \lambda) \rangle.
\end{align*}
Moreover we note  the support of $ \chi_{ \lesssim\tilde{\lambda}^{-\f12}}$ has length $\sim \tilde{\lambda}^{-\f12}$ and 
\begin{align*}
	| \langle \mathcal{H} e_{\pm}(y, \tilde{\lambda})  \chi_{ \lesssim\tilde{\lambda}^{-\f12}}(y) , \sigma_3 e_{+}(y, \lambda) \rangle| &\lesssim |\tilde{\lambda}^2\cdot  \tilde{\lambda}^{- \f12}|,\\[2pt]
	\langle [\mathcal{H}, y \partial_y] e_{\pm}(y, \tilde{\lambda})   \chi_{ \lesssim\tilde{\lambda}^{-\f12}}(y), \sigma_3 e_{+}(y, \lambda) \rangle &=  \langle \mathcal{H} y \partial_ye_{\pm}(y, \tilde{\lambda})   \chi_{ \lesssim\tilde{\lambda}^{-\f12}}(y), \sigma_3 e_{+}(y, \lambda) \rangle\\
	&~~~ -  \langle y \partial_y \mathcal{H}e_{\pm}(y, \tilde{\lambda})   \chi_{ \lesssim\tilde{\lambda}^{-\f12}}(y), \sigma_3 e_{+}(y, \lambda) \rangle.
\end{align*}
For the second term on the right  we integrate  by parts
\begin{align*}
	\langle y \partial_y \mathcal{H}e_{\pm}(y, \tilde{\lambda})   \chi_{ \lesssim\tilde{\lambda}^{-\f12}}(y), \sigma_3 e_{+}(y, \lambda) \rangle &= \mp\tilde{\lambda}^2  \langle  e_{\pm}(y, \tilde{\lambda}) , \sigma_3 y \tilde{\lambda}^{\f12} \chi'( y \tilde{\lambda}^{\f12}) e_{+}(y, \lambda) \rangle\\
	&~~~    \mp\tilde{\lambda}^2 \langle  e_{\pm}(y, \tilde{\lambda}) , \sigma_3 \chi_{ \lesssim\tilde{\lambda}^{-\f12}}(y) ( 1 +  y \partial_y)e_{+}(y, \lambda) \rangle.
\end{align*}
Again for the second term on the right, we split the integral $\mathbb{I}( y > 0) + \mathbb{I}(y < 0)$ and use respectively for $ y \partial_y e_{+}(y, \lambda)$ (say $\lambda \geq 0$ and $ \lambda < 0 $ is treated similarly)
\[
i (y \lambda) s(\lambda)e^{i \lambda y} + y \partial_y e^{\infty}(y, \lambda),~~ i (y \lambda)( r(\lambda)e^{i \lambda y} - e^{- i \lambda y}) + y \partial_y e^{-\infty}(y, \lambda).
\]
There holds $s(\lambda), r(\lambda) = O(1) $ globally and  the error terms are as well globally uniformly bounded. For the  oscillatory parts we write $ y \lambda  \chi_{ \lesssim\tilde{\lambda}^{-\f12}}(y) = \lambda \tilde{\lambda}^{- \f12} (y \tilde{\lambda}^{\f12})  \chi_{ \lesssim\tilde{\lambda}^{-\f12}}(y)$. Summing up we have the upper bound $ O(\tilde{\lambda}^{\f32}) + O(\tilde{\lambda}\lambda) $.  Now finally, we collect the terms
\begin{align*}
	\langle \mathcal{H} y \partial_ye_{\pm}(y, \tilde{\lambda})   \chi_{ \lesssim\tilde{\lambda}^{-\f12}}(y), \sigma_3 e_{+}(y, \lambda) \rangle &= \langle  y \partial_ye_{\pm}(y, \tilde{\lambda}), \sigma_3 \chi_{ \lesssim\tilde{\lambda}^{-\f12}}(y) \mathcal{H}e_{+}(y, \lambda) \rangle\\
	&~~~  -   \tilde{\lambda} \langle  y \partial_ye_{\pm}(y, \tilde{\lambda}), \sigma_3\chi''(\tilde{\lambda}^{\f12}y) e_{+}(y, \lambda) \rangle\\
	&~~  - 2   \tilde{\lambda}^{\f12} \langle  y \partial_ye_{\pm}(y, \tilde{\lambda}), \sigma_3\chi'(\tilde{\lambda}^{\f12}y) \partial_ye_{+}(y, \lambda) \rangle .
\end{align*}
Here, we use the above expansion which again schematically has the form
\[
y \partial_y e_{\pm}(y, \tilde{\lambda}) =   i (y\tilde{\lambda}) \text{osc}(y, \tilde{\lambda})+ y \partial_ye^{\pm \infty}(y, \tilde{\lambda}),~~ |\text{osc}(y, \tilde{\lambda})| \lesssim 1.
\]
Splitting again $ y\tilde{\lambda} = (y\tilde{\lambda}^{\f12}) \cdot \tilde{\lambda}^{\f12}$ for the oscillatory part implies 
\begin{align*}
	&|	\langle y \partial_ye_{\pm}(y, \tilde{\lambda}) , \sigma_3 \chi_{ \lesssim\tilde{\lambda}^{-\f12}}(y) \mathcal{H}e_{+}(y, \lambda) \rangle\|  \lesssim \lambda^2 \tilde{\lambda}^{-\f12},\\
	& | \tilde{\lambda} \langle i (y\tilde{\lambda}) \text{osc}(y, \tilde{\lambda})  , \sigma_3\chi''(\tilde{\lambda}^{\f12}y) e_{+}(y, \lambda) \rangle| \lesssim \tilde{\lambda},\\
	& | \tilde{\lambda}^{\f12} \langle  i (y\tilde{\lambda}) \text{osc}(y, \tilde{\lambda}), \sigma_3\chi'(\tilde{\lambda}^{\f12}y) \partial_ye_{+}(y, \lambda) \rangle| \lesssim \tilde{\lambda} \lambda + \tilde{\lambda}^{\f32},
\end{align*}
For the latter two lines, we further integrate the missing error carefully.  Note we have  $ y \sim \tilde{\lambda}^{- \f12}$ on the support of derivatives of $\chi(\cdot \tilde{\lambda}^{\f12})$ and 
\[
|  \partial_y e^{\pm \infty}(y, \tilde{\lambda}) | \lesssim \langle y \rangle^{-3} + |\tilde{\lambda} \log(\tilde{\lambda})| e^{\mp \tilde{\lambda} y}( |\tilde{\lambda}| + \langle y \rangle^{-3}).
\]
Integrating $ y \partial_y e^{\pm \infty}$ thus contributes at least  $O(\tilde{\lambda}^{\f12})$ , i.e. in both cases
\begin{align*}
	& | \tilde{\lambda} \langle y \partial_y e^{\infty}(y, \tilde{\lambda}) , \sigma_3\chi''(\tilde{\lambda}^{\f12}y) e_{+}(y, \lambda) \rangle| \lesssim \tilde{\lambda}^{\f12} \cdot \tilde{\lambda},\\
	& | \tilde{\lambda}^{\f12} \langle y \partial_y e^{\infty}(y, \tilde{\lambda}), \sigma_3\chi'(\tilde{\lambda}^{\f12}y) \partial_ye_{+}(y, \lambda) \rangle| \lesssim \tilde{\lambda}^{\f12} \lambda \cdot  \tilde{\lambda}^{\f12}.
\end{align*}
Overall we have the  bound $O(\tilde{\lambda})$ for the middle term and $ O(\lambda \tilde{\lambda})$  + $ O(\tilde{\lambda}^{\f32})$ for the last. Now we may select the bound $ O( \max\{ \lambda, \tilde{\lambda}, \lambda^2 \tilde{\lambda}^{-1}\})$ and note that 
$ O( \max\{ \lambda, \tilde{\lambda}, \tilde{\lambda}^2 \lambda^{- 1}\})$ is likewise an upper bound for $F_{\pm +}, F_{\pm -}$. This is seen either by replacing $\chi_{\lesssim \tilde{\lambda}^{-\f12}},~\chi_{\gtrsim \tilde{\lambda}^{-\f12}}$ by $\chi_{\lesssim \lambda^{-\f12}},~\chi_{\gtrsim \lambda^{-\f12}}$  in the above  steps or by symmetry $F_{\pm +}(\lambda, \tilde{\lambda}) = \overline{F_{+ \pm}(\tilde{\lambda}, \lambda)}$.\\[3pt]
Further we may improve the bound for $ \partial_{\lambda}F_{ \pm +}(\lambda, \tilde{\lambda})$. Note 
\begin{align*}
	(\lambda^2 \mp \tilde{\lambda}^2)^{2k} \partial_{\lambda}F_{\pm +}(\lambda, \tilde{\lambda}) =&~\big\langle  \bigg[ \sum_{j = 0}^{k} \tilde{\lambda}^{2j} U^{\text{even}}_{k, j}(y)  + \sum_{\ell = 0}^{k-1} \tilde{\lambda}^{2\ell} U^{\text{odd}}_{k,\ell}(y) \partial_y \bigg] e_{\pm}(y, \tilde{\lambda}), \sigma_3 \partial_{\lambda} e_{+}(y, \lambda) \big \rangle_{L^2_y},\\
	&~ - 4k \lambda (\lambda^2 \mp \tilde{\lambda}^2)^{2k-1}F_{\pm +}(\lambda, \tilde{\lambda}),\\[2pt]
	(\lambda^2 \pm \tilde{\lambda}^2)^{2K} \partial_{\lambda}F_{\pm -}(\lambda, \tilde{\lambda}) =&~\big\langle  \bigg[ \sum_{j = 0}^{k} \tilde{\lambda}^{2j} U^{\text{even}}_{k, j}(y)  + \sum_{\ell = 0}^{k-1} \tilde{\lambda}^{2\ell} U^{\text{odd}}_{k,\ell}(y) \partial_y \bigg] e_{\pm}(y, \tilde{\lambda}), \sigma_3 \partial_{\lambda}e_{-}(y, \lambda) \big \rangle_{L^2_y}\\
	&~ - 4k \lambda (\lambda^2 \pm \tilde{\lambda}^2)^{2k-1}F_{\pm -}(\lambda, \tilde{\lambda}).
\end{align*}
Hence for all $ k \in \Z_+$ there holds
\begin{alignat}{2} \label{ein}
	| \partial_{\lambda}F_{\pm +}(\lambda, \tilde{\lambda}) | &\lesssim ( 1 + \frac{|\lambda|}{(\tilde{\lambda}^2 \mp \lambda^2)})|\tilde{\lambda}|^{2k} ( \tilde{\lambda}^2 \mp \lambda^2  )^{-2k},   &1 \lesssim |\tilde{\lambda}|,  |\lambda|,\\[2pt] \label{zw}
	| \partial_{\lambda}F_{\pm +}(\lambda, \tilde{\lambda}) | &\lesssim ( 1 + |\lambda|^{-1})( \tilde{\lambda}^2 \mp \lambda^2  )^{-2k}, & |\tilde{\lambda}| \ll 1 \lesssim   |\lambda|,\\[2pt] \label{dre}
	| \partial_{\lambda}F_{\pm +}(\lambda, \tilde{\lambda}) | &\lesssim | \log (|\lambda| )| |\tilde{\lambda}|^{2k} ( \tilde{\lambda}^2 \mp \lambda^2  )^{-2k}, & |\lambda| \ll 1 \lesssim   |\tilde{\lambda}|.
\end{alignat}
The case of low frequencies $ \tilde{\lambda}^2+ \lambda^2  \lesssim1$ 
is treated along the above line of arguments. In particular, we need to reconsider the steps with $ \lambda  \lesssim 1$ (schematically say $  \lambda > 0$ ) and
\begin{align}
	y > 0:~~\partial_{\lambda}e_{\pm}(y, \lambda) &= s'(\lambda)e^{i y \lambda} + i y s(\lambda)e^{i \lambda y} + \partial_{\lambda}e^{\infty}(y, \lambda)\\ \nonumber
	& = O(\log(\lambda)) + \tilde{\lambda}^{- \f12} (iy \tilde{\lambda})  O(1) + \partial_{\lambda}e^{\infty}(y, \lambda),\\
	y < 0:~~\partial_{\lambda}e_{\pm}(y, \lambda) &= r'(\lambda)e^{i y \lambda} + iy (r(\lambda)e^{i \lambda y} - e^{- i y \lambda}) +  \partial_{\lambda}e^{-\infty}(y, \lambda)\\ \nonumber
	& = O(\log(\lambda)) + \tilde{\lambda}^{- \f12} (iy \tilde{\lambda})  O(1) + \partial_{\lambda}e^{-\infty}(y, \lambda),
\end{align}
where for $ \lambda \lesssim 1 $ 
\[
|\partial_{\lambda} e^{\pm \infty}(y, \lambda)| \lesssim  | \log(\lambda)| C(\langle y \rangle^{-1} + e^{\mp |\lambda| \frac{y}{2}}).
\]
This implies an bound in the low frequency regime of the form 
\[
G(\lambda, \tilde{\lambda}) O(|\log(\lambda)|) + O(\lambda \tilde{\lambda}^{- \f12}),
\]
where $G(\lambda, \tilde{\lambda}) $ consists of all the upper bounds collected in the calculation above for $ F_{\pm +}, F_{\pm -}$. To be precise we mind the extra term  $O(\lambda \tilde{\lambda}^{-\f12})$ from estimating
\begin{align*}
	\langle  y \partial_ye_{\pm}(y, \tilde{\lambda}), \sigma_3 \chi_{ \lesssim\tilde{\lambda}^{-\f12}}(y) \mathcal{H}\partial_{\lambda}e_{+}(y, \lambda) \rangle =&~ \lambda^2 \langle  y \partial_ye_{\pm}(y, \tilde{\lambda}), \sigma_3 \chi_{ \lesssim\tilde{\lambda}^{-\f12}}(y)\partial_{\lambda}e_{+}(y, \lambda) \rangle\\
	&~ + 2 \lambda  \langle  y \partial_ye_{\pm}(y, \tilde{\lambda}), \sigma_3 \chi_{ \lesssim\tilde{\lambda}^{-\f12}}(y) e_{+}(y, \lambda) \rangle.
\end{align*}
Now we consider $ \partial_{\tilde{\lambda}}F_{\pm +}(\lambda, \tilde{\lambda})$. Calculating this derivatives will imply upper bounds along the above lines with estimates for the two extra terms 
\begin{align*}
	& | \tilde{\lambda}  \langle  e_{\pm}(y, \tilde{\lambda}) , \sigma_3 y \tilde{\lambda}^{\f12} \chi'( y \tilde{\lambda}^{\f12}) e_{+}(y, \lambda) \rangle|,~~| \tilde{\lambda} \langle  e_{\pm}(y, \tilde{\lambda}) , \sigma_3 \chi_{ \lesssim\tilde{\lambda}^{-\f12}}(y) ( 1 +  y \partial_y)e_{+}(y, \lambda) \rangle|.
\end{align*}
By symmetry $ F_{\pm + }(\lambda, \tilde{\lambda}) = \overline{F_{+ \pm}(\tilde{\lambda}, \lambda)}$ we then infer the following upper bound for $\partial_{\lambda}F_{\pm +}(\lambda, \tilde{\lambda})$
\[
G(\tilde{\lambda}, \lambda) O(|\log(\lambda)|) + O(\lambda^{\f12}) + O(\tilde{\lambda}^{\f12}),
\]
which, in combination, implies the desired bound.\\[2pt]
Similarly and by symmetry we derive the bounds for  $ \partial_{\lambda}F_{\pm -},  \partial_{\tilde{\lambda}}F_{\pm +}, \partial_{\tilde{\lambda}}F_{\pm -}$. The estimates for the second order derivatives 
$\partial_{\lambda \tilde{\lambda}}^2 F_{\pm +}(\lambda, \tilde{\lambda}),~ \partial_{\lambda}^2 F_{\pm +}(\lambda, \tilde{\lambda})$ follow  by differentiating  the above identities for the first order derivative, applying prior upper bounds for $F_{\pm +}$, as well as \eqref{ein} - \eqref{dre}.\\[4pt]
Now let us consider the $\delta$ function on the diagonal of $K_{cc}$ kernel.  First we recall from Corollary \ref{Das-Cor} there holds the following differentiable asymptotic expressions, say for $\lambda \geq 0$
\begin{align*}
	&	e_+(y, \lambda) = s(\lambda) \begin{pmatrix}
		e^{i \lambda y}\\ 0
	\end{pmatrix} + O(\langle y \rangle^{-2}) + O( \tfrac{\lambda}{\langle \lambda \rangle} \log(2\tfrac{\lambda}{\langle \lambda \rangle} ) e^{- \lambda y}),~~y \gg1,\\ 
	& e_+(y, \lambda) = r(\lambda) \begin{pmatrix}
		e^{-i \lambda y}\\ 0
	\end{pmatrix} +  \begin{pmatrix}
		e^{i \lambda y}\\ 0
	\end{pmatrix}+ O(\langle y \rangle^{-2}) + O( \tfrac{\lambda}{\langle \lambda \rangle} \log(2\tfrac{\lambda}{\langle \lambda \rangle} ) e^{ \lambda y}),~~ - y \gg1,
	\end{align*}
	and for $ \lambda < 0$ we have
	\begin{align*}
& e_+(y, \lambda) = \overline{r(\lambda)} \begin{pmatrix}
	e^{-i \lambda y}\\ 0
\end{pmatrix} +  \begin{pmatrix}
	e^{i \lambda y}\\ 0
\end{pmatrix}+ O(\langle y \rangle^{-2}) +  O( \tfrac{\lambda}{\langle \lambda \rangle} \log(2\tfrac{\lambda}{\langle \lambda \rangle} ) e^{- \lambda y}),~~  y \gg1,\\
&	e_+(y, \lambda) = \overline{s(\lambda) }\begin{pmatrix}
e^{i \lambda y}\\ 0
\end{pmatrix} + O(\langle y \rangle^{-2}) +  O( \tfrac{\lambda}{\langle \lambda \rangle} \log(2\tfrac{\lambda}{\langle \lambda \rangle} ) e^{ \lambda y}),~~ - y \gg1.
\end{align*}
Likewise it holds
\begin{align*}
&	(y \partial_y - \lambda \partial_{\lambda})e_+(y, \lambda) = \lambda s'(\lambda) \begin{pmatrix}
e^{i \lambda y}\\ 0
\end{pmatrix}  + O(\langle y \rangle^{-1}) +  O( \tfrac{\lambda}{\langle \lambda \rangle} \log(2\tfrac{\lambda}{\langle \lambda \rangle} ) e^{- \lambda \frac{y}{2}}),~~ y \gg1,\\
& (y \partial_y - \lambda \partial_{\lambda})e_+(y, \lambda) =  \lambda r'(\lambda) \begin{pmatrix}
e^{-i \lambda y}\\ 0
\end{pmatrix}  + O(\langle y \rangle^{-1}) +  O( \tfrac{\lambda}{\langle \lambda \rangle} \log(2\tfrac{\lambda}{\langle \lambda \rangle} ) e^{ \lambda \frac{y}{2}}),~~ - y \gg1,
\end{align*}
and again  for $ \lambda < 0$ we have
\begin{align*}
&	(y \partial_y - \lambda \partial_{\lambda})e_+(y, \lambda) = \lambda \overline{r'(\lambda)} \begin{pmatrix}
e^{-i \lambda y}\\ 0
\end{pmatrix}  + O(\langle y \rangle^{-1}) +  O( \tfrac{\lambda}{\langle \lambda \rangle} \log(2\tfrac{\lambda}{\langle \lambda \rangle} ) e^{- \lambda \frac{y}{2}}),~~ y \gg1,\\
& (y \partial_y - \lambda \partial_{\lambda})e_+(y, \lambda) =  \lambda \overline{s'(\lambda)} \begin{pmatrix}
e^{i \lambda y}\\ 0
\end{pmatrix}  + O(\langle y \rangle^{-1}) +  O( \tfrac{\lambda}{\langle \lambda \rangle} \log(2\tfrac{\lambda}{\langle \lambda \rangle} ) e^{ \lambda \frac{y}{2}}),~~ - y \gg1.
\end{align*}
The asymptotics for $ e_-(y, \lambda)$ is equivalent to  considering  $\sigma_1 e_+(y, \lambda) $ above. For a cut-off $\chi \in C^{\infty}_c(\R),~ \chi \geq 0$ with $ \chi = 1 $ on $[-1,1]$ and $\chi = 0$ on $\R \backslash [-2, 2]$ we split
\begin{align*}
&\int_{-\infty}^{\infty}	\int_{-\infty}^{\infty}f^{\pm}(\tilde{\lambda}) [ y \partial_y - \tilde{\lambda}\partial_{\tilde{\lambda}}]e_{\pm}(y, \tilde{\lambda}) d \tilde{\lambda} \sigma_3 \overline{e_+(y, \lambda)}~dy\\
&= \int_{-\infty}^{\infty}	\int_{-\infty}^{\infty}f^{\pm}(\tilde{\lambda}) [ y \partial_y - \tilde{\lambda}\partial_{\tilde{\lambda}}]e_{\pm}(y, \tilde{\lambda}) d \tilde{\lambda} \chi(y) \sigma_3 \overline{e_+(y, \lambda)}~dy\\
& + \int_{-\infty}^{\infty}	\int_{-\infty}^{\infty}f^{\pm}(\tilde{\lambda}) [ y \partial_y - \tilde{\lambda}\partial_{\tilde{\lambda}}]e_{\pm}(y, \tilde{\lambda}) d \tilde{\lambda} (1 -\chi(y)) \sigma_3 \overline{e_+(y, \lambda)}~dy.
\end{align*} The first term on the right can be handled as before using
\[
\langle [\mathcal{H}, \chi][y \partial_y - \tilde{\lambda} \partial_{\tilde{\lambda}}] e_{\pm}(y, \tilde{\lambda}), \sigma_3 e_+(y, \lambda) \rangle,~~\langle \chi(y) [\mathcal{H}, y \partial_y]  e_{\pm}(y, \tilde{\lambda}), \sigma_3 e_+(y, \lambda) \rangle ,
\]
in order to see that the $\delta$ must be generated in the other integral. In this integral, we use the above expansions for $ e_{\pm}(y, \lambda), ~[y \partial_y - \tilde{\lambda} \partial_{\tilde{\lambda}}] e_{\pm}(y, \tilde{\lambda})$. This genuinely generates four integrals, separating the above oscillatory leading terms and the error. All the integrals involving an error term will then contribute to a bounded kernel (cf \cite{KST-slow}). In order to see this for the $ O(\langle y \rangle^{-1}) $ part, we integrate by parts, i.e. exemplarily
\begin{align*}
\int_{\pm y \gtrsim 1}[y \partial_y - \tilde{\lambda} \partial_{\tilde{\lambda}}] e^{\pm \infty}(y)  e^{i \lambda y}~dy = -  \frac{1}{i \lambda}\int_{\pm y \gtrsim 1} \big( \partial_y [y \partial_y - \tilde{\lambda} \partial_{\tilde{\lambda}}] e^{\pm \infty}(y) \big) e^{i \lambda y}~dy +  \frac{1}{i \lambda}C,
\end{align*}
and we use Corollary \ref{Das-Cor2}. Finally we focus on $\delta$-contributions. First, let us consider the $(+,+)$ kernel element 
\[
\int_{\R}f^{+}(\tilde{\lambda})\int_{\R} (1 - \chi(y))[y \partial_y - \tilde{\lambda} \partial_{\tilde{\lambda}}] e_{+}(y, \tilde{\lambda}) \cdot \sigma_3 \overline{ e_+(y, \lambda)}~d \tilde{\lambda} d y.
\]
In particular, we may ignore $(\tilde{\lambda}, \lambda)$ interactions in the above $e_{\pm}$ expansions where $ \lambda \geq 0, \tilde{\lambda} < 0$ and $\lambda < 0, \tilde{\lambda} \geq 0$, since any such $\delta(\tilde{\lambda} - \lambda)$ and $ \delta(\lambda - \tilde{\lambda}  )$ contribution must vanish. 
Then if $ \lambda \geq 0, \tilde\lambda \geq 0$, we have the following leading  integrals
\begin{align} \label{dass-first}
\int_{0}^{\infty} \int_{0}^{\infty}f^{\pm}(\tilde{\lambda})& (1 -\chi(y)) \tilde{\lambda}  s'(\tilde{\lambda}) \overline{s(\lambda)} e^{i(\tilde{\lambda} - \lambda)y}d \tilde{\lambda} ~dy\\ \nonumber
& +  \int_{ - \infty}^{0}	\int_{0}^{\infty}f^{\pm}(\tilde{\lambda}) (1 -\chi(y)) \tilde{\lambda}  r'(\tilde{\lambda}) \overline{r(\lambda)} e^{i(\lambda - \tilde{\lambda} )y}d \tilde{\lambda} ~dy\\ \nonumber
& + \int_{ - \infty}^{0}	\int_{0}^{\infty}f^{\pm}(\tilde{\lambda}) (1 -\chi(y)) \tilde{\lambda}  r'(\tilde{\lambda}) e^{- i(\lambda + \tilde{\lambda} )y}d \tilde{\lambda} ~dy.
\end{align}
Further if $ \lambda <0, \tilde\lambda < 0$ we have
\begin{align}\label{dass-zwei}
\int_{0}^{\infty}	\int_{- \infty}^{0}f^{\pm}(\tilde{\lambda})& (1 -\chi(y)) \tilde{\lambda}  \overline{r'(\tilde{\lambda})} r(\lambda) e^{i(\lambda - \tilde{\lambda} )y}d \tilde{\lambda} ~dy\\ \nonumber
& +  \int_{ - \infty}^{0}	\int_{- \infty}^{0}f^{\pm}(\tilde{\lambda}) (1 -\chi(y)) \tilde{\lambda}  \overline{s'(\tilde{\lambda})} s(\lambda) e^{i(\tilde{\lambda} - \lambda)y}d \tilde{\lambda} ~dy\\ \nonumber
&+  \int_{ 0}^{\infty}	\int_{- \infty}^{0}f^{\pm}(\tilde{\lambda}) (1 -\chi(y)) \tilde{\lambda}  \overline{r'(\tilde{\lambda})}  e^{- i(\tilde{\lambda} + \lambda)y}d \tilde{\lambda} ~dy.
\end{align}
These integrals exist in an approximate sense by means of the classical (free) Fourier inversion, which allows to change the order of integration. Also note we the last integrals in \eqref{dass-first} and \eqref{dass-zwei} involving $ e^{- i (\tilde{\lambda} + \lambda)y}$ contribute a bounded kernel using integration by parts and hence will be ignored. Finally we  obtain the $(+,+)$ kernel from  the remaining  integrals
\begin{align}
\lambda \geq 0:~~\pi(\lambda s'(\lambda) \overline{s(\lambda)} + \lambda r'(\lambda) \overline{r(\lambda)}) \delta(\tilde{\lambda} - \lambda),\\
\lambda < 0: ~~\pi(\lambda \overline{s'(\lambda)}s(\lambda) + \lambda \overline{r'(\lambda)} r(\lambda)) \delta( \lambda - \tilde{\lambda} ).
\end{align}
Now, by expanding the Fourier base,  the leading oscillatory terms in the integrals of the $ (+, - )$ and $(-, +)$  kernel elements are perpendicular in $\C^2 $ and the $(-,-)$ element has the same on-diagonal kernel as the $(+,+)$ element.
~\\ 
We are left with the $\mathcal{K}_{cc}$ contribution of the integrals 
\begin{align}\label{lefk55}
&\tilde{ \mathcal{K}}_{+ \pm}f(\lambda) =- \langle  \int_{- \infty}^{\infty}  f^+(\tilde{\lambda})  e_+(y , \tilde{\lambda})~ d \tilde{\lambda}, \sigma_3 e_{\pm}(y, \lambda) \rangle_{L^2_y},\\ \nonumber
&\tilde{ \mathcal{K}}_{- \pm}f(\lambda) = - \langle  \int_{- \infty}^{\infty}  f^-(\tilde{\lambda}) e_-(y , \tilde{\lambda})~ d \tilde{\lambda}, \sigma_3 e_{\pm}(y, \lambda) \rangle_{L^2_y}.
\end{align}
Concerning off-diagonal frequencies, we may apply the above reasoning for a modified $e_{\pm}$, i.e.  let $ e^{\varepsilon}_{\pm}(y, \lambda) : = e^{- \varepsilon y^2 } e_{\pm}(y, \lambda)$ and calculate
\begin{align}
\mathcal{H}e^{\varepsilon}_{\pm}(y, \lambda) = &~\pm \lambda^2 e^{\varepsilon}_{\pm}(y, \lambda)   - 4 \varepsilon y e^{- \varepsilon y^2 } \sigma_3 \partial_y e_{\pm}(y, \lambda) - 2 \varepsilon  e^{- \varepsilon y^2 } \sigma_3  e_{\pm}(y, \lambda)\\ \nonumber
&~ +  4 \varepsilon^2 y^2 e^{- \varepsilon y^2 } \sigma_3  e_{\pm}(y, \lambda).
\end{align}
Then we conclude
\begin{align}
\langle \int_{- \infty}^{\infty} &f^{\pm}(\tilde{\lambda}) (\lambda^2 \mp \tilde{\lambda}^2) e_{\pm}(y, \tilde{\lambda}) d \tilde{\lambda}, \sigma_3 e_{+}^{\varepsilon}(y, \lambda)\rangle_{L^2_y}\\ \nonumber
& =~ 	 \int_{- \infty}^{\infty}  f^{\pm}(\tilde{\lambda}) \langle   e_{\pm}(y, \tilde{\lambda}) , 4 \varepsilon \sigma_3 e^{- \varepsilon y^2 }\partial_y e_{+}(y, \lambda)\rangle_{L^2_y} d \tilde{\lambda}\\ \nonumber
&+~ \int_{- \infty}^{\infty}  f^{\pm}(\tilde{\lambda}) \langle   e_{\pm}(y, \tilde{\lambda}) , 2 \varepsilon \sigma_3 e^{- \varepsilon y^2 } e_{+}(y, \lambda)\rangle_{L^2_y} d \tilde{\lambda}\\ \nonumber
&-~ \int_{- \infty}^{\infty}  f^{\pm}(\tilde{\lambda}) \langle   e_{\pm}(y, \tilde{\lambda}) , 4 \varepsilon^2 y^2 \sigma_3 e^{- \varepsilon y^2 } e_{+}(y, \lambda)\rangle_{L^2_y} d \tilde{\lambda}.
\end{align}
Letting $\varepsilon \to 0^+$, and noting the convergence to zero of the kernels on the right, this shows the operator given by these terms can not have a bounded off-diagonal contribution.\\[2pt]
Instead we consider the $ \delta$-measure on the diagonal (where we again  use a cut-off $ \chi(y)$ supported at the origin). Let us consider the $ (+,+)$ kernel element
\[
\int_{\R}f^{+}(\tilde{\lambda})\int_{\R} (1 - \chi(y)) e_{+}(y, \tilde{\lambda}) \cdot \sigma_3 \overline{ e_+(y, \lambda)}~d \tilde{\lambda} d y.
\]
As above we neglect the error terms in the expansion of the base functions and if $ \lambda \geq 0$ we may first restrict to $ \tilde{\lambda } \geq 0$ and obtain the contribution
\begin{align*}
&\int_{0}^{\infty}f^{+}(\tilde{\lambda})\int_{0}^{\infty} (1 - \chi(y)) s(\tilde{\lambda}) \overline{s(\lambda)} e^{i (\tilde{\lambda} - \lambda)y}~d \tilde{\lambda} d y\\
& +  \int_{- \infty}^{0}f^{+}(\tilde{\lambda})\int_{0}^{\infty} (1 - \chi(y)) (r(\tilde{\lambda}) \overline{r(\lambda)} + 1) e^{i (\tilde{\lambda} - \lambda)y}~d \tilde{\lambda} d y,
\end{align*}
which allows to change order of integration by the classical (free) Fourier inversion. Hence the kernel reads 
\[
\pi |s(\lambda)|^2 \delta(\tilde{\lambda} - \lambda) + \pi ( |r(\lambda)|^2 + 1) \delta(\tilde{\lambda} - \lambda) = 2 \pi  \delta(\tilde{\lambda} - \lambda).
\]
Note again that products of the oscillatory leading terms also lead to terms involving $ e^{ \pm i (\tilde{\lambda} + \lambda)y}$. These, however,  contribute a bounded kernel from integration by parts and will be readily ignored. Further the $ \lambda \geq 0,~ \tilde{\lambda} < 0$ and $ \lambda < 0,~ \tilde{\lambda} > 0$ interaction has vanishing kernel, respectively, on the diagonal. Lastly, the $ \lambda < 0,~ \tilde{\lambda } \geq 0$ interaction part is seen to contribute as well $ 2 \pi \delta(\tilde{\lambda} - \lambda)$.\\[2pt]
For the $ (+,-)$ and $(-, +)$ kernel element (referring to $ e_{\pm} $), the leading oscillatory terms are perpendicular in $\C^2 $ and the $(-,-)$ kernel element has the same on-diagonal kernel $ 2 \pi \delta(\tilde{\lambda} - \lambda)$. Thus the kernel of \eqref{lefk55} is 
\[
\tilde{\mathcal{K}}(\lambda, \tilde{\lambda}) = \begin{pmatrix}
\tilde{\mathcal{K}}_{++}(\lambda, \tilde{\lambda}) & \tilde{\mathcal{K}}_{+ -}(\lambda, \tilde{\lambda})\\
\tilde{\mathcal{K}}_{-+}(\lambda, \tilde{\lambda})  & \tilde{\mathcal{K}}_{--}(\lambda, \tilde{\lambda}) 
\end{pmatrix} = - 2 \pi I_{2 \times 2} \cdot  \delta(\tilde{\lambda} - \lambda).
\]
\end{proof}
\begin{Rem} (i) We do not claim the estimates \eqref{firsto-ord1} - \eqref{firsto-ord3} to be sharp. However, as seen above, they are better than expected from the Fourier base. (ii)~The bounds \eqref{seco-ord1} - \eqref{seco-ord2}  for the second order derivatives $\partial_{\lambda} \partial_{\tilde{\lambda}}F_{\pm +},  \partial_{\lambda}^2 F_{\pm +}$ may be improved in the low frequency contributions $ |\lambda |\lesssim 1 ,~ |\tilde{\lambda}| \lesssim  1$ along the lines of the remaining estimates. However, we will not make use of second order derivatives and thus spare to give details.
\end{Rem}
We now define the weighted space $ L^{2, \sigma}$ via the completion of the norm
\begin{align}
\|  f \|_{L^{2, \sigma}}^2 =~& |f(i \kappa)|^2 + |f( - i \kappa)|^2 + \int_{-\infty}^{\infty} |f^+(\lambda)|^2 \langle \lambda \rangle^{2 \sigma}~d \lambda\\ \nonumber
&~~+  \int_{-\infty}^{\infty} |f^-(\lambda)|^2 \langle \lambda \rangle^{2 \sigma}~d \lambda,
\end{align}
defined for rapidly decaying functions on $\R \cup \{ \pm i \kappa\}$ (as usual such functions $f$ map into $\C^2 $ on $\R$, where we write $f_c(\lambda) = (f^+(\lambda), f^-(\lambda))$, and into $\C$ on $\{\pm i \kappa\}$).
\begin{Prop} \label{Propo-boundedne-K}The operators $\mathcal{K}_0$ and  $\mathcal{K} = \mathcal{K}_0 + \tilde{\mathcal{K}}$ are bounded as maps
\begin{align}
& \mathcal{K}_0 : L^{2, \sigma} \to L^{2, \sigma }\\
&  \mathcal{K} : L^{2, \sigma} \to L^{2, \sigma }
\end{align}
for all $ \sigma\in \R$.
\end{Prop}
\begin{proof} We consider the kernel $ \langle \lambda \rangle^{\sigma } [K_0]_{\pm +}(\lambda, \tilde{\lambda}) \langle \tilde{\lambda}\rangle^{- \sigma}$ defining an operator mapping $ L^2(d \lambda) $ into $ L^2(d \lambda) $ (and similarly we proceed for $ [K_0]_{\pm -}(\lambda, \tilde{\lambda}))$. By  Proposition \ref{Kerne-prop} we may inspect the maps
\[
[\tilde{K}_0]_{\pm +} : (\lambda, \tilde{\lambda}) \mapsto \frac{\langle \lambda\rangle^{\sigma } \langle \tilde{\lambda}\rangle^{- \sigma}  }{\lambda^2 \mp \tilde{\lambda}^2} F_{\pm +}(\lambda, \tilde{\lambda}).
\]
Let us define the cube $Q\subset \R^2$ for fixed $ k \in \Z_+$ to be 
$
Q = [-2^k, 2^k] \times [-2^k, 2^k].
$
Further for  $ j \in \Z$ we set 
\begin{align*}
Q_j =&~~ [2^{j-1}, 2^{j+1}] \times [2^{j-1}, 2^{j+1}] \cup [-2^{j+1}, -2^{j-1}]  \times [2^{j-1}, 2^{j+1}] \\
&~~ \cup [-2^{j+1}, -2^{j-1}]  \times [-2^{j+1}, -2^{j-1}]  \cup  [2^{j-1}, 2^{j+1}]  \times [-2^{j+1}, -2^{j-1}]\\
& = \{  (\lambda, \tilde{\lambda}) \in \R^2~|~ |\lambda|, |\tilde{\lambda}| \in [2^{j-1}, 2^{j+1}] \}.
\end{align*}
We may distinguish $ Q_j^{\pm, +} = Q_j \cap \{  \pm \lambda > 0\} \cap \{  \tilde{\lambda} > 0\}$ and $ Q_j^{\pm, -} = Q_j \cap \{  \pm \lambda > 0\} \cap \{  \tilde{\lambda} < 0\}$.\\[3pt]
Let
\begin{align}
B_1 : = \bigcup_{j = - \infty}^{k} Q_j,~~~B_2 := \bigcup_{j = k-1}^{\infty} Q_j.
\end{align}
Then we split the kernel function
\begin{align}
[\tilde{K}_0]_{\pm +}(\lambda, \tilde{\lambda}) =& \sum_{j \leq k} \chi_{Q \cap Q_j} [\tilde{K}_0]_{\pm +}(\lambda, \tilde{\lambda})  + \chi_{Q \cap B_1^C }[\tilde{K}_0]_{\pm +}(\lambda, \tilde{\lambda})\\ \nonumber
& + \sum_{j  \geq k} \chi_{Q^C \cap Q_j} [\tilde{K}_0]_{\pm +}(\lambda, \tilde{\lambda})  + \chi_{Q^C \cap B_2^C }[\tilde{K}_0]_{\pm +}(\lambda, \tilde{\lambda}).
\end{align}
\underline{Step~1}.~Let us consider $ \sum_{j \leq k} \chi_{Q \cap Q_j} [\tilde{K}_0]_{\pm +}(\lambda, \tilde{\lambda}) $ covering diagonal frequencies in  $Q$. In particular we take a look at $ \chi_{Q \cap Q_j}[\tilde{K}_0]_{++}(\lambda, \tilde{\lambda}), \chi_{Q \cap Q_j}[\tilde{K}_0]_{--}(\lambda, \tilde{\lambda})$ for any such $ j \leq k$. Then (we restrict to $ \lambda \geq 0$ for simplicity)
\begin{align}\label{das-int}
\int \chi_{Q \cap Q_j}(\lambda, \tilde{\lambda})[\tilde{K}_0]_{++}(\lambda, \tilde{\lambda}) g(\tilde{\lambda})~d \tilde{\lambda} =&~  \langle \lambda\rangle^{\sigma } \chi_{\sim 2^j}(\lambda) \int \frac{F_{++}(\lambda, \tilde{\lambda})}{\lambda^2 -  \tilde{\lambda}^2} f(\tilde \lambda)~d \tilde{\lambda}\\ \nonumber
=&~  \langle \lambda\rangle^{\sigma } \chi_{\sim 2^j}(\lambda) F_{++}(\lambda, \lambda)  \int \frac{f(\tilde{\lambda})}{\lambda^2 -  \tilde{\lambda}^2} ~d \tilde{\lambda}\\ \nonumber
&+ \langle \lambda\rangle^{\sigma } \chi_{\sim 2^j}(\lambda) \int \frac{F_{++}(\lambda, \tilde{\lambda}) - F_{++}(\lambda, \lambda)}{\lambda^2 -  \tilde{\lambda}^2} f(\tilde \lambda)~d \tilde{\lambda},
\end{align}
where 
$f (\tilde{\lambda}) = \langle \tilde{\lambda}\rangle^{- \sigma}   \chi_{ \sim 2^j}(|\tilde{\lambda}|) g(\tilde{\lambda}) =: \tilde{k}(\tilde{\lambda}) g(\tilde{\lambda})$.
Now for the first integral on the right we write
\begin{align}
F_{++}(\lambda, \lambda)  \int \frac{f(\tilde{\lambda})}{\lambda^2 -  \tilde{\lambda}^2} ~d \tilde{\lambda} =&~ \lambda^{-1}F_{++}(\lambda, \lambda) \pi ~p.v. \int \frac{f(\tilde{\lambda})}{\pi(\lambda-  \tilde{\lambda})} ~d \tilde{\lambda}\\
& - \lambda^{-1}F_{++}(\lambda, \lambda)   \int \frac{ f(\tilde{\lambda})}{\lambda-  \tilde{\lambda}}\cdot \frac{\tilde{\lambda}}{\lambda + \tilde{\lambda}} \ ~d \tilde{\lambda},
\end{align}
where the first integral is a Hilbert transform and $ F_{++}(\lambda, \lambda) = O(\lambda)$ by Proposition \ref{Kerne-prop} .
For the second integral we have 
\begin{align}
\int \frac{ f(\tilde{\lambda})}{\lambda-  \tilde{\lambda}}\cdot \frac{\tilde{\lambda}}{\lambda + \tilde{\lambda}} \ ~d \tilde{\lambda} =  \f12 \int \frac{ f(\tilde{\lambda})}{\lambda-  \tilde{\lambda}} ~d \tilde{\lambda} +   \f12 \int \frac{ f(\tilde{\lambda})}{\lambda +  \tilde{\lambda}} ~d \tilde{\lambda}.
\end{align}
Here  we may use Schur's test in the latter integral, i.e. we note 
\begin{align}
&\sup_{\tilde{\lambda}} \big|\int \frac{ \tilde{k}(\tilde{\lambda})}{\lambda +  \tilde{\lambda}} d\lambda \big| + \sup_{\lambda} \big|\int  \frac{ \tilde{k}(\tilde{\lambda})}{\lambda +  \tilde{\lambda}} d \tilde{\lambda}\big| \lesssim 
\big|\int_{2^{j-1}}^{2^{j+1}} \frac{d\lambda}{\lambda}\big| \sim |\log(\frac{2^{j+1}}{2^{j-1}})| \sim 1.
\end{align}
For the second integral in \eqref{das-int} we use
\begin{align}
\int \frac{F_{++}(\lambda, \tilde{\lambda}) - F_{++}(\lambda, \lambda)}{\lambda^2 -  \tilde{\lambda}^2} f(\tilde \lambda)~d \tilde{\lambda} =  \int \frac{\partial_{\tilde{\lambda}}F_{++}(\lambda, \eta) }{\lambda +  \tilde{\lambda}} f(\tilde \lambda)~d \tilde{\lambda}.
\end{align}
By Proposition \ref{Kerne-prop}, the latter kernel is a  linear combination of
\begin{align}
O(\log(\lambda)),~ ~O(\log(\tilde{\lambda})),~ ~O(\frac{\lambda^{\f12} + \tilde{\lambda}^{\f12}}{\lambda + \tilde{\lambda}}),
\end{align}
where  $ \log(\lambda) = O(1) + \log(\tilde{\lambda})$ since $ (\lambda, \tilde{\lambda}) \in Q_j$. Then we use again  Schur's test (or estimate the Hilbert-Schmidt norm for the logarithmic expression), which works uniform in $ j \leq k$  for the upper bound. 

We  may alternatively  conclude from Proposition \ref{Kerne-prop} that they define a \emph{singular kernel} for some $ 0 <  \sigma \leq 1,~ C_0 > 0$ fixed, i.e. 
\begin{align}
&(1)~~| \tilde{K}_0(\lambda, \tilde{\lambda})| \leq C_0 |\lambda - \tilde{\lambda}_0 |^{-1}\\
&(2)~~| \tilde{K}_0(\lambda, \tilde{\lambda}) - \tilde{K}_0(\lambda, \eta)| \leq C_0 |\tilde{\lambda} - \eta|^{\sigma} |\lambda - \eta |^{-1 - \sigma},~~\text{if}~~|\tilde{\lambda} - \eta | < \f12 |\lambda - \eta|\\
&(3)~~| \tilde{K}_0(\lambda, \tilde{\lambda}) - \tilde{K}_0(\xi, \tilde{\lambda})| \leq C_0 |\lambda - \xi|^{\sigma} | \xi - \tilde{\lambda} |^{-1 - \sigma},~~\text{if}~~| \lambda - \xi | < \f12 |\xi - \tilde{\lambda}|.
\end{align}
Further the operator $ T : \mathcal{S} \to \mathcal{S}'$ with $ Tg(\lambda) = \int \tilde{K}_0(\lambda, \tilde{\lambda}) g(\tilde{\lambda})~d \tilde{\lambda}$ satisfies for any bump function $ \Phi(\lambda)$ on $\R$ \begin{align}
\sup_{\lambda_0 \in \R} \|  T(\Phi((\cdot - \lambda_0)\slash \delta))   \|_{L^2(d \lambda)} \leq C_0 \sqrt{\delta},~~\sup_{\lambda_0 \in \R} \|  T^*(\Phi((\cdot - \lambda_0)\slash \delta))   \|_{L^2(d \lambda)} \leq C_0 \sqrt{\delta},
\end{align}
after taking $C_0> 0$ large enough (independent of $ \Phi$ and $ \delta > 0$). To be precise $\Phi$ is supposed to range over \emph{normalized bump functions}, see \cite[page 293]{Stein}. Hence the operators with kernel $\chi_{Q \cap Q_j} [\tilde{K}_0]_{++},~\chi_{Q \cap Q_j} [\tilde{K}_0]_{--},~ j \leq k$ are bounded  on $L^2(d \lambda)$ by the T1 theorem, see \cite[Chapter 7, Theorem 3]{Stein}. From this, boundedness of the kernel
$$ \sum_{j \leq k} \chi_{Q \cap Q_j} [\tilde{K}_0]_{\pm +}(\lambda, \tilde{\lambda}) $$
follows by orthogonality. Now for the non-singular Matrix entries near the diagonal, i.e.
$$ \chi(Q\cap B_1) [\tilde{K}_0]_{+-},~ \chi(Q \cap B_1) [\tilde{K}_0]_{-+},$$
we essentially need to consider 
\[
\int \frac{F_{+-}(\lambda, \tilde{\lambda})}{\lambda^2 + \tilde{\lambda}^2} f(\tilde{\lambda})~d \tilde{\lambda},
\]
which is directly treated with Schur's test as above using Proposition \ref{Kerne-prop}.\\[2pt]
\underline{Step~2}.~We look at $\chi_{Q \cap B_1^C }[\tilde{K}_0]_{\pm +}(\lambda, \tilde{\lambda})$, i.e. off-diagonal frequencies  in  $Q$. Here we require a more careful expansion of $F_{\pm + }(\lambda, \tilde{\lambda})$ at low frequencies obtained by inspection of the proof of Proposition \ref{Kerne-prop}. In particular we have by symmetry 
\begin{align}\label{expansiooons}
F_{\pm +}(\lambda, \tilde{\lambda}) = h(\lambda, \tilde{\lambda}) + \tilde{\lambda}\Lambda(\tilde{\lambda})( 1 + O(\tilde{\lambda}^{\f12}))\\ \label{expansiooons2}
F_{\pm +}(\lambda, \tilde{\lambda}) = \tilde{h}(\lambda, \tilde{\lambda}) + \lambda \Lambda(\lambda)( 1 + O(\tilde{\lambda}^{\f12}))
\end{align}
where\begin{itemize}\setlength\itemsep{3pt}
\item[$\circ$] $ |\Lambda(\lambda)| \sim  |s(\lambda)| + |r(\lambda)|$, thus $ \Lambda(\lambda) = O(1)$ is the leading oscillatory term in the expansion of $e_{\pm}(y, \lambda )$,
\item[$\circ$] $h(\lambda, \tilde{\lambda}),~ \tilde{h}(\lambda, \tilde{\lambda})$ are of class $ O(\lambda^{\f32}) + O(\tilde{\lambda}^{\f32})$ and well behaved in light of a Schur test.
\end{itemize}
Then for \[
\int \chi_{Q \cap B_1^C }[\tilde{K}_0]_{\pm +}(\lambda, \tilde{\lambda}) g(\tilde{\lambda})~d\tilde{\lambda} \sim \int_{(\lambda, \cdot) \in Q \cap B_1^C} \frac{F_{\pm +}(\lambda, \tilde{\lambda})}{\lambda^2 - \tilde{\lambda}^2} f(\tilde{\lambda})~d \tilde{\lambda},
\]
we focus on $ \tilde{\lambda} \Lambda(\tilde{\lambda}),~ \lambda \Lambda(\lambda)$ and distinguish $|\lambda \slash \tilde{\lambda}| \sim 1,~  |\lambda \slash \tilde{\lambda}| \ll1$ and $|\lambda \slash \tilde{\lambda}| \gg1 $. In the first case we may directly use the Hilbert transform via
$
\int \frac{ \Lambda(\tilde{\lambda})}{\lambda - \tilde{\lambda}} f(\tilde{\lambda})~d \tilde{\lambda} 
$
since   $ \tilde{\lambda}(\tilde{\lambda} + \lambda)^{-1} \sim 1$. In the latter two cases we use $ |\lambda^2 - \tilde{\lambda}^2| \lesssim \lambda^2 + \tilde{\lambda}^2 $ and then proceed as for $ F_{+ -}, F_{-+}$, i.e. we look at 
\[
\int \frac{F_{+-}(\lambda, \tilde{\lambda})}{\lambda^2 + \tilde{\lambda}^2} f(\tilde{\lambda})~d \tilde{\lambda} = \int \frac{F_{+-}(\lambda, \tilde{\lambda})}{\lambda^2 + \tilde{\lambda}^2} f(\tilde{\lambda}) \chi\{ |\tilde{\lambda}| \lesssim |\lambda| \}~d \tilde{\lambda} +  \int \frac{F_{+-}(\lambda, \tilde{\lambda})}{\lambda^2 + \tilde{\lambda}^2} f(\tilde{\lambda}) \chi\{ |\tilde{\lambda}| \gtrsim  |\lambda| \}~d \tilde{\lambda}.
\]
Thus with the leading terms $ \tilde{\lambda} \Lambda(\tilde{\lambda}),~ \lambda \Lambda(\lambda)$ in \eqref{expansiooons}, \eqref{expansiooons2} we use Schur's test 
\begin{align}
& \sup_{\lambda} \int \frac{\tilde{\lambda} \Lambda(\tilde{\lambda})}{\lambda^2 + \tilde{\lambda}^2}\chi(|\tilde{\lambda}| \lesssim |\lambda|) d \tilde{\lambda} \lesssim ~\sup_{\lambda} \frac{1}{\lambda^2}\int_{|\tilde{\lambda}| \lesssim |\lambda|} \tilde{\lambda} d\tilde{\lambda} \lesssim 1,\\
& \sup_{\tilde{\lambda}} \int \frac{\tilde{\lambda} \Lambda(\tilde{\lambda})}{\lambda^2 + \tilde{\lambda}^2}\chi(|\tilde{\lambda}| \lesssim |\lambda|) d \lambda \lesssim~ \sup_{\tilde{\lambda}} \tilde{\lambda}\int_{|\tilde{\lambda}| \lesssim |\lambda|} \frac{1}{\lambda^2}d\lambda \lesssim 1,\\
& \sup_{\lambda} \int \frac{\lambda \Lambda(\lambda)}{\lambda^2 + \tilde{\lambda}^2}\chi(|\tilde{\lambda}| \gtrsim |\lambda|) d \tilde{\lambda} \lesssim ~\sup_{\lambda} \lambda \int_{|\tilde{\lambda}| \gtrsim |\lambda|} \frac{1}{ \tilde{\lambda}^2} d\tilde{\lambda} \lesssim 1,\\
& \sup_{\tilde{\lambda}} \int \frac{\tilde{\lambda} \Lambda(\tilde{\lambda})}{\lambda^2 + \tilde{\lambda}^2}\chi(|\tilde{\lambda}| \gtrsim |\lambda|) d \lambda \lesssim ~\sup_{\tilde{\lambda}} \frac{1}{\tilde{\lambda}^2}\int_{|\tilde{\lambda}| \gtrsim |\lambda|} \lambda d\lambda \lesssim 1.
\end{align}
\underline{Step~3}.~We take a look at diagonal frequencies $ \sum_{j  \geq k} \chi_{Q^C \cap Q_j} [\tilde{K}_0]_{\pm +}(\lambda, \tilde{\lambda}) $  in  $Q^C$ and proceed similar as in the first step. For $(\lambda, \tilde{\lambda}) \in Q^C \cap Q_j$ we have $ |\lambda \slash \tilde{\lambda}| \sim1 $ whence $\langle \lambda\rangle ^{\sigma  } \langle \tilde{\lambda}\rangle^{- \sigma } \sim 1$. We write as above (for $\lambda \sim 2^j$, the case $\lambda < 0 $ works analogously)
\begin{align}
\int \frac{ \lambda F_{++}(\lambda, \tilde{\lambda})}{\lambda^2 - \tilde{\lambda}^2} f(\tilde{\lambda}) d \tilde{\lambda} =& F_{++}(\lambda, \lambda) \frac{1}{2 \lambda}\big[  \int  \frac{f(\tilde{\lambda})}{\lambda - \tilde{\lambda}}~d \tilde{\lambda} +   \int  \frac{f(\tilde{\lambda})}{\lambda + \tilde{\lambda}}~d \tilde{\lambda} \big]\\ \nonumber
& +  \int \frac{\partial_{\tilde{\lambda}} F_{+ + }(\lambda, \eta)}{\lambda + \tilde{\lambda}}~d \tilde {\lambda}
\end{align}


and use $ \partial_{\tilde{\lambda}}F_{+ -}(\lambda, \tilde{\lambda}) = O(1)$ by Proposition \ref{Kerne-prop} in Schur's test. For $ F_{+ -}(\lambda, \tilde{\lambda}), F_{- +}(\lambda, \tilde{\lambda})$ we similarly  verify Schur's condition 
\begin{align}
&\sup_{\lambda} \int \frac{d \tilde{\lambda}}{\lambda^2 + \tilde{\lambda}^2} \lesssim | \sup_{\lambda} \int_{\tilde{\lambda}\sim 2^j} \frac{ 1}{ \tilde{\lambda}^2} d\tilde{\lambda}| \sim \sup_{j \geq k} 2^{-j} ,\\
&\sup_{\tilde{\lambda}} \int \frac{d\lambda}{\lambda^2 + \tilde{\lambda}^2} \lesssim | \sup_{\tilde{\lambda}} \int_{\lambda \sim 2^j} \frac{1}{ \lambda^2}d \lambda| \sim \sup_{j \geq k} 2^{-j}.
\end{align} 
\underline{Step~4}.~Finally we consider off-diagonal frequencies $\chi_{Q^C \cap B_2^C }[\tilde{K}_0]_{\pm +}(\lambda, \tilde{\lambda})$  in  $Q^C$. Here it suffices to calculate the Hilbert-Schmidt norm in $Q^C$ since by Proposition \ref{Kerne-prop} for any $N \in \Z_+$
\begin{align*}
F_{\pm +}(\lambda, \tilde{\lambda}) \leq C (1 + |\lambda|)^{-N} (1 + |\tilde{\lambda}|)^{-N},\\
F_{\pm -}(\lambda, \tilde{\lambda}) \leq C (1 + |\lambda|)^{-N} (1 + |\tilde{\lambda}|)^{-N}.
\end{align*}
~~\\
Now for estimating $ \mathcal{K} = \mathcal{K}_0 + \tilde{\mathcal{K}}$, we note $\tilde{\mathcal{K}}$ simply consists of the remaining components $ \mathcal{K}_{dd}, \mathcal{K}_{cd}$ and $\mathcal{K}_{dc}$. Clearly $ \mathcal{K}_{dd}$ is bounded and  $ \mathcal{K}_{cd}(\lambda)$ needs to be integrated. Then integration by parts with $ y \partial_y$ and changing the order of integration in $ \mathcal{K}_{dc} $, shows that both, $\mathcal{K}_{dc}$ and $ \mathcal{K}_{cd}$, involve integrating Fourier coefficients of Schwartz functions.  For those we have  decay of arbitrary order by Lemma \ref{Fourier-decay}.
\end{proof}
\begin{Rem} The proof shows we may gain $\langle \lambda \rangle$, i.e. continuity $ \mathcal{K}_0 : L^{2, \sigma +1} \to L^{2, \sigma }$ if we estimate the integral in Step 3 
\begin{align}
\lambda\int \frac{ \lambda F_{++}(\lambda, \tilde{\lambda})}{\lambda^2 - \tilde{\lambda}^2} f(\tilde{\lambda}) d \tilde{\lambda} 
\end{align}
rather directly against Hilbert transformations of $|f(\tilde{\lambda})|$ for instance splitting $ \chi(\lambda \geq \tilde{\lambda}),~ \chi(\lambda \leq \tilde{\lambda})$. However we don't need this property.
\end{Rem}
\section{Spectral properties of the linearized NLS flow} \label{sec:spec}

\subsection{The operator $ \mathcal{L}_W$: General properties }\label{sec:linop}
We  first recall general properties and a coercivity estimate for $ \mathcal{L}_W$  with spectral consequences obtained in the work of Duyckaerts-Merle \cite{DM} .
In dimension $ d \geq 3$ the ground state solution of \eqref{NLSequation} is given by
\begin{align}
W(x) = \bigg( 1 + \f{|x|^2}{ d(d-2)}\bigg)^{\f{2-d}{2}},~ x \in \R^d
\end{align}
Splitting  $ u(t,r) $ in \eqref{NLSequation} into imaginary and real parts, we denote 
by
$$ \mathcal{L}_W= \begin{pmatrix}
0 &- \mathcal{L}_W^{(2)}\\
\mathcal{L}_W^{(1)} & 0
\end{pmatrix},~~D(\mathcal{L}_W) = H^2(\R^d, \C) \subset L^2(\R^d, \C)$$
the linearized operator of \eqref{NLSequation} at $ W$.  Here we have \begin{align}
\mathcal{L}_W^{(1)} = -\Delta - p_c W(x)^{p_c -1},~~ \mathcal{L}_W^{(2)} = -\Delta - W(x)^{p_c -1}
\end{align}
with $ D(\mathcal{L}_W^{(1)} )  = D(\mathcal{L}_W^{(2)})  = H^2 \subset L^2 $  and $ p_c = \frac{d +2}{d-2}$ is the scaling critical index. 
We note that a simple calculation gives
\begin{align}
W= W_d \in \begin{cases}
H^1 & \text{if}~ d \geq 5\\
\dot{H}^1 & \text{if}~ d \geq 3
\end{cases}
\end{align}
Clearly
$ \nabla W \in L^2 $ is an eigenvector (by translation invariance)  removed for the restriction to the   radial subspace $ H^2_{rad} \subset L^2_{rad}$. The generalized problem $ \mathcal{L}_Wv = 0$ is solved by the radial resonances
\[
i W =  \frac{d}{d \theta}\big(e^{i \theta}W\big)_{|_{\theta = 0} },~~~ W_1 = \big( \frac{d-2}{2} + r \partial_r \big)W =  \partial_{\lambda }W(\lambda \cdot )_{|_{\lambda = 1} },
\]
derived from remaining invariances. The operator $ \mathcal{L}_W$ can be interpreted as a form operator (cf.  \cite[Section 5.1]{DM}) on the (real) Hilbert space $ \dot{H}^1(\R^d, \C) $ via the quadratic forms ($f = f^{(1)} + i f^{(2)}$)
\begin{align*}
B(f,g) =~& \f12 \int (\nabla f^{(1)} \nabla g^{(1)} + \nabla f^{(2)} \nabla g^{(2)}) ~dx - \f12 \int W^{p_c -1}(p_c  f^{(1)}  g^{(1)}  +  f^{(2)} g^{(2)})~dx\\
=~&\f12 Im\big( \int (\mathcal{L}_Wf) \bar{g}~dx\big),
\end{align*}
where the latter expression is distributional.
\begin{Rem}\label{Remark}
We note that $ B(f,g) $ is well defined for $ f,g \in \dot{H}^1$ by the Sobolev inequality and the decay $ W^{p_c -1}(r) \sim r^{-4}$ as $ r \to \infty$.
\end{Rem}
The relevance of $ Q(f) = B(f,f) $ for the near threshold dynamics is highlighted in \cite{DM}. Let $ f \in \dot{H}^1(\R^d, \C)$, then
\[
E(W + f ) = E(W) + Q(f)  + O(\| f\|_{\dot{H}^1}^3).
\]
\begin{Lemma}\emph{(\cite[Section 5.3, Lemma 5.1, 5.2]{DM})}\label{main-coerc3} $ \mathcal{L}_W $ has two eigenfunctions $ \mathcal{Y}_+, \mathcal{Y}_- \in \mathcal{S}(\R^d)$  
\[
\mathcal{L}_W \mathcal{Y}_+ = \kappa \mathcal{Y}_+,~~\mathcal{L}_W \mathcal{Y}_- = - \kappa \mathcal{Y}_-,~~~ \kappa \in (0, \infty),~~ \mathcal{Y}_+ = \overline{\mathcal{Y}_-}
\]
Further there exists $ c > 0$ such that
\[
Q(f) \geq c \| f\|_{\dot{H}^1}^2~~\forall f \in G_{\perp},
\]
where $ G_{\perp } = \{ f \in \dot{H}^1~|~ \langle  W_1, f \rangle_{\dot{H}^1}  =   \langle i W, f \rangle_{\dot{H}^1}  = B(\mathcal{Y}_+, f) = B(\mathcal{Y}_-, f) = 0 \}.$
\end{Lemma}
The linearized operator $ \mathcal{L}_W $ has the following spectral properties.  The essential spectrum $ \sigma_{ess}(\mathcal{L}_W) = i\R$ by relative compactness of $ W^{p_c -1}$ and the spectrum $ \sigma(\mathcal{L}_W) \cap \R = \{ -\kappa, 0, \kappa \}$  as a consequence of Lemma \ref{main-coerc3} (cf. \cite[Corollary 5.3]{DM}).\\[5pt] 
\emph{Schr\"odinger operator}. ~For the linearized evolution problem of \eqref{NLSequation} involving $ \mathcal{L}_W$,  it is common (see \cite{Busl-Per}, \cite{K-S-stable}, \cite{OP}) to use  the matrix Schr\"odinger operator 
\begin{align}\label{op-prob}
& H = - \Delta \sigma_3   + V(r) ,~~~ V(r) = \begin{pmatrix}
V_1(r) & V_2(r)\\
- V_2(r) & -V_1(r)
\end{pmatrix},~~ \sigma _3 = \begin{pmatrix}
1 & 0 \\
0& -1
\end{pmatrix},\\[4pt] \nonumber
&D(H) = H^2(\R^d, \C^2) \subset L^2(\R^d, \C^2),~~ H_0 = - \Delta \sigma_3
\end{align}
where in Section \ref{sec:approx} - Section \ref{sec:linearized} we have
\begin{align}\label{schropote}
V_1(r) = -  3 W^{4}(r) ,~ ~V_2(r) = -  2 W^{4}(r).
\end{align}
This corresponds to the linearization of \eqref{NLSequation} around $W$ as a system 
\begin{align}\label{andere-lin2} i \partial_t w + H_0 w + \tilde{R}(w) = 0,~~ \tilde{R}(w) = \begin{psmallmatrix}
R(w)\\[2pt]
-\overline{R(w)}
\end{psmallmatrix},~~w = \begin{psmallmatrix}
u \\[2pt]
\bar{u}
\end{psmallmatrix},
\end{align} 
where $ R(w)$ is the nonlinearity. If $ \lambda \in \R $ is an eigenvalue of $\mathcal{L}_W$, then $ - i \lambda $ is an eigenvalue of $  H$ and vice versa. The eigenfunctions for $H$ with eigenvalue on $i \R$ are  of the form  $\tilde{f} = \begin{psmallmatrix}
f\\
\bar{f}
\end{psmallmatrix}$.
Hence by the remark above (Corollary 5.3 in \cite{DM})
$$ \sigma(H) \subset \R \cup~ i \R,~~ \sigma_{ess}(H) = \R,~~ \sigma(H) \cap i \R = \{- i \kappa, 0, i \kappa \},$$
where $ \{ \pm i \kappa\}$ are simple eigenvalues with eigenfunctions $  \phi_+ = \sigma_1 \phi_-  =  \overline{\phi_- } \in \mathcal{S}(\R^d)$. We define the bilinear  form
\begin{align} \label{bilin}\mathcal{B}(f, g ) = \langle  H f, \sigma_3 g \rangle_{L^2},~~f,g \in \dot{H}^1(\R^d, \C^2),\end{align}
in a distributional sense.  
Expressing the zero resonances $ i W, W_1$ for $\mathcal{L}_W$  under the linearization of \eqref{andere-lin2}, we recall from Section \ref{subsec: Scat}
\[
\mathcal{W}_0(r) = \f{1}{\sqrt{3}}W(r)\begin{pmatrix}
1\\[1pt]-1
\end{pmatrix},~~~ \mathcal{W}_1(r) = - \f{2}{\sqrt{3}}W_1(r)\begin{pmatrix}
1\\[1pt]1
\end{pmatrix},
\]
where the factors match, see Section \ref{subsec: Scat}, the asymptotics of the Jost functions $f_1(0, \cdot), f_3(0 , \cdot)$.
The following is the adaption of Lemma \ref{main-coerc3} for $ H = - \Delta \sigma_3 + V(r)$ in dimension $  d = 3$.
\begin{Lemma} \label{mod-Dm-Lem}There  exists $ c > 0$ such that
\[
\mathcal{B}(f, f ) \geq c \| f\|_{\dot{H}^1}^2~~\forall f \in G_{\perp},
\]
where $ G_{\perp } = \{ f \in \dot{H}^1_{\text{rad}}(\R^3, \C^2)~|~ \langle  \mathcal{W}_0, f \rangle_{\dot{H}^1}  =   \langle \mathcal{W}_1 , f \rangle_{\dot{H}^1}  = \mathcal{B}( \phi_+, f) = \mathcal{B}(\phi_-, f) = 0 \}.$
\end{Lemma}
~~\\
We now reduce the generalized radial 3D eigenvalue problem $ (H \pm \lambda^2 )f = 0$ as in Section \ref{subsec: Scat} via $\tilde{f}(r) = r  f(r)$ to  odd solutions of the 1D problem
\begin{align}
\mathcal{H} \tilde{f}(r) = - \partial_r^2 \sigma_3 \tilde{f}(r) + V(r) \tilde{f}(r) = \mp \lambda^2\tilde{f}(r).
\end{align}

\subsection{Spectrum of  1D Schr\"odinger systems}
Let us consider the following 1D Schr\"odinger operators for $ \mu \geq 0$

\begin{align} \label{subseq:1dspe}
&\mathcal{H}_0 = (- \partial_x^2 + \mu)\sigma_3 =   \begin{pmatrix}
- \partial_x^2 + \mu & 0\\
0&  \partial_x^2 - \mu
\end{pmatrix},~ V(x) = \begin{pmatrix}
V_1(x) & V_2(x)\\
- V_2(x) & -V_1(x)
\end{pmatrix},~~~~ 
\\[6pt] \nonumber
& \mathcal{H} = \mathcal{H}_0 + V(x),~~~D(\mathcal{H}) = W^{2,2}_{\text{odd}}(\R, \C^2) \subset L^2_{\text{odd}}(\R, \C^2),
\end{align}
~\\
where $ V_1, V_2$ is as above in \eqref{schropote}. We recall there clearly holds
\begin{itemize}  \setlength\itemsep{5pt}
\item[(a)]~$\sigma(\mathcal{H}) \subset \R \cup~ i \R$,~$\sigma_{ess}(\mathcal{H}) = (- \infty, -\mu] \cup [ \mu,  \infty)$,
\item[(b)] $\sigma_d(\mathcal{H}) = \{-i \kappa, 0, i \kappa\}$ if $\mu > 0$ and $ \sigma_d(\mathcal{H}) = \{ \pm i \kappa\}$ otherwise. Here $ \pm i \kappa $ are simple eigenvalues with odd eigenfunctions $ x\cdot \phi_+(x), ~x \cdot\phi_-(x)$.
\item[(c)] If $ \mu = 0$ then $ \lambda = \mu = 0$ is not an eigenvalue and $ x \cdot \mathcal{W}_0(x),~ x \cdot \mathcal{W}_1(x)$ are  two independent odd solutions of $ \mathcal{H}f = 0$ in $ L^{\infty}(\R)\backslash L^2(\R)$.
\end{itemize}
~~\\
The main purpose of this Section is concerned with the following three Propositions on the spectrum of $\mathcal{H}$. For the validity of the first Proposition we will give numerical evidence in the supplemented document \cite{Schmid}.
\begin{Prop}[numerical] \label{subsec: embe}The operator $ \mathcal{H} = - \partial_x^2 \sigma_3 + V(x)$ has  no embedded eigenvalues in $\R$.
\end{Prop}
\begin{proof}[Verification]
The eigenvalue problem $ \mathcal{H}h \pm \lambda^2h = 0$ may be written into the system
\begin{align}\label{subsec:syss}
\begin{pmatrix}
	- \partial_x^2- 3 W^4 & - 2 W^4\\
	2 W^4 & \partial_x^2 + 3 W^4
\end{pmatrix} \cdot \begin{pmatrix}
	h_1\\
	h_2
\end{pmatrix} = \mp \lambda^2 \begin{pmatrix}
	h_1\\
	h_2
\end{pmatrix} 
\end{align}
where $ h(x) = h_1 \underline{e} + h_2 \sigma_1 \underline{e}$. Further this system  has the invariance $ \tilde{h}(x) =  h(-x),~\tilde{h}(x) = - h(-x)$ and is uniquely solved with initial values at $ x = 0$, i.e. we set
\begin{align} \label{subsec: Dat}
(h_1(0), h_2(0))^t = (\beta_1, \beta_2)^t,~~~(h_1'(0), h_2'(0))^t = (\gamma_1, \gamma_2)^t.
\end{align}
Since the $(\mp)$ cases in \eqref{subsec:syss} are symmetric via $h \mapsto \sigma_1 h$ we can restrict to the $(+)$ case in the following. Now clearly from Section \ref{subsec: Scat}  a unique solution in the $(+)$case  in $ L^2(\R)$ must be such that
\[
h(x) = f_3(\lambda, x) \sim e^{- |\lambda| x},~~~ x \to \infty,~~~h(x) = g_3(\lambda, x) \sim e^{ |\lambda| x},~~ x \to - \infty,
\]
and  $\sigma_1 f_3,\sigma_1 g_3 $ replacing $ f_3, g_3$ in the $(-)$case respectively. Hence,  considering \eqref{subsec: Dat} for $f_3$ and $g_3$, the only such possibilities for \eqref{subsec:syss} are
\begin{itemize}
\item[(1)] $f_3(\lambda, 0) = 0$. Here   $h(x) = f_3(\lambda, x) = - h(-x)$ is an odd eigenfunction.
\item[(2)] $ f_3(\lambda, 0) \neq 0$ and $f_3'(\lambda, 0) = 0$. Here $ h(x) = f_3(\lambda, x)  = h(-x)$ is an even eigenfunction.
\end{itemize}
We restrict to odd functions and proceed with (1). Hence one possibility is to show that either $\lim_{x \to 0^+}h_1(x) \neq 0 $ or $\lim_{x \to 0^+}h_2(x) \neq 0$, for which we give evidence in a discrete sample set of $\lambda > 0$ in \cite{Schmid} .\\[3pt]
For a second possibility, depending only on the value of $L^2$ norms,  we identify $h(x) \sim |x|^{-1} e^{- |\lambda| |x|}$ as $ |x| \to \infty$ with a 3D radial eigenfunction for $H$ at $ \lambda $. We decompose in  $\dot{H}^1(\R^3)$ 
\[
g(x) := h(x) -  \frac{\langle \Delta \mathcal{W}_0, h \rangle_{L^2} }{\|  \mathcal{W}_0 \|^2_{\dot{H}^1}} \mathcal{W}_0(x)-  \frac{\langle \Delta \mathcal{W}_1, h \rangle_{L^2} }{\|  \mathcal{W}_1 \|^2_{\dot{H}^1}}\mathcal{W}_1(x),
\]
where $ \nabla \mathcal{W}_0, \nabla \mathcal{W}_1 \sim r^{-2}$ as $ r = |x| \to \infty$. Especially, by the exponential decay of $h$ we infer $\nabla g \neq 0$ 
and further check that $g(x) $ satisfies the assumption $ g \in G_{\perp}$ of Lemma \ref{mod-Dm-Lem}. Therefore we obtain 
\begin{align*}
 \lambda^2 \langle h,  \sigma_3 h \rangle_{L^2} =  \langle H h,  \sigma_3 h \rangle_{L^2} =  \langle H g,  \sigma_3 g \rangle_{L^2} \geq c  \| \nabla g \|_{L^2}^2 > 0.
\end{align*}
In particular this implies 
\[
\langle h,  \sigma_3 h \rangle_{L^2} = \int |h_1|^2\;dx - \int |h_2|^2\;dx > 0.
\]
In \cite{Schmid} we will also numerically compute this value to be negative. 
\end{proof}
\begin{Prop} \label{No-embedded-for-dist-op} For  $ 0 < \mu \ll1 $ small enough the operators $ \mathcal{H} = (- \partial_x^2 + \mu) \sigma_3 + V(x) $ have no embedded eigenvalues in $ \sigma_{\text{ess}}(\mathcal{H}) = (- \infty, -\mu] \cup [\mu , \infty)$.  
\end{Prop}
\begin{proof}  We likewise write the  eigenvalue problem $ \mathcal{H}h - \lambda^2h = 0$  into the system

\begin{align}\label{subsec:syss2}
\begin{pmatrix}
	- \partial_x^2+ \mu	- 3 W^4 & - 2 W^4\\
	2 W^4 & \partial_x^2 - \mu + 3 W^4
\end{pmatrix} \cdot \begin{pmatrix}
	h_1\\
	h_2
\end{pmatrix} =  \lambda^2  \begin{pmatrix}
	h_1\\
	h_2
\end{pmatrix} 
\end{align}
Then, if prove by contradiction, there exist pairs $(\mu_n, \lambda_n)$ with $ \mu_n \to 0^+$ and $\mathcal{H}h_n = \lambda_n^2h_n$ where $ \mu_n \leq \lambda_n$. Note here $h_n = f_3(\lambda_n, \mu_n)$ is the odd eigenfunction.\\[2pt]
If we further assume $ \lambda_n^2 \geq c > 0$ the continuous dependence of $f_3(\lambda_n, \mu_n)$ gives an odd eigenfunction $f_3(\lambda, 0)$ at $\mu = 0$  with $\lambda^2 \geq c$ accumulation point of $(\lambda_n^2)_n$ and this contradicts Proposition \ref{subsec: embe}.\\[2pt]
Hence we assume $\lambda_n^2 \to 0^+$ and let $(\mu, \lambda)$ be any such pair with $ 0< \mu < \lambda^2 \ll1$. The system

\begin{align}\label{subsec:syss3}
\begin{pmatrix}
	- \partial_x^2	- 3 W^4 & - 2 W^4\\
	2 W^4 &  \partial_x^2  + 3 W^4
\end{pmatrix} \cdot \begin{pmatrix}
	h_1\\
	h_2
\end{pmatrix} =  \begin{pmatrix}
	\lambda^2 - \mu & 0 \\
	0 & \lambda^2 + \mu
\end{pmatrix} \cdot   \begin{pmatrix}
	h_1\\
	h_2
\end{pmatrix} 
\end{align}
may now be understood as a perturbation of the  operator on the left side. In particular we note
$$ \mathcal{W}(x) = x \cdot \mathcal{W}_0(x) = \begin{pmatrix} \tilde{W}(x)\\
- \tilde{W}(x)
\end{pmatrix}, ~~\Lambda\mathcal{W}(x) = x \cdot \mathcal{W}_1(x) = \begin{pmatrix}
\tilde{W}_1(x)\\
\tilde{W}_1(x)
\end{pmatrix}, $$  are independent odd solutions of	\begin{align}\label{subsec:syss4}
\begin{pmatrix}
	- \partial_x^2	- 3 W^4 & - 2 W^4\\
	2 W^4 &  \partial_x^2  + 3 W^4
\end{pmatrix} \cdot	\begin{pmatrix}
	h_1\\
	h_2
\end{pmatrix} = \begin{pmatrix}
	0\\
	0
\end{pmatrix},
\end{align}
and can be completed to a fundamental base by functions
$$ \Theta_1 = \begin{pmatrix}
\theta_1\\
\theta_1
\end{pmatrix} , \Theta_2  =  \begin{pmatrix}
\theta_0\\
-\theta_0
\end{pmatrix} ,~~\theta_0(x) = O(x), \theta_1(x) = O(x),~~|x| \to \infty.$$
We now consider the conjugated system to \eqref{subsec:syss3} 
\begin{align}\label{subsec:syss5}
\begin{pmatrix}
	0 & L_-\\
	-	 L_+ &  0 
\end{pmatrix} \cdot \begin{pmatrix}
	\tilde{h}_1\\
	\tilde{h}_2
\end{pmatrix} =  \begin{pmatrix}
	\lambda^2  & \mu\\
	\mu &  \lambda^2
\end{pmatrix} \cdot   \begin{pmatrix}
	\tilde{h}_1\\
	\tilde{h}_2
\end{pmatrix} ,
\end{align}
where $ L_- = - \partial_x^2 - W^4,~ L_+ = - \partial_x^2 - 5W^4$ with the fundamental base  at $ \lambda = \mu = 0$
\[\big\{
\begin{pmatrix}
0\\
\tilde{W}
\end{pmatrix},~\begin{pmatrix}
\tilde{W_1}\\
0
\end{pmatrix},~\begin{pmatrix}
0\\
\theta_0
\end{pmatrix},~ \begin{pmatrix}
\theta_1\\
0
\end{pmatrix}\big\}.
\]
\begin{Lemma} \label{subsec:des-Lema}There exist fundamental solutions $ \{ \tilde{h}^{W}(\lambda, \cdot),\tilde{h}^{W_1}(\lambda, \cdot), \tilde{h}^{\theta_0}(\lambda, \cdot), \tilde{h}^{\theta_1}(\lambda, \cdot)\}$ of \eqref{subsec:syss5} such that in an absolute sense
\begin{align}
	&	\tilde{h}^{W}(\lambda,x) = \tilde{W}(x) \sigma_1\underline{e} + x^{-1}\sum_{j = 1}^{\infty} (\lambda x)^{2j} \zeta_j(x),~~ \zeta_j = (\zeta_j^{(1)}, \zeta_j^{(2)})^t,\\
	&	\tilde{h}^{\theta_0}(\lambda,x) = \theta_0(x) \sigma_1\underline{e} + \sum_{j = 1}^{\infty} (\lambda x)^{2j} \eta_j(x),~~ \eta_j = (\eta_j^{(1)}, \eta_j^{(2)})^t.
\end{align}
with similar expansions  for  for $\tilde{h}^{W_1}(\lambda, \cdot)$ as in the first line and  $\tilde{h}^{\theta_1}(\lambda, \cdot) $ as in the second, replacing $\sigma_1 \underline{e}$ by $\underline{e}$. Further 
\begin{align}
	&	| \zeta_j(u) |  \leq \frac{C^j}{(j-1)!} |u|^2 \langle u \rangle^{-1},\\
	&	| \eta_j(u) |  \leq \frac{C^j}{(j-1)!}  \langle u \rangle^{-1},
\end{align}
with the same estimates for the coefficients of $\tilde{h}^{W_1},\tilde{h}^{\theta_1}$.
\end{Lemma}
For the proof of the first two $\tilde{h}^{W}, \tilde{h}^{W_1}$ we make the formal ansatz
\[
\tilde{h}^{W}(\lambda, x) =x^{-1} \sum_{j = 0}^{\infty}\lambda^{2j} \tilde{f}_j(x),~~\tilde{h}^{W_1}(\lambda, x) =x^{-1} \sum_{j = 0}^{\infty}\lambda^{2j} \tilde{g}_j(x),
\]
and justify convergence. We focus on $\tilde{h}^{W}(\lambda, x)$ since the argument for both is the same. Applying the system \eqref{subsec:syss5}, we have
\begin{align*}
&L(x^{-1}\tilde{f}_j) = A_{\mu \lambda} x^{-1}\tilde{f}_{j-1},~~\tilde{f}_{-1} = 0,~~ \tilde{f}_{0}(x) = x \cdot \tilde{W}(x) \sigma_1 \underline{e},\\
& L = 	\begin{pmatrix}
	0 & L_-\\
	-	 L_+ &  0 
\end{pmatrix},~~ A_{\mu \lambda} = 	\begin{pmatrix}
	1 & \frac{\mu}{\lambda^2}\\
	\frac{\mu}{\lambda^2} &  1 
\end{pmatrix},
\end{align*}
where the iteration is solved with boundary values $ \tilde{f}_j(0) = \tilde{f}_j'(0) = 0$. This leads to  
\begin{align}
L_-(x^{-1} \tilde{f}^{(2)}_j) = x^{-1} \tilde{f}^{(1)}_{j-1} + \frac{\mu}{\lambda^2 x} \tilde{f}^{(2)}_{j-1},~~~~L_+(x^{-1} \tilde{f}_j^{(1)}) = - \frac{\mu}{\lambda^2 x} \tilde{f}^{(1)}_{j-1} - x^{-1} \tilde{f}^{(2)}_{j-1}.
\end{align}
Thus 
\begin{align}
&\tilde{f}^{(2)}_j(x) = \int_0^x \frac{x}{s}[\tilde{W}(x) \theta_0(s) - \tilde{W}(s) \theta_0(x)  ](\tilde{f}^{(1)}_{j-1}(s) + \frac{\mu}{\lambda^2 } \tilde{f}^{(2)}_{j-1}(s))~ds,\\
& \tilde{f}_j^{(1)}(x) = - \int_0^x \frac{x}{s}[\tilde{W_1}(x) \theta_1(s) - \tilde{W_1}(s) \theta_1(x)  ] (\frac{\mu}{\lambda^2 } \tilde{f}^{(1)}_{j-1}(s) +  \tilde{f}^{(2)}_{j-1}(s))~ds.
\end{align}
Now we write the Greens functions more explicitly via
\begin{align*}
&[\tilde{W}(x) \theta_0(s) - \tilde{W}(s) \theta_0(x)  ] =  \frac{x(a_0 + a_2s^2) - s(a_0 + a_2 x^2)}{(1 + \frac{x^2}{3})^{\f12} (1 + \frac{s^2}{3})^{\f12}}\\
&[\tilde{W_1}(x) \theta_1(s) - \tilde{W_1}(s) \theta_1(x)  ] = \frac{1}{2} \frac{x(1- \frac{x^2}{3})(b_0 + b_2s^2 + b_4 s^4) - s(1- \frac{s^2}{3})(b_0 + b_2x^2 + b_4 x^4) }{(1 + \frac{x^2}{3})^{\f32} (1 + \frac{s^2}{3})^{\f32}},
\end{align*}
where we use
\[
\theta_0(x) = (a_0 + a_2 x^2) W(x),~~ \theta_1(x) = W^3(x)(b_0 + b_2 x^2 + b_4 x^4).
\]
Since $\tilde{f}_0(x) = O( x^2)$ if $ x \ll1 $ and  $\tilde{f}_0(x) = O( x)$ if $ x \gg1$, we infer inductively
\[
\tilde{f}_j(x)  = O(x^{2j + 2}),~~ x \ll1,~~\tilde{f}_j(x)  = O( x^{2j + 1}),~~ x \gg1,
\]
and $\tilde{f}_j$ is an absolute series  if $ x > -1$. Now for $\tilde{h}^{\theta_0}, \tilde{h}^{\theta_1}$ we make the formal ansatz
\[
\tilde{h}^{\theta_0}(\lambda, x) =  \sum_{j = 0}^{\infty}\lambda^{2j} \tilde{f}_j(x),~~\tilde{h}^{\theta_1}(\lambda, x) = \sum_{j = 0}^{\infty}\lambda^{2j} \tilde{g}_j(x),~~ \tilde{f}_0(x) =  \theta_0(x),~\tilde{g}_0(x) = \theta_1(x).
\]
and justify convergence. As above we infer for higher iterates
\begin{align}
&\tilde{f}^{(2)}_j(x) = \int_0^x [\tilde{W}(x) \theta_0(s) - \tilde{W}(s) \theta_0(x)  ](\tilde{f}^{(1)}_{j-1}(s) + \frac{\mu}{\lambda^2 } \tilde{f}^{(2)}_{j-1}(s))~ds,\\
& \tilde{f}_j^{(1)}(x) = - \int_0^x [\tilde{W_1}(x) \theta_1(s) - \tilde{W_1}(s) \theta_1(x)  ] (\frac{\mu}{\lambda^2 } \tilde{f}^{(1)}_{j-1}(s) +  \tilde{f}^{(2)}_{j-1}(s))~ds.
\end{align}
Hence $\tilde{f}_j$ decays like $x^{2j}$ at $ x =0$ and grows like $ x^{2j + 1}$ as $ x \to \infty$. This establishes the Lemma.\\[3pt]
The solutions of Lemma \ref{subsec:des-Lema} can be transitioned to the system  \eqref{subsec:syss3} via the usual conjugation and we call these solutions
\[
\big\{ \psi^{W}(\lambda, \cdot), \psi^{W_1}(\lambda, \cdot), \psi^{\theta_0}(\lambda, \cdot), \psi^{\theta_1}\big\}.
\]
We now make the the claim that odd eigenfunctions can not belong to the hull
\[
\left< \begin{pmatrix}
\psi_1^{W}(\lambda, \cdot)\\
\psi_2^{W}(\lambda, \cdot)
\end{pmatrix} , \begin{pmatrix}
\psi_1^{W_1}(\lambda, \cdot)\\
\psi_2^{W_1}(\lambda, \cdot)
\end{pmatrix}  \right>.
\]
Assume otherwise, then by the asymptotics of Jost function in Section \ref{subsec: Scat} we  infer
\[
e^{- \sqrt{\mu + \lambda^2}x} \sigma_1 \underline{e} + O(\langle x\rangle^{-2}) = \alpha \psi^{W}(x) + \beta \psi^{W_1}(x),
\]
where the $O(\langle x\rangle^{-2}) $ is uniform in $ \mu, \lambda$. Therefore, restricting to the region where  $ x \sim \delta \lambda^{-1}$ for some $ \lambda^{\f14} \ll \delta \ll1$, we have 
\begin{align}\label{subsec:dde2}
\alpha \begin{pmatrix}
	\tilde{W_1}\\
	\tilde{W_1}
\end{pmatrix} + \beta \begin{pmatrix}
	\tilde{W}\\
	- \tilde{W}
\end{pmatrix} = \begin{pmatrix}
	0\\1
\end{pmatrix} + O(\delta^2).
\end{align}
From the asymptotics on the right as $ x \gg1 $ we have 
\begin{align} \label{subsec:dde}
\alpha + \beta = O(\delta^2),~~ \alpha - \beta = 1 + O(\delta^2).
\end{align}
Expanding the exponential eigenfunction we use
\[
e^{- \sqrt{\mu + \lambda^2}x} \begin{pmatrix}
0\\
1
\end{pmatrix} + O(\langle x\rangle^{-2})  = \begin{pmatrix}
0\\
1
\end{pmatrix} - x \sqrt{\mu + \lambda^2} + O(\delta^2 + \frac{\delta^2 \mu}{\lambda^2}),
\]
and hence from above \eqref{subsec:dde2} and \eqref{subsec:dde} we have
\[
\begin{pmatrix}
0\\
1
\end{pmatrix}+ O(\delta^2) = \begin{pmatrix}
0\\
1
\end{pmatrix} - x \sqrt{\mu + \lambda^2}         + O(\delta^2),
\]
where $  x \sqrt{\mu + \lambda^2}  = \delta \sqrt{\frac{\mu}{\lambda^2} +1} \sim \delta $ leads to a contradiction choosing $(\mu, \lambda)$ and $ \delta \ll1 $ small. Hence the eigenfunction involves linear combinations of $ \psi^{\theta_0}, \psi^{\theta_1}$ which are not vanishing as $ x \to 0$. 
\end{proof} 
\begin{Prop}\label{Thres-is-no-res-dist-op} For $ 0 < \mu \ll1 $ small enough the operators $ (- \partial_x^2 + \mu) \sigma_3 + V(x)$ satisfy $ ( - \mu, \mu) \cap \sigma_d(\mathcal{H}) = \{0\}$ and  the threshold $\pm \mu$ is not a resonance.
\end{Prop}
\begin{proof} We present a modified version of the argument in \cite[Section 9, properties (i) \& (ii)]{K-S-stable}, which  first appeared in similar form in the work \cite{Pere}.\\[2pt] Let $ \chi \in C^{\infty}(\R)$ be a real, even cut-off function  with $ 0 \leq \chi \leq 1$, i.e. we set $ \chi(x) = 1 $ if $ |x| \leq1 $,~ $\chi(x) = 0$ if $ |x| \geq 2$ and  $ \chi_{\epsilon}(x) = \chi(\epsilon x)$. Then assume the intersection is not zero. We  find an eigenvalue $ e \in (0, \alpha^2)$ of $ \mathcal{H}^2$  and hence from the conjugated form 
\[  L_- L_+ \psi = e \psi,~ \psi \neq 0,~ \psi \in L^2(\R),~\text{odd},\]
where $ L_+ = - \partial_x^2 + \mu - 5W^4,~L_- = - \partial_x^2 + \mu - W^4$.
Let us define $P $ to be the orthogonal projection
\[
Pf = f - \frac{\langle f, \chi_{\epsilon}\tilde{W_1} \rangle }{\|  \chi_{\epsilon}\tilde{W_1} \|_{L^2}^2}\chi_{\epsilon}\tilde{W_1} ,~~ f \in L^2_{\text{odd}}(\R)
\]
and the operator $A = P L_+ P $. Clearly we have $ \chi_{\epsilon } \tilde{W_1 } \in \ker(A)$ and $ \lambda = 0$ is an eigenvalue of $A$. For  $ 0 < \mu \ll1$ small we let $ e_0 \in (-\infty, 0)$ be the ground state of $L_+$, i.e.  $ L_+ g = e_0 g$ for some $ g \in L^2_{\text{odd}}$.
Now we define the function 
\[
h(\lambda) : = \langle (L_+ - \lambda )^{-1} \chi_{\epsilon} \tilde{W_1}, \chi_{\epsilon} \tilde{W_1} \rangle_{L^2},~~ \lambda \in (e_0, \mu).
\]
\begin{Lemma} We can choose $\chi$ and take $ 0 < \epsilon \ll1 $ small enough such that there holds 
\begin{align*}
	& h(0 ) > 0,~~ h'(\lambda) = \langle (L_+ - \lambda)^{-2}\chi_{\epsilon} \tilde{W_1},\chi_{\epsilon} \tilde{W_1} \rangle > 0,~ \forall \lambda,~~ \lim_{\lambda \to e_0}h(\lambda) = - \infty.
\end{align*}
\end{Lemma}
For the proof we note $L_+ \tilde{W_1} = 0$ and  calculate
\begin{align*}
&(L_+ - \lambda)(\chi_{\epsilon}\tilde{W_1}) = (\mu - \lambda)\chi_{\epsilon}\tilde{W_1} - \epsilon^2 \chi''(\cdot \epsilon)\tilde{W_1}  - 2 \epsilon \chi'(\cdot \epsilon) \tilde{W_1}'\\
&(L_+ - \lambda)^2(\chi_{\epsilon}\tilde{W_1}) = (\mu - \lambda)^2 \chi_{\epsilon}\tilde{W_1} - \epsilon^2 (\mu - \lambda)  \chi''(\cdot \epsilon)\tilde{W_1} + \epsilon^4  (\chi''(\cdot \epsilon))^2\tilde{W_1}\\
&~~~~~~~~~~~~~~~~~~~~~~~~~ + 2 \epsilon^3  \chi''(\cdot \epsilon)  \chi'(\cdot \epsilon)\tilde{W_1}' - 2 \epsilon (\mu - \lambda)  )  \chi'(\cdot \epsilon) \tilde{W_1}'  + 10 \epsilon W^4   \chi'(\cdot \epsilon)\tilde{W_1}'\\
&~~~~~~~~~~~~~~~~~~~~~~~~~ + 2 \epsilon  ( \chi'(\cdot \epsilon)\tilde{W_1}' )''.
\end{align*}
Solving these equations with appropriate boundary condition at $ x = 0$ we infer
\begin{align*}
h(0) = \langle L_+^{-1}(  \chi_{\epsilon}\tilde{W_1} ), \chi_{\epsilon}\tilde{W_1} \rangle = & \frac{1}{\mu} \|  \chi_{\epsilon}\tilde{W_1}  \|_{L^2}^2 + \frac{\epsilon^2}{\mu} \langle L_+^{-1}(  \chi''(\cdot \epsilon)\tilde{W_1} ), \chi_{\epsilon}\tilde{W_1} \rangle\\
& + \frac{2 \epsilon}{\mu}   \langle L_+^{-1}(  \chi'(\cdot \epsilon)\tilde{W_1}' ), \chi_{\epsilon}\tilde{W_1} \rangle,
\end{align*}
where the latter two terms are shown to be bounded as $ \epsilon \to 0^+$.
Moreover we similarly have
\begin{align*}
h' (\lambda) = \frac{1}{(\mu - \lambda)^2} \big(\|  \chi_{\epsilon}\tilde{W_1}  \|_{L^2}^2 + O(1) \langle \mu - \lambda \rangle \big).
\end{align*}
For the last property in the Lemma we note 
\[
\langle g,  \chi(\cdot \epsilon)\tilde{W_1}\rangle = \frac{1}{e_0 - \mu} \big(\epsilon^2 \langle g,  \chi''(\cdot \epsilon)\tilde{W_1} \rangle  + 2 \epsilon \langle g, \chi'(\cdot \epsilon) \tilde{W_1}'\rangle \big),
\]
which we want to be non-vanishing by choice of $\chi$. Thus, since then
\[
\chi_{\epsilon}\tilde{W_1} =  P_{g^{\perp}} \chi_{\epsilon}\tilde{W_1} + \frac{\langle g,  \chi_{\epsilon}\tilde{W_1}\rangle g}{\|  g \|_{L^2}^2} 
\]
gives a nontrivial contribution, we are done  since $L_+g = e_0g$ and $h$ is strictly increasing.\\[2pt]
Now, especially, by the above Lemma $h$ has a unique zero $\lambda_1 \in (e_0, 0)$. We set $ \tilde{\eta} := (L_+ - \lambda)^{-1}(\chi_{\epsilon}\tilde{W_1})$, then clearly $ \lambda_1$ is a simple eigenvalue for $ A$, i.e.
\[
A \tilde{\eta}  = P L_+P \tilde{\eta} = \lambda_1 \tilde{\eta},~~ \langle \tilde{\eta},\chi_{\epsilon}\tilde{W_1}\rangle = 0.
\]
Moreover any $f \in \ker(A)$ independent of $ \chi_{\epsilon } \tilde{W_1} $ satisfies $	P L_+ Pf = 0 $ and thus $ 	 L_+ Pf = \tilde{\lambda} \chi_{\epsilon} \tilde{W_1} $. Solving this with appropriate boundary values at $ x = 0$ gives 
\begin{align*}
0 =  \langle Pf, \chi_{\epsilon}\tilde{W_1} \rangle = \tilde{\lambda}\langle L_+^{-1}(\chi_{\epsilon}\tilde{W_1}), \chi_{\epsilon}\tilde{W_1} \rangle = \tilde{\lambda} h(0),
\end{align*}
whence $ Pf \in \ker(L_+)$ and thus  $ \tilde{L}_+Pf = (L_+ - \mu) Pf = - \mu Pf$. However $\lambda = - \mu$ with $ 0 < \mu \ll1$ can not be an eigenvalue below zero and above the ground state of $\tilde{L}_+$. Therefore $A$  has exactly two simple eigenvalues  $\lambda_1$ and $0$ in $(-\infty, \mu)$. Similar to \cite{K-S-stable}, we check that $\{ \psi, \tilde{\eta}, \chi_{\epsilon}\tilde{W_1} \}$ are independent. Take
\[
c_1 \psi + c_2 \tilde{\eta} + c_3 \chi_{\epsilon}\tilde{W_1} = 0.
\]
Applying $L_+$ and multiplying by $ \overline{\psi}$ and $\overline{ \tilde{\eta}}$ on the right, we infer the system
\begin{align}\label{syyyyw}
\begin{pmatrix}
	e \langle L_-^{-1} \psi, \psi\rangle  & \lambda_1 \langle \tilde{\eta}, \psi \rangle \\
	\lambda_1 \langle  \psi,  \tilde{\eta} \rangle & \lambda_1 \langle \tilde{\eta}, \tilde{\eta}\rangle 
\end{pmatrix} \cdot \begin{pmatrix}
	c_1\\ c_2
\end{pmatrix} = - c_3  \epsilon \begin{pmatrix} \epsilon \langle  \chi''(\cdot \epsilon)\tilde{W_1}  + 2  \chi'(\cdot \epsilon) \tilde{W_1}', \psi \rangle\\
	\epsilon \langle  \chi''(\cdot \epsilon)\tilde{W_1}  + 2  \chi'(\cdot \epsilon) \tilde{W_1}', \tilde{\eta}\rangle
\end{pmatrix}.
\end{align}
We check $\det(\cdot)< 0$ and  there are $a = O(1), b= O(1)$ as $\epsilon \to 0^+$ with
$
\begin{psmallmatrix}
c_1\\c_2
\end{psmallmatrix} = c_3 \epsilon \begin{psmallmatrix}
a\\b
\end{psmallmatrix}.
$ Thus 
\[
c_3( a \epsilon \psi + b \epsilon \tilde{\eta} + \chi_{\epsilon}\tilde{W_1}) = 0.
\]
Now suppose $ c_3 \neq 0$, then multiplying by $ \chi_{\epsilon}\tilde{W_1}$ gives 
\begin{align} \label{we} a \epsilon  \langle \psi, \chi_{\epsilon}\tilde{W_1} \rangle = - \|  \chi_{\epsilon}\tilde{W_1}\|_{L^2}^2. 
\end{align} Then we may see 
\begin{align*} e \langle \psi, \chi_{\epsilon}\tilde{W_1} \rangle =& \langle L_+\psi, L_-(\chi_{\epsilon}\tilde{W_1}) \rangle\\
=& \mu \langle L_+\psi, \chi_{\epsilon}\tilde{W_1} \rangle -  \langle L_+\psi, (\chi_{\epsilon}\tilde{W_1})''\rangle  -  \langle L_+\psi, W^4\chi_{\epsilon}\tilde{W_1} \rangle. \\
\mu \langle L_+\psi, \chi_{\epsilon}\tilde{W_1} \rangle =& \mu \langle \psi, L_+( \chi_{\epsilon}\tilde{W_1}) \rangle.
\end{align*}
Calculating these terms, we infer the left side of \eqref{we} to be well controlled as $ \epsilon \to 0^+$ whereas the right side grows. Hence $ c_3 = 0$, and by \eqref{syyyyw} also $c_1 = c_2 = 0$.
As in \cite{K-S-stable}, we check (we omit further details) 
\[
\sup_{\| f\|_{L^2} =1,~ f \in \big \langle \psi, \tilde{\eta}, \chi_{\epsilon}\tilde{W}_1 \big \rangle  } \langle Af, f\rangle < \mu,
\]
which, by the min-max method, implies the existence of three eigenvalues of $A$  below $\mu$.\\[2pt]
Now concerning the resonance, we assume the existence by contradiction and infer
\[
L_-L_+\psi = \mu^2 \psi,~~ \psi  \in L^{\infty} \backslash L^2,~~ \psi(x) = C_{\pm} + O(\langle x\rangle^{-2}),~~x \to \pm \infty,
\]
with differentiable asymptotics according to Section \ref{subsec: Scat}.
Let $ 0 < \epsilon_0 \ll1$ as required above and $ \tilde{W}_0 : = \chi_{\epsilon_0}\tilde{W}_1$. For $ 0 < \epsilon \ll1 $ we set
\[
\psi^{\epsilon} = \psi \chi_{\epsilon} - \frac{\langle \psi \chi_{\epsilon}, \tilde{W}_0 \rangle}{\|  \tilde{W_0} \|_{L^2}^2} \tilde{W}_0,
\]
and hence $\langle \psi^{\epsilon},\tilde{W}_0 \rangle = 0$. For all $\delta > 0$ we may choose $ 0 < \epsilon_0 \ll1$ small depending only on $ \tilde{W}_1$.  Then let $ \epsilon = \epsilon(\epsilon_0) \ll1 $ such that
\[
\tilde{h}(\epsilon, \epsilon_0): = \frac{\langle \psi \chi_{\epsilon}, \tilde{W}_0 \rangle}{\|  \tilde{W}_0 \|_{L^2}^2} < \delta.
\]
This is seen as above when replacing $ \psi$ by $\frac{1}{\mu} L_-L_+\psi$.  We now calculate
\begin{align*}
\langle L_+ \psi^{\epsilon}, \psi^{\epsilon} \rangle &= \mu \|  \psi^{\epsilon} \|_{L^2}^2  + \langle (L_+ - \mu) \psi^{\epsilon}, \psi^{\epsilon} \rangle\\
&= \mu \|  \psi^{\epsilon} \|_{L^2}^2 + \|   \partial_x \psi^{\epsilon}  \|_{L^2}^2 - 5 \int \chi_{\epsilon}(x) W^4(x) |\psi^{\epsilon}(x)|^2~dx\\
&=  \mu \|  \psi^{\epsilon} \|_{L^2}^2 + \|   \partial_x \psi \|_{L^2}^2 + o(1)  - 5 \int W^4(x) |\psi^{\epsilon}(x)|^2~dx + 5 \int (1 -\chi_{\epsilon} (x))W^4(x) |\psi^{\epsilon}(x)|^2~dx\\
&= \mu \|  \psi^{\epsilon} \|_{L^2}^2 + \|   \partial_x \psi \|_{L^2}^2  - 5 \int W^4(x) |\psi(x)|^2~dx + o(1). 
\end{align*}
Following once more integration by parts  with $\psi$ instead of $\psi^{\epsilon}$, we infer
\[
\langle L_+ \psi^{\epsilon}, \psi^{\epsilon} \rangle = \mu \|  \psi^{\epsilon} \|_{L^2}^2  + \langle (L_+ - \mu) \psi, \psi \rangle + o(1),~\epsilon \to 0^+.
\]
Define $ f : = (L_+ - \mu )\psi$. Then $ f + \psi  =  - (L_- - \mu)^{-1}f$ and 
\[
\langle f + \psi, f \rangle = \langle (L_- - \mu)^{-1}f, f\rangle  = \langle (L_- - \mu)^{-1}(L_+ - \mu)\psi, (L_+ - \mu)\psi \rangle.
\]
We calculate $ (L_- - \mu)^{-1}(L_+ - \mu) \psi = - \frac{1}{\mu}(L_+ - \mu)\psi + \psi$ and hence
\begin{align} \label{die-in}
\langle f + \psi, f \rangle  = - \frac{1}{\mu} \| (L_+ - \mu) \psi \|_{L^2}^2 + \langle \psi, (L_+ - \mu)\psi \rangle < 0,
\end{align}
since $(L_+ - \mu)\psi = \psi''  - 5W^4\psi = O(x^{-4})$ and we choose $ 0 < \mu \ll1$. We now check as above that the set
$
\big\{ \psi^{\epsilon}, \tilde{\eta}, \tilde{W}_0\big\}
$ is linearly independent if we choose $ 0 < \epsilon  \ll1 $ small. Therefore, we need to use $\langle L_+ \psi^{\epsilon}, \psi^{\epsilon} \rangle \to \infty $ as $ \epsilon \to 0^+$ in order to obtain $\det(\cdot) < 0$.  Then the above argument works immediately as above without re-choosing $\epsilon_0 > 0$  since $\psi^{\epsilon} \perp \tilde{W}_0$. Similarly we need to see for perhaps even smaller  $ 0 < \epsilon  \ll1 $  there holds
\[
\sup_{\| f \|_{L^2}=1,~ f \in \big< \psi^{\epsilon}, \tilde{\eta}, \tilde{W}_0  \big>} \langle Af, f \rangle < \mu,
\]
where $P = P_{\tilde{W}_0^{\perp}}$ and $ A = P L_+ P$ as above. Then $A$ has three eigenvalues in $( - \infty, \mu)$ which is wrong.\\[2pt]
Now let $f $ in $L^2(\R)$ be such a linear combination and since $ \psi^{\epsilon}, \tilde{\eta} \perp \tilde{W}_0$ we may assume $ f \perp \tilde{W}_0$, so $f = c_1 \psi^{\epsilon } + c_2 \tilde{\eta}$. Then 
\begin{align}
&\frac{\langle L_{+}(c_1 \psi^{\epsilon } + c_2 \tilde{\eta} ), c_1 \psi^{\epsilon } + c_2 \tilde{\eta} \rangle}{\| c_1 \psi^{\epsilon } + c_2 \tilde{\eta} \|_{L^2}^2}\\
&= \frac{|c_1|^2(\mu \|  \psi^{\epsilon}\|_{L^2}^2 +\langle (L_+ - \mu) \psi, \psi \rangle + o(1)) + 2 \lambda_1 \Re(\langle c_1 \psi^{\epsilon}, c_2 \tilde{\eta}\rangle ) +  \lambda_1 \|c_2 \tilde{\eta}\|_{L^2}^2 }{|c_1|^2 \|\psi^{\epsilon}\|_{L^2}^2 + 2 \Re(\langle c_1 \psi^{\epsilon}, c_2 \tilde{\eta}\rangle ) + \| c_2 \tilde{\eta}\|_{L^2}^2}\\
& \leq \max_{x \in \C^2}  \frac{|x_1|^2(\mu + \delta^2 \langle (L_+ - \mu) \psi, \psi \rangle + o(\delta^2)) + 2 \lambda_1 \delta \Re(\langle x_1 \psi^{\epsilon}, x_2 e \rangle ) +  \lambda_1|x_2|^2 }{|x_1|^2  + 2 \delta \Re(\langle x_1 \psi^{\epsilon}, x_2 \tilde{\eta}\rangle ) + |x_2|^2}
\end{align}
where $ \delta^2 : =  \|  \psi^{\epsilon}\|_{L^2}^{-2}$ and $e : = \frac{\tilde{\eta}}{\| \tilde{\eta} \|_{L^2}}$. We set $b^{\epsilon}:= \langle \psi^{\epsilon}, e \rangle \to \langle \psi, e \rangle =: b $ as $\epsilon \to 0^+$ by sufficient decay of  $e$. Further we identify the two Hermitian $2 \times 2$ matrices $B^{\epsilon},C^{\epsilon}$ via 
\begin{align*}
&C_{11}^{\epsilon} : = \mu + \delta^2\langle (L_+ - \mu) \psi, \psi \rangle + o(\delta^2),\\ &C^{\epsilon}_{12} = C^{\epsilon}_{2 1} := \lambda_1 \delta b^{\epsilon},~~ C_{22}^{\epsilon} : = \lambda_1\\[2pt]
&B^{\epsilon}_{12} = B^{\epsilon}_{2 1}  : = \delta b^{\epsilon},\\
&B^{\epsilon}_{11} = B^{\epsilon}_{2 2}  : = 0.
\end{align*}
Then we obtain for the maximum above and by definition of $B^{\epsilon},C^{\epsilon}$
\begin{align*}
\max_{f} \frac{\langle Af, f \rangle}{\langle f , f \rangle} \leq \max_{x \in \C^2} \frac{\langle C^{\epsilon}x, x \rangle }{\langle (I + B^{\epsilon})x,x\rangle} = \max_{x \in \C^2} \frac{\langle (I + B^{\epsilon})^{- \f12}C^{\epsilon}(I + B^{\epsilon})^{- \f12}x, x \rangle }{\langle x,x\rangle},
\end{align*}
for which we use the following expansion
\[
(I + B^{\epsilon})^{- \f12}C^{\epsilon}(I + B^{\epsilon})^{- \f12}  = C - \f12(CB + BC) + \frac{3}{8}(CB^2 + B^2 C) + \f14 BCB + O(\delta^3).
\]
In particular, by direct calculation  the right side has the form (dropping $\epsilon$)
\begin{align*}
&\begin{pmatrix}
	(1 + \f34(b \delta)^2) C_{11} - \f34 \lambda_1 (b \delta)^2 & \lambda_1 b \delta(\f12 + (b \delta)^2) - \f12 (b \delta) C_{11}\\
	\lambda_1 b \delta(\f12 + (b \delta)^2) - \f12 (b \delta) C_{11} & \lambda(1 - \f14 (b \delta)^2) + (b \delta)^2 C_{11}
\end{pmatrix} + O(\delta^3)\\
&~~~~~~~~~~~\hspace{3cm}~~ = \begin{pmatrix}
	\mu + \delta^2 M & \delta b( \frac{\lambda_1}{2} - \frac{\mu}{2})\\[2pt]
	\delta b( \frac{\lambda_1}{2} - \frac{\mu}{2}) & \lambda_1 + (b\delta)^2 ( \f14\mu - \f14 \lambda_1)
\end{pmatrix}  + o(\delta^2) =: D,
\end{align*}
where $ M = \langle (L_+ - \mu) \psi, \psi \rangle + \f34 b^2( \mu - \lambda_1) $. Here, the right side $D$ has simple  eigenvalues $\mu, \lambda_1$ at $ \delta = 0$.  In fact, we may observe this setting $b_1 = 0,~ b_2 = b$ in  the proof in \cite{K-S-stable}. Now for $ 0 < \delta \ll1 $ the maximal eigenvalue $\lambda \sim \mu + x$ where $ 0 < x \ll1 $ is small.  By the calculation of
\[
\det(D - (1+x) I_{2 \times 2}) = 0,
\]
we have  (we only estimate the first term $ < 0$)  since $\lambda_1 < 0$
\begin{align*}
(\mu - \lambda_1)x =~& - x^2 (b \delta)^2 \f14(\mu - \lambda_1) + x(b \delta)^2 \f14 (\mu - \lambda_1)(\delta^2 M + o(\delta^2)) \\
&~~~~+ (\mu - \lambda_1) (\delta^2M + o(\delta^2)) + \big(\delta b \f12 (\mu - \lambda_1) + o(\delta^2) \big)^2\\
& \leq  \delta^2 [ \langle (L_+ - \mu) \psi, \psi \rangle( \mu - \lambda_1) + (\mu - \lambda_1)^2 b^2] + o(\delta^2)\\
&=   \delta^2 [ \langle (L_+ - \mu) \psi, \psi \rangle( \mu - \lambda_1) + (\mu - \lambda_1)^2 (b^0)^2] + o(\delta^2)
\end{align*}
Finally we use $ b^0 = - (\mu - \lambda_1)^{-1} \langle f, e \rangle$ and hence  by definition of $ f$
\begin{align*}
(\mu - \lambda_1)x  &\leq  \delta^2( \mu - \lambda_1) [ \langle (L_+ - \mu) \psi, \psi \rangle + \langle f, e \rangle^2 ] + o(\delta^2)\\
& \leq   \delta^2( \mu - \lambda_1) [ \langle (L_+ - \mu) \psi, \psi \rangle + \langle f, f \rangle ] + o(\delta^2)\\
& \leq  \delta^2( \mu - \lambda_1) \langle f + \psi, f \rangle  + o(\delta).
\end{align*}
We know that this upper bound is $ < 0$ if $ 0 < \delta \ll1 $ is small, which it is for $ 0 < \epsilon \ll1 $ small.
\end{proof}

\vspace{1cm}

\appendix

\section{Asymptotics for singular ODE} \label{sec:SingODE}
We start with a differential equation on the complex plane with a singularity at $ z = \infty$, i.e. we consider
\begin{align}\label{ode}
	&\partial_z^2 w + b(z) \partial_z w + c(z) w = 0,~~ z \in \C,
\end{align}
where $ b,c$ are non-constant analytic functions such that
\begin{align}
	\label{exp-now}
	&b(z) = \sum_{k \geq 0} b_k z^{-k},~~c(z) = \sum_{k \geq 0} c_k z^{-k},~~~ |z| > R,
\end{align}
holds in an absolute sense for some $ R > 0$. Substitution and using \eqref{exp-now}  leads to the formal ansatz
\begin{align}\label{as-here}
	w(z) = e^{\lambda \cdot z} z^{\gamma} ( a_0 + a_1 z^{-1} + a_2 z^{-2} + \dots),~~ \lambda, \gamma \in \C,
\end{align}
and  hence we infer, evaluating partial sums,
\begin{align}\label{first-here}
	& \lambda^2 + b_0 \lambda + c_0 = 0\\[2pt]\label{second-here}
	& (b_0 + 2 \lambda) \gamma  + b_1 \lambda + c_1 = 0\\[2pt] \label{third-here}
	& ( b_0 + 2 \lambda) k a_k = (k - \gamma)(k -1 - \gamma) a_{k-1} + (\lambda b_2 + c_2 - (k -1 - \gamma)b_1) a_{k-1}\\ \nonumber
	&\hspace{2.4cm} + ( \lambda b_3 + c_3 - (k-2- \gamma)b_2) a_{k-2} + \dots + (\lambda b_{k+1 } + c_{k+1 } + \gamma c_k)a_0.
\end{align}
By  \eqref{first-here}  we have $ \lambda_{\pm} = - \f12 b_0 \pm (\f14 b_0^2 - c_0)^{\f12}$, where $ b_0^2 =4 c_0$ is excluded as an exceptional case (see \cite[chapter 7.1.3]{Olver} for a discussion) and $ \gamma_{\pm},~ a^{\pm}_k$ are defined through \eqref{second-here},~ \eqref{third-here}.\\[5pt]
Clearly from the first term on the right of \eqref{third-here}, we might expect $(a_k)$ to grow too fast to make sense of \eqref{as-here}. 
The following is a Lemma stated in \cite[Theorem 2.1]{Olver}. 
\begin{Lemma} \label{Olvers-Lemma} Let $b(z),c(z)$ be as above for some $ R > 0$ and $ b_0^2 \neq  4 c_0$. 
	Then \eqref{ode} has holomorphic solutions $ w_{\pm}(z)$ on the intersection of $\{z ~|~ |z| > R\}$ with the sector
	\[
	|\arg(\pm(\lambda_+ - \lambda_-) z)| \leq \pi.
	\]
	Further in the above region, for all $ n \in \N_0$ 
	there exists a representation
	\begin{align}\label{representation}
		&w_{\pm}(z) = e^{\lambda_{\pm}z} z^{\gamma_{\pm}} \bigg(\sum_{k = 0}^n a^{\pm}_k z^{-k} + R_n(z) \bigg),~~~ |z| \gg1,\\[2pt] \nonumber
		&\partial^l_zR_n(z)  = O( z^{-n-1 - l}),~~ l \in \N_{0}.
	\end{align}
\end{Lemma}
For a proof, we refer to \cite[chapter 7.2.1]{Olver}. The asymptotics for differentials of $R_n(z)$ is not stated explicitly, however follows directly from the proof in \cite[ Theorem 2.1]{Olver}.
\begin{Rem}
	In case $ b(\mu, z), c(\mu,z) $ depend on a parameter $ \mu \in U \subset \C$ in an open subset, we have $w_{\pm}(\mu, z)$  as in Lemma \ref{Olvers-Lemma} depends analytic on $\mu \in U$ if 
	\begin{itemize}
		\item[(i)] $b_0,c_0$ are independent of $\mu$. Further $a_0^{\pm}(\mu)$ and $ b_k(\mu),~c_k(\mu)$ for $ k \geq 1$ are holomorphic in $\mu \in U$.
		\item[(ii)] For $ K \subset U$ compact we have that $\sum_{k} \| b_k\|_{L^{\infty}(\mu,K)} < \infty$ and $\sum_{k} \| c_k\|_{L^{\infty}(\mu,K)} <  \infty$.
	\end{itemize}
	See \cite[Theorem 3.1]{Olver} for a statement and a discussion of the proof.
\end{Rem}


\vspace{3cm}

\begin{center}
\textbf{\large Acknowledgments}
\end{center}
~\\
The author likes to express his sincere gratitude to Joachim Krieger for the many helpful comments and discussions during the making of this manuscript. He further likes to thank him for suggesting the argument in the proof of Proposition \ref{No-embedded-for-dist-op} and Proposition \ref{Thres-is-no-res-dist-op}.

\vspace{1cm}

\bibliographystyle{alpha}

\end{document}